\documentclass{report}

\usepackage[english]{babel}

\usepackage[a4paper,top=2cm,bottom=2cm,left=3cm,right=3cm,marginparwidth=1.75cm]{geometry}

\usepackage{graphicx}
\usepackage{hyperref}
\hypersetup{hidelinks}
\usepackage{setspace}

\usepackage{bibentry}   
\usepackage{bbm}
\usepackage{amsmath}    
\usepackage{amssymb}    
\usepackage{amsthm}
\usepackage{mathtools}  
\usepackage{microtype}  
\usepackage{multirow}   
\usepackage{booktabs}   
\usepackage{subcaption} 
\usepackage[ruled,linesnumbered,algochapter]{algorithm2e} 
\usepackage[dvipsnames,table]{xcolor} 
\usepackage{nag}        
\usepackage{todonotes}  
\usepackage{verbatim}
\usepackage{enumerate}
\usepackage{graphicx}   
\usepackage{wrapfig}

\usepackage{pst-node}
\usepackage[explicit]{titlesec}
\usepackage{etoolbox}
\usepackage{tikz-cd} 
\usetikzlibrary{tikzmark}
\usepackage{defs}

\usepackage{morewrites} 

\newcommand\HH{\mathrm{H}}

\title{T-dual branes on hyperk\"ahler manifolds}
\author{M\'aria Anna Sis\'ak}
\date{October, 2024}

\begin{document}

\begin{titlepage}
\begin{spacing}{0.5}
    \vspace*{\fill}
\end{spacing}
\begin{spacing}{1.5}
\vspace*{\fill}
\end{spacing}
\begin{spacing}{1.5}
\vspace{4cm}
\centering \textbf{\huge T-dual branes on hyperk\"ahler manifolds}\\[0.5cm]
\textbf{by}\\
\textbf{\large M\'aria Anna Sis\'ak}\\
\textbf{October, 2024}
\vspace*{\fill}
\end{spacing}
\vspace{6cm}
\begin{spacing}{1.15}
\centering
\textnormal{\large A thesis submitted to the\\Graduate School\\of the\\
Institute of Science and Technology Austria\\
in partial fulfillment of the requirements\\
for the degree of\\ Doctor of Philosophy}\\
\vspace*{\fill}
\textnormal{\large Supervisor:\\
Tam\'as Hausel}

\vspace*{\fill}

\end{spacing}
\begin{center}
\includegraphics[width=0.5\textwidth]{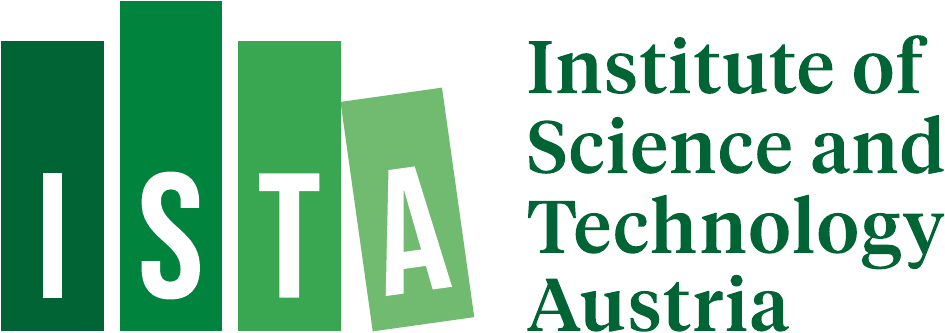}
\end{center}
\end{titlepage}

\begin{abstract}
In \cite{KW} Kapustin and Witten conjectured that there is a mirror symmetry relation between the hyperk\"ahler structures on certain Higgs bundle moduli spaces. As a consequence, they conjecture an equivalence between categories of $BBB$ and $BAA$-branes. At the classical level, this mirror symmetry is given by T-duality between semi-flat hyperk\"ahler structures on algebraic integrable systems.

In this thesis, we investigate the T-duality relation between hyperk\"ahler structures and the corresponding branes on affine torus bundles. We use the techniques of generalized geometry to show that semi-flat hyperk\"ahler structures are T-dual on algebraic integrable systems. We also describe T-duality for generalized branes. Motivated by Fourier-Mukai transform we upgrade the T-duality between generalized branes to T-duality of submanifolds endowed with $U(1)$-bundles and connections. This T-duality in the appropriate context specializes to T-duality between $BBB$ and $BAA$-branes.
\end{abstract}

\renewcommand{\contentsname}{Table of Contents} 
\tableofcontents




\pagestyle{plain} 

\chapter{Introduction}
Mirror symmetry is a conjectured mathematical duality between Calabi-Yau manifolds originating in string theory. It states that certain physical theories on mirror manifolds are isomorphic, in particular, there exists an equivalence between their categories of boundary conditions (branes). One of the main approaches to mirror symmetry is due to Str\"ominger, Yau and Zaslow \cite{SYZ} who proposed that on the "classical level" mirror symmetry is given by T-duality.

Mathematically T-duality describes a relationship between torus bundles fibered over the same base, roughly exchanging the radius of each circle in the fibers. A mathematical definition was first given by Bouwknegt, Evslin and Mathai \cite{BEM}. In physics, T-duality exchanges certain topological twists of supersymmetric sigma models on the T-dual manifolds, the $A$ and $B$ twists \cite{wittenTft}. The $A$ twist is defined using a symplectic structure while a $B$ twist is related to a complex structure, with corresponding boundary conditions called $A$ and $B$-branes.

Initially, the category of $A$-branes on a symplectic manifold was thought to be the category of Lagrangian submanifolds endowed with local systems, the Fukaya category. However, Kapustin and Orlov \cite{kapustinOrlov}, as well as Gualtieri \cite{MGpaper}, independently showed that there exist $A$-branes supported on coisotropic submanifolds endowed with Hermitian line bundles and non-flat connections. The inclusion of such branes into the Fukaya category remains an open problem. In \cite{GW} Gaiotto and Witten gave a conjecture to the space of homomorphisms between a rank one Lagrangian and a space-filling $A$-brane, a method extended by Bischoff and Gualtieri \cite{BG} to define morphisms between generalized branes. Nonetheless, examples of coisotropic branes are rare, especially those that are neither space-filling nor Lagrangian. Additionally, the theory of higher-rank coisotropic branes is still underdeveloped although a definition was given by Herbst in \cite{herbst}.

The category of $B$-branes on a complex manifold is much better understood. It is the derived category of coherent sheaves. On the zeroth level $B$-branes are complex submanifolds endowed with holomorphic vector bundles, but one can argue that "stacking" branes and other physical operations recover the entire derived category.

In the context of \cite{SYZ}, T-duality maps $A$-branes on a symplectic torus fibration with Lagrangian fibers to $B$-branes on a complex torus fibration with real fibers. The first mathematical formulation of this duality is due to Arinkin and Polishchuk \cite{AP} who transformed $A$-branes supported on Lagrangian sections of a torus fibration. Bruzzo, Marelli and Pioli \cite{BMP2} extended this result to Lagrangian submanifolds which intersect the fibers in subtori of arbitrary dimension. They started by describing a Fourier-Mukai type transform for rank one local systems supported on affine subtori of a single real torus \cite{BMP1}.

Glazebook, Jardim and Kamber \cite{GJK} developed a T-duality transformation for Hermitian vector bundles with connections on torus bundles which are flat on the fibers. Here the bundles are not assumed to be supported on Lagrangian submanifolds and it is not assumed that the connection is flat. Their work and \cite{BMP1} indicate that a T-duality transformation should be defined for Hermitian bundles with connections rather than just for $A$-branes. This approach was further developed in \cite{CLZ} by Chan, Leung and Zhang who defined T-duality for projectively flat Hermitian vector bundles supported on submanifolds which intersect the fibers in subtori, even for bundles not flat on the torus fibers. 

The motivation for this work came from the seminal paper of Kapustin and Witten \cite{KW}. The Higgs bundle moduli space with gauge group $G$ is mirror, in the sense of Strominger-Yau-Zaslow \cite{SYZ}, to the Higgs bundle moduli space with the Langlands dual gauge group $G^L$. Kapustin and Witten utilised this duality to reformulate the geometric Langlands program to the language of S-duality between gauge theories. In this process, they defined branes which are boundary conditions in a triple of sigma models corresponding to complex and symplectic structures coming from the hyperk\"ahler structure on the moduli spaces. They called these branes $BBB$ and $BAA$-branes and they conjectured that such branes should map to each other under mirror symmetry.

The hyperk\"ahler structure on the Higgs moduli space is constructed via infinite dimensional hyperk\"ahler quotient. Meanwhile, on certain open subsets of the moduli space, there is a second, simpler hyperk\"ahler structure called the semi-flat hyperk\"ahler structure \cite{freed}. Gaiotto, Moore and Nietzke \cite{GMN1, GMN2} have conjectured that one can understand the original hyperk\"ahler metric on the Higgs moduli space by adding instanton corrections to the semi-flat metric. Similarly, they proposed building mirror symmetry between $BBB$ and $BAA$-branes by starting with $BBB$ and $BAA$-branes of the semi-flat hyper\"ahler structure, where mirror symmetry is T-duality, and adding corrections. 

While T-duality between $A$ and $B$-branes is complicated, in the setting of \cite{KW} T-duality is also supposed to map certain $B$-branes to $B$-branes.  This $B$ to $B$ duality is supposed to be realized by fiberwise Fourier-Mukai transform \cite{mukai}. Although this transform remains poorly understood on the very singular fibers, Arinkin and Fedorov \cite{AF} have extended it up to mildly singular fibers and showed the derived equivalence between large subsets of the moduli spaces. 

This thesis focuses on torus bundles with smooth fibers, where Fourier-Mukai transform is well-defined. Therefore, if a brane is $BBB$ or $BAA$ one can take its T-dual as a $B$-brane and use the result as a motivation. As also hinted by the work of Glazebook, Jardim and Kamber \cite{GJK} and Chan, Leung and Zhang \cite{CLZ} a Fourier-type transform should exist for Hermitian vector bundles with connections. This unified transform should restrict to Fourier-Mukai transform on certain holomorphic objects and to $A$ to $B$ duality on the appropriate $A$-branes.

One of the most successful formalisms unifying $A$ and $B$-theory is generalized geometry introduced by Hitchin in the early 2000s. His theory has since been expanded by many of his students.  In \cite{MGthesis} Gualtieri introduced generalized complex structures which encompass complex and symplectic structures. Moreover, in \cite{MGpaper} he showed that branes have a natural interpretation in this context as well. In \cite{cavalcantiTdual} Cavalcanti and Gualtieri have also reformulated T-duality to the language of generalized geometry. This theory therefore provides the perfect toolkit for a unified understanding of duality between branes.

The aim of this project was to explore the mirror symmetry relationship between $BBB$ and $BAA$-branes in hyperk\"ahler manifolds with applications to the Higgs moduli space. In the end, we restrict our attention to the semi-flat case where mirror symmetry is given by T-duality. We give some insight into the structure of coisotropic $BAA$-branes and explain T-duality for semi-flat hyperk\"ahler structures based on T-duality in generalized geometry. 

We have two sets of results. First, we describe in detail T-duality of generalized branes on affine torus bundles with torsion Chern classes. This is an application of T-duality in generalized geometry but has not been fully worked out yet. We first define a class of generalized branes which we call "locally T-dualizable". Then, we study the conditions under which such branes admit T-duals in this class. 

Let $M\ra B$ and $\hat{M}\ra B$ be a T-dual pair of affine torus bundles with torsion Chern class. We call the fiber product of $M$ and $\hat{M}$ over the base the correspondence space.
\begin{equation}\label{intro1}
\begin{tikzcd}
   &  M\times_B \hat{M} \arrow{dr}{\hat{p}} \arrow{dl}[swap]{p} & \\
    M \arrow{dr}[swap]{\pi} & & \hat{M} \arrow{dl}{\hat{\pi}}\\
    & B & 
\end{tikzcd}
\end{equation}
\begin{defin}[Definition \ref{locally T-dualizable brane}]
A \emph{locally T-dualizable brane} in $M$ is a pair $\cL=(S,F)$, where $S\subset M$ is an affine torus subbundle and $F\in \Omega^2(S)$ is a closed invariant two-form on $S$ which represents a rational cohomology class on the fibers of $S$.
\end{defin} 
We show that when the base of $S$ is contractible a locally T-dualizable brane admits an entire family of T-duals.
\begin{thm}[Theorem \ref{local gg thm}]
Let $\cL=(S,F)$ be a locally T-dualizable brane in $M$ with $\pi(S)$ simply connected. Then, there exists a foliation of $S\times_{\pi(S)}\hat{M}$ by affine torus subbundles denoted by $Z$. For each such $Z$ we denote by $\hat{S}_Z\subset \hat{M}$ the image of $Z$ under the projection $S\times_{\pi(S)}\hat{M}\ra \hat{M}$.

Then, $\hat{S}_Z$ is an affine torus subbundle of $\hat{M}$ over $\pi(S)$. Moreover, there exists a unique closed invariant two-form $\hat{F}_Z$ on $\hat{S}_Z$ which represents a rational cohomology class on the fibers of $\hat{S}_Z$ and satisfies
$$p_Z^*F+P|_Z=\hat{p}_Z^*\hat{F},$$
where $P\in \Omega^2(M\times_B\hat{M})$ is a closed invariant two-form and $p_Z:Z\ra S$ and $\hat{p}_Z:Z\ra \hat{S}_Z$ are the projections.

Whenever two leaves $Z_1$ and $Z_2$ of the foliation have the same image $\hat{S}_{Z_1}=\hat{S}_{Z_2}$ we have
$$\hat{F}_{Z_1}=\hat{F}_{Z_2}. $$
Moreover, then $(\hat{S}_Z,\hat{F}_Z)$ is a T-dual of the generalized brane $(S,F)$.
\end{thm}
To prove this theorem we first construct an integrable distribution on $S\times_{\pi(B)}\hat{M}$ using the two-forms $F\in \Omega^2(S)$ and $P\in \Omega^2(M\times_B\hat{M})$. Then, we show that this distribution has closed leaves, which we denote by $Z$, and the images of these leaves in $\hat{M}$ give the T-dual branes $\hat{S}$. That is, for a T-dual pair of generalized branes $(S,F)$ and $(\hat{S},\hat{F})$, we have the following diagram.
\begin{equation}\label{intro2}
\begin{tikzcd}
   &  (Z, \hat{p}^*\hat{F}-p^*F=P) \arrow{dr}{\hat{p}} \arrow{dl}[swap]{p} & \\
    (S,F) \arrow{dr}[swap]{\pi} & & (\hat{S},\hat{F}) \arrow{dl}{\hat{\pi}}\\
    & \pi(S) & 
\end{tikzcd}
\end{equation}
When the base of $S$ is contractible there is no obstruction to the existence of multiple T-duals. We finally study whether locally T-dualizable branes in non-trivial affine torus bundles admit global T-duals. We find that there is a topological constraint to the existence of such T-duals which is given as follows.
\begin{thm}[Theorem \ref{global thm1}]
    Let $M$ and $\hat{M}$ be a T-dual pair and $\cL=(S,F)$ a locally T-dualizable brane of $M$. Then, $\cL$ has a T-dual if and only if
$$q(c_{\hat{M}})\in \HH^2(\pi(S),\Gamma_S^\vee\cap H(V_S)).$$    
\end{thm}
Here, $c_{\hat{M}}$ is the Chern class of $\hat{M}$, $q$ is a map induced by a morphism of local systems and $\Gamma_S^\vee\cap H(V_S)$ is a sub local system of the dual of $\Gamma_S$ depending on $F$. 

The existence of an affine torus subbundle in a non-trivial affine torus bundle is obstructed by the Chern class and the monodromy. Theorem \ref{global thm1} is a direct consequence of this constraint

Our second set of results is the description of a T-duality transformation for "physical branes" which are $U(n)$-bundles with connections supported on affine torus subbundles. We first focus on rank one branes, that is $U(1)$-bundles. We build on our treatment of generalized branes, so we restrict to connections whose curvatures are invariant two-forms. These results are partially covered by \cite{CLZ} where T-duality is worked out for projectively flat $U(d)$-bundles supported on affine torus subbundles in a special class of affine torus bundles. Our method differs from the method of \cite{CLZ} and we expand to a more general class of torus bundles, specifically to ones without smooth sections. On the other hand, we only treat rank one bundles and a very specific case of higher rank bundles.

Our first theorem parallels Theorem \ref{local gg thm} in that we first assume that the base $\pi(S)$ of the submanifold $S$ is contractible. We upgrade the two-form $P\in \Omega^2(M\times_B\hat{M})$ from Theorem \ref{local gg thm} to a $U(1)$-bundle with connection $\cP$, the relative Poincar\'e bundle.
\begin{thm}[Theorem \ref{tdual U(1) bundles trivial base}]
Let $S\subset M$ be an affine torus subbundle such that $\pi(S)$ is contractible. Let $L\ra S$ be a $U(1)$-bundle with connection as such that the curvature $2\pi i F$ is invariant. Then, $\cL=(S,F)$ is a locally T-dualizable generalized brane and the following hold.
\begin{enumerate}
    \item\label{intro thm1} There exists a unique generalized T-dual $(\hat{S},\hat{F})$ of $(S,F)$ such that for any leaf $Z$ of the foliation on $S\times_{\pi(S)}\hat{M}$ projecting onto $\hat{S}$, the $U(1)$-bundle with connection
$$ \hat{L}_Z:= p_Z^*L\otimes \cP|_Z$$
is trivial when restricted to the fibers of $\hat{p}_Z:Z\ra \hat{S}.$ 
\item\label{intro thm2} For any leaf $Z$ as in 1., we define the  pushforward 
$$\hat{E}:=(\hat{p}_Z)_*\hat{L}_Z,$$
as the direct image of the sheaf of fiber-wise flat sections of $\hat{L}_Z$. Then, $\hat{E}$ is a projectively flat $U(d^2)$-bundle with connection independent of $Z$, and the curvature of the connection is given by $\hat{F} \cdot Id \in \Omega^2(\hat{S},\gu(d^2))$.
\item\label{intro thm3}  Finally, there exists a projectively flat $U(d)$-bundle $\hat{L}\ra \hat{S}$ such that
$$\hat{E} \cong \hat{L}\otimes U(d).$$
The curvature of the connection on $\hat{L}$ is  $\hat{F}\cdot Id\in \Omega^2(\hat{S},\gu(d))$.
\end{enumerate}
We say that $(\hat{L},\hat{S})$ is the T-dual of $(S,L)$.
\end{thm}
The proof relies on techniques borrowed from many sources. We represent the $U(1)$-bundles with connections by factors of automorphy. This description was used in \cite{BMP1,BMP2}, in particular, the Poincar\'e bundle is described there.

We follow the technique of T-dualizing branes given in \cite{AP} and \cite{BMP1,BMP2}, that is we pull back the $U(1)$-bundle $L$ to the correspondence space and tensor it with the Poincar\'e bundle. When $L$ is flat on the fibers of $S$, the resulting bundle is flat on the fibers of the projection $S\times_{\pi(S)}\hat{M}\ra \hat{M}$. Then one can take fiber-wise flat sections to define the pushforward.

When $L$ is not flat on the fibers, however, neither is its pullback to the correspondence space. Here we use an idea from \cite{GJK} to restrict to the leaves $Z$ defined by the underlying generalized branes (\ref{intro2}). The restriction is now flat on the fibers of $Z\ra \hat{S}$ and we can take flat sections to push forward as we do in Part \ref{intro thm2}.

It turns out, however, that the resulting image is not the genuine T-dual. Instead, it is the T-dual tensored with a trivial bundle (cf. Part \ref{intro thm3}.). We demonstrate that one can recover the genuine T-dual using techniques inspired by the Fourier-Mukai transform.

We generalize the local theorem to T-dual affine torus bundles with non-zero Chern classes. To do this we introduce gerbes and topological T-duality based on Baraglia's work \cite{B2}. In this setting, T-dualizable branes are given by submanifolds endowed with bundles twisted by a gerbe and we also re-define the Poincar\'e bundle as a twisted $U(1)$-bundle. We start with a pair of topologically T-dual affine torus bundles with torsion Chern classes $M$ and $\hat{M}$ endowed with gerbes $\cG$ and $\hat{\cG}$. Our final theorem is as follows. 
\begin{thm}[Theorem \ref{tdual u(1) bundle general base}]
Let $S\subset M$ be an affine torus subbundle and $L\ra S$ a $U(1)$-bundle with connection twisted by $\cG|_S$. Let $F\in \Omega^2(S)$ be the curvature of the connection on $L$. Then, $(S,F)$ is a locally T-dualizable generalized brane and the following holds.
\begin{enumerate}
    \item There exist global affine torus subbundles $Z\subset S\times_{\pi(S)}\hat{M}$, such that over any contractible open set $U\subset \pi(S)$ the submanifold $Z|_{U}$ is a leaf of the local foliation of $S\times_U\hat{M}$ corresponding to the locally T-dualizable brane $(S|_{U},F)$.
    
    \item There exists a submanifold $Z\subset S\times_{\pi(S)}\hat{M}$ as in 1., such that the twisted line bundle $p^*L\otimes \cP$ is trivial on the fibers of $\hat{p}_Z:Z\ra \hat{S}$. Moreover, over $Z$ we have an isomorphism between twisted $U(1)$-bundles
    $$p^*L\otimes \cP \cong \hat{L}_Z$$
    such that $\hat{L}_Z$ is twisted by the gerbe $\hat{p}_Z^*\hat{\cG}$.
    \item The T-dual $(\hat{S},\hat{L})$, locally constructed as in Theorem \ref{tdual U(1) bundles trivial base}, is a rank $d$ T-dualizable brane on $\hat{M}$.
\end{enumerate}
\end{thm}

In the final section of this thesis, we study the relationship between T-dual generalized branes and T-dual physical branes. Our final theorem concerns higher-rank T-dual branes with projectively flat connections. These T-dual pairs were treated in \cite{CLZ} for affine torus bundles that admit sections. We extend these results to torus bundles with torsion Chern classes and to branes twisted by gerbes.

\begin{thm}[Theorem \ref{last global}] Let $(S,F)$ and $(\hat{S},\hat{F})$ be a T-dual pair of generalized branes such that there exists a global submanifold $Z\subset S\times_{\pi(S)}\hat{S}$ as in (\ref{intro2}). Let $H$ be the fiberwise component of $F$, and suppose that in a local symplectic frame $\{\mu\}$ of $\Gamma_S$ it can be written as
    $$H=\sum_{i=1}^r \frac{n_i}{m_i}\mu_i^*\wedge \mu_{i+r}^*. $$
    Let $m=\prod_{i=1}^rm_i$ and $n=\prod_{i+1}^rn_i$.
 Then, there exist flat gerbes $\cG'$ and $\hat{\cG}'$ on $S$ and $\hat{S}$ such that the following holds.
    \begin{enumerate}
        \item Over any $Z\subset S\times_{\pi(S)}\hat{S}$ as in (\ref{intro2}) the twisted Poincar\'e bundle is a trivialization of $\hat{p}_Z^*\hat{\cG}'\otimes (p_Z^*\cG')^{-1}$. Moreover,
        $$\cG'\otimes (\cG|_S)^{-1}\cong\pi^*\cG_0\ \ \ \text{and}\ \ \ \hat{\cG}'\otimes (\hat{\cG}|_{\hat{S}})^{-1}\cong\hat{\pi}^*\cG_0 $$
        where $\cG_0$ is a gerbe on $\pi(S)$.
        \item For any $Z\subset S\times_{\pi(S)}\hat{S}$ as in (\ref{intro2}) there exists $U(1)$-bundles $L$ and $\hat{L}$ on $Z$ twisted by $p_Z^*\cG'$ and $\hat{p}_Z^*\hat{\cG}'$ such that we have an isomorphism of gerbe trivializations
    \begin{align}\label{intro tdual 1}\delta(\hat{L})\otimes \delta(L)^{-1} \cong \delta(\cP).\end{align}
        Moreover, $L$ is
        trivial on the fibers of $p_Z:Z\ra S$ and $\hat{L}$ is trivial on the fibers of $\hat{p}_Z:Z\ra \hat{S}$. 
        \item There exists a $\cG'$-module $E$ of rank $m$ on $S$ and a $\hat{\cG}'$-module $\hat{E}$ of rank $n$ on $\hat{S}$ such that
         \begin{align}\label{intro tdual 2}E\otimes U(m)\cong (p_Z)_*L\ \ \text{and}\ \ \hat{E}\otimes U(n)\cong (\hat{p}_Z)_*\hat{L}_Z.\end{align}
         Moreover, the connections on $E$ and $\hat{E}$ are projectively flat with curvatures given by 
         $$F_E=2\pi i F\cdot Id \in \Omega^2(S,\gu(m))\ \ \ \text{and}\ \ \ F_{\hat{E}}=2\pi i \hat{F} \cdot Id \in \Omega^2(\hat{S},\gu(n)).$$
    \end{enumerate}
    \end{thm}
    That is, for any pair of generalized T-duals, we can find higher-rank T-dualizable branes over $S$ and $\hat{S}$ which are T-dual in the sense of (\ref{intro tdual 1}) and (\ref{intro tdual 2}). Moreover, the connections of these higher rank branes are projectively flat and the underlying generalized branes determine their curvatures.

Our results apply to general branes, not just $A$ or $B$-branes. We show that, in the appropriate context, we recover the Fourier-Mukai transform and we explore T-duality between $BBB$ and $BAA$-branes through specific examples.

\textbf{Organization of the thesis.}
Chapter 2 contains the necessary background on affine torus bundles and algebraic integrable systems and we introduce the semi-flat hyperk\"ahler structure. We define it with respect to a flat connection on an algebraic integrable system. This slightly extends the original definition given by Freed \cite{freed}. 

In Chapter 3 we give the necessary background on generalized geometry and we introduce generalized branes. We spend a section studying the structure of $BAA$-branes. We show that whenever the leaf space of a coisotropic brane is a manifold, the brane structure descends to a hyperk\"ahler structure with an indefinite metric on the leaf space. 
\begin{prop}[Proposition \ref{leaf space of coisotropic 1} and Proposition \ref{leaf space of coisotropic 2}]
The leaf space of a coisotropic $BAA$ brane is hypercomplex. Moreover, there exists a symmetric non-degenerate two-tensor on the leaf space, that is a pseudo-Riemannian metric, compatible with all three complex structures.
\end{prop}
This constrains the geometry of coisotropic branes. We also generalize a result of Kamenova and Verbitsky \cite[Theorem 3.1]{kamenovaVerbitsky} to coisotropic submanifolds of integrable systems.
\begin{thm}[Theorem \ref{coisotropic in integrable system}]
Let $\pi:M\ra B$ be an algebraic integrable system. Let $Z\subset M$ be a connected complex coisotropic submanifold such that $Z$ projects to $\pi(Z)$ smoothly and regularly and let $x\in \pi(Z)$. Then, either $Z\cap \pi^{-1}(x)=\pi^{-1}(x)$ or for any leaf $L_z$, of the characteristic foliation, passing through a point $z\in (\pi|_Z)^{-1}(x)$ the intersection $L_z\cap \pi^{-1}(x)$ is a disjoint union of translates of a subtorus in $\pi^{-1}(x)$, which is independent of $z$. In particular, $Z_x:=Z\cap \pi^{-1}(x)$ is foliated by translates of a subtorus in $\pi^{-1}(x)$. 

If moreover, the characteristic foliation on $Z$ has closed leaves, $\pi(Z)$ inherits a special K\"ahler structure from $B$.
\end{thm}

At the beginning of Chapter 4, we introduce T-duality in generalized geometry. This chapter contains the first part of our main results. We first show that on T-dual affine torus bundles endowed with flat connections, the semi-flat hyperk\"ahler structures are indeed T-dual. 
\begin{thm}[Theorem \ref{semiflat tdual}]
    Let $M$ and $\hat{M}$ be algebraically integrable systems which are T-dual in the sense of generalized geometry. Assume that $M$ and $\hat{M}$ are endowed with flat connections and the corresponding semi-flat hyperk\"ahler structures. Then, the semi-flat hyperk\"ahler structure on $M$ is T-dual to the semi-flat hyperk\"ahler structure on $\hat{M}$.
\end{thm}
A local version of this theorem was already shown by Hitchin in \cite{hitchinCplxLag}, the novelty of our treatment is that we put the metric on affine torus bundles globalizing the result of Hitchin. 

After this, we turn our attention to generalized branes. We introduce T-duality for generalized branes without restricting to $A$ or $B$-type branes based on \cite{cavalcantiTdual}. We define the class of "locally T-dualizable" generalized branes. We construct the T-duals of locally T-dualizable generalized branes over a contractible base (Theorem \ref{local gg thm}), and finally, we study whether locally T-dualizable branes admit global T-duals over a non-contractible base. We find that there is a topological constraint to the existence of such T-duals (Theorem \ref{global thm1}).

Chapter 5 is dedicated to the study of holomorphic line bundles on a complex torus. This chapter motivates our treatment of T-duality. We describe the Fourier-Mukai transform of holomorphic line bundles supported on affine subtori in an algebraic torus. In this chapter, we also introduce factors of automorphy which are our main tools in generalizing this section's results. We describe the Fourier-Mukai transform using factors of automorphy.

In Chapter 6 we start working with "physical branes". We treat them as submanifolds endowed with principal $U(d)$-bundles and connections. That is, we work with the frame bundles of Hermitian vector bundles with connections instead of the vector bundles themselves. We develop the theory of factors of automorphy for principal $U(d)$-bundles with connections. This is analogous to the holomorphic description and also was used by Bruzzo, Marelli and Pioli \cite{BMP1, BMP2} in their treatment of T-duality for $A$-branes. We prove several technical theorems analogous to the ones in Chapter 5.

Chapter 7 contains our main results about the T-duality of $U(1)$-bundles. We first describe T-duality for $U(1)$-bundles supported on a single torus (Theorem \ref{tdual u(1) bundles on a torus}). We show that for bundles compatible with a complex structure the T-dual agrees with the Fourier-Mukai transform. We then extend the factor of automorphy description to $U(1)$-bundles on a trivial family of tori and prove Theorem \ref{tdual U(1) bundles trivial base}. In Section \ref{last chapter section general base}, we introduce gerbes and topological T-duality and prove Theorem \ref{tdual u(1) bundle general base}. Finally, we prove Theorem \ref{last global} by first proving a local version and then extending to a general base.

In Chapter 8 we conclude this thesis with a few possible extensions of our results.

\paragraph{Acknowledgment:}  A special thanks to my supervisor Tamas Hausel for his guidance during our weekly meetings. I would also like to thank Emilio Franco, Marco Gulatieri, Mirko Mauri, Goncalo Oliveira and Kamil Rychlewicz for helpful discussions. 

This project was supported in part by the DOC Fellowship of the Austrian Academy of Sciences titled "Branes on hyperk\"ahler manifolds", which was awarded under grant number 26069.

\chapter{Affine torus bundles and integrable systems}
The starting points of SYZ mirror symmetry \cite{SYZ} are symplectic manifolds which admit special Lagrangian torus fibrations. In this thesis we restrict our attention to fibrations with smooth fibers and affine structure groups, that is to affine torus bundles and integrable systems. In this chapter, we introduce these geometric objects and their properties which are needed in the next chapters.

In the first section, we recall the linear algebra of compact tori, both real and complex, based on their treatment in \cite{BL}. In the second section, we introduce affine torus bundles following \cite{B2} and we define connections on them with special attention to flat ones. We also explain how the Leray spectral sequence can be used to calculate the cohomology and its degeneration when the torus bundle admits a flat connection. The third section focuses on algebraic integrable systems, which are special cases of affine torus bundles admitting flat connections. We recite the work of Freed  \cite{freed} who showed that the base of algebraic integrable systems carry geometric structures called special K\"ahler structures. The end of section 3 and 
section 4 introduces semi-flat hyperk\"ahler structures on algebraic integrable systems which are induced by a special K\"ahler structure and a flat connection.

\section{Tori}\label{section tori}
In this section, we define real and complex compact tori and their duals. Moreover, we introduce the most important morphisms of tori, translations, projections and isogenies.

 A real torus $T=V/\Gamma$ is defined as the quotient of a real vector space $V$  by a full rank lattice $\Gamma\subset V$. It inherits the abelian group structure from $V$. We denote the translation by an element $x\in T$ as
\begin{align*}
t_x: T&\ra T\\
y&\mapsto y+x.
\end{align*}

Let $T=V/\Gamma_T$ and $T'=W/\Gamma_{T'}$ be tori. A homomorphism $f:T\ra T'$ can be lifted to a homomorphism of the universal covers $F:V\ra W$ which we call the \emph{analytic representation} of $f$. 
\begin{definition}
    An \emph{isogeny} is a homomorphism $f:T\ra T'$ with finite kernel. The \emph{degree of an isogeny} is $deg(f):=|Ker(f)|$.
\end{definition}
The analytic representation of an isogeny is an isomorphism $F:V\ra W$ which induces an injection $F:\Gamma_T\ra \Gamma_{T'}$. Hence, we may assume that $F=id$ and $\Gamma_T\subset \Gamma_{T'}$. In particular, we have a short exact sequence of abelian groups
$$0\ra  \Gamma_T \ra \Gamma_{T'}
\ra G \ra 0$$
where $|G|=[\Gamma_{T'}:\Gamma_T]=d$, the degree of $f$. With the identification $V=W$ we also have $d\cdot \Gamma_{T'}\subset \Gamma_T$.

Let $f:T\ra T'$ be a homomorphism. Then, $im(f)\subset T'$ is a subtorus and $ker(f)\subset T$ is a subgroup. Moreover, if we denote by $ker(f)_0$ the connected component of $0$ in $ker(f)$, then $ker(f)_0$ is a subtorus of $T$ and a finite index subgroup of $ker(f)$. In particular, we can decompose $f$ as a projection $q$ to $T/ker(f)_0$ and an isogeny $f'$ onto its image.
\begin{equation}\label{stein factorization}
\begin{tikzcd}
T \arrow{rr}{f}\arrow{dr}[swap]{q} & & im(f)\\
& T/ker(f)_0 \arrow{ur}[swap]{f'}
\end{tikzcd}
\end{equation}
This decomposition is called the \emph{Stein factorization} of $f$. The proofs of these statements can be found in \cite[Chapter 1.1]{BL} applied to complex tori. 

\paragraph{Complex tori.}
Let $V\cong \dC^g$ be a $g$-dimensional complex vector space and $\Gamma\subset V$ a full rank lattice. Then, the torus
$$X:=V/\Gamma$$
inherits the complex structure $V$. We call $X$ a \emph{complex torus}. A complex torus can be endowed with the structure of an abelian variety if it admits a \emph{polarization}, that is a Hermitian inner product $H: V\times V\ra \dC$ which is positive definite and its imaginary part takes integer values on $\Gamma$. Homomorphisms of complex tori are holomorphic maps.

Let $Y=W/\Gamma_Y$ be another complex torus. If there is an isogeny $f:X\ra Y$ and $X$ is an algebraic variety so is $Y$. Indeed, if $H$ is the polarization on $X$ and $d=deg(f)$, then $d\cdot H$ is a positive definite Hermitian form on $V\cong W$ whose imaginary part takes integer values on $\Gamma_Y$, since $d\cdot \Gamma_Y\subset \Gamma_X$.

\paragraph{Dual tori.}
Let $T=V/\Gamma$ be a real torus. Then, the dual torus is defined as 
\begin{align}\hat{T}_\dR:=V^*/\Gamma^\vee,\end{align}
where $\Gamma^\vee\subset V^*$ is the \emph{dual lattice} defined by
\begin{align*}\Gamma^\vee:=\{\gamma^*\in V^*\ |\ \gamma^*(\Gamma)\subset \dZ\}.\end{align*}
If $X=V/\Gamma $ is a complex torus we can also endow the dual torus with a complex structure. We follow \cite[Chapter 2.3]{BL}. Let 
\begin{align}\overline{\Omega}:=\text{Hom}_{\bar{\dC}}(V,\dC)\end{align}
that is, the space of $\dC$-antilinear homomorphism from $V$ to $\dC$. Let us define the dual complex lattice as 
\begin{align}\label{dual cplx lattice}\Gamma^\vee_\dC:=\{ \gamma^*\in \overline{\Omega}\ |\ im(\gamma^*(\Gamma))\subset \dZ \}.\end{align}
The dual complex torus is then given by
$$\hat{X}_\dC:=\overline{\Omega}/\Gamma_\dC^\vee.$$
We view $V$ as a real vector space with a complex structure $I$. Then 
$$V^{1,0}=\Big\{ \frac{1}{2}(v-iIv)\ | \ v\in V\Big \}$$
is again a real vector space with complex structure given by multiplication with $i\in \dC$. We have $(V,I)\cong (V^{1,0},i)$ via the map $v\mapsto \frac{1}{2}(v-iIv)$. The dual vector space $V^*=\text{Hom}_\dR(V,\dR)$ is endowed with the dual complex structure $I^*$. For $f\in V^*$, we have $I^*f(v)=f(Iv)$ for all $v\in V$. Then, the space of $\dC$-antilinear forms is given by 
$$\overline{\Omega}=(V^*)^{0,1}=\{ f: V\ra \dC\ |\ f(Iv)=-if(v)\}=\{f+iI^*f\ | \ f\in V^*\}.$$
That is, $(V^*)^{0,1}\cong (V,-I^*)$ via the map $f\mapsto i(f+iI^*f)$. The multiplication by $i$ is added so that we have
\begin{equation} \begin{tikzcd}\Gamma^\vee \arrow{r}{\cong} & \Gamma_\dC^\vee, & f\ \mapsto\ i(f+iI^*f). \end{tikzcd} \end{equation}
In particular,
\begin{align}\label{dual torus}\hat{X}_\dC\cong (V^*,-I^*)/\Gamma^\vee=\hat{X}_\dR.\end{align}
We will denote the dual torus simply by $\hat{X}$ both in the real and complex cases.

\section{Affine torus bundles}
Fiber bundles with torus fibers are key players in T-duality. The most general mathematical framework of T-duality is due to Baraglia \cite{B2} who extended the theory to affine torus bundles endowed with gerbes. In this section, we focus on affine torus bundles. We define these objects, their monodromy local systems, vertical bundles and Chern classes and introduce invariant forms on them. We describe several ways a connection on such fiber bundles can be viewed and finally introduce the Leray spectral sequence which abuts to the cohomology.

Let $T^n=\dR^n/\dZ^n$ be the standard torus. We denote by $\text{Aff}(T^n)$ the group of affine transformations of $T^n$, that is $\text{Aff}(T^n)=GL(n,\dZ)\ltimes T^n$ with $GL(n,\dZ)$ acting on $\dR^n$ via the standard representation preserving the lattice $\dZ^n$.
\begin{definition}
    An \emph{affine torus bundle} on a manifold $B$ is a torus bundle $\pi:M\ra B$ with structure group $\text{Aff}(T^n)$.
\end{definition}
In particular, there exists an $\text{Aff}(T^n)$ principal bundle $P\ra B$ such that 
\begin{align}\label{P for affine torus bundle}P\times_{\text{Aff}(T^n)}T^n=M.\end{align}
We may also identify $M$ with $P/GL(n,\dZ)$.  Indeed, suppose we trivialize $P$ on a cover $\{U_i\}$ of $B$ with transition functions $(A_{ij},c_{ij})\in GL(n,\dZ)\ltimes T^n$. Over $U_i$ the maps
\begin{equation*}
\begin{tikzcd}[column sep = small]
   & P\cong U_i\times \text{Aff}(T^n) \arrow{rd}{q} \arrow{ld}[swap]{p} &  \\
   M|_{U_i}\cong U_i\times \text{Aff}(T^n)\times_{\text{Aff}(T^n)} T^n \arrow{rr}{r} & &
    P/GL(n,\dZ)|_{U_i}\cong U_i\times T^n   \\
\end{tikzcd}
\end{equation*}
are given by  $p(b,(A,t))=[b,(A,t),0]$, $q((b,(A,t))=t$ and $r[b,(A,t),x]=(b,Ax+t)$ so we have 
\begin{equation*}
\begin{tikzcd}
    & (b,(A,t)  \arrow{dr}{q} \arrow{dl}[swap]{p} & \\ \left [b,(A,t),0\right ]\sim \left [b,(A,t)(1,-t'),t' \right ] \arrow{rr}{r} &
    & (b,t) =(b, A\cdot 0 +t).
\end{tikzcd}
\end{equation*}
Therefore locally $M|_{U_i}\cong P/GL(n,\dZ)|_{U_i}$ and the maps are compatible with the transition functions.

Using the homomorphism $\phi:\text{Aff}(T^n)\ra GL(n,\dZ)$ we can define a principal $GL(n,\dZ)$-bundle 
\begin{align}\label{P0 for aff torus bundle}P_0:=P\times_\phi GL(n,\dZ)\end{align} 
which encodes the ``linear part" of the transition functions of $P$ and of $M$. 
\begin{definition}
    The \emph{monodromy local system} of $M$ is the local system $\Gamma:=P_0\times_{GL(n,\dZ)} \dZ^n\ra B$  and the \emph{vertical bundle} of $M$ is the flat vector bundle $V:=P_0\times_{GL(n,\dZ)}\dR^n$, where the associated bundles are constructed via the standard representations.
\end{definition}
In this work, we assume every affine torus bundle to be \emph{orientable} in the sense that $V$ is orientable. 

There are many other ways to define $\Gamma$ and $V$. We will explain a few here that will be used later. First, notice that the vector bundle $V$ is the adjoint bundle of $P$. Indeed, $\text{Lie}(\text{Aff}(T^n))=\dR^n$ and the adjoint action of $\text{Aff}(T^n)$ is given by
$$Ad_{(A,t)}s_0=\frac{d}{ds}\Big|_{s=s_0}(A,t)(1,s)(A^{-1},-A^{-1}t)=\frac{d}{ds}\Big|_{s=s_0}(1,As)=As_0 ,$$
the action of $\phi((A,t))$ on $\dR^n$. We then have $$P\times_\phi GL(n,\dZ)\times_{GL(n,\dZ)}\dR^n=P\times_{Ad}\text{Lie}(\text{Aff}(T^n))=ad(P).$$

Secondly, for any $b\in B$ the constant vector fields on the torus $\pi^{-1}(b)$ form a vector space $V_b$ which glue together to form the vector bundle $V\ra B$. The lattice $\Gamma\subset V$ is given by the constant vector fields whose flow is 1-periodic. Finally, we have $V\cong (R^1\pi_*\dR)^*$ and $\Gamma\cong (R^1\pi_*\dZ)^*$ where $\dR$ and $\dZ$ are the constant sheaves on $M$, since the stalk of $R^1\pi_*\dZ$ at $b\in B$ can be identified with $\HH^1(\pi^{-1}(b),\dZ)$ and that of $R^1\pi_*\dR$ with $\HH^1(\pi^{-1}(b),\dR)$. Moreover, the vertical bundle pulled back to $M$ can be identified with the subbundle of vertical vectors in $TM$, that is the kernel of $\pi_*$. 

The vector bundle $V$ and lattice $\Gamma$ define a group bundle 
\begin{align}\label{group bundle M0}M_0:=V/\Gamma\end{align}
whose local sections act on $M$ by fiberwise translation. Choosing local sections of $M$ over $U_i$ we may identify
$M_0|_{U_i}\cong M|_{U_i}$ and the difference between $M_0$ and $M$ can be measured by the linear part of the transition functions viewed as sections $c_{ij}\in \Gamma(U_{ij},M_0)$. Let $\cC^\infty_V$ and $\cC^\infty_{M_0}$ be the sheaves of smooth sections of $V\ra B$ and $M_0\ra B$ respectively. The $\{c_{ij}\}$ form a class in $\HH^1(B,\cC^\infty_{M_0})$, and via the short exact sequence of sheaves 
$$0 \ra \Gamma \ra \cC^\infty_V\ra \cC^\infty_{M_0} \ra 0$$
using that $\cC^\infty_{V}$ is acyclic we find that $\{c_{ij}\}$ determines a class $c\in \HH^2(B,\Gamma)$.
\begin{definition}
    The class $c\in \HH^2(B,\Gamma)$ is called the \emph{Chern class} of $M$.
\end{definition}
The monodromy local system and the Chern class completely determine $M$ according to the following theorem.
\begin{theorem}[\cite{B1} Proposition 3.1, 3.2]\label{iso of afft}
Let $B$ be locally contractible and paracompact. To every pair $(\Gamma,c)$, where $\Gamma$ is a $\dZ^n$-valued local system on $B$ and $c\in \HH^2(B,\Gamma)$, there is an affine $T^n$-bundle $\pi:M\ra B$ with monodromy local system $\Gamma$ and twisted Chern class $c$. Two pairs $(\Gamma_1,c_1)$, $(\Gamma_2,c_2)$ determine the same $T^n$-bundle (up to bundle
isomorphisms covering the identity on $B$) if and only if there is an isomorphism
$\phi:\Gamma_1\ra \Gamma_2$ of local systems which sends $c_1$ to $c_2$ under the induced homomorphism $\phi:\HH^2(B,\Gamma_1) \ra \HH^2(M,\Gamma_2)$.
\end{theorem}

\paragraph{Invariant forms.}
The group bundle $M_0\ra B$ associated to an affine torus bundle $M\ra B$ has a canonical Ehresmann connection. Indeed, the vector bundle $V\ra B$ has a canonical flat connection by requiring the sections of $\Gamma$ to be flat. This connection descends to $M_0=V/\Gamma$
and we call it the Gauss-Manin connection. Another way to describe it is that $M_0\cong P_0\times_{GL(n,\dZ)}T^n$ and $P_0$ is a principal $GL(n,\dZ)$-bundle on $B$ which has a unique flat connection. The Gauss-Manin connection is the induced connection. A principal $GL(n,\dZ)$-bundle is defined by a representation of the fundamental group of $B$ in $GL(n,\dZ)$. This representation is also the monodromy of the Gauss-Manin connection. 

Even though $M$ is not a principal bundle we can define invariant forms on it. Recall that $P$ can be viewed as a principal $GL(n,\dZ)$-bundle over $M$ and let us denote by $p:P\ra M$ the projection. We use Definition 4.1. from \cite{B2}.
\begin{definition}
    A differential form $\alpha\in \Omega^k(M)$ is said to be \emph{invariant} if $p^*\alpha$ is an $\text{Aff}(T^n)$-invariant form on $P$.
\end{definition}
A form $\alpha$ on $M$ pulls back to a $GL(n,\dZ)$-invariant form and we can write any $(A,t)\in \text{Aff}(T^n)$ as $ (1,t)\cdot (A,0)$ so
$$R_{(A,t)}^*p^*\alpha=(R_{(A,0)}R_{(1,t)})^*p^*\alpha=R_{(1,t)}^*R_{(A,0)}^*p^*\alpha=R_{(1,t)}^*p^*\alpha,$$
and $p^*\alpha$ is invariant if and only if it is invariant under the action of $T^n\leq \text{Aff}(T^n)$.

We can view this action from the perspective of the action of $M_0$ on $M$. Let us denote by $t_s$ the translation along the fibers of $M$ by a (possibly local) section $s$ of $M_0$. Locally, over $U_i\subset B$ we can write $P|_{U_i}\cong U_i\times \text{Aff}(T^n)$ and the action of $(1,t)\in \text{Aff}(T^n)$ is given by 
$$(b,(A_0,t_0))\cdot(1,t)=(b,(A_0t+t_0))$$
so we have that $$p(R_{(1,t)}(b,(A_0t+t_0)))=t_{A_0t}\circ p((b,(A_0,t))),$$
where by $A_0t$ we mean the flat section of $M_0$ passing through $(b,A_0t)$.  Notice that the locally defined map
$$\mu: P\times T^n \ra X_0\ \ \text{mapping}\ \ (b,(A_0,t_0))\times t \mapsto (b,A_0t)$$
is just the quotient map $P_0\times T^n\ra M_0$ precomposed with $P\ra P_0$ so it is well-defined globally. We thus have
$$p(R_{(1,t)}(x))=t_{\mu(x,t)}p(x)\ \ \ \text{for any }x\in P,$$
so the action of $T^n$ on $P$ corresponds to the action of flat sections of $M_0$ on $M$. Therefore the following lemma is clear.
\begin{lemma}
A differential form $\alpha$ on $M$ is invariant if and only if $t_s^*\alpha=\alpha$ for any flat local section $s$ of $M_0$. 
\end{lemma}
Similarly to principal bundles, the cohomology of an affine torus bundle can be computed using invariant forms.  We will use the following statement.
\begin{lemma}[\cite{B2} Corollary 4.1.]
    Any cohomology class $h\in \HH^k(M,\dR)$ has an invariant representative $H\in \Omega^k(M)$.
\end{lemma}

\paragraph{Connections.} We can think of a connection on an affine torus bundle $\pi:M\ra B$ in several different ways. Firstly, as an Ehresmann connection, that is a split of the tangent bundle $TM$ into horizontal and vertical subbundles
$$TM\cong \pi^*TB\oplus \pi^*V,$$
where $\pi^*V=ker(\pi_*)$. We may furthermore require the split to be invariant under the action of $M_0$ on $M$, that is for any flat local section $s$ passing through $x\in M$ we have $(t_s)_*(\pi^*TB)_x=(\pi^*TB)_{t_s(x)}$.

From the discussion in the previous section, it is clear that such a connection can equivalently be represented by an invariant 1-form $A\in \Omega^1(M,\pi^*V)$ which induces an isomorphism $A: ker(\pi_*)\ra \pi^*V$. Here by invariant form we mean a form $A$ such that its pullback $p^*A\in \Omega^1(P,p^*\pi^*V)=\Omega^1(P,Lie(\text{Aff}(T^n)))$ is equivariant with respect to 
the action of $\text{Aff}(T^n)$. The horizontal subbundle is given by $ker(A)$. Clearly, $p^*A$ is a principal bundle connection on $P$ and any connection on $M$ is induced by a principal bundle connection on $P$. 

In particular, the curvature of the connection $p^*A$ is a closed two-form $F\in \Omega^2(B,ad(P))=\Omega^2(B,V)$ which can also be understood as the obstruction to the integrability of the horizontal vector bundle on $M$. More precisely, for any $X,Y\in TB$, if  $(\ )^H$ represents the horizontal lift we have
$$F(X,Y)=[X^H,Y^H]-[X,Y]^H,$$
where we identify $F(X,Y)\in V$ with a constant vertical vector field on $M$.

The curvature form $F$ represents a cohomology class in $\HH^2(B,V)$ and it is clear that it is the image of the Chern class $c$ under the morphism $\HH^2(B,\Gamma)\ra \HH^2(B,V)$. In particular, when the Chern class of $M$ is torsion, the decomposition 
$$TM=\pi^*TB\oplus \pi^*V$$
is into integrable subbundles and we may find local coordinates representing the split.

Indeed, suppose that $A\in \Omega^1(M,\pi^*V)$ is a flat connection on $M$. Let $U\in B$ be a contractible open set and choose a local frame $\{v_1,...,v_n\}$ of $V|_U$. The $v_i$ pull back to an invariant frame $\{v_i^A\}$ of $\pi^*V$ and we may use $A$ to identify the $v_i$ with an invariant frame of $ker(\pi_*)$. We have
$$[v_i^A,v_j^A]=[v_i,v_j]^A=0,$$
where we use that $V$ is the adjoint bundle of the overlying principal bundle and so its sections can be endowed with a Lie bracket. 

Similarly, a coordinate frame $\{e_i\}$ for $TB$ can be lifted to an invariant frame $\{e_i^H\}$ of the horizontal bundle. Since $A$ is flat and $\{e_i\}$ come from a coordinate frame, we have
$$[e^H_i,e^H_j]=[e_i,e_j]^H=0.$$
In particular, $\{e_i^H,v_i^A\}$ is an integrable frame of $TM|_U$ but the $v_i^A$ cannot yield global coordinates on the fiber. Nonetheless, choosing $\{v_i\}$ as a frame of $\Gamma|_U$ inside $V|_U$ allows us to integrate the $v_i^A$ to 1-periodic coordinates.

Choosing these coordinates for a cover $\{U_i\}$ of $B$ gives us an atlas $\{x^i,p^j\}$ of 1-periodic coordinates $\{p^i\}$ on the fiber and base coordinates  $\{x^j\}$ such that $\{\partial_{x_j}\}$ span the horizontal distribution. In these coordinates the transition functions must be of the form $x'=\phi_{ij}(x)$ and $p'=A_{ij}p+c_{ij}$ moreover, $c_{ij}\in \Gamma(U_{ij},V/\Gamma)$ must be constant.

In conclusion, a choice of flat connection is equivalent to a choice of local 1-periodic coordinates realizing the isomorphism
$$M|_{U_i}\cong V/\Gamma|_{U_i}$$
such that the transition functions are constant.

\paragraph{Leray spectral sequence.} Let $f:M\ra B$ be a continuous map of topological spaces and let $\cF$ be a sheaf on $M$. To the data $(M,B,f,\cF)$ we associate the Leray spectral sequence  $E^{p,q}_r(f,\cF)$ associated to the composition of $f_*$ with the global sections functor converging to the sheaf cohomology of $\cF$ on $M$. In particular, there is a filtration
$$0=F^{n+1,n}\subset F^{n,n}\subset ...\subset F^{1,n}\subset F^{0,n}=\HH^n(M,\cF)$$
such that $E^{p,q}_\infty(f,\cF)\cong F^{p,p+q}/F^{p+1,p+q}$ and the second page is given by $\HH^p(B,R^qf_*\cF)$.

Let $\pi:M\ra B$ be an affine torus bundle with monodromy local system $\Gamma\ra B$ and Chern class $c\in \HH^2(B,\Gamma)$. Let us first consider the Leray-Serre spectral sequence associated to the constant sheaf $\dZ$. The second page is given by 
$$E_2^{p,q}(\pi,\dZ)=\HH^2(B,R^q\pi_*\dZ)\cong \HH^2(B,\wedge^q\Gamma^*)$$
and the differentials are given by cupping with the Chern class and contracting the coefficients (\cite{B1} Proposition 3.3.). 

Considering the same spectral sequence with $\dR$ coefficients we have that
$$E^{p,q}_2(\pi,\dR)=\HH^p(B,\wedge^q V^*).$$
Moreover, the change of coefficients from $\dZ$ to $\dR$ induces a morphism of spectral sequences $E^{p,q}_2(\pi,\dZ)\ra E^{p,q}_2(\pi,\dR)$ so the second page differential $d_2^\dR$ is given by cupping with the image of $c$ in $\HH^2(B,V)$. 

\begin{lemma}
Let $\pi:M\ra B$ be an affine torus bundle. Then, the Leray spectral sequence $E^{p,q}_r(\pi,\dR)$ degenerates on the second page if and only if the Chern class of $M$ is torsion.
\end{lemma}
\begin{proof}
    Since we assume $V$ to be orientable, if $rank(V)=n$ then $\wedge^n V^*\cong \dR$. Moreover, via the pairing $\wedge^k V^* \otimes \wedge^{n-k}V^*\ra \wedge^n V^*\cong \dR$ we get an isomorphism $\wedge^{k}V^*\cong \wedge^{n-k}V$. Consider now the segment $d_2^\dR:\HH^0(B,\wedge^nV^*) \ra \HH^2(B,\wedge^{n-1}V^*)$ of $E^{p,q}_2(\pi,\dR)$. Using the established isomorphisms it reads as 
$$d_2^\dR:\HH^0(B,\dR)\ra \HH^2(B,V),\ \ \ 1\mapsto [c]$$
where $[c]$ represents the image of the Chern class $c$ in $\HH^2(B,V)$. Therefore, if $d_2^\dR$ is zero, so is $[c]$ and $c$ must be torsion.
\end{proof}

On an affine torus bundle, there is also a filtration on differential forms
$$0=F^{n+1}\Omega^n\subset F^{n}\Omega^n\subset ...\subset F^0\Omega^n=\Omega^n(M) $$
where
$$F^i\Omega^n=\{\omega\in \Omega^n(M)|\ \iota_{V_1}\iota_{V_2}...\iota_{V_{n+1-i}}\omega=0\ \forall\ V_1,...,V_i\in \pi^*V \}.$$
and a choice of connection splits this filtration. The indexing is done in the opposite way to match the filtration coming from the Leray spectral sequence. Indeed, when the connection is flat, it also splits the filtration on cohomology coming from the Leray spectral sequence. In particular, we have
$$H^n(M,\dR)=\bigoplus_{p+q=n}\HH^p(B,\wedge^qV^*).$$

\section{Algebraic integrable systems}
Algebraic integrable systems are a special case of affine torus bundles endowed with interesting geometric structures. The base of an algebraic integrable system is special Kh\"aler and the total space carries a semi-flat hyperk\"ahler structure. In this section we introduce these structures based on work of Freed \cite{freed}.

\begin{definition}\label{integrable system}
    An algebraic integrable system is a holomorphic surjection $\pi: M\ra B$ between smooth complex manifolds such that the following hold\\
    (1) $M$ carries a holomorphic symplectic structure $\eta\in \Omega^{2,0}(M)$,\\
    (2) the fibers of $\pi$ are compact complex Lagrangian tori, and\\
    (3) there exists a class $\alpha\in \HH^2(M,\dR)$ such that its restriction $\alpha_b$ to each fiber $M_b$ lies in $\HH^2(M_b,\dZ)\cap \HH^{1,1}(M_b)\subset \HH^2(M_b,\dR)$ and defines a positive polarization.  
\end{definition}
A positive polarization on a complex torus $M_b\cong V_b/\Gamma_b$ is a non-degenerate Hermitian form on $V_b$ which is positive definite and whose imaginary part takes integer values on $\Gamma_b$. A polarization can also be defined by a cohomology class $\alpha_b\in \HH^2(M_b,\dZ)\cap \HH^{1,1}(M_b)$ as follows. Let $E_b\in \Omega^2(M_b,\dR)$ be the unique invariant representative of $\alpha_b$. Then, we may consider $E_b$ as an alternating bilinear form on $V_b$ and define the corresponding Hermitian form as
$$H_b(v,w)=E_b(Iv,w)+iE_b(v,w).$$
We say that $\alpha_b\in \HH^2(M,\dR)$ defines a positive polarization if $H_b$ does.  Note that this convention differs from the convention we use for K\"ahler geometry. In a K\"ahler triple $(g,I,\omega)$ we take $\omega=gI$ so the associated Hermitian metric is $h=g-i\omega$ with $Im(h)=-\omega$. 

Viewed as real manifolds $\pi:M\ra B$ is an integrable system with respect to both $Re(\eta)=\sigma$ and $Im(\eta)$ so we can utilize the theory of real integrable systems, see for example \cite{Dui}. In particular, if the fibers are compact Lagrangian submanifolds they are necessarily affine tori. Indeed, let us choose local coordinates $\{x^i\}$ on some open $U\subset B$. Then, $\{f_i=x^i\circ \pi\}$ are $dim_{\dR}(M)/2$ Poisson commuting functions. The Hamiltonian flow along the corresponding vector fields $X_{f_i}=\sigma^{-1}df_i$ endows the fibers with the affine torus structure. 

For each $b\in B$ then we may consider the vector space $V_b$ of constant vector fields with respect to the Hamiltonian action. Inside $V_b$ there is a lattice $\Gamma_b$ of vector fields whose flow is one-periodic. The local vector fields $V_b$ glue together to a vector bundle $V\ra B$ and the lattices to a local system $\Gamma\ra B$. In conclusion, $M$ is an affine torus bundle with monodromy local system $\Gamma$ and Chern class defined as before.

Let $X\in V_b$ be a constant vector field on the fiber $\pi^{-1}(b)$. Then, since the fibers are Lagrangian $\iota_X\sigma$ is a basic one-form and we get a symplectic isomorphism 
$$V\cong T^*B|_U.$$
The image of $\Gamma$ under this symplectomorphism is a Lagrangian submanifold, which we still denote by $\Gamma$, in $T^*B$. Since $\Gamma$ is Lagrangian it must be spanned by closed one-forms viewed as sections of $T^*B$. Finally, via a Lagrangian section $s:U\ra M$ we have a symplectomorphism
$$T^*B/\Gamma|_U\cong M|_U.$$

In the algebraic case, we have a holomorphic symplectic form $\eta$ with $Re(\eta)=\sigma$ and $Im(\eta)=I^*\sigma$. Since $\pi$ is a holomorphic map pulling back holomorphic coordinates $\{z^i\}$ from the base induces complex valued Poisson commuting functions $\{\varphi_i=z^i\circ \pi\}$ on $M$. The complex structure on the fibers induces a complex structure on $V$ and $V\cong T^*B$ becomes a holomorphic symplectomorphism.

In particular, the fibers of $\pi:M\ra B$ are complex tori but the existence of polarization is a non-trivial addition which allows us to define the special K\"ahler structure on the base. The polarization also endows the fibers with a K\"ahler structure and by the relative hard Lefschetz theorem the Leray spectral sequence for $\pi:M\ra B$ degenerates on the second page. In particular, as we have discussed before $M$ as an affine torus bundle must have torsion Chern class.

\paragraph{Special K\"ahler manifolds.} Special K\"ahler geometry was first considered in the physics literature as the geometry of scalars in four-dimensional $N=2$ supersymmetric gauge theories. A clear mathematical description was given by Freed in \cite{freed} which we follow here.
\begin{definition}
    A \emph{special K\"ahler structure} on a manifold $B$ is a triple $(\omega,I,\nabla)$ where $\omega$ is a symplectic form, $I$ a compatible complex structure and $\nabla$ is a torsion-free, flat, symplectic connection satisfying $d_\nabla I=0 $. 
\end{definition}
In particular, $(\omega,I)$ form a K\"ahler structure on $B$. In the definition we consider $I$ as an element of $\Omega^1(TB)$ and $d_\nabla$ as the operator $\Omega^1(TB)\ra \Omega^2(TB)$.  

Since $\nabla$ is flat we may choose a flat local frame $\{e_i\}$ for $TB$. If $\{e^i\}$ is the dual frame of $T^*B$ we must have $de^i=0$. Indeed, for $X,Y\in TB$ we have
$$de^i(X,Y)=X(Y^i)-Y(X^i)-[X,Y]^i,$$
which, since $\nabla_XY=X(Y^i)e_i+Y^i\nabla_Xe_i$, is exactly the $e_i$ component of $\nabla_XY-\nabla_YX-[X,Y]=0$. In particular, locally $e^i=dx^i$ and the $\{x^i\}$ form coordinates since $[e_i,e_j]=\nabla_{e_i}e_j-\nabla_{e_i}e_j=0$. We call coordinates $\{x^i\}$ flat whenever the corresponding vector fields $\{\partial_{x^i}\}$ are flat with respect to $\nabla$.

The connection is symplectic, that is $\omega$ is a parallel form with respect to $\nabla$. In particular, in the flat coordinates, $\omega$ must have constant coefficients and by scaling we may choose the $\{x^i\}$ to be \emph{flat Darboux coordinates}. We can form an atlas of these flat coordinates and find that the transition functions must be affine transformations preserving $\omega$, more precisely if $dim(M)=2n$ we have
$$x'=A\cdot x +b\ \ \ A\in Sp(2n,\dR),\ b\in \dR^{2n}. $$
That is, a special K\"ahler manifold carries an \emph{affine structure}. 

Using flat coordinates, which do not have to be Darboux, one can show that the metric is induced by a \emph{potential}. This result is due to Hitchin.
\begin{theorem}\label{special kahler potential}\emph{(\cite[Section 4]{hitchinCplxLag})}
Let $(B,\nabla,\omega,I)$ be a special K\"ahler manifold and let $\{x^i\}$ be flat coordinates. Then the Riemannian metric $g$ is 
$$g=\frac{\partial^2\phi}{\partial x^i\partial x^j}dx^i\otimes dx^j$$
for some real function $\phi$. In fact, $\phi$ is a K\"ahler potential.    
\end{theorem}
Indeed, if $I=I^i_j\frac{\partial}{\partial u^i}\otimes du^j$ in flat coordinates then the condition $d_\nabla I=0$ reads
$$\frac{\partial I^i_j}{\partial u^k}-\frac{\partial I^i_k}{\partial u^j}=0. $$
Then, the metric $g=-\omega I$ can be written as
$$g_{ij}=-\omega_{ik}I^k_j $$
where $\omega_{ij}$ are constant due to $\nabla \omega=0$. Finally, 
$$\frac{\partial g_{ij}}{\partial u^l}=-\omega^{ik}\frac{\partial I^k_j}{\partial u^l}=-\omega_{ik}\frac{\partial I^k_l}{\partial u^j}=\frac{\partial g_{il}}{\partial u^j} $$
and hence the matrix $g$ is locally the derivative of a function $\Psi:\dR^n\ra \dR^n$. Since $g_{ij}$ is also symmetric, $\Psi$ is also the derivative of a function $\phi_\alpha:\dR^n\ra \dR$.

Finally, the dual of the flat connection $\nabla$ on $T^*B$ is again a flat connection. It can also be seen as a connection on the fiber bundle $p:T^*B\ra B$ so it induces a splitting of $TT^*B$ into horizontal and vertical subspaces 
$$TT^*B\cong p^*TB\oplus p^*T^*B .$$
The horizontal subspace at $x \in T^*B$ is given by the space $(s_x)_*(T_{p(\xi)}B)$ where $s_x: B\ra TB$ is a flat section through $x$. This splitting turns out to respect the complex structure of $T^*B$. Moreover, since $\nabla$ is flat this splitting is into integrable subbundles. That is, if $\{x^i,y_i\}$ are flat Darboux coordinates on $B$, let $\{p_i,q^i\}$ be the dual fiberwise coordinates on $T^*B$. Then, $\{x^i,y_i,p_i,q^i\}$ are coordinates on $T^*B$ realizing the splitting.

Let $x\in T^*B$ with $p(x)=b$ so the tangent space at $x$ is given by $T_xT^*B=T^*_bB\oplus T_bB$. The base is K\"ahler so $(g,I,\omega)$ endows $T_bB$ with a Hermitian structure. Moreover, $T^*B$ is holomorphic symplectic and the holomorphic symplectic form $\eta$ is in standard form in the splitting. In particular, we can write the complex structure $\dI$ of $T^*B$ and the holomorphic symplectic structure $\eta$ in matrix form as
\begin{align}\label{semiflat hklr 1}\dI=\begin{pmatrix}I & 0 \\ 0 & I^* \end{pmatrix}, \ \ \ Re(\eta)=\begin{pmatrix}0 & Id\\ -Id & 0\end{pmatrix},\ \ \ Im(\eta)=-\dI^*Re(\eta)=\begin{pmatrix} 0 & -I^* \\ I & 0\end{pmatrix}.\end{align}
Using the K\"ahler form $\omega$ on $B$ the triple at $T_xT^*B$ can be extended to a hypek\"ahler structure with $Re(\eta)=\omega_\dJ$ and $Im(\eta)=\omega_\dK$. The rest of the structures are given in matrix form as
\begin{align}\label{semiflat hklr 2}\omega_\dI=\begin{pmatrix}\omega & 0 \\ 0 & \omega^{-1}\end{pmatrix} \ \ \ \dJ=\begin{pmatrix}0 & g^{-1}\\ -g & 0 \end{pmatrix}\ \ \ \dK=\dI\dJ=\begin{pmatrix} 0 & -\omega^{-1} \\ \omega & 0\end{pmatrix}.\end{align}
These expressions are coordinate independent so $\omega_\dI$, $\dJ$ and $\dK$ are globally defined and since the splitting of $TT^*B$ is integrable, $\dJ$ and $\dK$ are complex structures. To show that the triple $(\dI, \dJ,\dK)$ defines a hyperk\"ahler structure on $T^*B$ it remains to show that $\omega_\dI$, $\omega_\dJ$ and $\omega_\dK$ are closed \cite{hhklrlemma}. Indeed, since $d\eta=0$ we have $d\omega_\dJ=d\omega_\dK=0$ and in the coordinates $\{x^i,y_i,p_i,q^i\}$ we may write $\omega_\dI$ as 
$$\omega_\dI=dx^i\wedge dy_i- dp_i\wedge dq^i$$
which is closed. The corresponding metric is given by
$$\cG=\begin{pmatrix}g & 0 \\ 0 & g^{-1}\end{pmatrix}.$$

The hyperk\"ahler metric defined here is what we call \emph{semi-flat hyperk\"ahler metric}. In the next section, we will see that such a metric can also be defined on algebraic integrable systems.

\paragraph{Special K\"ahler geometry on the base of an algebraic integrable system.} 
We have seen that an algebraic integrable system $\pi:M\ra B$ defines a flat connection on $B$. Moreover, the horizontal distribution of the dual connection, considered as an Ehresmann connection on $T^*B\cong V$, is complex Lagrangian. In particular, flat sections of $T^*B$ must be closed one-forms.

Indeed, let $s:B\ra TB$ be a flat section, that is $s_*TB$ is complex Lagrangian in $T(T^*B)$. If $\{x^i,p_i\}_{i=1}^{2n}$ are dual coordinates on $T^*B$, then the real standard symplectic form is given by $\sigma=d\tau=d(p_idx^i)$. Moreover, $s_*T^*B$ is Lagrangian if and only if $s^*\sigma=0$. In coordinates $s(x)=(x^i,s_i(x))$ so we have 
$$s^*\sigma=s^*d(\tau)=ds^*(\tau)=d(s_idx^i)=ds.$$

The canonical holomorphic symplectic form on $T^*B$ is given by 
$\eta=\sigma-i\dI^*\sigma$
where $\dI$ represents the complex structure on $T^*B$. Note that in general for a real two-from, $F$ the tensor $\dI^*F$ is not a two-form, it happens precisely when $F$ is the real part of a holomorphic two-form. 

Similar calculations show that a section $s:B\ra T^*B$ defines a complex horizontal subspace if and only if $s-iI^*s$ is a holomorphic $(1,0)$-form. Finally, the horizontal subspace is Lagrangian with respect to $Im(\eta)$ if and only if $dI^*s=0$. That is, $s$ is a complex Lagrangian section precisely when $ds=0$, $dI^*s=0$ and $\overline{\partial}(s-iI^*s)=0$. It is easy to see that any two of these conditions imply the third.



As we have discussed before, the cohomology class $\alpha\in \HH^2(M,\dR)$ can be viewed as a family of alternating bilinear forms $\{E_b\}$ on $V_b\cong T^*_bB$ taking integer values on $\Gamma_b$. The dual forms $\{-E_b^{-1}\}$ then give smoothly varying alternating bilinear forms on $T_bB$, that is a two-form $\omega$. The significance of the sign will become clear later. Locally, we can find a \emph{symplectic frame} $\{\mu_i,\nu_i\}$ for the $\{E_b\}$, that is
$$E_b(\mu_i,\nu_j)=-d_i\delta_{ij},\ \ \ E_b(\mu_i,\mu_j)=E_b(\nu_i,\nu_j)=0,$$
where $d_i\in \dZ_{>0}$ and $d_i|d_{i+1}$. The vector $d=(d_1,...,d_n)$ is called the type of the polarization $E_b$ and it must stay constant over a connected base. In particular, $\omega$ can be written locally as
$\omega=d_i^{-1} \nu_i\wedge \mu_i.$ Moreover, since $\{\mu_i,\nu_i\}$ are flat we may integrate them to flat coordinates $\{x^i,y_i\}$ such that $\{dx^i,dy_i\}$ is an integral frame of $\Gamma$ and such that
\begin{align}\label{omega in flat darboux coords}\omega=\frac{1}{d_i}dy_i\wedge dx^i.\end{align}
Clearly, $\nabla\omega=0$ and since $\nabla$ is torsion-free $d\omega=0$ as well.

Finally, $\omega$ is compatible with the complex structure on $TB$, since $E_b$, which agrees fiberwise with $-\omega^{-1}$, is compatible with the complex structure on $T^*_bB$. The reason for the minus sign is clear now. Due to our convention for the K\"ahler form $(I,\omega)$ defines a K\"ahler structure on $B$ with positive definite K\"ahler metric $I^*\omega$.

The compatibility with the complex structure $I$ is equivalent to the existence of conjugate special complex coordinates $\{z^i,w_i\}$ adapted to flat Darboux coordinates $\{x^i,y_i\}$. That is $\{z^i\}$ and $\{w_i\}$ are two sets of complex coordinates on $B$ satisfying $Re(z^i)=x^i$ and $Re(w_i)=y_i$. These coordinates can be found as follows. The flat Darboux coordinates induce a frame $\{dx^i,dy_i\}$ of $\Gamma$ which on each fiber can be identified with a cycles $\{\gamma^i,\delta_i\}$ generating $\HH_1(M_b,\dZ)$. The holomorphic coordinates are then given by integrating the holomorphic symplectic form over these families of cycles
$$dz^i=\int_{\gamma_i}\eta,\ \ \ dw_i=-\int_{\delta^i}\eta.$$
One can also show that in these coordinates 
$$\frac{\partial}{\partial z^i}=\frac{1}{2}\Big(\frac{\partial}{\partial x^i}-\tau_{ij}\frac{\partial}{\partial y_j}\Big).$$
where $\tau_{ij}$ is the period matrix of the fibers.

The conclusion is the following theorem from \cite{freed}, first stated in \cite{DW}.
\begin{theorem}[\cite{freed} Theorem 3.4.]
Let $(M\ra B,\eta,\alpha)$ be an algebraic integrable system. Then, the K\"ahler form $\omega$ and the connection $\nabla$ constructed above comprise of a special K\"ahler structure on $B$. Furthermore, there is a lattice $\Gamma^\vee\subset TM$
whose dual $\Gamma\subset T^*B$ is a complex Lagrangian submanifold, and the holonomy of $\nabla$ is contained in the integral symplectic group defined by $\Gamma$.
\end{theorem}
There is a second part of this theorem which states that a special K\"ahler manifold $B$ together with a complex Lagrangian lattice $\Gamma$ whose dual is flat, $T^*B/\Gamma=M_0$ with the structure of an algebraic integrable system. Indeed, the holomorphic symplectic structure of $T^*B$ descends to $T^*B/\Gamma$ and the dual of the K\"ahler form $\omega$ induces the polarizations on the fibers. This algebraic integrable system has a well-defined zero section which is complex Lagrangian.

Even more is true, however. We saw that holomorphic symplectic structure on $T^*B$ can be extended to a semi-flat hyperk\"ahler structure which is invariant under translation by $\Gamma \subset T^*B$. In particular, it descends to a hyperk\"ahler structure on $T^*B/\Gamma$ extending the holomorphic symplectic structure and such that for each $b\in B$ the K\"ahler form $\omega_\dI$ restricted to the fiber $T_b^*B/\Gamma_b=(M_0)_b$ agrees with minus the polarization. The minus sign is due to our conventions.

\section{Semi-flat hyperk\"ahler structure on an algebraic integrable system}\label{section deformation of symplectic}
In this section we show that via a flat connection, the semi-flat hyperk\"ahler structure can also be extended to the algebraic integrable system $\pi: M\ra B$. On the other hand, this hyperk\"ahler structure may not extend the holomorphic symplectic structure. In the real case, if an integrable system has a smooth section then it is diffeomorphic to $T^*B/\Gamma$ but it is only symplectomorphic if it has a Lagrangian section \cite[Theorem 2.1]{Dui}. In the holomorphic case, Lagrangian sections are replaced by complex Lagrangian ones but the obstruction remains. This issue is further explored in the second half of this section.

As we have discussed before $M$ has torsion Chern class, so any connection will define an integrable horizontal distribution. We have seen that a choice of such a connection can be viewed as a choice of an atlas in which the transitions between the fiber coordinates are constant affine transformations. Since the base of $M$ is special K\"ahler we may also choose the base coordinates to be flat Darboux and the coordinates on the vertical bundle as the dual coordinates on $T^*B$.

The construction of these coordinates is via smooth sections as follows. Let  $s_i:U_i\ra M|_{U_i}$ be any smooth sections, where $\{U_i\}$ is a good cover of $B$. The Chern class is given by $c=\{c_{ij}\}\in \HH^1(B,\cC^\infty(T^*B/\Gamma))\cong \HH^2(B,\Gamma)$ defined as
$$s'_i=s'_j+c_{ij}\ \ \ c_{ij}:U_{ij}\ra T^*U_{ij}/\Gamma.$$
Since $c$ is torsion there exists some $r\in \dZ_{>0}$ such that $r\cdot c$ is trivial. In particular there exist local sections $n_{ij}:U_{ij}\ra \Gamma$ such that 
$$c_{ij}-\frac{1}{r}n_{ij}=l_i-l_j\ \ \ \text{for sections }\  l_i:U_i\ra T^*B/\Gamma.$$
Then, via the action of $M_0$ on $M$, $s_i=s'_i-l_i$ are local smooth sections of $M$ which differ by flat sections of $T^*B/\Gamma$ over double intersections. This translates to constant affine transformations between coordinates.

These smooth sections together with flat coordinates on the base and their dual coordinates on the cotangent bundle decompose the tangent bundle of $M$ into \emph{integrable} vertical and horizontal subbundles 
$$TM\cong \pi^*TB\oplus \pi^*T^*B.$$
Therefore, the semi-flat hyperk\"ahler structure on $V+V^*$ extended to the vector bundle on $TB\oplus T^*B$ pulls back to an $M_0$-invariant structure on $M$. This is indeed a hyperk\"ahler structure by the integrability of the horizontal distribution. In this hyperk\"ahler structure the local sections $s_i$ defining the flat connection are complex Lagrangian and via these sections $M|_{U_i}\cong T^*U_i/\Gamma$ as hyperk\"ahler manifolds. The conclusion is the following theorem, which is a slight extension of Theorem 3.8 of \cite{freed}.
\begin{theorem}
Let $(M\ra B, \eta, \alpha)$ be an algebraic integrable system. Then to any flat connection on $M$ we can associate a semi-flat hyperk\"ahler structure on $M$. In particular, $M_0=T^*B/\Gamma$ carries a canonical semi-flat hyperk\"ahler structure via the Gauss-Manin connection.
\end{theorem}
It is clear from this perspective that the semi-flat structure extends the holomorphic symplectic structure on $M$ if we can find coordinates for the horizontal distribution via holomorphic Lagrangian sections. In particular, only if the flat horizontal distribution of $M$ can be chosen to be complex Lagrangian. 

The obstruction to this can be explained via the short exact sequence of shaves
$$0 \ra \Gamma \ra \Lambda(T^*B)\ra \Lambda(T^*B/\Gamma) \ra 0,$$
where $\Lambda$ denotes complex Lagrangian sections with respect to the canonical holomorphic symplectic structure $\eta$ on $T^*B$. 

A section of $T^*B$ is a one-form $\xi$ on $B$. It is Lagrangian with respect to $Re(\eta)=\sigma$ if $d\xi=0$. Indeed, $\sigma=d\tau$ where $\tau$ is the Liouville one-form, so we have $\xi^*\sigma=\xi^*(d\tau)=d(\xi^*\tau)=d\xi$. The section is holomorphic, if and only if  $\xi-iI^*\xi$ is a holomorphic one-form on $B$. Finally, if $\xi$ is Lagrangian with respect to $Re(\eta)$ and holomorphic, it is also Lagrangian with respect to $Im(\eta)$.

The relevant part of the long exact sequence of cohomology is
$$\hdots \ra \HH^1(B,\Lambda(T^*B)) \ra \HH^1(B, \Lambda(T^*B/\Gamma)) \ra \HH^2(B,\Gamma) \ra  \hdots $$
Suppose that $M\ra B$ has a smooth section $s:B\ra X$. Let $\mu_i:U_i\ra M|_{U_i}$ be local complex Lagrangian sections which define a class in $ \HH^1(B,\Lambda(T^*B/\Gamma))$. Since it is in the kernel of the boundary morphism it can be lifted to a class  $\mu\in \HH^1(B,\Lambda(T^*B))$ represented by sections $\mu_{ij}:U_{ij}\ra \Lambda(T^*U_{ij})$. Since $\cC^\infty(T^*B)$ is an acyclic sheaf $\HH^1(B,\cC^\infty(T^*B))=0$ so there exist sections $\xi_i: U_i\ra T^*U_i$ satisfying $\mu_{ij}=\xi_i-\xi_j$. That is, the sections $s$ and $\mu_i$ over $U_i$ are related as $s|_{U_i}+\xi_i=\mu_i$ where we take the image of $\xi_i$ under $T^*B\ra T^*B/\Gamma$.

Let us take dual coordinates $(x,p)=(x^\alpha,p_\alpha)$ with respect to the sections $s|_{U_i}$ and $(\tilde{x},\tilde{p})=(\tilde{x}^i,\tilde{p}_i)$ with respect to $\mu_i$ on $M|_{U_i}$. The two coordinates are related by $\tilde{x}=x$ and $\tilde{p}=p+\xi_i$.  Writing $\xi_i$ as a one-form $\xi_i=\chi_\alpha dx^\alpha$ we define $F_{\alpha\beta}=\partial_\alpha\chi_\beta$ as the matrix of differentials, so we have $d\xi_i=(F-F^T)_{\alpha\beta}dx^\alpha\wedge dx^\beta$. In the coordinates $(\tilde{x},\tilde{p})$ the holomorphic symplectic structure $(\dI,\sigma-\dI^*\sigma)$ on $M$ is in the standard form. Transforming to the coordinates $(x,p)$ we find
\begin{align}\label{nonstandard symplectic}
    \sigma=\begin{pmatrix}F-F^T & 1\\ -1 & 0\end{pmatrix}, \ \ \ \dI=\begin{pmatrix} I & 0 \\ I^*F-FI & I^* \end{pmatrix},\ \ \ -\dI^*\sigma=\begin{pmatrix} F^TI-I^*F & -I^* \\ I & 0 \end{pmatrix}.
\end{align}
We can identify $(F-F^T)_{\alpha\beta}dx^\alpha\wedge dx^\beta+i(I^*F-F^TI)_{\alpha\beta}dx^\alpha\wedge dx^\beta$ with the pullback of the holomorphic symplectic form on $M$ via the section $s:B\ra M$. It is shown in \cite[Proposition 2.10]{BDV} that this pullback must be closed and have Hodge type $(2,0)+(1,1)$. 

We see that the obstruction to the existence of a global complex Lagrangian section lies in $\HH^1(B,\Lambda(T^*B))$ and can be represented by a complex two-from on the base. If $M\ra B$ is an algebraic integrable system with torsion but non-zero Chern class, we have no chance of finding a smooth section. On the other hand, if $r\cdot c=0$ for some $r\in \dZ_>0$ we may choose local sections $s_i:U_i\ra M$ such that $s_i=s_j+r^{-1}+r^{-1}\cdot n_{ij}$ where $r^{-1}\cdot n_{ij}$ are flat sections of  $T^*U_{ij}$ and also r-torsion elements of $T^*U_{ij}/\Gamma$. 

We can explain this in a coordinate-free way as follows. On $B$ there is a short exact sequence of local system
\[
\begin{tikzcd}
    0 \arrow{r}& \Gamma \arrow{r}{\cdot r} & \Gamma \arrow{r} & \Gamma[r]\arrow{r} & 0
\end{tikzcd}
\]
where the local system $\Gamma[r]$ can be identified with the r-torsion points of $T^*B/\Gamma$. Via the induced map on cohomology $\HH^2(B,\Gamma)\ra \HH^2(B,\Gamma)$ the Chern class $c$ of $M$ maps to zero. That is, $c$ lies in the image of the boundary morphism $\HH^1(B,\Gamma[r])\ra \HH^2(B,\Gamma)$ of the long exact sequence which is equivalent to the representation of $c$ as the collection of sections $\{r^{-1}\cdot n_{ij}\}$.

The morphism of lattices $\cdot r: \Gamma\ra \Gamma$ induces a map of affine torus bundles $\rho:M\ra M_0$ where $M_0$ is the affine torus bundle corresponding to the local system $\Gamma$ with zero Chern class (\ref{group bundle M0}). In the coordinates associated to the local sections $s_i$ of $M$ and a global flat section of $M_0$ we have $\rho(x,p)=(x,r\cdot p)$. The map $\rho$ is a degree $r^{2n}$ isogeny on the fibers. 

Locally we can write the sections $s_i$ as $\mu_i+\xi_i$ for some complex Lagrangian sections $\mu_i$ of $M$ and one-forms $\xi_i$ on the base. Then, the holomorphic symplectic structure of $M$ is given by (\ref{nonstandard symplectic}) in the coordinates corresponding to the $s_i$. We can then endow $M_0$ with a holomorphic symplectic structure such that $\rho$ is a local holomorphic symplectomorphism.

Now suppose that $M_0$ endowed with the induced holomorphic symplectic structure has a global complex Lagrangian section $\mu:B\ra M_0$. Then, over a contractible cover, we may choose preimages $\mu_i$ of $\mu$ which will be complex Lagrangian and which will differ by an r-torsion element over double intersections. The conclusion of the previous section is the following theorem.

\begin{proposition}
    On an algebraic integrable system $M$ there exists a semi-flat hyperk\"ahler structure extending the holomorphic symplectic structure if and only if $M_0$ has a complex Lagrangian section.
\end{proposition}
 
\begin{example}[Higgs bundle moduli spaces]\label{higgs bundes 1} Let $\Sigma$ be a compact Riemann surface of genus $g\geq 2$. Let $K$ be the canonical bundle. A \emph{Higgs bundle of degree $d$ and rank $r$ on $\Sigma$} is a pair $(V,\Phi)$, where $V$ is a degree $d$ rank $r$ holomorphic vector bundle and $\Phi$ is a holomorphic section of $End(V)\otimes K$. We call $\Phi$ the \emph{Higgs field}. A Higgs bundle is \emph{stable} (resp. \emph{semi-stable}) if any proper $\phi$-invariant subbundle $F\subset E$ satisfies
$$\frac{deg(F)}{rank(F)}<\frac{deg(E)}{rank(E)}$$
(resp. $deg(F)/rank(F)\leq deg(E)/rank(E)$). For a fixed $r$ and $d$ there exists a coarse moduli space parametrizing the isomorphism classes of semi-stable Higgs bundles $\cM(r,d)$. It was first constructed by Hitchin \cite{hitchinHiggs} using infinite dimensional hyperk\"ahler quotient and later by Nitsure \cite{nitsure} and Simpson \cite{simpson1, simpson2} via GIT quotient. The smooth points of $\cM(r,d)$, denoted by $\cM^s(r,d)$, correspond to stable Higgs bundles and carry a hyperk\"ahler structure.

The characteristic polynomial of the Higgs field
$$\det(I\cdot x-\Phi)=x^r+a_1x^{r-1}+...+a_1$$
defines a map
\begin{equation}\label{hitchin map}
    \begin{aligned}
        h:\cM(d,r)\ \  &\ra\ \  \cA= \bigoplus_{i=1}^r \HH^0(\Sigma, K^{\otimes i})\\
        (E,\Phi)\ \  &\mapsto\ \  (a_1,...,a_r)
    \end{aligned}
\end{equation}
called the \emph{Hitchin map} to an affine space, called the \emph{Hitchin base}. In one of the holomorphic symplectic structures, the fibers of $\cM^s(d,r)$ are (singular) polarized complex Lagrangian tori, that is $\cM^s(d,r)$ is an algebraic integrable system with singular fibers. The polarization is induced by one of the K\"ahler forms. Let us denote by $\cA^{reg}$ the locus where the fibers of $h$ are smooth and let $\cM^{reg}(r,d)\subset \cM^s(r,d)$ be the preimage of $\cA^{reg}$ via $h$. Then 
\begin{align}\label{hitchin map reg}\cM^{reg}(r,d) \ra \cA^{reg}\end{align}
is an integrable system in the sense of Definition \ref{integrable system}. 
 
An \emph{$SL(r,\dC)$-Higgs bundle} is a Higgs bundle $(V,\Phi)$ of rank $r$ such that $det(V)$ is trivial and $\Phi\in End_0(V)\otimes K$ is trace-free. Once again, there exists a moduli space parametrizing $SL(r,\dC)$-Higgs bundles which we denote by $\cM_{SL}(r,\cO)$. Here, $r$ denotes the rank and $\cO$ is the trivial line bundle, that is, the determinant of $V$. Clearly, 
$$\cM_{SL}(r,\cO)\subset \cM(r, 0).$$
The restriction of the Hitchin map 
\begin{align}\label{hitchin map for SL}\cM_{SL}(r,\cO)\ra \cA_0=\bigoplus_{i=2}^n\HH^i(\Sigma,K^i)\end{align}
is again an algebraic integrable system with singular fibers. It turns out that this map has a section over all of $\cA_0$ called the \emph{Hitchin section}.

In rank $r=2$ the Hitchin section is constructed as follows \cite[Section 3]{hitchinSection}. The Hitchin base is given by $\cA_0=\HH^0(\Sigma, K^2)$ the space of quadratic differentials. Let $K^{1/2}$ be a square root of the canonical bundle, that is a spin structure on $\Sigma$. Define the Higgs bundle $(V_a,\Phi_a)$ for an $a\in \cA_0$ as
$$V_a=K^{1/2}+K^{-1/2},\ \ \Phi_a=\begin{pmatrix}1 & 0 \\ a & 1\end{pmatrix}.$$
It is easy to see that $(V_a,\Phi_a)$ is stable for all $a\in \cA_0$. This construction can be generalized to any $r\geq 3$ and it is a complex Lagrangian section (see for example \cite[Proposition 2.10]{hauselHitchin}).

To the Hitchin section over $\cA_0^{reg}=\cA_0\cap \cA^{reg}$ we can associate a semi-flat hyperk\"ahler structure which has been extensively studied. It was conjectured by Gaiotto, Moore and Nietzke  \cite{GMN1, GMN2} that the semi-flat metric is exponentially close to the original hyperk\"ahler metric far away from the locus of singular fibers $\cA_0\backslash \cA^{reg}_0$. This has only been partially proved.

\end{example}


\chapter{Generalized geometry}
The term generalized geometry was introduced by Nigel Hitchin in 2002 referring to the study of the bundle $TM+T^*M$ on a differentiable manifold $M$. He was motivated by the double field theory formalism of supersymmetric sigma models. 
Generalized geometry was further developed by two of his students Marco Gualtieri and Gil Cavalcanti during their PhD. Gualtieri demonstrated that generalized geometry unifies complex and symplectic structures into 
so-called generalized complex structures. He also showed that the formalism of supersymmetric sigma models can be encoded into generalized K\"ahler structures. 

In this chapter, we first introduce the basic notions of generalized geometry. The remainder of this section is dedicated to one of the most important notions of this thesis: branes. These objects appear in all sorts of physical theories as boundary conditions, in our context they come from supersymmetric sigma models. Gualtieri has also defined branes in the framework of generalized geometry. We study different types of generalized branes which are compatible with different generalized complex structures. The last subsection of this chapter contains the first small results of this thesis regarding the structure of branes called BAA-branes.

\section{Basic notions}
In this section we follow the work of Gualtieri \cite{MGpaper, MGthesis}. We introduce Courant algebroids, generalized complex/K\"ahler/hyperk\"ahler structures and  B-field transformations.

Let $M$ be a smooth manifold.
\begin{definition}
    A Courant algebroid on $M$ is a quadruple $(E,\rho,\langle,\rangle, [,])$ where $E\ra M$ is a vector bundle, $\rho:E\ra TM$ is a morphism of vector bundles called the anchor, $\langle, \rangle$ is a non-degenerate bilinear pairing on sections of $E$ and $[,]$ is a bracket, called the Courant bracket, satisfying the following for all $x,y,z\in \Gamma(E)$:
    \begin{enumerate} 
        \item $[x,[y,z]]=[[x,y],z]+[y,[x,z]]$, 
        \item $\rho([x,y])=[\rho(x),\rho(y)]$,  
        \item $[x,fy]=f[x,y]+df(\rho(x)) y$ for $f\in \cC^\infty(M)$,
        \item $\rho(x)\langle y,z\rangle = \langle [x,y],z\rangle +\langle y,[x,z]\rangle$,
        \item $[x,x]=\rho^*d\langle x,x\rangle$.
    \end{enumerate}
\end{definition}
An exact Courant algebroid is a Courant algebroid fitting into the short exact sequence
\begin{equation}\label{SES exact Calg}
\begin{tikzcd}
    0 \arrow{r}& T^*M \arrow{r}{\rho^*}& E \arrow{r}{\rho} & TM \arrow{r} & 0.
\end{tikzcd}
\end{equation}
One can show that in an exact Courant algebroid, the image of $T^*M$ is isotropic and the bilinear pairing $\langle,\rangle$ has split signature. Moreover, there exist isotropic splittings $s:TM\ra E$ and via such a splitting $E\cong TM+T^*M$, the anchor map is projection to the first summand, the pairing is given by
$$\langle X+\xi, Y+\eta\rangle= \frac12 (\xi(Y)-eta(X))$$
and the bracket is
$$[X+\xi,Y+\eta]=[X,Y]+\cL_X\eta -\iota_Yd\xi+\iota_X\iota_YH$$
for a closed three-form $H\in \Omega^2(M)$. 

Choosing a different isotropic splitting $s':TM\ra E$ changes $H$ by an exact three form $dB$. A change in isotropic splitting is called a $B$-field transformation. The isomorphism classes of Courant algebroids are classified by the de Rham cohomology class of the three-form $H$ and each representative can be attained by a choice of isotropic splitting.

The novelty of generalized geometry is that several geometrical structures have analogues as structures on $E=TM\oplus T^*M$. Most importantly, complex and symplectic structures can be viewed as examples of generalized complex structures. 
\begin{definition}
A \emph{generalized almost complex structure} on $E$ is an endomorphism $\cJ: E\ra E$ which is orthogonal with respect to the natural pairing and satisfies $\cJ^2=-\text{Id}$.
\end{definition}

A generalized almost complex structure decomposes the complexified bundle $E\otimes \dC$ into $\pm i$ eigenbundles. That is,
$$E\otimes \dC=L\oplus \overline{L} $$
where $L=\text{Im}(\frac12(\text{Id}+i\cJ))$ and $\overline{L}=\text{Im}(\frac12(\text{Id}-i\cJ))$. We say that a generalized almost complex structure is \emph{integrable} if $L$ is involutive with respect to the Courant bracket.
\begin{definition}
A \emph{generalized complex structure} (GCS) on $E$ is an integrable generalized almost complex structure.
\end{definition} 
Alternatively, a generalized complex structure is a maximal isotropic subbundle $L\subset E\otimes \dC $ which is involutive with respect to the Courant bracket and satisfies $L\cap \overline{L} = 0$, where $\overline{L}$ is the conjugate of $L$. Moreover, as $L$ is an isotropic subbundle it is endowed with the structure of a Lie algebroid via the Courant bracket.

The complexification of the anchor map gives a morphism of complex vector bundles $\rho:E\otimes \dC\ra TM\otimes \dC$ The \emph{type} of a generalized complex structure is the complex codimension $k$ of the complex distribution $\rho(L)\subset TM\otimes \dC$. The type of a generalized complex structure may not be constant.

The most commonly used examples of GCSs are induced by complex and symplectic structures. If $J$ is a complex structure on $M$ then the \emph{complex type generalized almost complex structure} corresponding to $J$ is 
\begin{align}\label{cplx type}
\cJ_J=\begin{pmatrix}
J & 0 \\
0 & -J^*
\end{pmatrix}\end{align}
which is integrable whenever $H^{3,0}=0$. The $+i$ eigenbundle of $\cJ_J$ is given  by
$$L_J=T^{1,0}M\oplus \Omega^{0,1}(M).$$
This GCS has type $k=dim_\dR(M)/2$. 

If $\omega\in \Omega^2(M)$ is a symplectic form on $M$, then the \emph{symplectic type generalized almost complex} structure is
\begin{align}\label{symplectic type}\cJ_\omega=\begin{pmatrix}
0 & -\omega^{-1} \\
\omega & 0
\end{pmatrix}\end{align}
which is integrable if and only if $H=0$. The $+i$ eigenbundle of $\cJ_\omega$ is
$$L_\omega=\{X-i\omega(X)\ |\ X\in TM\otimes \dC\}.$$
This GCS has type $k=0$. 
If $(M,J,\omega)$ is a K\"ahler manifold, there is both a symplectic and complex type GCS associated to the K\"ahler structure for the Courant bracket with $H=0$. The pair $(\cJ_\omega, \cJ_J)$ is an example of a generalized K\"ahler structure.

In the case of a hyperk\"ahler manifold $(M,g,I,J,K)$ if we denote the K\"ahler forms by $\omega_I,\ \omega_J$ and $\omega_K$ there are three complex type $\cJ_I,\ \cJ_J,\ \cJ_K$ and three symplectic type $\cJ_{\omega_I},\ \cJ_{\omega_J},\ \cJ_{\omega_K}$ generalized complex structures on $TM\oplus T^*M$ with $H=0$. Moreover, since $I,\ J$ and $K$ obey the quaternionic relations, the corresponding generalized complex structures satisfy
\begin{align}\label{gen quat1}
    \cJ_I\cJ_J\cJ_K&=-1,\\
    \label{gen quat2}
    \cJ_{\omega_I} \cJ_{\omega_J} =\cJ_K,\ \ \  \cJ_{\omega_J}\cJ_{\omega_K}&=\cJ_I,\ \ \  \cJ_{\omega_K}\cJ_{\omega_I}=\cJ_J.
\end{align}
Such a structure is also called \emph{generalized hyperk\"ahler} and it can also be defined for $H\neq 0$ as three generalized K\"ahler structures which satisfy the generalized quaternionic relations (\ref{gen quat1}) and (\ref{gen quat2}).

An important automorphism of exact Courant algebroids is the \emph{B-field transform} corresponding to some $B\in \Omega^2(M)$. This can be seen as changing the isotropic splitting (\ref{SES exact Calg}). It acts on $TM+T^*M$ as the matrix
$$e^B=\begin{pmatrix} 1 & 0 \\ B & 1 \end{pmatrix}$$
that is, it maps 
$$e^B(X+\xi)=X+\xi+\iota_XB,\ \ \ X\in TM,\ \xi\in T^*M.$$
The B-field transform changes the three-form which defines the Courant bracket as $H\mapsto H+dB$, so it is an automorphism precisely when $B$ is closed. The B-field transform acts on generalized complex structures $\cJ$ via conjugation
\begin{align}\label{Bfield}
    e^B(\cJ)=e^{B}\cJ e^{-B},
\end{align}
and it changes the $+i$ eigenbundle as
\begin{align}\label{Bfield+i}
    L_{e^B(\cJ)}=e^BL_\cJ=\{ X+\xi+\iota_B X \ | \ X+\xi \in L_\cJ\}.
\end{align}
For example, if $I$ and $\omega$ are a complex and a symplectic structure on $M$ and $\cJ_I$, $\cJ_\omega$ are the corresponding generalized complex structures we have
\begin{align*}e^B\cJ_I e^{-B}=\begin{pmatrix}I & 0 \\ BI+I^*B & -I^*\end{pmatrix},\ \ \ e^B\cJ_\omega e^{-B}=\begin{pmatrix} \omega^{-1}B & -\omega^{-1} \\ \omega+B\omega^{-1} B & -B\omega^{-1}  \end{pmatrix} . \end{align*}
It can also be shown that any GCS of type $k=0$ is the $B$-field transform of a symplectic type GCS.

\section{Generalized branes}
The term ``brane" originates from physics where it roughly means boundary conditions for a certain theory. The branes we are concerned with are coming from topological twists of two-dimensional nonlinear supersymmetric sigma models \cite{twisted}. In such a model the bosonic fields are given by smooth maps 
$$\Phi: \Sigma \ra M $$
from a base manifold $\Sigma$ to a target manifold $M$. Nonlinear means that the target space is a manifold, not a vector space and two-dimensional refers to the dimension of the base $\Sigma$. The base $\Sigma$ is taken to be a Riemann surface and the target $M$ is endowed with a metric and a $B$-field, a collection of local two-forms $B_{i}\in \Omega^2(U_i)$ with respect to a good cover $\{U_i\}$ of $M$, which satisfy $dB_i=dB_j $ on double intersections. Such a B-field can be understood as the curving of a connection on a $U(1)$-bundle gerbe whose curvature is the global three form $H\in \Omega^3(M)$ defined as $H|_{U_i}=dB_i$. We call $H$ the $H$-flux.

The fermionic fields are sections $\psi$ of a bundle over the space of smooth maps $\Sigma\ra M$. More precisely, we choose a spin bundle $S$ on $\Sigma$ and at the point $\Phi$ the field $\psi$ takes value in $S\otimes \Phi^*(TM)$. The physical theory is given by an action functional on the space of bosonic and fermionic fields depending on the metrics on and the $B$-field. Classically, the extremal points of the action functional provide the physical trajectories of the particles.

Supersymmetry (SUSY) transformations mix the local components of the fermionic $\psi$ and the bosonic $\Phi$ fields and we say that the theory is supersymmetric whenever the action functional is invariant under these transformations. The generators of the SUSY transformations are sections of spin bundles on $\Sigma$. In two dimensions the spin representation is not irreducible, so spin bundles decompose as $S=S^+\oplus S^-$ into ``left-handed'' and ``right-handed'' parts. When we say a theory is $\cN=(p,q)$-supersymmetric we mean that there are $p$ right-handed and $q$ left-handed independent supersymmetry transformations under which the Lagrangian is invariant.  

On a two-dimensional nonlinear sigma model, the metric of $M$ induces a pair of supersymmetry transformations and any further ones act via complex structures on $M$. These complex structures must be covariantly constant under certain connections which depend on the $H$-flux. The induced geometry on the target space of an $\cN=(p,q)$ supersymmetric sigma model is called \emph{$(p,q)$ hermitian geometry} \cite{pqGeo}. 

It has been shown that these geometries have a natural interpretation in generalized complex geometry with the $H$-twisted Courant bracket on $TM\oplus T^*M$. In \cite{MGthesis} Gualtieri proved that $(p,q)=(2,2)$ geometry is equivalent to the existence of a generalized K\"ahler structure on $(TM\oplus T^*M, H)$. The general case was summarized in \cite{pqGeo}, in particular, $(p,q)=(4,4)$ corresponds to a generalized hyperk\"ahler structure. 

In this work, we will focus on the case when the $B$-field is flat, that is when $H=0$, and $\cN=(2,2)$ or $\cN=(4,4)$. Then, the complex structures have to be parallel with respect to the Levi-Civita connection on $M$. Therefore, $(2,2)$-hermitian structure translates to a K\"ahler structure and $(4,4)$-hermitian structure to a hyperk\"ahler structure on $M$. If moreover the $B$-field vanishes, the corresponding generalized K\"ahler and hyperk\"ahler structures are precisely (\ref{cplx type}), (\ref{symplectic type}) and (\ref{gen quat1}), (\ref{gen quat2}) respectively. Turning on the $B$-field amounts to transforming the generalized complex structures via the $B$-field transform (\ref{Bfield}).

Witten in \cite{wittenTft} defined a twisting procedure which creates topological field theories from supersymmetric sigma models. In the terminology ``topological" means that the theory is independent of the metric on the base $\Sigma$. The gist of the construction is that we require the fermions to take values in different bundles than before but the action functional is kept unchanged. In this procedure half of the supersymmetry is lost that is, the generators are set to zero, but the other half is made into global symmetries with global generators. The two-dimensional $\cN=(2,2)$ supersymmetric sigma model admits two different twists that result in different physical theories which are called the $A$ and $B$ twists. In the K\"ahler case ($H=0$) the $A$ twist results in a theory that only depends on the symplectic structure, while the $B$ twist only depends on the complex part of the original K\"ahler structure. 

In the $\cN=(4,4)$ supersymmetric $H=0$ case the target manifold is hyperk\"ahler and there are a $\dC\dP^1$ worth of K\"ahler structures. Moreover, the topological twist can be performed along any two of these K\"ahler structures. In conclusion a two-dimensional $\cN=(4,4)$ supersymmetric sigma model admits a topological twist corresponding to each point of $\dC\dP^1\times\dC\dP^1$.

 When the base space $\Sigma$ is a Riemann surface with a boundary, varying the action with respect to the fields and supersymmetry generators yields boundary equations of motion. The equations of the bosonic field $\Phi$ restrict where the boundary  $\partial\Sigma$ of the base is mapped to \cite{N=2susy}. More precisely, the constraints define a (local) distribution in the tangent bundle of the target $M$ and the boundary of the base space must map to a leaf of this distribution. A submanifold to which the boundary $\partial\Sigma$ can be mapped is called the support of a brane.   In the simplest, rank one, case the fields on the boundary couple to a $\text{U}(1)$ gauge field, which can be viewed as a connection on a hermitian line bundle over the support of the brane. This is what physicists call the \emph{Chan Pathon bundle} of the brane. The gauge field is related to the difference between two trivializations of the ambient gerbe providing the $B$-field. We call a pair $(S,\nabla)$ a \emph{rank one brane} where $S$ is a submanifold of $M$ and $\nabla$ is a connection on a hermitian line bundle on $S$. In our discussion of branes in relation to generalized geometry only the curvature $F\in \Omega^2(S)$ of this connection plays a role. In this context, we will think of a brane as a pair $(S,F)$ of a submanifold and a closed two-form.

The geometry of a brane is determined by the type of supersymmetry the boundary conditions conserve. In particular, in the $\cN=(2,2)$ case there are two different types of branes which conserve half of the supersymmetry. These are called $A$ and $B$-branes and the name indicates that when one considers the twisted theories, $A$-branes are compatible with the $A$ twist and $B$-branes with the $B$ twist. In the $\cN=(4,4)$ supersymmetric sigma model there are special branes which conserve more supersymmetry and therefore are compatible with a triple of topological twist. In this section, we give a mathematical definition to all of these brane types and study their geometry.

The generalized geometry description of branes was introduced by Gualtieri \cite{MGpaper}. This definition led to the discovery of coisotropic $A$-branes independently from the work of Kapustin and Orlov \cite{kapustinOrlov}. 
\begin{definition}
A \emph{generalized submanifold} of the manifold $M$ endowed with the exact Courant algebroid $(TM\oplus T^*M,H)$, is a pair $\cL=(S, F)$ of a submanifold $S\subset M$ and a two-form $F\in \Omega^2(S)$ such that $dF=H|_S$. We will also call generalized submanifolds \emph{generalized branes}.
\end{definition}
If $M$ is the target space of a sigma model there is a gerbe with a connection on $M$ whose curvature is the three-form $H$ twisting the Courant bracket. In the sigma models, branes come equipped with vector bundles twisted by the ambient gerbe. The assumption $dF=H|_S$ requires the gerbe to be torsion when restricted to $S$, meaning that there are finite-dimensional vector bundles on $S$ twisted by the gerbe. In particular, a generalized brane can support a brane in the classical sense. 

To a generalized submanifold $\cL$ we can associate an involutive subbundle 
\begin{align}\label{gen tangent bundle}
    \tau_\cL=\{X+\xi\in TS\oplus T^*M|_S\ | \ \iota_XF=\xi|_S \} 
\end{align}
of $TM\oplus T^*M$ over $S$, which is called the \emph{generalized tangent bundle} of the generalized submanifold $\cL$. Suppose now that there is a generalized complex structure $\cJ: TM\oplus T^*M\ra TM\oplus T^*M$ on $M$. 

Let us denote by $N^*S$ the conormal bundle $Ann(TS)\subset T^*M$ of $S$. Then, the generalized tangent bundle fits into the short exact sequence
\begin{align*}
0 \ra N^*S \ra \tau_\cL \ra TS \ra 0
\end{align*}
and the two-from $F\in \Omega^2(S)$ can be viewed as the extension class of $\tau_\cL$ inside $TM\oplus T^*M$.

\begin{definition}
A generalized submanifold $\cL=(S,F)$ is a \emph{generalized complex submanifold} if  its generalized tangent bundle is invariant under $\cJ$, that is $\cJ(\tau_\cL)=\tau_\cL$.
\end{definition}

\subsection{A and B-branes}

For the $\cN=(2,2)$ supersymmetric sigma model with vanishing $B$-field the target space is a K\"ahler manifold $(M,g,J)$ and we have two generalized almost complex structures
$$\cJ_J= \begin{pmatrix}
J & 0\\
0 & -J^*
\end{pmatrix},\ \ \text{and}\ \ \cJ_\omega = \begin{pmatrix}
0 & -\omega^{-1} \\
\omega & 0
\end{pmatrix}$$
which are both integrable. We use the following terminology:
\begin{definition} A \emph{generalized $A$-brane} is a generalized complex submanifold of the generalized complex manifold $(M,\cJ_\omega,H=0)$. A \emph{generalized $B$-brane} is a generalized complex submanifold of the generalized complex manifold $(M,\cJ_J,H=0)$. 
\end{definition}
In these cases, if $\cL=(S,F)$ is an $A$ or $B$ type brane then the two-form is closed since $dF=H|_S=0 $. 
The data corresponding to branes in $\cN=(2,2)$ supersymmetric sigma models was reformulated in terms of generalized geometry by Zabzine in \cite{N=2susygcg} and by Kapustin in \cite{Kap1}. They showed that rank one physical branes are also generalized complex branes. More precisely, when the $B$-field vanishes the topological $A$ and $B$ models are governed by the symplectic and complex structures on the target. Then the generalized $A$ and $B$-branes with $F$ representing an integral cohomology class correspond to the rank one $A$ and $B$-branes of the topological sigma models.

\paragraph{A-branes.}
Let $(S,F)$ be a generalized $A$-brane on a symplectic manifold $(M,\omega)$. For $F=0$, the generalized tangent bundle is $\tau_\cL\cong N^*S\oplus TS$, so it is preserved by $\cJ_\omega$ if
$$\omega^{-1}N^*S\subset TS \ \ \text{and}\ \ \omega TS\subset N^*S . $$
That is if $S$ is both coisotropic and isotropic and therefore Lagrangian submanifold of $M$ with respect to the symplectic structure $\omega$. For $F\neq 0$, we still have $N^*S\subset \tau_\cL$ and 
$$\omega^{-1}N^*S\subset TS, $$
so the submanifold $S$ must be coisotropic, and therefore at least $dim(M)/2$ dimensional. 

As the generalized complex structure $\cJ_\omega$ restricts to $\tau_\cL$ we may write $\tau_\cL\otimes \dC = \ell \oplus \overline{\ell}$ where $\ell$ is the image of $\tau_\cL$ under $\frac12 (I-i\cJ_\omega)$. The subbundle $\ell$ for any $F$ is just the intersection of the $+i$ eigenbundle of $\cJ_\omega$ with $\tau_\cL$. Let us write $j:S\ra M$ for the inclusion. Then we have,
\begin{align*}
    \ell&= \{ X-i\omega X\ | \ X\in T_\dC M\}\cap \{X+\xi\in (TS\oplus  T^*M|_S)\otimes \dC: \ \iota_XF=j^*\xi \} \\
    &= \{ X-i\omega X\ |\ X\in T_\dC S,\ \ \iota_X(F+ij^*\omega)=0\}.
\end{align*}
where $T_\dC M=TM\otimes\dC$ and $T_\dC S=TS\otimes \dC$. Moreover,
$$\overline{\ell}=\{ X+i\omega X\ |\ X\in T_\dC S,\ \ \iota_X(F-ij^*\omega)=0\}. $$
Denote by $A\subset T_\dC S$ the image of $\ell$ under the anchor map $\rho: \tau_\cL\otimes \dC\ra T_\dC S$ and by $\overline{A}$ its complex conjugate $\rho(\overline{\ell})$. The real distribution 
$$\Delta = A\cap \overline{A} $$
is called the \emph{characteristic distribution} of the coisotropic submanifold $S$. From the above description of $A$ and $\overline{A}$ 
\begin{align}\label{char distr}
    \Delta = \{X\in TS\otimes \dC\ | \ \iota_XF=0\ \text{and}\ \iota_Xj^*\omega=0\}.
\end{align}
Since $A+\overline{A}=TS\otimes \dC, $ we can also write
\begin{align}\label{char distr2}
    \Delta=\{X\in T_\dC S\ |\ \iota_Xj^*\omega=0\ \}=\{X\in T_\dC S\ |\ \iota_XF=0\ \}. 
\end{align}
Then, $\Delta$ is an integrable distribution. Indeed, for $X,Y\in \Delta$  and $ Z\in TS$ we have 
\begin{align*}
    j^*\omega([X,Y],Z)=&Xj^*\omega(Y,Z)-Yj^*\omega(X,Z)+Zj^*\omega(X,Y)+\\
    &+j^*\omega([X,Z],Y)-j^*\omega([Y,Z],X)-dj^*\omega(X,Y,Z)\\
    =&0,
\end{align*} 
since $TS$ is integrable, $dj^*\omega=0$ and $\Delta\subset TS$ is the symplectic orthogonal complement of $TS$. The characteristic distribution integrates to a foliation of $S$ which we call the \emph{characteristic foliation}. Let 
 $\bar{S}\cong S/\Delta$
 be the leaf space and assume it is a manifold.  Denote the projection by $p:S\ra \bar{S}$. 

The complex two-form $F+ij^*\omega$ is basic with respect to the characteristic foliation, so the two-forms $F$ and $\omega$ descend to non-degenerate closed two-forms $\bar{F},\bar{\omega}\in \Omega^2(\bar{S})$  via the usual equalities
$$p^*\bar{\omega}=j^*\omega\ \ \text{and} \ \ \ p^*\bar{F}= F.$$
The complex two-form
$\bar{F}+i\bar{\omega}$ then defines an almost  complex structure  
$$J:=\bar{\omega}^{-1}\bar{F}. $$
on $\bar{S}$. It is integrable, as $X\in T^{1,0}_J\bar{S}$ if and only if $\bar{F}X=i\bar{\omega}X$ so for $X,Y\in T^{1,0}_J\bar{S}$ we have 
\begin{align*}\bar{F}([X,Y])&=[\cL_X,\iota_Y]\bar{F}\\
&=d\iota_X\iota_Y\bar{F}+\iota_Xd\iota_Y\bar{F}-\iota_Yd\iota_X\bar{F}-\iota_Y\iota_Xd\bar{F}\\
&=id\iota_X\iota_Y\bar{\omega}+i\iota_Xd\iota_Y\bar{\omega}-i\iota_Yd\iota_X\bar{\omega}\\
&=i \iota_{[X,Y]}\bar{\omega},\end{align*} since both $\bar{F}$ and $\bar{\omega}$ are closed. In the complex structure $J$ the two-form $\bar{F}+i\bar{\omega}$ is holomorphic symplectic.

For $F=0$ $A$-branes correspond to Lagrangian submanifolds with flat line bundles. On the other hand, for $F\neq 0$ the $A$-brane is supported on a coisotropic submanifold. These branes are also endowed with a line bundle but with a connection that is only flat along the leaves of the characteristic foliation. In particular, there are branes supported on the full target space which we call space-filling branes.

\begin{example}[Space filling coisotropic brane] Suppose $(S,F)$ is an $A$-brane with $S=M$. Then, the generalized tangent bundle is
$$\tau_\cL=\{X+\iota_XF\ |\ X\in T_\dC M\ \}. $$
By (\ref{char distr2}) the two-form $F$ must be non-degenerate and the leaf space is the whole manifold $M$. Then, $F+i\omega$ defines a new complex structure $I=\omega^{-1}F$ on $M$.
\end{example}

\paragraph{B-branes.} Let now $(S,F)$ be a generalized $B$-brane on a complex manifold $(M,I)$.
Then, for any $X+\xi\in \tau_\cL$
$$\cJ_I(X+\xi)=IX-I^*\xi\in \tau_\cL. $$
Therefore,
$$I(TS)\subset TS $$
so $S$ is a complex submanifold of $M$. Denote by $j:S\ra M$ the inclusion. Then $j^*I^*=I^*j^*$ and for any $(X+\xi)\in \tau_\cL$, we have
$$\iota_{IX}F=-j^*I^*\xi= -I^*j^*\xi=-I^*\iota_XF $$
that is $F$ of type $(1,1)$.

In conclusion, a $B$-brane corresponds to a complex submanifold $S$ of $M$ together with a type $(1,1)$ closed two-form $F$. When $F$ represents an integral cohomology class, it can be interpreted as the curvature of a holomorphic line bundle on $S$.

\noindent \textbf{Remark:} On a K\"ahler manifold $(X,I,\omega)$ a $B$-brane $(S,F)$ is a complex, hence K\"ahler, submanifold of with non-degenerate K\"ahler form  $\omega|_S$. On the other hand an $A$-brane is coisotropic, so $\omega|_S$ is degenerate unless $S$ is a space-filling brane. Therefore, a submanifold which supports both $A$ and $B$ type branes corresponding to a K\"ahler structure must fill the whole target manifold.

\subsection{Hyperk\"ahler branes}
Let $(M,g,I,J,K)$ be a hyperk\"ahler manifold with K\"ahler forms  $\omega_I$, $\omega_J$ and $\omega_K$. Then, on $M$ there are actually a sphere $\dS^2$ worth of K\"ahler structures, since for any vector $(a,b,c)\in \dR^3$ such that $a^2+b^2+c^2=1$ the linear map
$$aI+bJ+cK:\ TM\ra TM $$
is a complex structure with K\"ahler form $a\omega_I+b\omega_J+c\omega_K$. We denote this two-sphere of K\"ahler structures by $\dS^2_h$.

If the $B$-field vanishes then on a hyperk\"ahler manifold we can define six different generalized complex structures (\ref{gen quat1}), (\ref{gen quat2}) corresponding to the three K\"ahler structures $(I,\omega_I)$, $(J,\omega_J)$ and $(K,\omega_K)$. We may define a generalized K\"ahler structure with respect to any $v=(a,b,c)\in \dS^2_h$ as well. Indeed, if $J_v=aI+bJ+cK$ 
\begin{align}\label{general gcs relations}
    \cJ_{J_v}=a\cJ_I+b\cJ_J+c \cJ_K\ \ \ \text{and}\ \ \ \cJ_{\omega_v}=a\cJ_{\omega_I}+ b\cJ_{\omega_J}+c\cJ_{\omega_K}.
\end{align}
since
\begin{align*}
    (a&\omega_I^{-1}+b\omega_J^{-1}+c\omega_K^{-1})(a\omega_I+b\omega_J+c\omega_K)= \\
    &=a^2+b^2+c^2+ab(\omega_I^{-1}\omega_J+\omega_J^{-1}\omega_I)+bc(\omega_J^{-1}\omega_K+\omega_K^{-1}\omega_J)+
    +ac(\omega_I^{-1}\omega_K+\omega_K^{-1}\omega_I)\\
    &= 1+ab(-IJ-JI)+bc(-JK-KJ)+ac(-IK-KI)\\
    &=1.
\end{align*}

We want to consider branes which are either $A$ or $B$ type with respect to the three K\"ahler structures $I$, $J$ and $K$. Naively, we would have eight kinds of special branes on a hyperk\"ahler manifold. However, from the relations (\ref{gen quat1}) and (\ref{gen quat2}) we see that a brane that is type $B$ in two complex structures is type $B$ in the third as well and if a brane is $A$ type in two complex structures it is automatically $B$ type in the third. Therefore the possible branes are:
\begin{enumerate}
    \item $BBB$-branes: $B$ type in all three complex structures,
    \item $AAB$, $ABA$, $BAA$-branes: $A$ type in two complex structures and $B$ type in the third.
\end{enumerate}
These brane types can also be defined for any triple of orthogonal complex structures in $\dS^2_h$. It is clear that a $BBB$ brane is a $B$ brane in all of the complex structures on $M$. Meanwhile, if a brane is the second type from the list, it is an $A$ brane with respect to a circle of K\"ahler structures in $\dS^2_h$, and a $B$-brane for K\"ahler structures furthest away from the circle. In this case, the brane is neither $A$ nor $B$ type for the rest of the complex structures.

\paragraph{$BBB$-branes.}
A $BBB$-brane is a generalized complex submanifold with respect to the generalized complex structures $\cJ_I$, $\cJ_J$ and $\cJ_K$, or equivalently a hyperk\"ahler submanifold $S$ of $M$ together with a closed 2-form $F\in \Omega^2(S)$ which is type $(1,1)$ with respect to all three complex structures. Such a 2-form, when it represents an integral cohomology class, can be interpreted as the curvature of a line bundle which is holomorphic with respect to all three complex structures. Such line bundles are called \emph{hyperholomorphic}.

By the Atiyah-Ward correspondence \cite{atiyahWard}, hyperholomorphic line bundles are in one-to-one correspondence with holomorphic line bundles on the twistor space of $M$ which are trivial on the twistor lines. Hyperk\"ahler submanifolds of $M$ correspond to complex submanifolds which are also foliated by the twistor lines. That is, $BBB$-branes on a hyperk\"ahler manifold $M$ are in one-to-one correspondence with foliated submanifolds of the twistor space $M$ together with holomorphic line bundles that are trivial on the leaves of the foliation. This point of view on $BBB$ branes was fleshed out in the case of the Higgs moduli stack in \cite{francoHanson}.

\begin{example}\label{space filling BBB} Let $\pi:M\ra B$ be an algebraic integrable system endowed with a flat connection and the corresponding semi-flat hyperk\"ahler structure (\ref{semiflat hklr 1}) (\ref{semiflat hklr 2}). Then, 
$$F=\begin{pmatrix}\omega & 0 \\ 0 & -\omega^{-1}\end{pmatrix}$$
is type $(1,1)$ with respect to all complex structures. It is clearly type $(1,1)$ with respect to $\dI$ and we have
$$\dJ^*F+F\dJ=\begin{pmatrix}0 & g\omega^{-1}+\omega g^{-1}\\ g^{-1}\omega+\omega^{-1}g\end{pmatrix}=\begin{pmatrix} 0 & I^*-I^*\\ I-I & 0\end{pmatrix}=0.$$
Then, $\dK^*F+F\dK=0$ as well. In particular, $(M,F)$ is a space-filling $BBB$-brane.
\end{example}

\paragraph{$AAB/ABA/BAA$-branes.}

A hyperk\"ahler manifold is also a \emph{holomorphic symplectic manifold} in any of its complex structures. In particular, in complex structure $I$
$$\Omega_I=\omega_J+i\omega_K$$
is a holomorphic symplectic form. A complex submanifold $S$ of $(M,I,\Omega_I)$ is called a \emph{holomorphic Lagrangian submanifold} if $\Omega_I|_S=0$.

Suppose now that $\cL=(S,F)$ is a $BAA$ brane. Then, 
\begin{enumerate}
    \item $S$ is a complex submanifold in the complex structure $I$ and $F$ is type $(1,1)$,
    \item $S$ is a coisotropic submanifold in the symplectic structure $\omega_J$ and there is a characteristic distribution
    $$\Delta_J=\{ X\in T_\dC S\ | \ \iota_Xj^*\omega_J=0\ \}=\{ \ X\in T_\dC S\ | \ \iota_XF=0\ \}, $$
    \item $S$ is also a coisotropic submanifold in the symplectic structure $\omega_K$ and the characteristic distribution is 
    $$\Delta_K=\{ X\in T_\dC S\ | \ \iota_Xj^*\omega_K=0\ \}=\{ \ X\in T_\dC S\ | \ \iota_XF=0\ \} . $$
\end{enumerate}
Clearly, 
$$\Delta:=\Delta_J=\Delta_K, $$
moreover, $\Delta$ is preserved by complex structure $I$. Indeed, if we denote by $j:S\ra M$ the inclusion and $X\in \Delta_J$, then $IX\in \Delta_K$ as
$$\iota_{IX}j^*\omega_K=\iota_{j_*IX}\omega_K=gKI(j_*X)=gJ(j_*X)=\iota_Xj^*\omega_J. $$
Here we used that $S$ is a complex submanifold, that is $j_*IX=Ij_*X$. 

In conclusion, if $\Delta$ has constant rank, the leaves of the characteristic foliation are holomorphic isotropic submanifolds of the holomorphic symplectic manifold $(I,\Omega_I)$. Meanwhile, the leaf space, whenever it is a manifold, has two different complex structures coming from $F+ij^*\omega_J$ and \mbox{$F+ij^*\omega_K$}. The line bundle corresponding to this brane with curvature $F$ is holomorphic in the complex structure $I$ and restricts to a holomorphic flat bundle to each leaf. 

The same arguments hold for $ABA$ and $AAB$
 branes by switching the three complex structures.

\begin{example}
If $F=0$ then  $BAA$-branes correspond to holomorphic Lagrangian submanifolds of the holomorphic symplectic manifold $(M,I,\Omega_I)$ with flat holomorphic bundles. An example of such a brane is a fibre of the Hitchin fibration in $\cM(r,d)$ \cite{KW}.
\end{example}
\begin{example}\label{canonical coisotropic} Space filling $BAA$-brane: If we take $F=\omega_I$ then $M$ becomes a $B$-brane in complex structure $I$, since $\omega_I$ is type $(1,1)$. Moreover, $$\omega_J^{-1}\omega_I=-Jg^{-1}gI=K$$
    is a complex structure, therefore $(M,\omega_I)$ is a space filling $A$-brane with respect to $\omega_J$. Similarly, it is an $A$-brane with respect to $\omega_K$.
     Analogously, $(M,\omega_J)$ is an $ABA$-brane and $(M,\omega_K)$ is an $AAB$-brane.
 \end{example}

\subsection{On the structure of coisotropic BAA-branes}
Since their discovery in the early 2000s coisotropic branes have remained mysterious. Mirror symmetry, a conjectural equivalence of categories between $A$ and $B$-branes, suggests that the Fukaya category should be enhanced to also contain coisotropic branes. This problem has yet to find a complete solution. One of the most influential ideas is by Gaiotto and Witten \cite{GW} who postulate that the space of morphisms between a Lagrangian and a space-filling $A$-brane should be related to a quantization of the Lagrangian brane. In \cite{BG} Bischoff and Gualtieri constructed a Lagrangian-space filling pair of branes from any pair of generalized branes and used the definition of \cite{GW} to define morphisms of generalized branes. Despite all these advancements, in every case where mirror symmetry has been proven as a categorical equivalence coisotropic branes are not present. 

In this section, we do not consider morphisms only the objects, more precisely $BAA$-branes, and look at the structure of those which are neither space-filling nor Lagrangian. We show first that whenever the leaf space of the characteristic foliation is a manifold, the brane structure descends to a space-filling $BAA$-brane structure on the leaf space. This result can be used to assess whether a coisotropic submanifold carries a $BAA$-brane structure or not. Secondly, we look at the geometry of holomorphic coisotropic submanifolds, without a brane structure, in algebraic integrable systems. This result is a slight modification of Kamenova and Verbitsky's \cite{kamenovaVerbitsky} theorem about the structure of holomorphic Lagrangians.

\paragraph{Leaf space of a $BAA$-brane.} Let $V$ be a symplectic vector space with symplectic form $\omega$. Let $Z\subset V$ be a coisotropic subspace with symplectic orthogonal complement $Z^\omega=\Delta\subset Z$. Then the symplectic form induces a symplectic form $\bar{\omega}$ on $Z/\Delta$ which we may view as a map $\bar{\omega}: Z/\Delta \ra (Z/\Delta)^*$. We have two short exact sequences of vector spaces
\begin{equation}\label{coisotropicSES1}
\begin{tikzcd}
    0 \arrow{r} & \Delta \arrow{r} & Z \arrow{r} & Z/\Delta \arrow{r} & 0
\end{tikzcd}
\end{equation}
\begin{equation}\label{coisotropicSES2}
\begin{tikzcd}
    0 \arrow{r} & (Z/\Delta)^* \arrow{r} & Z^* \arrow{r} & \Delta^* \arrow{r} & 0,
\end{tikzcd}
\end{equation}
and we know that
$$Z^*\cong V^*/Ann_V(Z)\ \ \ \text{and}\ \ \ \Delta^*\cong V^*/Ann_V(\Delta)\cong Z^*/Ann_Z(\Delta). $$
Note, that $Ann_Z(\Delta)$ is the image of $Ann_V(\Delta)$ under the projection $V^*\ra Z^*$. In particular,
\begin{align}\label{Z/Delta dual}(Z/\Delta)^*\cong Ann_Z(\Delta)\cong Ann_V(\Delta)/Ann_V(Z).\end{align}
The symplectic form induces an isomorphism
$$\bar{\omega}: Z/\Delta \ra (Z/\Delta)^* $$
which is the restriction of the ambient isomorphism
$$\omega: Z \ra Ann_V(\Delta)\subset V^* $$
since $\omega: \Delta \ra Ann_V(Z)$ is also an isomorphism. 

Let $(M,I,\eta)$ be a holomorphic symplectic manifold. Let $\eta=\omega_1+i\omega_2$. Since $\eta$ is type $(2,0)$ with respect to $I$, we have
$$I^*\eta=\eta I. $$
Therefore, 
$$I^*\omega_1=\omega_1 I\ \ \ \text{and}\ \ \ I^*\omega_2=\omega_2 I. $$
Let $j:S\subset M$ be a holomorphic coisotropic submanifold together with a closed real two-form $F\in \Omega^{1,1}(S)$, such that $(S,F)$ is a $BAA$ brane. Denote by $\rho_1=j^*\omega_1$ and $\rho_2=j^*\omega_2$ the restriction of the real symplectic forms to the submanifold $S$. 
Let $\Delta$ be the real characteristic distribution of $S$, that is
\begin{align*}TS \supset \Delta &=\{X\in TS|\ \rho_1(X)=0\in T^*S\}\\
&= \{X\in TS|\ \rho_2(X)=0\in T^*S\}\\
&=\{X\in TS| \ F(X)=0\in T^*S\} .\end{align*}
Let  $$\bar{S}\cong S/\Delta$$
 be the leaf space of the characteristic foliation. Assume that $\bar{S}$ is a manifold and denote the projection by $p:S\ra \bar{S}$. As we have seen, the two-forms $F, \rho_1,\rho_2$ descend to non-degenerate closed two-forms $\bar{F},\bar{\rho}_1,\bar{\rho}_2\in \Omega^2(\bar{S})$ on the leaf space and the complex two-forms 
$\bar{F}+i\bar{\rho}_1$ and $\bar{F}+i\bar{\rho}_2 $
then almost complex structures which we denote by
$$J:=\bar{\rho}_1^{-1}\bar{F}\ \ \ \text{and}\ \ \ K:=\bar{\rho}_2^{-1}\bar{F}. $$
The complex structure $I$ also descends to $\bar{S}$ which we keep denoting by $I$.

\begin{proposition}\label{leaf space of coisotropic 1}
The leaf space of a coisotropic $BAA$ brane is hypercomplex. That is, $I,J$ and $K$ satisfy the quaternionic relations
$$IJ=-JI=K. $$
\end{proposition}
\begin{proof}
The claim of the lemma can be proven locally. At each point of $S$
we have the following inclusion of vector subspaces $\Delta\subset TS \subset TM$ which are all preserved by the complex structure $I$. There is a short exact sequence 
$$0 \ra \Delta \ra TS \ra T\bar{S} \ra 0 $$
corresponding to the projection $p:S\ra \bar{S}=S/\Delta$. This is pointwise exactly (\ref{coisotropicSES1}) and we have the dual sequence (\ref{coisotropicSES2}). To prove $IJ=-JI$ we will lift to $TS \subset TM$ and use the relations between $F, I$ and $\omega_1$. 

In particular, if $\bar{X}\in T\bar{S}$ and $X\in TS$ is any lift of it, then
\begin{align}\label{eq1}
    \bar{\rho}_1\bar{X}=\omega_1(X)+N^*S\ \  \in Ann_{TM}(\Delta)/N^*S,
\end{align}
and for any $\xi\in T^*\bar{S}$ and any lift $\xi^\ell\in Ann_{TM}(\Delta)$ we have
\begin{align}\label{eq2}
    \bar{\rho}_1^{-1}\xi = p_*\omega_1^{-1}\xi^\ell\ \ \in TS.
\end{align}
Similarly, using (\ref{Z/Delta dual}) the image of $F: TS \ra T^*S=T^*M/N^*S$ lies in $Ann_{TM}\Delta/N^*S\cong T^*\bar{S}.$  Then, for $\bar{X}\in T\bar{S}$ and any lift $X\in TS$ we can write
\begin{align}\label{eq3}
    \bar{F}(\bar{X})=F(X)\ \ \in Ann_{TM}(\Delta)/N^*S.
\end{align}

First, we show $IJ=-JI$. Let $\bar{X}\in T\bar{S}$ and $X\in TS$ a lift if it. Then,
\begin{align*}
  IJ(\bar{X})&=I\bar{\rho}_1^{-1}\bar{F}(\bar{X})\\
  &=I\bar{\rho}_1F(X)\ \ \ \ \text{eqn. (\ref{eq3})} \\
  &=Ip_*\omega_1^{-1}F(X)^\ell\ \ \ \text{eqn. (\ref{eq2})}\\
  &=p_*I\omega_1^{-1}F(X)^\ell \ \ \ \text{since $TS$ and $\Delta$ are $I$-invariant}\\
  &=p_*\omega_1^{-1}I^*F(X)^\ell\ \ \ \text{$I^*$ preserves $N^*S$ and $Ann_{TM}(\Delta)$}\\
  &=p_*\omega_1^{-1}(I^*F(X))^\ell\ \ \ \text{$F$ is type (1,1)}\\
  &=p_*\omega_1^{-1}(-FI(X))^\ell \\
  &=-p_*\omega_1^{-1}(F(IX))^\ell\\
  &=-\bar{\rho}_1^{-1}F(IX) \ \ \ \text{$I$ preserves $TS$ and $\Delta$}\\
  &=-\bar{\rho}_1^{-1}\bar{F}(I\bar{X})\\
  &=-JI(\bar{X})
\end{align*}
To show $IJ=K$ note that all the above discussion applies to $K$, $\bar{\rho}_2$ and $\omega_2$ as well. For $\bar{X}\in T\bar{S}$ and $X\in TS$ as before we have
\begin{align*}
    IJ(\bar{X})&=Ip_*\omega_1^{-1}F(X)^\ell\\
    &=p_*I\omega_1^{-1}F(X)^\ell\ \ \ \text{from $-I^*\omega_1=\omega_2$}\\
    &=p_*\omega_2^{-1}F(X)^\ell\\
    &=K(\bar{X}).
\end{align*}
\end{proof}

\begin{proposition}\label{leaf space of coisotropic 2}
On the leaf space $\bar{S}$ the tensor $\bar{g}:=I^*\bar{F}$ is symmetric and non-degenerate, that is a pseudo-Riemannian metric compatible with all three complex structures $I$, $J$ and $K$. Moreover, the forms $\bar{F}$,  $\bar{\rho}_1$ and $\bar{\rho}_2$ are the pseudo-K\"ahler forms corresponding to $I$, $J$ and $K$.
\end{proposition}
\begin{proof}
The form $\bar{F}$ is type $(1,1)$ with respect to $I$ on $\bar{S}$, by construction.  As $F$ is anti-symmetric seen as a linear map $\bar{F}:T\bar{S}\ra T^*\bar{S}$ we have $F^*=-F$. Therefore, the two-tensor $I^*F$ is symmetric
$$(I^*F)^*=-FI=I^*F. $$
Similarly we have $J^*=(\bar{\rho}_1^{-1}\bar{F})^*=-\bar{F}(-\bar{\rho}_1^{-1})=\bar{F}\bar{\rho}_1^{-1}$ and  $K^*=(\bar{\rho}_2^{-1}\bar{F})^*=-\bar{F}(-\bar{\rho}_2^{-1})=\bar{F}\bar{\rho}_2^{-1}$. 

Compatibility with the complex structures is as follows.
\begin{align*}
    I^*\bar{g}I&=I^*I^*\bar{F}I=-\bar{F}I=I^*\bar{F}=\bar{g},\\
    J^*\bar{g}J&=J^*I^*\bar{F}\bar{\rho}_1^{-1}\bar{F}=J^*I^*J^*\bar{F}=-I^*(J^*)^2\bar{F}=I^*\bar{F}=\bar{g},
\end{align*}
Then, $K^*\bar{g}K=\bar{g}$ is automatically true.

The K\"ahler form associated to $I$ is
\begin{align*}
    \omega_I&=\bar{g}I=I^*\bar{F}I=-\bar{F}I^2=\bar{F}.
\end{align*}
For the next step notice that $\bar{F}^{-1}\bar{\rho}_1=J^{-1}=-J=-\bar{\rho}_1^{-1}\bar{F}$ and analogously for $K$. Therefore,
\begin{align*}
    \bar{g}J&=-\bar{g}J^{-1}=-I^*\bar{F}\bar{F}^{-1}\bar{\rho}_1=-I^*\bar{\rho}_1=\bar{\rho}_2,\\
    \bar{g}K&=-\bar{g}K^{-1}=-I^*\bar{F}\bar{F}^{-1}\bar{\rho}_2=-\bar{\rho}_1.
\end{align*}
Note that in the equation $-I^*\omega_1=\omega_2$ each term respects the subbundles $TS$ and $\Delta$ of $TM$ and $Ann_{TM}(\Delta)$ and $N^*S$ of $T^*M$, therefore the equality descends to $\bar{S}$ and we have $-I^*\bar{\rho}_1=\bar{\rho}_2$.

The holomorphic symplectic form corresponding to this pseudo-hyperk\"ahler structure is $\bar{\rho}_2-i\bar{\rho}_1$ which is the image of $-i\eta=-i(\omega_1+i\omega_2)=\omega_2-i\omega_1$ the $-i$-rotated holomorphic symplectic form of the ambient manifold.
\end{proof}
\begin{example} Let $M$ be hyperk\"ahler with metric $g$ complex structures $I,J,K$ and K\"ahler forms $\omega_I,\omega_J,\omega_K$. Then $(M,\omega_I)$ is a space filling coisotropic $BAA$ brane. Its leaf space is $M$ and the induced pseudo-hyperk\"ahler structure is
$$I'=I,\ \ \ J'=\omega_J^{-1}\omega_I=-Jg^{-1}gI=-JI=K,\ \ \ K'=\omega_K^{-1}\omega_I=-Kg^{-1}gI=-KI=-J. $$
We see that the corresponding K\"ahler forms are really $\omega_K$ and $-\omega_J$ and the holomorphic symplectic form $\omega_J+i\omega_K$ was rotated to $\omega_K-i\omega_J$, i.e. it was multiplied by $-i$.
\end{example}
\begin{example}[Non-example]\label{attractor 1} The Higgs bundle moduli space $\cM(r,d)$ over a Riemann surface $\Sigma$ (see Example \ref{higgs bundes 1}) carries a $\dC^\times$-action sending $(E,\Phi)$ to $(E,\lambda\Phi)$ for $\lambda\in \dC^\times$. This is a holomorphic action in one of the complex structures, which we denote by $I$. For a fixed point $x\in \cM(r,d)^{\dC^\times}$ on can define its \emph{downward flow} as
$$W_x^+=\{y\in \cM(r,d)\ | \ \lim_{\lambda\ra 0}\lambda.y=x\}.$$
If $x$ is a smooth point then $W_x^+\cap \cM^s(r,d)$ is a holomorphic Lagrangian subvariety of the holomorphic symplectic structure corresponding to $I$ \cite[Proposition 2.10]{hauselHitchin} . One can also define the \emph{attractor} for a component $F$ of the fixed point set $\cM(r,d)^{\dC^\times}$ as
$$W_F^+=\{y\in \cM(r,d)\ | \ \lim_{\lambda\ra 0}\lambda.y\in F\}=\coprod_{x\in F} W_x^+.$$
It can be shown  \cite{hauselUnpub} that there exists a component of the fixed point set which is isomorphic to $\Sigma$. The attractor $W_\Sigma^+$ is a coisotropic submanifold of $\cM^s(r,d)$. Moreover, the leaf space of its characteristic foliation is the cotangent bundle $T^*\Sigma$ of $\Sigma$. In \cite{feix} it was shown that $T^*\Sigma$ can not carry a complete hyperk\"ahler metric, therefore $W_\Sigma^+$ can not carry a $BAA$-brane structure.
\end{example}

\paragraph{BAA branes in algebraic integrable systems.}
Kamenova and Verbitsky \cite{kamenovaVerbitsky} showed that when $Z$ is a complex Lagrangian submanifold in an algebraic integrable system $\pi:M\ra B$, such that $Z$ projects to $\pi(Z)$ smoothly then
the intersection of $Z$ with any fiber of $X$ has to be the disjoint union of translates of a subtorus. Such complex Lagrangians can be seen as $BAA$-branes. On the other hand, we have seen that there exist $BAA$-branes which are not Lagrangians but coisotropic. In this section, we follow \cite{kamenovaVerbitsky} and generalize their proof for a coisotropic submanifold $Z$. It is not true anymore that the intersections are subtori. On the other hand, when $Z$ is foliated by the leaves of the characteristic distribution, each isotropic leaf must intersect the fibers in translates of a subtorus. In \cite{kamenovaVerbitsky} it was also shown that for a Lagrangian $Z$ the image $\pi(Z)$ must be a special K\"ahler submanifold. In the coisotropic case, this is true for the isotropic leaves.

Let $V$ be a $g$ dimensional real (or complex) vector space, and $V^*$ its dual. The space $V+V^*$ is endowed with the canonical symplectic form $\omega$ given by
$$\omega(X+\xi, Y+\eta)=\xi(Y)-\eta(X). $$
Let $p: V+V^*\ra V$ be the projection to the first factor and denote by $Ann(R)\subset V^*$ the annihilator of a linear subspace $R\subset V$. We first show a generalization of \cite[Lemma 3.3]{kamenovaVerbitsky}.

\begin{lemma}\label{coisotropic lemma}
Let $W\subset V+V^*$ be a coisotropic subspace and let $\Delta\subset W$ be its characteristic subspace, that is 
$$\Delta=W^\omega=\{ w\in W\ |\ \omega(w,W)=0\}. $$ Then, $W\cap V^*=Ann(p(\Delta))$ and $\Delta \cap V^*=Ann(p(W))$.
\end{lemma}
\begin{proof}
Since $\Delta=W^\omega\subset W$ we have $\Delta^\omega=W$ as well. For the first claim, if $\xi\in W\cap V^*$, then for any $w\in \Delta \subset W$ we have
$$\omega(w,\xi)=\xi(p(w))=0, $$
and therefore $\xi\in Ann(p(\Delta))$. On the other hand, if $\xi \in Ann(p(\Delta))$, then for any $w\in \Delta$
$$\omega(w,\xi)=\xi(p(w))=0, $$
so $\xi\in \Delta^\omega\cap V^*=W\cap V^*$.

For the second claim, if $\xi \in Ann(p(W))$, then for any $w\in W$ we have $\omega(w,\xi)=\xi(p(w))=0$, therefore $\xi \in W^\omega\cap V^* =\Delta\cap V^*$. On the other hand, if $\xi \in \Delta \cap V^*$ then $\omega(w,\xi)=0$ for any $w\in W$, but then $\xi(p(w))=0$ so $\xi\in Ann(p(W))$.
\end{proof}

Let $\pi: M\ra B$ be an algebraic integrable system and $i: Z\subset M$ a complex coisotropic submanifold The characteristic distribution is denoted by $\Delta$. Let us denote by $L_z$ the leaf of the characteristic foliation passing through the point $z\in Z$. These are immersed submanifolds of $Z$ and $M$. Finally, by modifying the proof of \cite[Theorem 3.2]{kamenovaVerbitsky} slightly we can prove the following statement.

\begin{theorem}\label{coisotropic in integrable system}
Let $Z\subset M$ be a connected complex coisotropic submanifold such that $Z$ projects to $\pi(Z)$ smoothly and regularly and let $x\in \pi(Z)$. Then, either $Z\cap \pi^{-1}(x)=\pi^{-1}(x)$ or for any leaf $L_z$ passing through a point $z\in (\pi|_Z)^{-1}(x)$ the intersection $L_z\cap \pi^{-1}(x)$ is a disjoint union of translates of a subtorus in $\pi^{-1}(x)$, which is independent of $z$. In particular, $Z_x:=Z\cap \pi^{-1}(x)$ is foliated by translates of a subtorus in $\pi^{-1}(x)$. 

If moreover, the characteristic foliation on $Z$ has
closed leaves, $\pi(Z)$ inherits a special K\"ahler structure from $B$.
\end{theorem}
\begin{proof}
We show: $T_zZ_x=Ann(\pi_*\Delta_z)\subset F_zM$, when $x\in \pi(Z)$ is smooth.

Via a complex Lagrangian section, the tangent bundle of an algebraic integrable system locally splits as 
$$TM=\pi^*TB\oplus \pi^*T^*B. $$
Moreover, the real part of the holomorphic symplectic form is just the canonical symplectic form on $V+V^*$ at each point with respect to this splitting. 

At each point, $z\in Z_x$ the tangent space $T_zZ\subset T_zM\cong T_xB\oplus T^*_xB$ is a coisotropic subspace and $\Delta_z$ is its symplectic orthogonal subspace. The vectors tangent to the fibre $Z_x\subset \pi^{-1}(x)$ are given by $T_zZ\cap F_zM=T^*_xB$. By Lemma \ref{coisotropic lemma} we know that
$$T_zZ_x=Ann(\pi_*\Delta_z)\subset T^*_xB. $$
Since $\Delta_z=T_zL_z$ we can identify $\pi_*\Delta_z$ with $T_x(\pi(L_z))$ if the map $L_z\ra \pi(L_z)$ is also regular. Then along a certain leaf of the characteristic distribution, the tangent space $T_{z}Z_x$ is constant. On the other hand, generally, $T_zZ_x$ can vary as we vary the leaves, even if $\pi$ is regular restricted to all the leaves, since different leaves can have different images in the base. 

Consider now a leaf $L$ intersecting $Z_x$ non-trivially. Its tangent bundle is given by $\Delta\subset TZ\subset TM$ restricted to $L$. Using the second part of Lemma \ref{coisotropic lemma} we can understand the intersection $L\cap Z_x$ as follows. The vectors tangent to $L\cap Z_x$ at the point $z$ are given by $\Delta_z\cap T^*_xB$, which by the Lemma above equals $Ann(\pi_*T_zZ_x)=Ann(T_x\pi(Z))$. In particular, 
$$T_{z}(L\cap Z_x)=Ann(T_x\pi(Z))\ \ \ \forall z\in L\cap Z_x. $$
That is, the tangent space of $L\cap Z_x$ is constant along the fiber, therefore it is a disjoint union of translates of a certain subtorus of $Z_x$. Moreover, $Ann(T_x\pi(Z))$ is independent of the leaf intersecting the fiber. Note that by this theorem the foliating subtorus may be dense. In that case, $Z_x$ has to be the whole fiber as $Z_x$ is a closed submanifold of $M$. Otherwise, $Z_x$ is foliated by the translates of a closed subtorus. 

Alternatively, if we assume that the leaves are closed then the tori foliating $Z\cap \pi^{-1}(x)$ must also be closed. 

To show that $\pi(Z)$ inherits a special K\"ahler structure from $B$ it suffices to show that the connection restricts to a special K\"ahler connection. Indeed, $\pi(Z)$ is a complex, hence K\"ahler submanifold of $B$, so the connection is the only missing piece. This holds if $T\pi(Z)\subset TB$ is a flat subbundle, since the fact that $\omega$ is flat, that $d_\nabla I=0$ and that $\nabla$ is torsion free follows from $\pi(Z)$ being a submanifold.

At each smooth point $x\in \pi(Z)$ the tangent space $T_x\pi(Z)$ is isomorphic to $Ann(T(L\cap Z_x))$ where $L$ is any leaf of the characteristic foliation which intersects the fibre over $x$. The intersection $L\cap Z_x$ is union of translates of a closed subtorus in $Z_x$, therefore $H_1(L \cap Z_x, \dZ)$ is a sublattice in $H_1(Z_x,\dZ)$. This sublattice cannot change in a neighbourhood without passing through open fibers so locally $L$ is an affine torus subbundle of $M$ over $\pi(Z)$. In particular, it inherits the Gauss-Manin connection.
\end{proof}

\begin{example} Let $B=\dC^2$ with complex coordinates $(z,w)$ and real coordinates $(x,y,u,v)$ with $z=x+iy$ and $w=u+iw$. Let $\Gamma\subset T^*B$ be the lattice $span_\dZ\{dx,dy,du,dv\}$. Then flat K\"ahler structure on the base then defines a semi-flat hyperk\"ahler structure on $M=T^*B/\Gamma$ so it is an algebraic integrable system. Denote the natural complex structure on $T^*B$ by $I$ and the holomorphic symplectic form by $\Omega_I$. Let $\{x,y,u,v,p,q,r,s\}$ be dual coordinates on $M$ with $\{p,q,r,s\}$ 1-periodic. We have
\begin{align*}
\Omega_I=(dp-idq)\wedge (dx+idy)+(dr-ids)\wedge (du+idv).
\end{align*}
For any $\alpha\in \dR$ the affine torus subbundle
$$S_\alpha=\{w=\alpha \cdot z\}\subset M$$
is complex and since it is of complex codimension 1 it is also coisotropic. In coordinates
\begin{align*}
i:S\ \ &\hookrightarrow{}\ \  M\\
[t_1,t_2,p,q,r,s]\ \  &\mapsto\ \ [t_1,t_2,\alpha t_1,\alpha t_2,p,q,r,s]
\end{align*}
We have
$$i^*\Omega_I=(dp-idq)\wedge (dt_1+idt_2)+\alpha (dr-ids)\wedge (dt_1+idt_2).$$
$$Ker(i^*\Omega_I)=\Delta=\Big\{\frac{\partial}{\partial r}-\alpha \frac{\partial}{\partial p},\ \frac{\partial}{\partial s}-\alpha \frac{\partial}{\partial q}\Big\}$$
The leaves of the distribution $\Delta$ are dense in the fiber if $\alpha$ is irrational.
\end{example}

\begin{example}\label{jacobian fiber} Recall from Example \ref{attractor 1} that there is a coisotropic submanifold $W_\Sigma^+$ inside the smooth locus of the Higgs moduli space $\cM^s(r,d)$ (\ref{higgs bundes 1}). Its intersection with the fibers of $h^{reg}:\cM^{reg}\ra \cA^{reg}$ can be understood as follows. To each $a\in \cA^{reg}$ one can associate a \emph{spectral curve} $S_a\subset T^*\Sigma$ which is a smooth projective curve determined by $\Phi$. The BNR correspondence \cite{BNR} asserts that the fiber of $h$ over $a$ is isomorphic to the Jacobian of the spectral curve $Jac(S_a)$ (for definition see Section 11.1 of \cite{BL}). There is an injection
$$\alpha:C\ra J(C)$$
called the Abel-Jacobi map \cite[Corollary 11.15]{BL} and the intersection is given by
$$W_\Sigma^+\cap h^{-1}(a)=S_a\subset Jac(S_a),$$
(conjectured by Bousseau, proved in \cite{hauselUnpub}). In particular, $W_\Sigma^+$ is a coisotropic submanifold of an integrable system whose intersection with the fibers is not a union of affine subtori. The submanifold $W_\Sigma^+$ is foliated by Lagrangian submanifolds, the upward flows of points in $\Sigma$. The null foliation is a refinement of these foliations. In \cite{hauselHitchin} it was shown that the intersection of the upward flows with the generic fibers $h^{-1}(a)$ is finitely many points, in particular, the isotropic leaves also intersect the fibers in finitely many points. That is, in translates of affine subtori.
\end{example}

\chapter{T-duality in generalized geometry}\label{chapter gen geom tdual} 
Topological T-duality is a relation between affine torus bundles endowed with bundle gerbes or equivalently, degree three integral cohomology classes. T-duality in generalized geometry can be viewed as the ``de Rham shadow" of it, a relation between affine torus bundles endowed with degree three de Rham cohomology classes. In particular, two affine torus bundles with some de Rham classes may be T-dual in the sense of generalized geometry but there might be no integral classes allowing for a topological relation.

This chapter is organized as follows. In the first section, we review the relevant definitions and theorems about T-duality in generalized geometry using works of Cavalcanti and Gualtieri \cite{cavalcantiTdual} and Baraglia \cite{B1, B2}. In particular, we introduce the T-duality map, an isomorphism of Courant algebroids derived from the T-duality relation.

In the second section, we first describe the Legendre transform of special K\"ahler structures (Theorem \ref{Leg1}). Then, we apply the T-duality map to algebraic integrable systems endowed with semi-flat hyperk\"ahler structures. We show that the T-dual of a semi-flat hyperk\"ahler structure is also a semi-flat hyperk\"ahler structure and that they are connected by the Legendre transform on the base (Theorem \ref{semiflat tdual}). The connection between T-duality and Legendre transform has already been observed by Hitchin in \cite{hitchinCplxLag}. The novelty of our treatment is that we put the semi-flat hyperk\"ahler structure on torus bundles endowed with flat connections. 

In the last section, we apply T-duality to generalized branes. We first define a general method of constructing T-duals locally, then determine a class of branes to which our method applies. We call these branes "locally T-dualizable" (Definition \ref{locally T-dualizable brane}). We show that locally T-dualizable branes in a trivial affine torus bundle admit an entire family of T-duals (Theorem \ref{local gg thm}). Finally, we study the conditions under which a locally T-dualizable brane admits a T-dual in a non-trivial affine torus bundle (Theorem \ref{global thm1}).

\section{Formalism}
In this section, we review differential T-duality in the context of generalized geometry. T-duality is a duality derived from physics between torus bundles endowed with H-fluxes. The first mathematical description is due to Bouwknegt, Evslin and Mathai \cite{BEM} who showed that the T-duality relation induces an isomorphism on twisted cohomologies. The generalization of this result was used by Cavalcanti and Gualtieri \cite{cavalcantiTdual} who reformulated the isomorphism of twisted cohomologies to an isomorphism of Courant algebroids on T-dual pairs of principal torus bundles. Finally, Baraglia \cite{B1,B2} extended the results of \cite{cavalcantiTdual} to T-dual pairs of affine torus bundles.

Let $\pi:M\ra B$ and $\hat{\pi}:\hat{M}\ra B$ be affine torus bundles and denote by $C=M\times_B\hat{M}$ the fiber product of $M$ and $\hat{M}$ over $B$. Let $H\in \Omega^3(M)$ and $\hat{H}\in \Omega^3(\hat{M})$ be closed invariant 3-forms. We have then the following diagram.
\begin{equation}\label{tduality diamond 1}
\begin{tikzcd}
    & C \arrow{dl}[swap]{p} \arrow{dr}{\hat{p}} \arrow{dd}{q} & \\
    M \arrow{dr}[swap]{\pi} & & \hat{M} \arrow{dl}{\hat{\pi}}\\
    & B &
\end{tikzcd}
\end{equation}
\begin{definition}\label{gen geom tdual}
   We say that $(M,H)$ and $(\hat{M},\hat{H})$ are T-dual in the sense of generalized geometry if there exists an invariant two-form $P\in \Omega^2(C)$ which induces a non-degenerate pairing
   $$P:q^*V\otimes \hat{q}^*\hat{V}\ra \dR$$
   and such that
   $$\hat{p}^*\hat{H}-p^*H=dP.$$
\end{definition}
Note that even though the definition is given in terms of a specific representative, if $H'=H+dB$ is another invariant representative of the de Rham class of $H$, then $(M,H')$ is also T-dual to $(M,\hat{H})$ via the two-form $P'=P-B$. We could therefore define the relation between pairs $(M,[H])$ and $(\hat{M},[\hat{H}])$ where $[H]\in \HH^3(M,\dR)$ and $[\hat{H}]\in \HH^3(\hat{M},\dR)$ but for our purposes we want to set $H=0$ and $\hat{H}=0$ and use the exact isomorphism outlined in this chapter.

\begin{example}\label{Tdual torsion coordinates} Let $\pi:M\ra B$ and $\hat{\pi}:\hat{M}\ra B$ be affine torus bundles with monodromy local systems $\Gamma_M$ and $\Gamma_{\hat{M}}$ and Chern classes $c_M$ and $c_{\hat{M}}$. Suppose moreover that $\Gamma_M^\vee\cong \Gamma_{\hat{M}}$ and the Chern classes of $M$ and $\hat{M}$ are torsion. Then $(M,0)$ and $(\hat{M},0)$ are T-dual.

Indeed, we have an element $p\in \HH^0(B,\wedge^2(\Gamma_M+\Gamma_{\hat{M}}))$ corresponding to the identity. The image $[P]$ of $p$ in $\HH^0(B,\wedge^2(q^*V+q^*\hat{V}))$ satisfies $d_2([P])=0$ since the spectral sequence with real coefficients degenerates on page 2. That is, there is a global de Rham class projecting onto $[P]$ under $F^{2,2}(q,\dR)\ra E^{2,0}_\infty(q,\dR)$. 

A choice of flat connections $A$  on $M$ and $\hat{A}$ on $\hat{M}$ gives a representative
$$P=\langle \hat{p}^*\hat{A} \wedge p^*A\rangle.$$
We can write this representative of $P$ via flat coordinates associated to the connections $A$ and $\hat{A}$. Since $\Gamma_{\hat{M}}\cong \Gamma_M^\vee$ we can identify $\hat{V}\cong V^*$ as well. Then we may choose dual frames of $\Gamma_M$ and $\Gamma_{\hat{M}}$ which we can integrate to 1-periodic fiber coordinates $\{p_i\}$ on $M$ and $\{\hat{p}_i\}$ on $\hat{M}$. In these coordinates $P=\hat{p}^*(d\hat{p}_i)\wedge p^*(dp^i)$.
\end{example}

 \paragraph{Remark.} Note that we can relax the above example to $M$ and $\hat{M}$ endowed with torsion Chern classes and an isomorphism $\phi: V^*\ra \hat{V}$. The same argument shows that there exists a T-duality relation between $(M,0)$  and $(\hat{M},0)$. On the other hand, if we want to upgrade this differential T-duality to topological T-duality we must require $\Gamma_M^\vee\cong\Gamma_{\hat{M}}$ which will be clear later.

Let us now consider the geometric consequences of a T-duality relation between pairs $(M,H)$ and $(\hat{M},\hat{H})$. To the pair $(M,H)$ we may associate an exact Courant algebroid $E=TM+T^*M$ together with the $H$-twisted Courant bracket. As we have defined invariant differential forms we can also define invariant vector fields under the action of $M_0$ on $M$. Therefore $M_0$ acts on the vector bundle $E$ and its invariant sections form a vector bundle $E_{red}$ over $B$. In \cite{BCG} it was shown that if $H$ is also invariant, the Courant algebroid structure on $E$ reduces to one on $E_{red}$. This Courant algebroid is not exact anymore, indeed given a connection on $M$ we can split $E_{red}$ as
$$E_{red}\cong TB\oplus V \oplus T^*B \oplus V^*.$$
We can repeat this construction on the other side to find another Courant algebroid $\hat{E}_{red}$ on $B$. The following theorem was shown for principal torus bundles in \cite{cavalcantiTdual} and more generally for affine torus bundles in \cite{B2}.
\begin{theorem}\label{tdualitymap}\emph{ (\cite[Theorem 3.1]{cavalcantiTdual}, \cite[Theorem 5.3]{B2}) }
Let $(M,H)$ and $(\hat{M},\hat{H})$ be a T-dual pair of affine torus bundles endowed with $H$-fluxes. Then, there exists an isomorphism of Courant algebroids
$$T:E_{red}\ra \hat{E}_{red}$$
which we call the T-duality map.    
\end{theorem}
We do not need every detail of the proof for our purposes but we need to understand the T-duality map as it plays a crucial role in the next chapter.

As in most T-duality and adjacent relations, the morphism between the two objects goes through the correspondence space via pulling, twisting with a universal object and pushing down.  Consider therefore a section of  $E_{red}$ as an invariant section $X+\xi$ of $TM+T^*M=E$. We would like to lift this to an invariant section of $TC+T^*C$ but this lift is not well-defined for the vector component. Then we want to ``twist" with the two-form $P$ and push forward to get an invariant section of $T\hat{M}+T^*\hat{M}$. Once again, pushing forward is not well-defined for the covector unless it is basic. 

These two ambiguities can be used to cancel each other and come to a unique solution. Let us choose a lift 
$$\hat{X}+p^*\xi \in TC+T^*C$$
such that $p_*\hat{X}=X$ and moreover such that in
$$\hat{X}+p^*\xi +\iota_{\hat{X}}P \in TC+T^*C$$
the covector component $p^*\xi+\iota_{\hat{X}}P$ is basic with respect to $\hat{p}:C\ra \hat{M}$.

Since $P$ is non-degenerate there is a unique such lift $\hat{X}$ and we may define
\begin{align}\label{Tmap}T(X+\xi)=\hat{p}_*(\hat{X}+p^*\xi +\iota_{\hat{X}}P).\end{align}
Although the isomorphism is between Courant algebroids on the base $B$, the invariant sections span $E$ and $\hat{E}$ so we can also transfer certain invariant features of $E$ to $\hat{E}$. In particular, via the T-duality map one can also transport invariant generalized complex structures.
\begin{theorem}\label{thm gil structures}\emph{\cite[Theorem 4.1]{cavalcantiTdual}}
Let $(M, H)$ and $(\hat{M},\hat{H})$ be T-dual spaces. Then a generalized
complex (or generalized K\"ahler) structure on $M$ which is invariant under the
action of $M_0$ is transformed via $T$ into an invariant generalized complex (or generalized K\"ahler) structure on $\hat{M}$.
\end{theorem}
\begin{proof}(sketch) 
A generalized complex structure decomposes $E\otimes \dC$  into complex maximal isotropic subbundles $L$ and $\overline{L}$ of satisfying $L\cap \overline{L}=0$. A GCS is ``invariant" if these subbundles are invariant. Therefore, we get a decomposition of $E_{red}\otimes \dC$ into $L_{red}$ and $\overline{L}_{red}$. The T-dual of these $T(L_{red})=\hat{L}_{red}$ and $T(\overline{L}_{red})=\overline{\hat{L}}_{red}$ is another GCS which can be lifted to an invariant GCS on $\hat{E}$. One can also show that $T$ respects multiplication, that is for two generalized complex structures $\cJ_1$ and $\cJ_2$ we have 
$$T(\cJ_1\cdot \cJ_2)=T(\cJ_1)\cdot T(\cJ_2).$$
\end{proof}
This theorem appeared only as a theorem for $M$ and $\hat{M}$ being principal torus bundles but the version concerning affine torus bundles is a straightforward consequence of Baraglia's \cite[Theorem 5.3]{B2}.

It is explained in  \cite{cavalcantiTdual} that the T-duality map although maps generalized complex structures to generalized complex structures, does not preserve their type. If $\pi:M\ra B$ is endowed with invariant complex structure $I$ and symplectic structure $\omega$ then we denote the corresponding generalized complex structures by $\cJ_I$ and $\cJ_\omega$. The type of $\cJ_I$ is $k=dim_\dR(M)/2$ and the type of $\cJ_\omega $ is $k=0$. The table \cite[Table 1]{cavalcantiTdual} shows that if the fibers of $\pi:M\ra B$ are complex, the T-dual of $\cJ_I$ is again a complex type generalized complex structure, Moreover, the fibers of $\hat{M}$ are also complex submanifolds with respect to the induced complex structures. Meanwhile, if the fibers are ``real" the T-dual of $\cJ_I$ is a symplectic type generalized complex structure with respect to which the fibers of $\hat{M}$ are Lagrangian. Finally, if the fibers of $M$ are symplectic with respect to $\omega$, the T-dual of $\cJ_\omega$ is a symplectic type GCS and the fibers stay symplectic.

This is an instance of the $A$-$B$ type switching which we will see in our examples in the next section.

\section{T-duality of semi-flat hyperk\"ahler structure} 
 
The Legendre transform is a classical transform which assigns to a convex real-valued function another convex real-valued function. It has been generalized to many different geometric situations, it appears for example as the translation between Lagrangian and Hamiltonian dynamics. A very important application is \cite{hklr} where it was shown that one can cook up a new hyperk\"ahler structure from a hyperk\"ahler structure on $\dR^{4n}$ invariant under the action of $\dR^{2n}$. The connection between this hyperk\"ahler Legendre transform and T-duality has been understood before \cite{hitchinCplxLag} as a transformation between structures on vector bundles. Here we first recount the relevant classical background, then we show that T-duality applied to affine torus bundles in the generalized geometry context also recovers the expected results.

\subsection{Legendre transform}
In this section, we reprove the Legendre transform of special K\"ahler structures. Legendre transform is a classical transformation applied to convex functions. It appeared in the context of hyperk\"ahler metrics in \cite[Section 3 (G)]{hklr} and in the context of special K\"ahler geometry and T-duality in \cite[Section 5]{hitchinCplxLag}. The theorem below is a reformulation of the aforementioned results.
\begin{theorem}[Legendre transform]\label{Leg1}
Let $(B,\omega,I,\nabla)$ be a special K\"ahler manifold. Then, the dual connection $\hat{\nabla}$ on $T^*B$ pulled back via the isomorphism induced by the K\"ahler metric $g: TB\ra T^*B$ defines a new special K\"ahler structure $(\omega, I, \hat{\nabla})$ on $B$. Moreover, a potential of this new special K\"ahler structure can be obtained as the Legendre transform of a potential of $(\omega, I, \nabla)$.
\end{theorem}
\begin{proof}
Let $g$ be the K\"ahler metric on $B$. Choose a good cover $\{U_\alpha\}$ of $B$ and flat Darboux coordinates $\{u^i\}_{i=1}^{2n}=\{x^i,y_i\}_{i=1}^n$ on each coordinate patch $U_\alpha$. That is the symplectic form is in canonical form
$$\omega|_{U_\alpha}=dx^i\wedge dy_i $$
and the coordinates satisfy $\nabla dx^i=\nabla dy_i=0$. 
The transition functions between such special coordinate charts are affine transformations, that is for $\{x',y'\}$ coordinates on $U_\beta$ we have
$$\begin{pmatrix} x \\ y
\end{pmatrix}=A \begin{pmatrix} x' \\ y'
\end{pmatrix}+b\ \ \ \text{with} \ A\in Sp(2n,\dR),\ b\in \dR^{2n}.$$
An atlas of these flat Darboux coordinate charts defines the affine structure on $B$ which encodes the connection $\nabla$. 

We now can define a new set of special coordinates which will define an affine structure and a connection. Let $\phi_\alpha: U_\alpha \ra \dR$ be a potential for $g$, that is, in the flat coordinates $\{u_i\}$
$$g_{ij}=\frac{\partial^2 \phi_\alpha}{\partial u^i \partial u^j}. $$
Let us now use the gradient $\Psi_\alpha$ on $U_\alpha\subset \dR^n$ of $\phi_\alpha$ to define the new set of coordinates as
\begin{align*}
   v_i:=(\Psi_\alpha)_i=\frac{\partial \phi_\alpha}{\partial u^i}.
\end{align*} 
Since the derivative of $\Psi_\alpha$ is given by the matrix of $g$ which is invertible, this transformation is a well-defined change of coordinates.  Writing out in detail we have
\begin{align*}
    dv_i&=\frac{\partial v_i}{\partial u^j}du^j=\frac{\partial^2 \phi_\alpha}{\partial u^j \partial u^i} du^j=g_{ji}du^j=g_{ij}du^j,\\
    \frac{\partial }{\partial v_i}&= \frac{\partial u^j}{\partial v_i}\frac{\partial }{\partial u_j}=g^{ji}\frac{\partial}{\partial u^j}=g^{ij}\frac{\partial}{\partial u^j},
\end{align*}
indeed, since $\delta^i_j=dv_j(\frac{\partial}{\partial v_i})=\frac{\partial v_j}{\partial u_k}du^k(\frac{\partial u^l}{\partial v_i}\frac{\partial}{\partial u^l})= \frac{\partial v_j}{\partial u_k}\frac{\partial u^l}{\partial v_i}\delta^k_l=\frac{\partial v_j}{\partial u_k}\frac{\partial u^k}{\partial v_i}$ so $\frac{\partial u^j}{\partial v_i}=g^{ji} $. In particular, we see that $dv_i$ is the frame of $T^*B$ obtained from the frame $\frac{\partial}{\partial u^i}$ of $TB$ under the isomorphism $g: TB\ra T^*B$ induced by the K\"ahler metric.

The K\"ahler structure in the new coordinates is then clearly given (as matrices)
  $$\hat{I}=gIg^{-1}=-I^*,\  \ \  \hat{\omega}=g^{-1}\omega g^{-1}=g^{-1}gIg^{-1}=Ig^{-1}=-\omega^{-1}, \ \ \ \hat{g}=g^{-1}gg^{-1}=g^{-1},$$
   where $-I^*$ is the matrix the transpose of the matrix of $-I$ in coordinates $\{u^i\}$ and so on. In particular, since $\omega$ is in standard form in the flat Darboux coordinates, the matrix of $-\omega^{-1}$ is again the standard symplectic form.

  Finally, the new connection is defined as the flat connection with respect to which the coordinates $\{v_i\}_{i=1}^{2n}$ are flat. Then, $\hat{\nabla}$ is torsion-free since the local flat framing $\{dv_i\}$ of $T^*M$ is associated with a coordinate system, so $$\Big[\frac{\partial}{\partial v_i},\frac{\partial}{\partial v_j}\Big]=0=\nabla_{\frac{\partial}{\partial v_i}}\frac{\partial}{\partial v_j}-\nabla_{\frac{\partial}{\partial v_j}}\frac{\partial}{\partial v_i}.$$  
  The connection is symplectic, $\hat{\nabla} \hat{\omega} =0 $
since the coefficients of $\hat{\omega}$ are constant in the flat coordinates. To prove the last identity
$$d_{\hat{\nabla}}\hat{I}=0 ,$$
one has to show that $\hat{g}$ can also be given locally as the Hessian of a potential $\hat{\phi}_\alpha$. Indeed, $\hat{\phi}_\alpha$ is the \textbf{Legendre transform} of the potential $\phi_\alpha$. Let
$$\hat{\phi}_\alpha(v_1,...,v_{2n}):=u^iv_i-\phi_\alpha(u_1,...,u_{2n}) \subset \dR$$
where at any $p\in B$ point $p=(u_1,...,u_{2n})=(v_1,...,v_{2n})$ are the dual coordinates.  This is again a clear calculation.

It remains to show that newly constructed special coordinates define an affine structure, that is the locally defined $(\hat{g},\hat{\omega},\hat{I},\hat{\nabla})$ extends globally to $B$.

If $\phi_\alpha$ and $\phi_\beta$ are potentials for the K\"ahler metric $g$ over $U_\alpha$ in coordinates $(u^1,...,u^{2n})$ and $U_\beta$ in coordinates $(\bar{u}^1,...,\bar{u}^{2n})$ then by definition
\begin{align*}
   (v_1,...,v_{2n})=\left(\frac{\partial \phi_\alpha}{\partial u^1},...,\frac{\partial \phi_\alpha}{\partial u^{2n}}\right)\\
    (\bar{v}_1,...,\bar{v}_{2n})=\left(\frac{\partial \phi_\beta}{\partial \bar{u}^1},...,\frac{\partial \phi_\beta}{\partial \bar{u}^{2n}}\right).
\end{align*}
The transition between flat Darboux coordinates is an affine transformation, that is
$$\bar{u}^i=A^i_ju^j+b^j\ \ \ A\in Sp(2n,\dR),\ b\in \dR^{2n}. $$
We know that in coordinates $\{u^i\}$, resp. $\{\bar{u}^i\}$, the metric tensor is given by the Hessian of $\phi_\alpha$, resp. $\phi_\beta$. On the overlap $U_\alpha\cap U_\beta$
\begin{align*}
    g_{ij}&=\frac{\partial^2 \phi_\alpha}{\partial u^i \partial u^j},\\
    \bar{g}_{ij}&=\frac{\partial^2 \phi_\beta}{\partial \bar{u}^i\partial \bar{u}^j}=\frac{\partial }{\partial \bar{u}^i}\frac{\partial }{\partial \bar{u}^j}\phi_\beta(u(\bar{u}))=\frac{\partial }{\partial \bar{u}^i}\Big(\frac{\partial \phi_\beta}{\partial u^k}\frac{\partial u^k}{\partial \bar{u}^j}\Big)=\frac{\partial^2 \phi_\beta}{\partial u^l \partial u^k}\frac{\partial u^l}{\partial \bar{u}^i}\frac{\partial u^k}{\partial \bar{u}^j}+\frac{\partial \phi_\beta}{\partial u^k}\frac{\partial^2 u^k}{\partial \bar{u}^i\partial \bar{u}^j}\\
    &=\frac{\partial^2 \phi_\beta}{\partial u^l \partial u^k}\frac{\partial u^l}{\partial \bar{u}^i}\frac{\partial u^k}{\partial \bar{u}^j}
\end{align*}
since the transition is affine. As $g$ is a $(0,2)$-tensor we have
\begin{align*}
    g_{ij}du^i\otimes du^j=\bar{g}_{ij}d\bar{u}^i\otimes d\bar{u}^j=\bar{g}_{ij}\frac{\partial \bar{u}^i}{\partial u^k}\frac{\partial \bar{u}^j}{\partial u^l}du^k\otimes du^l
\end{align*}
that is,
\begin{align*}
  \frac{\partial^2 \phi_\alpha}{\partial u^i \partial u^j}=g_{ij} = \bar{g}_{kl}\frac{\partial \bar{u}^k}{\partial u^i}\frac{\partial \bar{u}^l}{\partial u^j}=\frac{\partial^2 \phi_\beta}{\partial u^m \partial u^n}\frac{\partial u^m}{\partial \bar{u}^k}\frac{\partial u^n}{\partial \bar{u}^l}\frac{\partial \bar{u}^k}{\partial u^i}\frac{\partial \bar{u}^l}{\partial u^j}=\frac{\partial^2 \phi_\beta}{\partial u^m \partial u^n}\delta^m_i\delta^n_j=\frac{\partial^2 \phi_\beta}{\partial u^i \partial u^j}.
\end{align*}
In particular, the difference between $\phi_\alpha$ and $\phi_\beta$ is at most linear in $\{u^i\}$ and we have
$$\frac{\partial \phi_\alpha}{\partial u^i}=\frac{\partial \phi_\beta}{\partial u^i}+c_i $$
with $c\in \dR^{2n}$. In conclusion,
$$v_i=\frac{\partial \phi_\alpha}{\partial u^i}=\frac{\partial \phi_\beta}{\partial u^i}+c_i =\frac{\partial \phi_\beta}{\partial \bar{u}^k}\frac{\partial \bar{u}^k}{\partial u^i}+c_i=\bar{v}_kA^k_i+c_i. $$
In flat coordinates $\{v_i\}$ on $\hat{U}_\alpha$ and $\{\bar{v}_i\}$ on $\hat{U}_\beta$ we have
$$v_i=\bar{v}_kA^k_i+c_i $$
that is, the transition between the dual coordinates is again affine
$$\bar{v}=\hat{A}v+\hat{c} $$
but with $\hat{A}=(A^{-1})^T\in Sp(2n,\dR)$ and the translation part is governed by the linear difference between $\phi_\alpha$ and $\phi_\beta$.
\end{proof}

\paragraph{Monodromy.} The linear part of the transition functions of $\{u^i\}$ is the inverse transpose of that for $\{v_j\}$. As these transition functions around nontrivial loops define the monodromy of $\nabla$ and $\hat{\nabla}$ these monodromies can be viewed as dual representations of $\pi_1(B)$ acting on $T_xB$ or $T^*_xB$ for some $x\in B$. On the other hand, monodromy is only defined up to conjugation which amounts to linear change in the coordinates. Namely, let
$$v'_i=v_{i+n},\ \ \ v'_{i+n}=-v_i\ \ \ \text{for all }i=1,...,n $$
on each coordinate patch, so that the new affine structure will have transition functions whose linear part is
$$\Omega(\hat{A})\Omega^{-1}=(\hat{A}^{-1})^T=A ,$$
where $\Omega$ is the standard symplectic matrix. Indeed, since $A\in Sp(2n,\dR)$ it satisfies
$$A^T\Omega A=\Omega. $$
That is the monodromy of $\nabla$ and $\hat{\nabla}$ agree up to conjugation which is just the fact that any $Sp(2n,\dR)$-representation is self-dual. On the other hand, the two connections are generally distinct, as the coordinates $\{v_i\}$ are only flat for $\nabla$ if $g$ is constant. Indeed,
$$\nabla dv_i=\nabla g_{ij}du^j=\frac{\partial g_{ij}}{\partial u^k} du^k\otimes du^j. $$

\begin{remark} [Semi-flat hyperk\"ahler metrics] The two special K\"ahler structures on $B$ define semi-flat hyperk\"ahler structures on $T^*B$ and on $TB$. In the coordinates $\{u^i\}$ if we denote by $\{p_i\}$ the canonical coordinates on the cotangent fibres we have 
\begin{align}\label{semiflat cotangent1}
    \dI=\begin{pmatrix}I & 0 \\ 0 & I^* \end{pmatrix},\ \ \ \dJ=\begin{pmatrix}0 &  g^{-1}\\ -g & 0 \end{pmatrix},\ \ \  \dK=\begin{pmatrix}0 & -\omega^{-1} \\ \omega & 0 \end{pmatrix},\\ \label{semiflat cotangent2}
    \omega_\dI=\begin{pmatrix}\omega & 0 \\ 0 & \omega^{-1} \end{pmatrix},\ \ \ \omega_\dJ=\begin{pmatrix}0 & 1 \\ -1 & 0 \end{pmatrix},\ \ \ \omega_\dK=\begin{pmatrix}0 & -I^* \\ I & 0 \end{pmatrix}.
\end{align}
Similarly, in coordinates $\{v_i\}$ and their canonical dual coordinates $\{q^i\}$ we have another hyperk\"ahler structure $\hat{\dI}, \hat{\dJ},\hat{\dK}, \omega_{\hat{\dI}}, \omega_{\hat{\dJ}}, \omega_{\hat{\dK}}$ in the same block-diagonal form but with $g,\omega$ and $I$ replaced by $\hat{g}=g^{-1}$, $\hat{\omega}=-\omega^{-1}$ and $\hat{I}=-I^*$. Rewriting in coordinates $\{u^i, q^i\}$ we find a semi-flat hyperk\"ahler structure on $TB$ given by
\begin{align}\label{semiflat tangent1}
\hat{\dI}=\begin{pmatrix}I & 0 \\ 0 & -I \end{pmatrix},\ \ \ \hat{\dJ}=\begin{pmatrix}0 &  1\\ -1 & 0 \end{pmatrix},\ \ \  \hat{\dK}=\begin{pmatrix}0 & I \\ I & 0 \end{pmatrix},\\ \label{semiflat tangent2}
    \omega_{\hat{\dI}}=\begin{pmatrix}\omega & 0 \\ 0 & -\omega \end{pmatrix},\ \ \ \omega_{\hat{\dJ}}=\begin{pmatrix}0 & g \\ -g & 0 \end{pmatrix},\ \ \ \omega_{\hat{\dK}}=\begin{pmatrix}0 & \omega \\ \omega & 0 \end{pmatrix}.
\end{align}
This hyperk\"ahler structure appears also in Hitchin's paper \cite{hitchinCplxLag} where he shows that it is obtained by the hyperk\"ahler Legendre transform \cite{hklr}.
\end{remark}

 Let $\pi: M\ra B$ be an algebraic integrable system and let $(I,\omega,\nabla)$ be the induced special K\"ahler structure on the base $B$. Let $\{x^i,y_i\}$ be flat coordinates on $B$ such that $\{dx^i,dy_i\}$ is a symplectic frame of the monodromy lattice $\Gamma\subset T^*B$ with respect to the fiberwise polarization. Then, 
 $$\omega = \frac{1}{d_i}dy_i\wedge dx^i. $$
  where $d_i\in \dN$ with $d_i|d_{i+1}$ for $i=1,..,n-1$ give the type of the polarization. We call these coordinates flat Darboux coordinates as in the corresponding frame the polarization is in standard form. 
 
 The special K\"ahler structure on $B$ induces an affine structure which can also be understood as the monodromy of the flat connection $\nabla$. The transition functions between flat Darboux coordinate charts are constant affine transformations
 $$\begin{pmatrix}x'\\ y' \end{pmatrix}=A\begin{pmatrix}x \\ y \end{pmatrix}+c.$$
  such that $A$ is a transformation between frames of the lattice $\Gamma$, thus the K\"ahler form is preserved. In conclusion, $A$ lies in the group
 $$Sp(\Gamma^\vee)=\{A\in SL(2n,\dZ)\ | \ A^T\omega A=\omega \}, $$
 the symplectic group associated to the dual lattice. Indeed, if $d=\prod_i d_i$ then $-d\cdot \omega$ is the polarization dual to $E_b$ on the dual lattice $\Gamma^\vee\subset TB$. 

The dual lattice also defines a special K\"ahler structure on the base via the Legendre transform. If $B$ in special coordinates is defined via transition functions $(A_{\alpha\beta})$, then the Legendre transformed affine structure in the new coordinates have transition functions $((A_{\alpha\beta}^{-1})^T)$. These do not lie in $Sp(\Gamma^\vee)$ anymore, but in $Sp(\Gamma)$ defined as
$$Sp(\Gamma)=\{A\in SL(2n,\dZ)\ | \ A^T\omega^{-1}A=\omega^{-1}\}. $$
The following proposition is clear from the proof of Theorem \ref{Leg1}.

\begin{proposition}[Legendre transform 2.]\label{Leg2}
Let $B$ be the base of an algebraic integrable system with induced special K\"ahler structure $(g,I,\omega,\nabla)$. Then the connection $\hat{\nabla}$ of the dual special K\"ahler structure has monodromy representation $\hat{\rho}:\pi_1(\hat{B})\ra Sp(\Gamma)$ dual to the monodromy representation of $\nabla$.
\end{proposition}
The two groups $Sp(\Gamma)$ and $Sp(\Gamma^\vee)$ coincide and equal $Sp(2n,\dZ)$ if and only if $\omega$ is a principal polarization, that is when its type is $d=(1,...,1)$. Then the two monodromy representations are again equivalent via conjugation by $\omega$.

\subsection{T-duality of semi-flat hyperk\"ahler structures}
In this section, we show that on a T-dual pair of algebraic integrable systems, the T-dual of a semi-flat hyperk\"ahler structure is again a semi-flat hyperk\"ahler structure. We moreover show that the corresponding transformation on the base is the Legendre transform of special K\"ahler structures.

Let us view the integrable system $\pi: M\ra B$ as an affine torus fibration with monodromy local system $\Gamma \subset T^*B\cong V$ and torsion Chern class $c\in \HH^2(B,\Gamma)$. Let $A$ be a flat connection and endow $M$ with the corresponding semi-flat hyperk\"ahler structure (\ref{semiflat cotangent1}), (\ref{semiflat cotangent2}).

The dual lattice $\Gamma^\vee\ra B$ together with a torsion class $\hat{c}\in \HH^2(B,\Gamma^\vee)$ defines another affine torus fibration $\hat{\pi}:\hat{M}\ra B$. The vertical bundle of $\hat{M}$ is $TB$ and the Gauss-Manin connection is the Legendre transformed connection $\hat{\nabla}$ on $B$. Let $\hat{A}$ be a connection on $\hat{M}$ and endow $\hat{M}$ with the corresponding semi-flat hyperk\"ahler structure. In coordinates corresponding to coordinates on $TB$ the semi-flat hyperk\"ahler structure on $\hat{M}$ is given by (\ref{semiflat tangent1}), (\ref{semiflat tangent2}).

In complex structure $\hat{\dI}$ the map $\hat{\pi}$ is holomorphic and the fibres are compact Lagrangian submanifolds with respect to the holomorphic symplectic structure $\omega_{\hat{\dJ}}+i\omega_{\hat{\dK}}$. Moreover, over any point, $b\in B$ the fibre $\hat{M}_b$ and $M_b$ are torsors under dual tori. 

Given a family of polarizations on the fibres, $\hat{M}$ becomes an algebraic integrable system. The fiber $\hat{M}_b$ of $\hat{\pi}$ over $b\in B$ is a complex torus $\hat{M}_b\cong V^*_b/\Gamma^\vee_b$, where $M_b\cong V/\Gamma$ is the fiber of $\pi$ over $B$. The restriction of $\omega_{\hat{\dI}}$ to the fiber $\hat{M}_b$ is an invariant K\"ahler form. Then $-\omega_{\hat{\dI}}|_{\hat{M}_b}$ can be viewed as an alternating bilinear pairing on $V^*_b$ compatible with the complex structure inducing a positive definite Hermitian structure. That is, a smoothly varying family of inner products $\{\hat{E}_b\}_{b\in B}$. Note, however, that these do not take integral values on the lattice $\Gamma^\vee$. Indeed for any $b\in B$ we have $-\omega_{\hat{\dI}}|_{\hat{M}_b}=\omega$. In the basis provided by the fiber coordinates it is given by
$$\hat{E}_b= \begin{pmatrix}0 & D^{-1}\\ -D^{-1} & 0 \end{pmatrix},$$
where $D=diag(d_1,...,d_n)$ as before. This inner product is not integral but it can be scaled to be integral.

Let us scale $E_b$ to get polarized fibres. If $d=d_1\cdot...\cdot d_n$ then the matrix
$$d\cdot E_b=\begin{pmatrix}0 & dD^{-1}\\ -dD^{-1} & 0 \end{pmatrix}$$
defines the \emph{dual polarization} on the fibres of $\hat{M}$ as defined in \cite{BLpaper}.

We can also scale the special K\"ahler structure $(B,I,d\cdot g, d\cdot \omega,\hat{\nabla})$. The associated hyperk\"ahler structure in the splitting above is given by
\begin{align}
    \hat{\dI}^d=\hat{\dI}=\begin{pmatrix}\hat{I} & 0 \\ 0 & \hat{I}^* \end{pmatrix},\ \ \ \hat{\dJ}^d=\begin{pmatrix}0 &  \frac{1}{d}\hat{g}^{-1}\\ -d\hat{g} & 0 \end{pmatrix},\ \ \  \hat{\dK}^d=\begin{pmatrix}0 & -\frac{1}{d}\hat{\omega}^{-1} \\ d\hat{\omega} & 0 \end{pmatrix},\\
    \omega_{\hat{\dI}^d}=\begin{pmatrix}d\hat{\omega} & 0 \\ 0 & \frac{1}{d}\hat{\omega}^{-1} \end{pmatrix},\ \ \ \omega_{\hat{\dJ}^d}=\omega_{\hat{\dJ}}=\begin{pmatrix}0 & 1 \\ -1 & 0 \end{pmatrix},\ \ \ \omega_{\hat{\dK}^d}=\omega_{\hat{\dK}}=\begin{pmatrix}0 & -\hat{I}^* \\ \hat{I} & 0 \end{pmatrix}.
\end{align}
Since only the metric and the K\"ahler form is changed on the base, the holomorphic symplectic structure of $\hat{M}$ is the same in the new hyperk\"ahler structure. Denote by $\hat{M}^d$ the integrable system with the scaled metric, which now has a smoothly varying family of polarizations on the fibres. We saw that the fibres $M_b$ and $\hat{M}^d_b$ are torsors under dual tori and now we see that they are torsors under dual \emph{polarized} tori.

In the next theorem, we will show that the Legendre transformed semi-flat hyperk\"ahler structures are T-dual to each other. Under T-duality the volume of the fiber inverts, which can be seen from the fact that the $\{\hat{E}_b\}_{b\in B}$ do not induce polarizations on the fibers unless the fibers have volume 1. The relationship between Legendre transforms and T-duality has been observed by Hitchin already in \cite{hitchinCplxLag}.

Consider the usual T-duality diamond
\[
\begin{tikzcd}
    & M\times_B\hat{M} \arrow{dl}[swap]{p} \arrow{dr}{\hat{p}} \arrow{dd}{q} & \\
    M \arrow{dr}[swap]{\pi} & & \hat{M} \arrow{dl}{\hat{\pi}}\\
    & B &
\end{tikzcd}
\]
We have seen already that $(M,0)$ and $(\hat{M},0)$ are T-dual in the sense of generalized geometry, by setting $P=\langle \hat{p}^*\hat{A}\wedge p^*A\rangle$. In Baraglia's terminology \cite{B2} $(A,\hat{A},0)$ is a T-duality triple. We can show the following.
\begin{theorem}\label{semiflat tdual}
    Let $(M,A)$ and $(\hat{M},\hat{A})$ be as above. Then, via $P=\langle \hat{p}^*\hat{A}\wedge p^*A\rangle$, the semi-flat hyperk\"ahler structure on $M$ is T-dual to the semi-flat hyperk\"ahler structure on $\hat{M}$.
\end{theorem}
\begin{proof}
    We work locally using coordinates associated to the flat connections. Over an open set $U\subset B$ let $\{x^i,p_i\}$ be dual coordinates on $M$ spanning the vertical and horizontal distributions. Let $\{y_i\}$ be the Legendre transformed flat coordinates on $B$ as described in the proof of \ref{Leg1}. Then if $\{y_i,\hat{p}^i\}$ are dual coordinates $d\hat{p}^i$ is a frame of $\Gamma^\vee$ and $\{y_i,\hat{p}^i\}$ are coordinates on $\hat{M}$ spanning the vertical and horizontal distributions of $\hat{A}$. The flat connections are defined by the choice of 1-periodic fiber coordinates, so we may change the base coordinates to $\{x^i\}$.

    We may define 1-periodic coordinates on $M\times_B\hat{M}$ by pulling back $p_i$ and $\hat{p}^i$ which defines a flat connection on the correspondence space. Let us use coordinates $\{x^i,p_i,\hat{p}^i\}$ on $M\times_B\hat{M}$. We have
    $$P=d\hat{p}^i\wedge dp^i.$$
    In these coordinates an invariant section of $TM\oplus T^*M$ is given by
    $$X+\xi=X_b^i\frac{\partial}{\partial x^i}+X^f_i\frac{\partial}{\partial p_i}+\xi^b_idx^i+\xi_f^idp_i,$$
    and via the T-duality map (\ref{Tmap}) it becomes
    $$T(X+\xi)=X_b^i\frac{\partial}{\partial x^i}+\xi_f^i\frac{\partial}{\partial \hat{p}^i}+\xi^b_idx^i+X^f_id\hat{p}_i.$$
    That is, if we use the connections to split the tangent and cotangent bundles, we have
    $$E_{red}\cong TB\oplus T^*B\oplus T^*B\oplus TB,\ \ \ \text{and}\ \ \ \hat{E}_{red}=TB\oplus TB\oplus T^*B\oplus T^*B$$
    and the T-duality map is just swapping the second and third components.

    To calculate the T-dual of the semi-flat hyperk\"ahler structure $(\dI,\dJ,\dK,\omega_\dI,\omega_\dJ,\omega_\dK)$ we have to promote it to a generalized hyperk\"ahler structure $(\cJ_\dI,\cJ_\dJ,\cJ_\dK,\cJ_{\omega_\dI},\cJ_{\omega_{\dJ}},\cJ_{\omega_{\dK}})$ as in (\ref{gen quat1}). The corresponding reduced $+i$ eigenbundles are as follows.
    \begin{align*}
    L_\dI&=\{(X-iIX,\xi-iI^*\xi,\eta+iI^*\eta,X+iIX)\ |\ X,Y\in TB,\ \xi,\eta\in T^*B\},\\
    L_{\omega_\dI}&=\{(X,\xi,-i\omega X,-i\omega^{-1}\xi)\ |\ X\in T_\dC B,\ \xi\in T^*_\dC B\},\\
    L_\dJ&=\{ (X-ig^{-1}\xi,\xi+igX,\eta-igY,Y+ig^{-1}\eta) \ |\ X,Y\in TB, \ \xi,\eta\in T^*B \}, \\
    L_{\omega_\dJ}&=\{ (X,\xi,-i\xi,iX) \ |\ X\in T_\dC B, \ \xi\in T^*_\dC B \},\\
    L_\dK&=\{ (X+i\omega^{-1}\xi,\xi-i\omega X,\eta-i\omega Y,Y+i\omega^{-1}\eta) \ |\ X,Y\in TB, \ \xi,\eta\in T^*B \},\\
    L_{\omega_\dK}&=\{ (X, \xi, iI^*\xi, -iIX)\ |\ X\in T_\dC B, \ \xi\in T^*_\dC B \}.
    \end{align*}
    We can do the same for the semi-flat hyperk\"ahler structure $(\hat{\dI},\hat{\dJ},\hat{\dK},\omega_{\hat{\dI}},\omega_{\hat{\dJ}},\omega_{\hat{\dK}})$ on $\hat{M}$. We have to use the form (\ref{semiflat tangent1}), (\ref{semiflat tangent2}) so that T-duality is just swapping components of the split. The reduced $+i$ eigenbundles in $\hat{E}$ are
    \begin{align*}
    \hat{L}_{\hat{\dI}}&=\{(X-iIX,Y+iIY,\xi+iI^*\xi,\eta-iI^*X)\ |\ X,Y\in TB,\ \xi,\eta\in T^*B\},\\ 
    \hat{L}_{\omega_{\hat{\dI}}}&=\{(X,Y,-i\omega X,i\omega Y)\ |\ X\in T_\dC B,\ \xi\in T^*_\dC B\},\\
    \hat{L}_{\hat{\dJ}}&=\{ (X+iY,-Y+iX,\xi+i\eta,-\eta+i\xi) \ |\ X,Y\in TB, \ \xi,\eta\in T^*B \}, \\ 
    \hat{L}_{\omega_{\hat{\dJ}}}&=\{ (X,Y,-igY,igX) \ |\ X\in T_\dC B, \ \xi\in T^*_\dC B \},\\
    \hat{L}_{\hat{\dK}}&=\{ (X,-iIX,\xi, iI^*\xi) \ |\ X\in T_\dC B, \ \xi\in T^*_\dC B \},\\
    \hat{L}_{\omega_{\hat{\dK}}}&=\{ (X, Y, -i\omega X, -i\omega Y )\ |\ X\in T_\dC B, \ \xi\in T^*_\dC B \}.
    \end{align*}
    Indeed under $T$ we have
    \begin{align}\label{BBI}
        T(L_\dI)=\hat{L}_{\hat{\dI}},\ \ T(L_{\omega_\dI})=\hat{L}_{\omega_{\hat{\dI}}}, \\ \label{BAJ}
        T(L_\dJ)=\hat{L}_{\omega_{\hat{\dJ}}},\ \ T( L_{\omega_\dJ})=\hat{L}_{\hat{\dJ}},\\ 
        \label{BAK}
        T(L_\dK)=\hat{L}_{\omega_{\hat{\dK}}},\ \ T(L_{\omega_\dK})=\hat{L}_{\hat{\dK}}.
    \end{align}
\end{proof}
\begin{remark}
Even though the two semi-flat hyperk\"ahler structures are T-dual, two of the complex structures map to symplectic structures (\ref{BAJ}) and (\ref{BAK}). This is the switching of type described in \cite[Table 1]{cavalcantiTdual} and the switching between $A$ and $B$-models in physics. In our case, the switching is between $BAA$ and $BBB$ models or between $ABA$ and $AAB$ models. The former was the base of Kapustin and Witten's treatment of the geometric Langlands program in \cite{KW}.
\end{remark}

We have seen in Theorem \ref{Leg2} the connections $\nabla$ and $\hat{\nabla}$ have dual monodromy, and they are only equivalent when the polarizations on the fibers are \emph{principal}, that is when the fibers have volume 1. 

Let $\pi:M\ra B$ be an algebraic integrable system with principally polarized fibers. Then, the cohomology class $\alpha\in \HH^2(B,\dR)$ restricted to the fibers induces an isomorphism
$$E:\Gamma\ra \Gamma^\vee.$$
Let $c\in \HH^2(B,\Gamma)$ be the Chern class of $M$ and define $\hat{c}:=E(c)\in \HH^2(B,\Gamma^\vee)$. Then, by Theorem \ref{iso of afft} the affine torus bundle $\hat{M}$ defined by $\Gamma^\vee$ and $\hat{c}$ is isomorphic to $M$. Let $A$ be a connection on $M$ and $\hat{A}$ the pullback of $A$ via the isomorphism on $\hat{M}$. Then we have the following. 
\begin{theorem}\label{semiflat tdual principal pol}
Let $M$ be a principally polarized algebraic integrable system. Then, $M$ is self T-dual and T-duality of the semi-flat hyperk\"ahler structure is hyperk\"ahler rotation.  
\end{theorem}
\begin{proof}
 Let $(\Gamma,c)$ be the local system and Chern class associated to $M$ and denote by $\hat{M}$ the affine torus bundle corresponding to $(\Gamma^\vee,E(c))$. Let $\{x^i,p_i,p^i\}$ be dual coordinates on $M\times_B\hat{M}$ as in the proof of \ref{semiflat tdual}. Then, the induced isomorphism $\phi:M\ra \hat{M}$ in coordinates is given by
 $$\phi:(x^i,p_i)\mapsto (x^i,E(p_i))$$
 where $E$ is the matrix representation of the polarization. In particular, in dual flat coordinates on $M$ we have $E=-\omega^{-1}$. Since the coordinates are flat, the differential of $\phi$ is 
 \begin{align*}
     D\phi=\phi_*=\begin{pmatrix} 1 & 0 \\ 0 & -\omega^{-1} \end{pmatrix},\ \ \phi^*=(D\phi)^{-1}=\begin{pmatrix} 1 & 0 \\ 0 & - \omega\end{pmatrix}.
 \end{align*} 
 Then, the pullback of the T-dual semi-flat hyperk\"ahler structure (\ref{semiflat tangent1}), (\ref{semiflat tangent2}) is as follows.
 \begin{align*}
     \phi^*\hat{\dI}=\phi^{-1}_*\hat{\dI}\phi_*=\begin{pmatrix}I & 0 \\ 0 & I^*\end{pmatrix},\ \ \ \phi^*\hat{\dJ}=\begin{pmatrix}0 & -\omega^{-1}\\ \omega & 0 \end{pmatrix},\  \ \ \phi^*\hat{\dK}=\begin{pmatrix} 0 & -g^{-1}\\ g & 0\end{pmatrix}\\
     \phi^*\omega_{\hat{\dI}}=(\phi_*)^T\omega_{\hat{\dI}}\phi_*=\begin{pmatrix}\omega & 0 \\ 0 & \omega^{-1}\end{pmatrix},\ \ \ \phi^*\omega_{\hat{\dJ}}=\begin{pmatrix} 0 & -I^* \\ I & 0\end{pmatrix} ,\ \ \ \phi^*\omega_{\hat{\dK}}=\begin{pmatrix}0 & -1\\ 1 & 0 \end{pmatrix}.
 \end{align*}
 In conclusion, the pullback of the T-dual semiflat hyperk\"ahler structure is precisely the hyperk\"ahler rotation $\dI\mapsto \dI$, $\dJ\mapsto \dK$, $\dK\mapsto -\dJ$.
\end{proof}
\begin{example}
    The generic fiber of the Higgs bundle moduli space $\cM(r,d)\ra \cA$  is the Jacobian of the spectral curve which is a principally polarized Abelian variety (Example \ref{jacobian fiber}).  
\end{example}
\begin{remark}
    The fact that T-duality does not map the polarization to a polarization in the non-principally polarized case can be seen in light of Fourier-Mukai transform \cite{mukai}. T-duality preserves the types of generalized complex structures (\ref{BBI}) corresponding to one of the K\"ahler structures. This is supposed to be a generalized geometry shadow of the Fourier-Mukai transform. The Fourier transform of an ample line bundle connected to a non-principal polarization is not an ample line bundle but a vector bundle whose determinant bundle is ample. This is a hint that T-duality of generalized complex structures may be a derivative of T-duality of branes rather than the other way around.
\end{remark}

\subsection{Deformations of holomorphic symplectic structure under T-duality.}
Let $A$ be a flat connection on an algebraic integrable system $\pi:M\ra B$. Recall, from Section \ref{section deformation of symplectic} that if the horizontal distribution of $A$ is not complex Lagrangian then the holomorphic symplectic structure of the semi-flat hyperk\"ahler structure does not agree with the original holomorphic symplectic structure. Instead, we can write the holomorphic symplectic structure of $M$ as a deformation (\ref{nonstandard symplectic}) of the semi-flat holomorphic symplectic structure via a global two-form on the base which is type (1,1)+(0,2).

In this section, we will calculate the T-dual of the deformed holomorphic symplectic structure. We will find that a deformation of the semi-flat holomorphic symplectic structure maps to a B-field transform (\ref{Bfield}) of the T-dual. 

We work locally. Let $U\subset B$ be an open set, $s:U\ra M$ a smooth section such that $s_*(TB)$ spans the horizontal distribution of $A$. Let $\mu: U\ra M$ be a complex Lagrangian section of$M$. Then there exists some $\xi: U\ra T^*U$ such that $\mu=s+\xi$. The connection $A$ splits the tangent bundle of $T^*B$ and the differential of $\xi$ in block diagonal form can be written as
\begin{align*}
    D\xi: T_bU\ra T_bU+T^*_bU,\ \ \ D\xi=\begin{pmatrix}1 \\ F \end{pmatrix},
\end{align*}
where $F-F^T$ gives the coefficients of $d\xi$.

Then, in the split provided by the connection, we can write the holomorphic symplectic structure of $M$ as the deformation of the semi-flat structures $\dI$, $\omega_\dJ$ and $\omega_\dK$ 
\begin{align*}
    \dI^F=\begin{pmatrix} I & 0 \\ I^*F-FI & I^* \end{pmatrix},\ \ \ \omega_{\dJ}^F=\begin{pmatrix}F-F^T & 1\\ -1 & 0\end{pmatrix}, \ \ \ \ \ \ \omega_{\dK}^F=\begin{pmatrix} F^TI-I^*F & -I^* \\ I & 0 \end{pmatrix}.
\end{align*}
We write $\cJ_{\dI^F}$, $\cJ_{\omega_{\dJ}^F}$ and $\cJ_{\omega_{\dK}^F}$ for the corresponding generalized complex structures.
\begin{theorem}\label{local def tdual}
    In the setting of Theorem \ref{semiflat tdual} the T-duals of $\cJ_{\dI^F}$, $\cJ_{\omega_{\dJ}^F}$ and $\cJ_{\omega_{\dK}^F}$ are locally the B-field transforms of $\cJ_{\hat{\dI}}$, $\cJ_{\hat{\dJ}}$ and $\cJ_{\hat{\dK}}$ with B-field
    $$B=\begin{pmatrix}0 & F^T\\ -F & 0 \end{pmatrix}.$$
\end{theorem}
\begin{proof}
    Once again, we calculate the reduced $L_{red}$ for $\cJ_{\omega_{\dJ}^F}$ and $\cJ_{\omega_{\dK}^F}$ which we denote by $L_{\omega_{\dJ}^F}$ and $L_{\omega_{\dK}^F}$. We have
    \begin{align*}
    L_{\omega_{\dJ}^F}&=\{ (X,\xi, -i\xi -iFX+iF^TX, +iX)\ | \ X\in T_\dC B,\ \xi\in T^*_\dC B \}\\
    L_{\omega_{\dK}^F}&=\{ (X, \xi, -iF^TIX+iI^*FX+iI^*\xi, -iIX)\ | \ X\in T_\dC B,\ \xi\in T^*_\dC B \}
    \end{align*}
   The T-duals are given by
    \begin{align*}
        T(L_{\omega_{\dJ}^F})&=\{ (X,iX, -i\xi-iFX+iF^TX, \xi)\ | \ X\in T_\dC B,\ \xi\in T^*_\dC B \}\\
        \text{let }&\ \xi'= \xi+FX\\
        &=\{(X,iX,-i\xi'+iF^TX , \xi'-FX)\ | \ X\in T_\dC B,\ \xi\in T^*_\dC B \}\\
        T(L_{\omega_{\dK}^F})&=\{ (X, -iIX, iI^*\xi-iF^TIX+iI^*FX,\xi )\ | \  X\in T_\dC B,\ \xi\in T^*_\dC B \}\\
        \text{taking }& \ \xi'=iI^*\xi+I^*FX\\
        &= \{ (X,-iIX,\xi'-iF^TIX ,iI^*\xi'-F^TX)\ | \  X\in T_\dC B,\ \xi\in T^*_\dC B \}.
    \end{align*}
     Using that the B-field transform acts on the $+i$ eigenbundles as $L\mapsto e^BL$ we have that 
     $$T(\cJ_{\omega_{\dJ}^F})= e^B\cJ_{\hat{\dJ}}e^{-B},\ \ T(\cJ_{\omega_{\dK}^F})=e^B\cJ_{\hat{\dK}}e^{-B}.$$
     Then,
     $$T(\cJ_{\dI^F})=T(\cJ_{\omega_{\dJ}^F}\cdot \cJ_{\omega_{\dK}^F})=T(\cJ_{\omega_{\dJ}^F})\cdot T(\cJ_{\omega_{\dK}^F})=e^{B}\cJ_{\hat{\dJ}}e^{-B}e^B\cJ_{\hat{\dK}}e^{-B}=e^B\cJ_{\hat{\dI}}e^{-B}.$$
\end{proof}
\begin{remark}
    The B-field transforms of $\cJ_{\omega_{\hat{\dI}}}$, $\cJ_{\omega_{\hat{\dJ}}}$ and $\cJ_{\omega_{\hat{\dK}}}$ extend the triple $T(\cJ_{\dI^F})$, $T(\cJ_{\omega_{\dJ}^F})$ and $T(\cJ_{\omega_{\dK}^F})$ to a local generalized hyperk\"ahler structure on $\hat{X}$. Taking $T^{-1}$ then extends $\cJ_{\dI^F}, \cJ_{\omega_{\dJ}^F},\cJ_{\omega_{\dK}^F}$ to one. This extension is just the semi-flat hyperk\"ahler structure corresponding to a local flat connection in which $\mu$ is a flat complex Lagrangian section. This indeed cannot extend to a global structure if $\mu$ is not a flat section of a global flat connection. The obstruction to extending to a global semi-flat structure can also be seen in the local nature of $F$. If there are local sections $\mu_i$ such that $\mu_i=s_i+\xi_i$ and $F_i=F_j$ then $\xi_i-\xi_j$ must be constant, and the $\mu_i$ define a flat connection. The interplay between $B$-field transform and deformations of complex structures was also observed in \cite{sawon}
\end{remark}

Even though $F$ is not well-defined globally, $\dI^F$, $\omega_\dJ^F$ and $\omega_\dK^F$ are. In particular, the T-duals are well-defined generalized complex structures and we may write 
\begin{align*}
T(\cJ_{\omega_{\dJ}^F})=e^{B_1}\cJ_{\hat{\dJ}}e^{-B_1},\ \ \ T(\cJ_{\omega_{\dK}^F})=e^{B_2}\cJ_{\hat{\dK}}e^{-B_2}
\end{align*}
for
$$B_1=\begin{pmatrix} 0 & F^T-F\\ 0 & 0 \end{pmatrix} ,\ \ \ B_2=\begin{pmatrix} 0 & I^*FI+F^T \\ 0 & 0\end{pmatrix}.$$
We also have, 
\begin{align*}T&(\cJ_{\dI^F})=\\
&=T(\cJ_{\omega_{\dJ}^F})T(\cJ_{\omega_{\dK}^F})=e^{-B_1}\cJ_{\hat{\dJ}}e^{B_1}e^{-B_2}\cJ_{\hat{\dK}}e^{B_2}=\begin{pmatrix} I & 0 & 0 & 0 \\ 0 & -I & 0 & 0 \\ 0 & I^*F^T-F^TI & -I^* & 0 \\ I^*F-FI & 0 & 0 & I^* \end{pmatrix}.\end{align*}

\section{T-duality of generalized branes}\label{section T-duality of generalized branes}
Let $(M,H)$ and $(\hat{M},H)$ be a T-dual pair in the sense of generalized geometry fibered over the base $B$. The T-duality relationship can be reformulated as an isomorphism (\ref{tdualitymap}) of Courant algebroids on the base reduced from $M$ and $\hat{M}$. We T-dualize an invariant generalized complex structure on $M$ by reducing its $+i$ eigenbundle and applying the T-duality map (\ref{tdualitymap}). Theorem \ref{thm gil structures} states that the resulting subbundle is the reduction of the $+i$ eigenbundle of a generalized complex structure on $\hat{M}$. We say that these generalized complex structures are T-dual.

We also associate a maximal isotropic subbundle to a generalized brane, its generalized tangent bundle (\ref{gen tangent bundle}). Therefore, we may T-dualize a generalized brane analogously, by first reducing then using the T-duality map and finally proving that the resulting subbundle is the reduction of a generalized tangent bundle. There are several obstructions to this program and the rest of this chapter is dedicated to understanding to what degree it can be carried out. Some consequences of the T-duality relation for generalized complex branes have also been explored by Ben-Bassat in \cite{BenBassat}.

The first obstruction is that T-duality applies to invariant structures. Therefore, we must restrict our attention to generalized branes which are invariant in a way. We will use the following definition.
\begin{definition}\label{def invariant brane}
    Let $\pi:M\ra B$ be an affine torus bundle and $H\in \Omega^3(M)$ a closed invariant three-form. Let $\cL=(S,F)$ be a generalized brane of $(M,H)$. We say that $\cL$ is \emph{invariant} if there exist invariant sections of $TM\oplus T^*M$ which span $\tau_\cL$ over $S$.
\end{definition}
If $\cL$ is an invariant generalized brane, then the invariant sections spanning $\tau_\cL$ define a subbundle
$$(\tau_\cL)_{red}\subset E_{red}.$$
\begin{definition} \label{def tdual gg branes}
Let $(M,H)$ and $(\hat{M},\hat{H})$ be a T-dual pair in the sense of Definition \ref{gen geom tdual}. Let $\cL=(S,F)$ and $\hat{\cL}=(\hat{S},\hat{F})$ be invariant generalized branes on $M$ and $\hat{M}$. We say that $\cL$ and $\hat{\cL}$ are \emph{T-dual} if 
$$T((\tau_{\cL})_{red})=(\tau_{\hat{\cL}})_{red}.$$
\end{definition}
This definition applies in a very general setting but we will only consider T-dual pairs of affine torus bundles with torsion Chern classes endowed with zero $H$-flux. Similarly to the construction of the T-duality map (\ref{tdualitymap}) we will construct T-dual branes by pulling the generalized tangent bundle back to the correspondence space, transforming via a two-form and then pushing down.

In the next section, we first describe the general construction of a T-dual brane. Then, we will determine a class of invariant generalized complex branes to which our method can be applied. We call these generalized complex branes \emph{locally T-dualizable}. In the last two sections, we determine the T-duals of a locally T-dualizable generalized complex brane first on a topologically trivial affine torus bundle and then in a general setting.

\subsection{General construction}\label{construction of tdual gc branes}

Let $\pi: M\ra B$ be an affine torus bundle with torsion Chern class $c$ and let $\hat{\pi}:\hat{M}\ra B$ be a T-dual of it in the sense of Example \ref{Tdual torsion coordinates}. That is, $\Gamma_{\hat{M}}\cong \Gamma_M^\vee$ and $\hat{c}$ is also torsion. Write $p:M\times_B\hat{M}\ra M$ and $\hat{p}:M\times_B\hat{M}\ra \hat{M}$ for the projections. Let $A$ and $\hat{A}$ be flat connections on $M$ and $\hat{M}$ and let $P=\langle \hat{p}^*\hat{A}\wedge p^*A\rangle \in \Omega^2(M\times_B\hat{M})$ be the associated closed two-form.

Our plan of constructing a T-dual for a generalized brane $\cL=(S,F)$ is the following. Let $i: S\times_{\pi(S)}\hat{M}\xhookrightarrow{} M\times_B \hat{M}$ be the inclusion. Let $p_S$ be the restriction of $p:M\times_B\hat{M}\ra M$ to $S\times_{\pi(S)}\hat{M}$ and $\hat{p}_S$ the restriction of $\hat{p}:M\times_B \hat{M}\ra \hat{M}$, 
\[
\begin{tikzcd}
   & S\times_{\pi(S)}\hat{M}\arrow{dl}[swap]{p_S} \arrow{dr}{\hat{p}_S} &\\
   S & & \hat{M}|_{\pi(S)}
\end{tikzcd}.
\]

Then, $$p_S^*F+i^*P\in \Omega^2(S\times_{\pi(S)}\hat{M})$$ is a closed two form. We have the short exact sequence
\begin{equation}\label{relativeSES}
\begin{tikzcd}
0 \arrow{r} &  ker(\hat{p}_S) \arrow{r} & T(S\times_{\pi(S)}\hat{M})  \arrow{r}{(\hat{p}_S)_*} & T\hat{M} \arrow{r} & 0
\end{tikzcd}
\end{equation}
and its dual
\begin{equation}\label{relative diff SES}
\begin{tikzcd}
0 \arrow{r} &  T^*\hat{M} \arrow{r}{(\hat{p}_S)^*} & T^*(S\times_{\pi(S)}\hat{M})  \arrow{r}{q} &  \Omega^1_{S\times_{\pi(S)}\hat{M}/\hat{M}}\arrow{r} & 0
\end{tikzcd}
\end{equation}
where $\Omega^1_{S\times_{\pi(S)}\hat{M}/\hat{M}}$ is the space of relative differentials.

Consider now the linear map
\begin{align}\label{F+P}
    (F+P)^1:= q\circ(p_S^*F-i^*P): T(S\times_{\pi(S)}\hat{M})\ra T^*(S\times_{\pi(S)}\hat{M}) \ra  \Omega^1_{S\times_{\pi(S)}\hat{M}/\hat{M}}
\end{align}
and let \begin{align}\label{Delta}\Delta:=ker(F+P)^1\subset T(S\times_{\pi(S)}\hat{M})\end{align} be its kernel. In Section \ref{section delta} we will find a class of generalized branes, with $\pi(S)$ contractible, such that the distribution $\Delta$ is integrable with closed leaves. We will call these branes \emph{locally T-dualizable}.

Given, that the distribution $\Delta$ is integrable with closed leaves, we consider a closed leaf $Z$ of the induced foliation. Then $\hat{p}|_Z: Z\ra \hat{M}$ has closed image which we denote by $\hat{S}$. In this setting, we have the following diagram.
\begin{equation}\label{big diagram}
\begin{tikzcd}
& Z \arrow{dl}[swap]{p_Z} \arrow{dr}{\hat{p}_Z} \arrow[hook]{rrr}{i_Z} &  & & M\times_{\pi(S)}\hat{M} \arrow{dl}[swap]{p} \arrow{dr}{\hat{p}} \\
S\arrow[hook, bend right =30]{rrr}{i_S} & & \hat{S} \arrow[hook, bend right =20]{rrr}{i_{\hat{S}}} & \pi^{-1}(\pi(S)) & & \hat{\pi}^{-1}(\pi(S))
\end{tikzcd}
\end{equation}
In Section \ref{section local tdual}, we show that for a locally T-dualizable generalized brane with $\pi(S)$ contractible, and for any $\hat{S}$ fitting into the diagram \ref{big diagram} there exists a unique closed two-form $\hat{F}\in \Omega^2(\hat{S})$ such that
$$p_z^*F+i_Z^*P=\hat{p}^*\hat{F}. $$ 
Finally, we will show that $\cL=(S,F)$ and $\hat{\cL}=(\hat{S},\hat{F})$ are T-dual generalized submanifolds. This is the content of Theorem \ref{local gg thm}. 

In the last section, we consider generalized branes without the assumption that $\pi(S)$ is contractible. We show that being locally T-dualizable is not enough for a global T-dual to exist. We give the precise constraints in Theorem \ref{global thm1}.

Generalized branes are only the de Rham shadows of physical branes which are submanifolds endowed with Hermitian vector bundles with connections. We can also view these bundles as principal $U(d)$-bundles with connections. We eventually, in Chapter \ref{last chapter}, will generalize the above construction to rank 1 branes, that is to $U(1)$-bundles with connections as follows. Let $L\ra S$ be a $U(1)$-bundle with a connection $\nabla$ whose curvature is $2\pi i F$. On the correspondence space, $M\times_B\hat{M}$ there exists a $U(1)$-bundle with connection twisted by a gerbe, the \emph{Poincar\'e bundle} $\cP$, whose curvature is $2\pi i P$. On a leaf $Z$ of $\Delta$ let us define the $U(1)$-bundle $\hat{L}_Z$ with connection $\hat{\nabla}_Z$ via the equation
\begin{align} \hat{L}_Z \cong p_Z^*L\otimes i_Z^*\cP^*.\end{align}

Via the standard representation of $U(1)$ on $\dC$, $\hat{L}_Z$ defines a Hermitian line bundle on $S\times_{\pi(S)}\hat{M}$ and $\hat{\nabla}_Z$ a Hermitian connection. We can decompose it as 
$$\hat{\nabla}_Z=\hat{\nabla}^1_Z+\hat{\nabla}^2_Z$$
such that 
$$\hat{\nabla}^1_Z:\Gamma(\hat{L}_Z)\ra \Gamma(\hat{L}_Z\otimes \Omega^1_{S\times_{\pi(S)}\hat{M}/\hat{M}})$$ and $$\hat{\nabla}_Z^2:\Gamma(\hat{L}_Z)\ra \Gamma(\hat{L}_Z\otimes (\hat{p}_S)^*T^*\hat{M}).$$ 
We show that there exists a leaf $Z$ of $\Delta$ such that $\hat{\nabla}^1_Z$ is trivial and we define $\hat{S}:=\hat{p}_Z(Z )$, as before and 
\begin{align}\label{naive T-dual}\hat{L}:=(\hat{p}_Z)_*ker(\hat{\nabla}_Z^1).\end{align}
The connection $\hat{\nabla}_Z^2$ induces a connection $\hat{\nabla}$ on $\hat{L}$. 

Taking flat sections is the original method of Arinkin and Polishchuk \cite{AP} for Lagrangian sections which was extended to Lagrangian affine torus subbundles by Bruzzo Marelli and Pioli in \cite{BMP1, BMP2} using factors of automorphy. Finally, Glazebook, Jardim and Kamber \cite{GJK} described the Fourier-Mukai transform of vector bundles which are flat only on the fibers by restricting to a leaf as we are here.

It turns out that the full picture for $U(1)$-bundles which are not flat on the fibers of $\pi: M\ra B$ is more complicated. In \cite{CLZ} Chan, Leung and Zhang gave a distinct construction for such bundles. Our method is simpler but also more restrictive in terms of application.  Chapter \ref{last chapter} will explain the necessary modifications to the method above.

\subsection{The integrability of the distribution $\Delta$}\label{section delta}
Our construction relies on the distribution $\Delta$ being integrable with closed leaves. This restricts the geometry of $S$ significantly, but we still end up with an interesting class of examples. In this section, we define this class and show that the corresponding distribution $\Delta$ is integrable with closed leaves.

Let $S\subset M$ be a submanifold such that $\pi(S)\subset B$ is a submanifold and $S$ intersects the fibers of $M|_{\pi(S)}=\pi^{-1}(\pi(S))$ in equidimentional affine subtori. Then, for any $b\in \pi(S)$, via the inclusion $S\ra M$ restricted to the fiber $$\HH_1(S_b,\dZ)\subset \HH_1(M_b,\dZ)$$ defines a sublattice. Since the fibers of $S$ are closed in the fibers of $M$ this sublattice is locally trivial. Therefore we have an inclusion of local systems 
$$\Gamma_S\subset \Gamma_M|_{\pi(S)}$$
over $\pi(S)$. That is, $S$ is a \emph{affine torus subbundle} of $M$. Kamenova and Verbitsky \cite[Theorem 3.2]{kamenovaVerbitsky} showed that in an integrable system over regular points of $\pi_S$ any holomorphic Lagrangian is of this form, with possibly disconnected fibers. Here we restrict our attention to connected fibers. 

Suppose that $S\subset M$ is an affine torus subbundle. Suppose that $B=\pi(S)$ for ease of notation, otherwise we restrict $M$ to $\pi(S)$. We have a short exact sequence of lattices 
\begin{equation}\label{SES lattice sub affine torus bundle}
\begin{tikzcd}
0 \arrow{r} &  \Gamma_S \arrow{r} & \Gamma_M \arrow{r} &  \Gamma_M/\Gamma_S \arrow{r} & 0,
\end{tikzcd}
\end{equation}
and that of flat vector bundles
\begin{equation}\label{SES vertical sub affine torus bundle}
\begin{tikzcd}
0 \arrow{r} &  V_S \arrow{r} & V_M \arrow{r} &  V_M/V_S \arrow{r} & 0 .
\end{tikzcd}
\end{equation}
Then $S$ is a torsor under $S_0=V_S/\Gamma_S$ and $M$ under $M_0=V_M/\Gamma_M$. The short exact sequences (\ref{SES lattice sub affine torus bundle}) and (\ref{SES vertical sub affine torus bundle}) induce an inclusion $S_0\ra M_0$. If $\pi(S)$ is simply connected a choice of sections $s_S:\pi(S)\ra S$ and $s_M:\pi(S)\ra M$ induce isomorphisms $M\cong M_0$ and $S\cong S_0$. On the other hand, if $s_S$ and $s_M$ do not agree, viewed as sections of $M$, the inclusion $S\ra M$ does not agree with $S_0\ra M_0$. Let us denote by $t_b$ the translation along the fibers of $M$ by a section $b$ of $M_0$. We have the following commutative diagram.
\begin{equation}\label{diagram S in M}
    \begin{tikzcd}
        S_0 \arrow{r} \arrow{d}{\cong} & M_0\cong M \arrow{d}{t_{s_S-s_M}}\\
        S \arrow{r} & M
    \end{tikzcd}
\end{equation}
Of course, locally we could promote the section $s_S$ to a section of $M$.  On the other hand, T-duality is connected to a choice of flat connection on $M$ which can be given by local trivializations. It is also more intuitive to fix a trivialization of the ambient space and study submanifolds in these fixed coordinates. 

The translation of $S_0$ in $M$ depends only on the image of $s_B-s_M$ in $V_M/V_S$ up to the lattice $\Gamma_M/\Gamma_S$. That is, with a fixed choice of section for $M$ we can still choose the $s_S$ such that $s_S-s_M$ lies in an orthogonal complement of $V_S\subset V_M$ up to an orthogonal complement of the lattice $\Gamma_S\subset \Gamma_M$.

\begin{lemma}\label{invariant}
    Let $S\subset M$ be an affine torus subbundle with $\pi(S)$ simply connected and let $F\in \Omega^2(S)$ be a closed invariant two-form. Then, $\Delta$ is integrable.
\end{lemma}
\begin{proof}
 Let $U=\pi(S)$. Therefore,
$$S\cong U\times T^l $$
and there exist coordinates $[y_1,...,y_k,q^1,...,q^l]$ on $S$ and $[x_1,...,x_m,p^1,...,p^n]$ on $M|_U$ such that $q^i$ and $p^i$ are 1-periodic and the inclusion of $S$ to $M$ in coordinates is
\begin{align*}
    & \ \ \ \ \ i_S: S \xhookrightarrow{} M\\
    [y,q]&\mapsto [\phi(y),q^1,..,q^l,b^{l+1},...,b^n],
\end{align*}
where $b^i: U\ra \dR/\dZ$ are smooth functions.  Indeed, following the commutative diagram (\ref{diagram S in M}) the coordinates on $M|_U$ and $S$ are given by local sections $s_M$ and $s_S$ of $M|_U$ and $S$, and frames of $\Gamma_M|_U$ and $\Gamma_S$ respectively. A frame of $\Gamma_S$ can always be extended to a frame of $\Gamma_M|_U$ and this extension gives an orthogonal decomposition $V_M\cong V_S\oplus (V_M/V_S)$. Then, the difference $s_S$ can be chosen such that  $s_S-s_M=b\in V_M/V_S$ up to $\Gamma_M/\Gamma_S$.

The correspondence space $M\times_U\hat{M}$ has dual coordinates $[x_1,..,x_n,p^1,..,p^n,\hat{p}_1,..,\hat{p}_n]$ as in Example \ref{Tdual torsion coordinates} and the inclusion is given by
\begin{align*}
   i: S\times_{U}\hat{M} &\xrightarrow{}  M\times_U\hat{M}\\
   [y,q,\hat{p}]&\mapsto [\phi(y),q,b,\hat{p}].
\end{align*}
In these coordinates $P= d\hat{p}_i\wedge dp^i$ and therefore $i^*P$ is given by
$$i^*P=\sum_{\mu=1}^l d\hat{p}_\mu\wedge dq^\mu+\sum_{\mu=l+1}^nd\hat{p}_\mu\wedge db^\mu. $$
Also, if $p_S:S\times_U \hat{M}\ra S$ is the projection we have
$$p_S^*F=\sum_{i,j=1}^kF_{ij}dy_i\wedge dy_j +\sum_{\substack{i=1,..k\\ \mu=1,...l}} G_{i\mu}dy_i\wedge dq^\mu + \sum_{\mu,\nu=1}^l H_{\mu \nu}dq^\mu \wedge dq^\nu . $$
In these coordinates, $p_S^*F+i^*P$ can be written in matrix form as 
\begin{align*}
   p_S^*F+i^*P= \begin{pmatrix}
    F_{ij} & -G^T_{\mu i} & -(\partial_ib^\mu)^T & -(\partial_ib^\mu)^T \\
    G_{i\mu} & H_{\mu\nu} & -\dI_{\mu\nu} & 0 \\
    \partial_ib^\mu & \dI_{\mu\nu} & 0 & 0\\
    \partial_ib^\mu & 0 & 0 & 0
    \end{pmatrix}
\end{align*}
where the blocks are made up of $(k,l,n-l , n-l)$ columns and $\dI$ is the  $l\times l$ identity matrix.

A vector field $X\in \Delta$ is in the kernel of $(F+P)^1$ if it satisfies
\begin{align}\label{being in Delta}GX^b+HX^f=-\begin{pmatrix} \dI & 0 \end{pmatrix} \hat{X}^f \end{align}
where $X^b, X^f$ and $\hat{X}^f$ are the components of $X$ in the basis $\{\frac{\partial}{\partial y_i}, \frac{\partial}{\partial q^\mu} \frac{\partial}{\partial \hat{p}_\nu}\}$ . That is,
$$\Delta=span\Big\{\frac{\partial}{\partial y_i}-G_{i\mu}\frac{\partial}{\partial \hat{p}_\mu},\ i=1,..k; \ \ \frac{\partial}{\partial q^\mu}-H_{\mu\nu}\frac{\partial}{\partial \hat{p}_\nu},\ \mu=1,...l;\ \ \frac{\partial}{\partial \hat{p}_\mu},\ \mu=l+1,...,n  \Big\}.$$
Let us calculate the Lie brackets of this frame. Firstly, since $H$ and $G$ only depend on $q$ and $y$
$$\Big[\frac{\partial}{\partial y_i}-G_{i\mu}\frac{\partial}{\partial \hat{p}_\mu},\frac{\partial}{\partial \hat{p}_\tau}\Big]=\Big[\frac{\partial}{\partial q^\mu}-H_{\mu\nu}\frac{\partial}{\partial \hat{p}_\nu},\frac{\partial}{\partial \hat{p}_\tau}\Big]=0. $$
The other brackets are 
$$\Big[\frac{\partial}{\partial y_i}-G_{i\mu}\frac{\partial}{\partial \hat{p}_\mu},\frac{\partial}{\partial q^\mu}-H_{\mu\nu}\frac{\partial}{\partial \hat{p}_\nu}\Big]=(\partial_{q^\mu}G_{i\nu}-\partial_{q^\mu}H_{\tau\nu})\frac{\partial}{\partial \hat{p}_\nu}, \ \ \ \ i=1,..k;\ \mu,\nu=1,...,l;$$
$$\Big[\frac{\partial}{\partial y_i}-G_{i\mu}\frac{\partial}{\partial \hat{p}_\mu},\frac{\partial}{\partial y_j}-G_{j\mu}\frac{\partial}{\partial \hat{p}_\mu} \Big]=(\partial_{y_j}G_{i\mu}-\partial_{y_i}G_{j\mu})\frac{\partial}{\partial \hat{p}_\mu}, \ \ \ \ i,j=1,..k;\ \mu=1,...,l;$$
$$\Big[\frac{\partial}{\partial q^\mu}-H_{\mu\nu}\frac{\partial}{\partial \hat{p}_\nu},\frac{\partial}{\partial q^\tau}-H_{\tau\nu}\frac{\partial}{\partial \hat{p}_\nu} \Big]=(\partial_{q^\tau}H_{\mu\nu}-\partial_{q^\mu}H_{\tau\nu})\frac{\partial}{\partial \hat{p}_\mu} ,\ \ \ \ \mu,\nu,\tau=1,...,l. $$
Since $F$ is  invariant $F_{ij}, G_{i\mu}$ and $H_{\mu\nu}$ are independent of $q^\mu$ and the first and third brackets vanish. Moreover, since $F$ is closed,
$$\partial_{y_j}G_{i\mu}-\partial_{y_i}G_{j\mu}=\partial_{q^\mu}F_{ij}=0.$$
Therefore, $\Delta$ is integrable.    
\end{proof}

\begin{lemma}\label{rational} 
Let $(S,F)$ be as in Lemma \ref{invariant} and assume that $F$ represents a rational cohomology class on $S$ (i.e. on each fiber of $S$). Then $\Delta$ has closed leaves which foliate $S\times_{\pi(S)}\hat{M}$.
\end{lemma}
\begin{proof}
In the coordinates used in the previous lemma, we have shown that the distribution $\Delta $ is spanned by 
$$\Delta=span\Big\{\frac{\partial}{\partial y_i}-G^i_{\mu}\frac{\partial}{\partial \hat{p}_\mu},\ i=1,..k; \ \ \frac{\partial}{\partial q^\mu}-H_{\mu\nu}\frac{\partial}{\partial \hat{p}_\nu},\ \mu=1,...l;\ \ \frac{\partial}{\partial \hat{p}_\mu},\ \mu=l+1,...,n  \Big\}.$$

Maximal integral submanifolds of this distribution are closed only if $H_{\mu\nu}$ are rational numbers, otherwise, the leaves would be dense in the fibers.  Since $H_{\mu\nu}dq^\mu\wedge dq^\nu$ is the unique invariant representative of the cohomology class of $F$ on the fiber its rationality is equivalent to $H_{\mu\nu}$ being rational. Indeed, over $b\in \pi(S)$ the cohomology $\HH^2(S_b,\dR)$ of the fiber is given by $\wedge^2\gt^*$, where $\gt$ is the Lie algebra of $S_b$. For any set of 1-periodic coordinates $\{q^1,...,q^l\}$ a basis of $\wedge^2\gt^*$ is given by $\{dq^\mu\wedge dq^\nu\}_{\mu<\nu}$. Moreover, $\{dq^\mu\wedge dq^\nu\}_{\mu<\nu}$ is also a basis of $\HH^2(S_b,\dZ)$ and $\HH^2(S_b,\dQ)$.
\end{proof}

\begin{definition}\label{locally T-dualizable brane}
    Let $\cL=(S,F)$ be a generalized brane in an affine torus bundle $M$ with torsion Chern class and zero $H$-flux. We call $\cL$ \emph{locally T-dualizable} if $S$ is an affine torus subbundle of $M$ and $F$ is an invariant closed two-form on $S$ representing rational cohomology classes when restricted to the fibers.
\end{definition}
Note that the generalized tangent bundle of a locally T-dualizable brane $\cL$ is indeed spanned by invariant sections, that is $\cL$ is invariant in the sense of Definition \ref{def invariant brane}.

\subsection{Local structure of leaves and T-duals}
From the previous section, it is clear that starting with locally T-dualizable brane $\cL=(S,F)$ we can construct the diagram (\ref{big diagram}). That is there exists some submanifolds $Z\subset S\times_{\pi(S)}\hat{M}$ and $\hat{S}\subset \hat{M}|_{\pi(S)}$ which fit into (\ref{big diagram}). In this section, we will show that $Z$ and $\hat{S}$ are also affine torus subbundles of $S\times_{\pi(S)}\hat{M}$ and $\hat{M}$ respectively and we give a precise definition of their monodromy local systems in terms of $\Gamma_S$ and $\Gamma_M$. We will also determine the space closed leaves $Z$ of $\Delta$ and the space of potential T-duals $\hat{S}$.

We first work in coordinates. Let $M, \hat{M}, B, S, F$ as in Lemma \ref{rational}. Then, there exist coordinates $[y_1,...,y_k,q^1,...,q^l]$ on $S$ and $[x_1,...,x_n,p^1,...,p^n]$ on $M$ such that $q^i$ and $p^i$ are 1-periodic and the inclusion of $S$ to $M$ in coordinates is
\begin{align*}
    & \ \ \ \ \ i_0: S \xhookrightarrow{} M\\
    [y,q]&\mapsto [\phi(y),q^1,..,q^l,b^{l+1},...,b^n],
\end{align*}
where $b^i: B\ra \dR/\dZ$ are smooth functions.  Moreover, in these coordinates, $F$ is given by
\begin{equation}\label{F in cordinates}
\begin{aligned}
    F=\sum_{i,j=1}^kF_0^{ij}dy_i\wedge dy_j +\sum_{\substack{i=1,..k\\ \mu=1,...l}}G^i_\mu dy_i\wedge dq^\mu +\sum_{\mu,\nu=1}^lH_{\mu\nu}dq^\mu \wedge dq^\nu.
\end{aligned}
\end{equation}
Since $F$ is closed and its coefficients do not depend on the fiber coordinates we find that $F_0$ is the pullback of a closed two-form from the base, moreover
$$\partial_iG^j_\mu=\partial_jG^i_\mu \ \ \ \forall i,j=1,...,k;\ \ \ \text{and}\ \ \ \partial_iH_{\mu\nu}=0 \ \ \ \forall i=1,..k,\ \mu,\nu=1,...,l. $$
This means firstly, $H_{\mu\nu}$ is a constant anti-symmetric matrix in these coordinates and secondly, there exists $G_\mu: B\ra \dR, \ \mu=1,...,l$ such that 
\begin{align}
    G^i_\mu=\partial_iG_\mu.
\end{align}
We may further assume that the periodic coordinates $(q^1,...,q^l)$ form a symplectic basis for $H$, that is in these coordinates
\begin{align}\label{H in coordinates}
H_{\mu\nu}dq^\mu \wedge dq^\nu=\sum_{i=1}^rd_i dq^i\wedge dq^{r+i}\end{align}
with $d_i=n_i/m_i\in \dQ_{>0}$ and $n_i,m_i\in \dZ_{>0}$ with $gcd(n_i,m_i)=1$.

The vector fields spanning $\Delta$ all commute so we may integrate them into coordinates and we can also construct periodic coordinates for the fibers of $Z\ra B$. These observations are summarized in the following lemma.
\begin{lemma}\label{local leaves}
Any leaf $Z$ of the foliation defined by $\Delta$ can be endowed with coordinates
$$Z=[y_1,...,y_k,q^1,...,q^l,\hat{q}_{l+1},...,\hat{q}_n] $$
were $q^i$ and $\hat{q}_i$ are 1-periodic and in which the inclusion
$$i_Z: Z\xhookrightarrow{} S\times_{\pi(S)}\hat{M} $$
is given by
\begin{align*}
    [y,q,\hat{q}]\mapsto [y,H_1q,-H_2q-G+c,\hat{q}]. 
\end{align*}
Where  $H_1,\ H_2\in \dZ^{l\times l}$ are integer matrices, $H_1=diag(m_1,...,m_r,m_1,...,m_r,1,...,1)$ and $H_2$ is a degenerate symplectic matrix of type $(n_1,...,n_r)$, that is
\begin{align*} 
    H_2=\begin{pmatrix} & & & & & n_1 & & &  \\ & & & &  \vdots & & &  \\ & &  & n_r & & & & \\  & &  -n_1 & & & & & & &\\ &\vdots & & &  & & &  \\ -n_r & & & & &  & & &  \\ & & & & & & 0 & & \\ & & & & & & & \ddots &\\ & & & & & & & & & 0 \end{pmatrix},
\end{align*}
and $c\in (\dR/\dZ)^l$ is a constant vector. The map is an injection because $n_i$ and $m_i$ are relatively primes.
\end{lemma}
The maximal integral submanifolds all intersect the fibers of $S\times_{\pi(S)}\hat{M}$ in $n$-dimensional subtori and the leaf space of $\Delta$ is parametrized by the values of $c_\mu\in \dR/\dZ$, that is a torus of rank $l$. A priori it is just an affine torus, in which choosing $G_\mu$ is equivalent to choosing zero.

Denote by $\hat{p}_Z$ the restriction of $\hat{p}$ to a leaf $Z$ of the foliation and the image by $\hat{S}_Z$ as in (\ref{big diagram}). It is a closed submanifold of $\hat{M}$ and has coordinates $[y_1,...,y_k,s_1,...,s_{2r},\hat{q}_{l+1},...,\hat{q}_n]$ where $\{s_\mu\}$ and $\{\hat{q}_\mu\}$ are 1-periodic on the fibers. The inclusion $i_{\hat{S}_Z}: \hat{S}_Z \xhookrightarrow{} \hat{M}$ is given in coordinates by
\begin{equation}\label{hat S in hat M}
\begin{aligned}
    i_{\hat{S}_Z}:& [y,s_1,...,s_{2r},\hat{q}_{l+1},...,\hat{q}_n]\mapsto \\
    &\mapsto [\phi(y),s_1-G_1+c_1,...,s_{2r}-G_{2r}+c_{2r},-G_{2r+1}+c_{2r+1},...,-G_{l}+c_{l},\hat{q}_{l+1},...,\hat{q}_n].
    \end{aligned}
\end{equation}
Here the inclusion of $s_\mu$ is somewhat arbitrary but it makes our calculations easier. Once again, choosing coordinates on $\hat{S}_Z$ amounts to changing the local section which we may do freely without changing the connection on $\hat{M}$.
In particular, in these coordinates the projection
$$\hat{p}_Z: Z \ra \hat{S}_Z $$
is given simply by
\begin{align}\label{Z to hat(S) in coords}[y,q,\hat{q}] \mapsto [y, n_1 q^{r+1},...,n_r q^{2r},-n_1 q^1,...,-n_r q^r, \hat{q}_{l+1},...,\hat{q}_n] .\end{align}
The image $\hat{S}_Z\subset \hat{M}$ is a closed submanifold, moreover it is again a trivial affine torus subbundle of $\hat{M}$.

 As we have discussed in Lemma \ref{local leaves} any leaf  $Z\subset  S\times_{\pi(S)}\hat{M}$ of $\Delta$ is an affine torus subbundle. That is, $Z$ is a torsor under some group bundle $Z_0=V_Z/\Gamma_Z\ra \pi(S)$ with inclusions $V_Z\subset V_S\oplus V_{\hat{M}}$ and $\Gamma_Z\subset \Gamma_S\oplus \Gamma_{\hat{M}}$. Let us determine $\Gamma_Z$ and $V_Z$ now.

Since we assume that the torus bundle $\pi_S: S\ra \pi(S)$ is trivial, the corresponding Leray spectral sequence degenerates on the second page with both $\dR$ and $\dZ$ coefficients. In particular, the two-form $F\in \HH^2(S,\dR)$ determines a class 
\begin{align}\label{H} 
 H \in \HH^0(\pi(S),\wedge^2 V_S^*),\end{align}
via the map $\HH^2(S,\dR)\cong F^{0,2}(\pi_S,\dR)\ra E^{0,2}_\infty(\pi_S,\dR)\cong E^{0,2}_2(\pi_S,\dR)$. The rationality condition says that $H$ is in the image of $$\HH^0(\pi(S),\wedge^2 \Gamma_S^\vee\otimes \dQ)\ra \HH^0(\pi(S),\wedge^2 V_S^*).$$
We have, therefore, an alternating bilinear form $H$ on $V_S$ which takes rational values on $\Gamma_S$. We may choose a frame $\{\mu_1,...,\mu_n\}$ for $\Gamma_S$ which is symplectic with respect to $H$, that is, $H$ is of the form
\begin{align*}H=\sum_{i=1}^r\frac{n_i}{m_i}\mu_i^*\wedge \mu_{r+i}^*\end{align*}
for some $2r\leq n$ with $n_i,m_i\in \dN_{>0}$ and $gcd(n_i,m_i)=1$. More precisely, there exists a minimal $m$ such that $m\cdot H$ takes integer values on the lattice $\Gamma_S$. Then by the elementary divisor theorem, we may bring $m\cdot H$ to normal form $m\cdot H=\sum_{i=1}^rd_i \mu_i^*\wedge \mu_{r+i}^*$ where $d_i\in \dN_{>0}$ and $d_i|d_{i+1}$. Finally, we divide out by the greatest common divisors element-wise so
$$n_i=\frac{d_i}{gcd(m,d_i)} \ \ \text{and}\ \ m_i= \frac{m}{gcd(m, d_i)}.$$
Using the symplectic frame to generate 1-periodic coordinates for $S$ recovers the expression (\ref{F in cordinates}) for $F$ with which the fiber component is given by (\ref{H in coordinates}). Indeed, the $n_i$ and $m_i$ in Lemma \ref{local leaves} are the same as here.

\begin{definition}\label{Gamma_H}
    For $H\in \HH^0(\pi(S),\wedge^2\Gamma_S^\vee\otimes \dQ)$ let $\Gamma_H\subset \Gamma_S$ be the sublattice $$\Gamma_H=span_\dZ\{m_1\mu_1,...,m_r\mu_r,m_1\mu_{r+1},...,m_r \mu_{2r}, \mu_{2r+1},...,\mu_{n}\}$$ with respect to a choice of symplectic frame.
\end{definition}
\begin{lemma}
    The sublattice $\Gamma_H\subset \Gamma_S$ is independent of the choice of symplectic frame.
\end{lemma}
\begin{proof}
    Let $\{\nu_1,...,\nu_n\}$ be another symplectic frame. Then, there exists a matrix $A\in SL(n,\dZ)$ such that $A\mu=\nu$ and $A^T\nu^*=\mu^*$ which satisfies $AHA^T=H$. Writing $A$ and $H$ as block matrices
    $$A=\begin{pmatrix}
        a &  b \\
        c & d
    \end{pmatrix}\ \ \text{and} \ \ H=\begin{pmatrix}
        H_0 &  0 \\
        0 & 0
    \end{pmatrix},$$
    where $a$ and $H_0$ are $2r\times 2r$ matrices, we find that $aH_0 a^T=0$ and $c=0$. 

    Denote by $\Gamma_H^\mu$ and $\Gamma_H^\nu$ the sublattices of $\Gamma_S$ as above, corresponding to the two frames. Let $M$ be the $n\times n$ diagonal matrix  $M=diag(m_1,...,m_r,m_1,...,m_r,1,...,1)$. Then, the frame for $\Gamma_H^\mu$ is $M\mu$ and for $\Gamma_H^\nu$ is $M\nu$. Therefore,
    $$M\nu= MA\mu=MAM^{-1} (M\mu).$$
    So $\Gamma_H^\mu=\Gamma_H^\nu$ if and only if $MAM^{-1}$ is in $SL(n,\dZ)$. Since
    $$MAM^{-1}=\begin{pmatrix} \gm a\gm^{-1} & \gm b \\ 0 & d \end{pmatrix}\ \ \text{with}\ \ \mathfrak{m}=diag(m_1,...,m_r,m_1,...,m_r),$$
    $MAM^{-1}$ is in $SL(n,\dZ)$ if and only if $\gm a\gm^{-1}$ is integral. 
    Let $\gn$ and $\gd$ be the matrices
    $$\gn=\begin{pmatrix}0 & diag(n_1,...,n_r)\\ -diag(n_1,...,n_r) & 0 \end{pmatrix},\ \ \ \gd=\begin{pmatrix}0 & diag(d_1,...,d_r)\\ -diag(d_1,...,d_r) & 0 \end{pmatrix}$$
    so we have
    $$H_0= \gm^{-1} \gn= m^{-1}\gd.$$
    Therefore,
    $$a\gm^{-1} \gn a^T=\gm^{-1}\gn \ \ \Rightarrow \ \ \gm a\gm^{-1}=\gn(a^T)^{-1}\gn^{-1} $$
    that is, $\gm a\gm^{-1}$ is integral if and only if $\gn(a^T)^{-1}\gn^{-1}$ is integral.

    We use that $a=H_0(a^T)^{-1}H_0^{-1}=\gd (a^{T})^{-1} \gd^{-1}$ and $(a^{-1})^T$ are integral. Let $a_{ij}$ be the entries of the top left $r\times r$ submatrix of $A$. Then we know that
    $$\frac{d_i}{d_j}a_{ij}\in \dZ$$
    for all $i,j=1,...,r$. Then,
$$\frac{n_i}{n_j}a_{ij} \ \ \text{is} \ \ \begin{cases}
        \in \dZ\ \text{ since $n_i|n_j$ } & j\leq i\\
        \frac{d_i}{d_j}a_{ij}\frac{gcd(m,d_j)}{gcd(m,d_i)}\in \dZ\  \text{ as $gcd(m,d_j)|gcd(m,d_i)$ since $d_j|d_i$} &j>i.
    \end{cases}$$
    The same argument works for the other three $r\times r$ submatrices of $a$, so $\gm a\gm^{-1}$ is integral.
\end{proof}
The short exact sequences dual to (\ref{SES lattice sub affine torus bundle}) and (\ref{SES vertical sub affine torus bundle}) are given by
\begin{equation}\label{SES S in M  lattice dual}
    \begin{tikzcd}
    0 \arrow{r} & Ann(\Gamma_S) \arrow{r} & \Gamma_M^\vee \arrow{r} & \Gamma_S^\vee \arrow{r} & 0,
    \end{tikzcd}
\end{equation}
and
\begin{equation}\label{SES S in M vertical dual}
    \begin{tikzcd}
    0 \arrow{r} & Ann(V_S) \arrow{r} & V_M^* \arrow{r} & V_S^* \arrow{r} & 0.
    \end{tikzcd}
\end{equation}
Then, $S\times_{\pi(S)}\hat{M}$ is an affine torus subbundle of $M\times_{\pi(S)}\hat{M}$ and its local system fits into the short exact sequence
\begin{equation}\label{ses SxM in MxM lattice}
    \begin{tikzcd}
    0 \arrow{r} & Ann(\Gamma_S) \arrow{r} & \Gamma_S\oplus \Gamma_M^\vee \arrow{r} & \Gamma_S\oplus\Gamma_S^\vee \arrow{r} & 0,
    \end{tikzcd}
\end{equation}
with corresponding vertical bundles
\begin{equation}\label{ses SxM in MxM vertical}
    \begin{tikzcd}
    0 \arrow{r} & Ann(V_S) \arrow{r} & V_S\oplus V_M^* \arrow{r} & V_S\oplus V_S^* \arrow{r} & 0.
    \end{tikzcd}
\end{equation}
In particular, following the diagram (\ref{diagram S in M}) and the discussion in Lemma \ref{invariant} we may write
$$S\times_{\pi(S)}\hat{M} = t_{(b,0)}\Big(S_0\times_{\pi(S)}\hat{M}\Big),$$
where now $t$ represents translation in $M\times_{\pi(S)}\hat{M}$ by sections of $M_0\times_{\pi(S)}\hat{M}_0$.

On $\Gamma_H\subset \Gamma_S$ the bilinear form $H$ is integral and so it induces a map $\Gamma_H\ra \Gamma_S^\vee$. Denote by $Graph^\Gamma(-H)$ the graph of $-H$ in $\Gamma_H\oplus \Gamma_S^\vee$ which is a primitive sublattice since $gcd(n_i,m_i)=1$. Denote by $Graph(-H)$ the graph of $-H:V_S\ra V_S^*$ inside $V_S\oplus V_S^*$. 

We can now reformulate Lemma \ref{local leaves} in a coordinate-free way as follows.
\begin{lemma}\label{local leaves coordfree}
  The leaves $Z$ of $\Delta$ are affine torus subbundles of $S\times_{\pi(S)} \hat{M}$ and the corresponding group bundle $Z_0=V_Z/ \Gamma_Z$ can be defined through the restrictions of the short exact sequences (\ref{ses SxM in MxM lattice}) and (\ref{ses SxM in MxM vertical}) to the graph of $-H$. That is,
  \begin{equation}\label{Gamma_Z SES}
  \begin{tikzcd}
      0 \arrow{r} & Ann(\Gamma_S) \arrow{r} & \Gamma_Z \arrow{r} & Graph^\Gamma(-H) \arrow{r} & 0
  \end{tikzcd}
  \end{equation}
  \begin{equation}\label{V_Z SES}
  \begin{tikzcd}
   0 \arrow{r} & Ann(V_S) \arrow{r} & V_Z \arrow{r} & Graph(-H) \arrow{r} & 0
  \end{tikzcd}
  \end{equation}
\end{lemma}
That is, given a choice of sections for $S$ and $\hat{M}$ the submanifold $Z$ can be identified with 
$$Z=t_{\hat{b}}(Z_0)\subset S\times_{\pi(S)}\hat{M}$$ 
where now $t$ is the translation inside $S\times_{\pi(S)}\hat{M}$ by sections of $(V_S\oplus V_M^*)/(\Gamma_S\oplus \Gamma_M^\vee)$. The leaf $Z$ only depends on the image of  $\hat{b}$
 in $V_S\oplus V^*_S/Graph(-H)$ up to elements of $\Gamma_S\oplus \Gamma_S^\vee/Graph^\Gamma(-H)$. From Lemma \ref{local leaves} it is clear that the slope of this section is fixed by $F\in \Omega^2(S)$.

We have $Graph(-H)\cong V_S$ and $Graph^\Gamma(-H)\cong \Gamma_H$. Therefore, $V_S\oplus V^*_S/Graph(-H)\cong V_S^*$.   
Let $\{\mu_1,...\mu_l\}$ be a symplectic basis of $V_S$ for $H$ and $\{\mu_1^*,...,\mu_l^*\}$ the dual basis in $V_S^*$, then the lattice dual to $\Gamma_H$ is given by 
\begin{align}\label{dual of Gamma H}
    \Gamma_S^\vee\subset \Gamma_H^\vee=span_\dZ\Big\langle \frac{\mu_1^*}{m_1},...,\frac{\mu_r^*}{m_r},\frac{\mu_{r+1}^*}{m_1},...,\frac{\mu_{2r}^*}{m_r},\mu_{2r+1}^*,...,\mu_l^*\Big \rangle.
\end{align}
Since $\Gamma_H$ is independent of the choice of basis so is $\Gamma_H^\vee$. From this description,  $$\Gamma_S\oplus \Gamma_S^\vee/Graph^\Gamma(-H)\cong \Gamma_H^\vee.$$ We have the following proposition.
\begin{proposition}\label{space of leaves of Delta}
    Over contractible $\pi(S)$, the space of leaves of $\Delta$ is parametrized by flat sections of the torus bundle $V_S^*/\Gamma_H^\vee$. These leaves foliate $S\times_{\pi(S)}\hat{M}$.
\end{proposition}


Let us now examine the image of leaves $Z$ under the projection $\hat{p}:Z\ra \hat{M}$. From the coordinate expression (\ref{Z to hat(S) in coords}) again we readily see the following lemma.
\begin{lemma}\label{local images of Z}
For any $Z$ leaf of $\Delta$ the image $\hat{S}_Z$ of $\hat{p}:Z\ra \hat{M}$ is an affine torus subbundle of $\hat{M}$. The corresponding group bundle $\hat{S}_0=V_{\hat{S}}/\Gamma_{\hat{S}}$ with $V_{\hat{S}}\subset V_M^*$ and $\Gamma_{\hat{S}}\subset \Gamma_M^\vee$ is independent of $Z$ and is given by the following short exact sequences.
\begin{equation}\label{hat(S) vertical}
        \begin{tikzcd}
            0 \arrow{r} & Ann(V_S) \arrow{r} & V_{\hat{S}} \arrow{r} & H(V_S) \arrow{r} & 0
        \end{tikzcd}
\end{equation}
\begin{equation}\label{hat(S) lattice}
        \begin{tikzcd}
            0 \arrow{r} & Ann(\Gamma_S) \arrow{r} & \Gamma_{\hat{S}} \arrow{r} &  \Gamma_S^\vee \cap H(V_S) \arrow{r} & 0
        \end{tikzcd}
\end{equation}
These are once again the restrictions of (\ref{SES S in M vertical dual}) and (\ref{SES S in M  lattice dual}) to the subbundle $H(V_S)= Im(-H: V_S\ra V_S^*)$ and the primitive sublattice cut out by it.
\end{lemma}
Again the space of images is parametrized by constant sections of $V_M^*/V_{\hat{S}}\cong V_S^*/H(V_S)\cong coker(H)$ up to sections of $\Gamma_M^\vee/\Gamma_{\hat{S}}\cong \Gamma_S^\vee/(H(V_S)\cap \Gamma_S^\vee)$. That is, up to the image of $\Gamma_S^\vee$ in $coker(H)$. Let us denote it by $coker(H,\Gamma^\vee_S)$.

\begin{proposition}\label{space of T-duals}
    When $\pi(S)$ is contractible, the space of images $\hat{S}_Z$ corresponding to leaves of $\Delta$ under $\hat{p}$ is parametrized by flat section of $coker(H)/coker(H,\Gamma^\vee_S)$. The images foliate $\hat{M}|_{\pi(S)}$
\end{proposition}
Then, the space of leaves $Z$ mapping onto a fixed $\hat{S}$ is parametrized by $Ker(V_S^*\ra V_S^*/H(V_S))=H(V_S)$ up to elements of $Ker(\Gamma_H^\vee \ra \Gamma_S^\vee/(H(V_s)\cap \Gamma_S^\vee))$. By expressing $\Gamma_H^\vee$ and $\Gamma_S^\vee$ in a symplectic basis as in (\ref{dual of Gamma H}) it is clear that $\Gamma_S^\vee/(H(V_S)\cap \Gamma_S^\vee)\cong \Gamma_H^\vee/(H(V_S)\cap \Gamma_H^\vee)$. Moreover, $H(V_S)\cap \Gamma_H^\vee=H(\Gamma_S)$. Finally, the following corollary is clear.
\begin{corollary}\label{space of leaves mapping to a T-dual}
    When $\pi(S)$ is contractible, space of leaves of $\Delta$  mapping onto the same $\hat{S}$ is parametrized by  constant sections of $H(V_S)/H(\Gamma_S)$. The leaves mapping onto $\hat{S}$ foliate $S\times_{\pi(S)}\hat{S}$.
\end{corollary}
The maps $p_Z:Z\ra S$ and $\hat{p}_Z:Z\ra \hat{S}_Z$ are the composition of a fiber-wise homomorphism of group bundles and a translation by a section in $M$ or $\hat{M}$. The corresponding fiber-wise homomorphism $p_0$ and $\hat{p}_0$ are given in the following diagrams
\begin{equation}
\begin{tikzcd}
    0 \arrow{r} & Ann(V_S)\arrow{d}{\cong} \arrow{r} & V_Z \arrow{d}{\hat{p}_0} \arrow{r}{p_0} & V_S \arrow{r} \arrow{d}{-H} & 0\\
    0 \arrow{r} & Ann(V_S) \arrow{r} & V_{\hat{S}} \arrow{r}{q} & H(V_S) \arrow{r} & 0
\end{tikzcd}
\end{equation}
\begin{equation}
\begin{tikzcd}
    0 \arrow{r} & Ann(\Gamma_S)\arrow{d}{\cong} \arrow{r} & \Gamma_Z \arrow{d}{\hat{p}_0} \arrow{r}{p_0} & \Gamma_H \arrow{r} \arrow{d}{-H} & 0\\
    0 \arrow{r} & Ann(\Gamma_S) \arrow{r} & \Gamma_{\hat{S}} \arrow{r}{q} & H(V_S)\cap \Gamma_S^\vee  \arrow{r} & 0
\end{tikzcd}
\end{equation}
where we use that $\Gamma_H\subset \Gamma_S$.

From the diagrams, it is clear that the fibers of $\hat{p}_Z: Z\ra \hat{S}_Z$ are $\prod_{i=1}^rn_i^2$ disjoint copies of $l-2r$ dimensional tori and the fibers of $p_Z: Z\ra S$  are the disjoint union of $\prod_{i=1}^rm_i^2$ copies of $n-l$ dimensional tori. 

\subsection{T-dual of generalized branes over contractible base}\label{section local tdual}
Finally, we can prove the following theorem which states that a locally T-dualizable generalized brane admits T-duals in the sense of Definition \ref{def tdual gg branes}.
\begin{theorem}\label{local gg thm}
Let $S\subset M$ be a submanifold which is also an affine torus subbundle of $M$ over $\pi(S)$. Assume that $\pi(S)$ is simply connected and let $F\in \Omega^2(S)$ be an invariant closed two-form such that the restriction of $F$ to each fiber represents a rational cohomology class. Consider the distribution $\Delta$ defined in the previous section and let $Z$ be a leaf of the foliation given by $\Delta$. Let $\hat{S}_Z\subset \hat{M}$ be the image of $\hat{p}_Z: Z\ra \hat{M}$.

Then, $\hat{S}_Z$ is an affine torus subbundle of $\hat{M}$ over $\pi(S)$. Moreover, there exists a unique closed invariant two-form $\hat{F}_Z$ on $\hat{S}_Z$ which represents a rational cohomology class on the fibers of $\hat{S}_Z$ and satisfies
$$p_Z^*F+i_Z^*P=\hat{p}_Z^*\hat{F}.$$
Whenever two leaves $Z_1$ and $Z_2$ of $\Delta$ have the same image $\hat{S}_{Z_1}=\hat{S}_{Z_2}$ we have
$$\hat{F}_{Z_1}=\hat{F}_{Z_2}. $$
Moreover, then $(\hat{S}_Z,\hat{F}_Z)$ is a T-dual of the generalized brane $(S,F)$.
\end{theorem}
Clearly, based on this picture a single generalized brane may have several T-duals. Any leaf $Z$ of the distribution $\Delta$ produces a T-dual and some may produce the same. By Proposition \ref{space of T-duals} the space of T-duals is parametrized by an $l-2r$ dimensional torus and by Proposition \ref{space of leaves mapping to a T-dual} the leaves mapping to a certain T-dual are parameterized by a $2r$ dimensional torus. Furthermore, the leaves projecting to the same $\hat{S}$ foliate $S\times_{\pi(S)}\hat{S}$.
\begin{proof}
We have already seen, assuming simply connectedness of the base, that $\hat{S}$ is a trivial affine torus subbundle of $\hat{M}$. Moreover, any leaf of $\Delta$ can be parameterized as $Z=[y,q,\hat{q}]$ with $q$ and $\hat{q}$ being  1-periodic and we have 
\begin{align*}
  i_Z: Z &\hookrightarrow S\times_{\pi(S)}\hat{M}\\
  [y,q,\hat{q}]&\mapsto [\phi(y),H_1q,b,-H_2q-G+c,\hat{q}] . 
\end{align*}
The constants $c_\mu\in \dR/\dZ$, $\mu=1,...,l$ parametrize the leaf space of $\Delta$ as described in Proposition \ref{space of leaves of Delta}. Let us first calculate $\hat{F}_Z$. We have
\begin{align*}
    i_Z^*P=&i_Z^*(d\hat{p}_\mu\wedge dp^\mu)\\
    =&\sum_{\mu=1}^r m_\mu n_\mu  dq^{r+\mu}\wedge dq^{\mu} + \sum_{\mu=1}^r -m_\mu n_\mu dq^{\mu}\wedge  dq^{\mu+r}+\\
    &+    \sum_{\substack{\mu=1...r\\ i=1...k}} (m_\mu (-\partial_i G_\mu dy_i) \wedge dq^\mu+ m_\mu (-\partial_i G_{\mu+r} dy_i) \wedge dq^{\mu+r})  + \\
    &+\sum_{\substack{\mu=2r,...l\\ i=1...k}}  (-\partial_i G_\mu dy_i) \wedge dq^\mu + \sum_{\mu=l+1}^n d\hat{q}_\mu \wedge db^\mu\\
    =&-2\sum_{\mu=1}^rn_\mu m_\mu  dq^{\mu}\wedge dq^{\mu+r}-
    \sum_{\substack{\mu=1...l\\ i=1...k}}h_\mu \partial_i G_\mu dq^\mu \wedge  dy_i -\sum_{\substack{\mu=l+1,...,n\\ i=1,...,k}} \partial_ib^\mu dy_i\wedge d\hat{q}_\mu,\\
\end{align*} 
where $h_\mu=m_\mu$ and $h_{\mu+r}=m_\mu$ for $\mu=1,...,r$ and $h_\mu=1$ for $\mu=2r+1,...,l$. We have 
$$ p_Z^*F=\sum_{i,j=1}^k F_0^{ij}dy_i\wedge dy_j + \sum_{{\substack{\mu=1...l\\ i=1...k}}} h_\mu G^i_\mu dy_i\wedge dq^\mu +\sum_{\mu=1}^{r}\frac{n_\mu}{m_\mu}m_\mu^2 dq^\mu \wedge dq^{\mu+r}$$
with $G^i_\mu=\partial_iG_\mu$. Therefore,
$$p_Z^*F+i_Z^*P=\sum_{i,j=1}^k F_0^{ij}dy_i\wedge dy_j-\sum_{\substack{\mu=l+1,...,n\\ i=1,...,k}} \partial_ib^\mu dy_i\wedge d\hat{q}_\mu- \sum_{\mu=1}^rn_\mu m_\mu dq^{\mu}\wedge dq^{\mu+r}.$$
In the coordinates $[y,s,\hat{q}]$ on $\hat{S}$ we have
$$p_Z^*F+i_Z^*P=\hat{p}_Z^*\hat{F}_Z$$
for 
\begin{equation}\label{hat(F) in coordinates} \hat{F}_Z=\sum_{i,j=1}^k F_0^{ij}dy_i\wedge dy_j-\sum_{\substack{\mu=l+1,...,n\\ i=1,...,k}} \partial_ib^\mu dy_i\wedge d\hat{q}_\mu- \sum_{\mu=1}^r\frac{m_\mu}{n_\mu} ds_{\mu}\wedge ds_{\mu+r}.\end{equation}
The two-form $\hat{F}$ is unique since $\hat{p}_Z:Z\ra \hat{S}_Z$ is a submersion. 

The maximal integral submanifolds of $\Delta$ are all diffeomorphic and the different leaves are given by translates of the $Z_0$. According to Corollary \ref{space of leaves mapping to a T-dual} the difference between leaves that map to the same $\hat{S}$ is a translation by a constant $(c_1,...,c_{2r})\in (\dR/\dZ)^{2r}$. This constant does not affect $p_Z^*F+i_Z^*P$ even though $p_Z$ and $i_Z$ are different. Moreover, if we fix a section of $\hat{S}_Z$ for a leaf $Z$ such that $\hat{p}_Z$ is given by (\ref{Z to hat(S) in coords}), the projection from any other leaf is given by a constant translation in the $s$-coordinates. This again, does not affect the resulting $\hat{F}_Z$. 

It remains to show that $\cL=(S,F)$ and $\hat{\cL}=(\hat{S},\hat{F})$ are T-dual generalized branes. We will show explicitly that $T(\tau_\cL)_{red}=(\tau_{\hat{\cL}})_{red}$ but the statement can be seen from the construction. The T-dual of an invariant subbundle $\tau_{\cL}$ in $(TM+T^*M)|_S$ is constructed by pulling back to an appropriate subbundle $\widehat{\tau_{\cL}}$ in the correspondence space, taking the $B$-field transform $e^P\widehat{\tau_{\cL}}$ and pushing down. We will see that $\Delta$ is precisely the image of $e^P\widehat{\tau_{\cL}}$ in the tangent bundle of $S\times_{\pi(S)}\hat{M}$. Moreover, the equation $p_Z^*F+i^*_ZP=\hat{p}_Z^*\hat{F}$ ensures that the pushforward of $e^P\widehat{\tau_{\cL}}$ restricted to $\hat{S}$ is the generalized tangent bundle of $\hat{\cL}$.

Let us work in the usual coordinates on $S$ and $M$ from Lemma \ref{invariant}, that is 
\begin{align*}
    \ \ \ \ \ i_S: S &\xhookrightarrow{} M,\\
    [y_1,...,y_k,q^1,...,q^l]\mapsto [\phi_1(y),...,\phi_m(y),&q^1,..,q^l,b^{l+1},...,b^n]=[x_1,..,x_m,p^1,...,p^n].
\end{align*}
Then, the tangent bundle of $S$ in $TM|_S$ is given by
\begin{align}\label{TS}
    TS=span_{\cC^\infty(S)}\left\{\begin{array}{lr}\frac{\partial}{\partial y_j}=\sum_{i=1}^m\frac{\partial \phi_i}{\partial y_j}\frac{\partial}{\partial x_i}+\sum_{\mu=l+1}^n\frac{\partial b^\mu}{\partial y_j}\frac{\partial}{\partial p^\mu}\ \text{for}\  j=1,...,k\\ \frac{\partial}{\partial q^\mu}=\frac{\partial }{\partial p^\mu}\ \text{for}\ \mu=1,...,l\end{array}\right\}
\end{align}
and the projection $T^*M\ra T^*S$ is given by
\begin{align*}
    dx_i\mapsto \sum_{j=1}^k\frac{\partial \phi_i}{\partial y_j}dy_j\ \ &(i=1,...m); \ \ \ \
    dp^\mu \mapsto dq^\mu\ \ (\mu=1,...,l);\\ \ \ \ &dp^\mu\mapsto \sum_{j=1}^k\frac{\partial b^\mu}{\partial y_j}dy_j\ \ (\mu=l+1,...,n).
\end{align*}
The generalized tangent bundle of $\cL$ is a subbundle of $E=TM+T^*M$ over $S$ given by
$$\tau_\cL=\{X+\xi\in TS\oplus T^*M|_S\ |\ \iota_XF=\xi|_S\}. $$
That is, the reduced generalized tangent bundle is spanned by invariant sections, that is
\begin{align*}(\tau_\cL)_{red}=\Big\{X+\xi= \sum_{j=1}^kX_j \frac{\partial}{\partial y_j}+ &\sum_{\mu=1}^lX^\mu\frac{\partial}{\partial q^\mu}+\sum_{i=1}^m\xi^i dx_i +\sum_{\mu=1}^n\xi_\mu dp^\mu \Big |\\
&\Big|\ X_j,X^\mu,\xi^i,\xi_\mu\in \cC^\infty(\pi(S))\text{\ satisfying\ }(\ref{iota X F = xi})  \Big\}\end{align*}
where, using the expression (\ref{F in cordinates}) for $F$ the equality $\iota_XF=\xi|_S$ reads as
\begin{equation}
\begin{aligned}\label{iota X F = xi}
    \sum_{i=1}^kF^{ij}X_i-\sum_{\mu=1}^lG^j_\mu X^\mu&=\sum_{i=1}^m\xi^i\frac{\partial \phi_i}{\partial y_j}+\sum_{\mu=1}^n\xi_\mu\frac{\partial b^\mu}{\partial y_j},\ \ (j=1,...,k)\\
    G^j_\mu X_j+H_{\nu\mu}X^\nu&=\xi_\mu,\ \ (\mu=1,...,l).
\end{aligned}
\end{equation}
These invariant sections indeed extend to a neighbourhood of $S\subset M$. Their T-dual is calculated by pulling back to the correspondence space, $B$-field transform by $P$ and then pushing down to $\hat{M}$. A lift of a section in $(\tau_\cL)_{red}$ is given by
\begin{align*}
    \hat{X}+p^*\xi=\sum_{j=1}^kX_j \frac{\partial}{\partial y_j}+ \sum_{\mu=1}^lX^\mu\frac{\partial}{\partial q^\mu}+\sum_{\mu=1}^n\hat{X}_\mu\frac{\partial }{\partial \hat{p}_\mu}+\sum_{i=1}^m\xi^i dx_i +\sum_{\mu=1}^n\xi_\mu dp^\mu,\ \ \ \ \ \ \ \\ \hat{X}_\mu,X_j,X^\mu,\xi^i,\xi_\mu\in \cC^\infty(\pi(S)),
\end{align*}
and we have to choose $\hat{X}_\mu$ such that $\iota_{\hat{X}}P+p^*\xi$ is basic with respect to $\hat{p}:M\times_{B}\hat{M}\ra \hat{M}$. Using the expression (\ref{TS}) of $TS\subset TM|_S$, we have
\begin{align*}
\iota_{\hat{X}}P+p^*\xi=-\sum_{\mu=1}^lX^\mu d\hat{p}_\mu-\sum_{\substack{j=1,...,k\\ \mu=l+1,...,n}}X_j\frac{\partial b^\mu}{\partial y_j}d\hat{p}_\mu+\sum_{\mu=1,...,n}\hat{X}_\mu dp^\mu+\sum_{i=1}^m\xi^idx_i+\sum_{\mu=1}^n\xi_\mu dp^\mu.
\end{align*}
This is basic with respect to $\hat{p}$ if and only if
$$-\hat{X}_\mu=\xi_\mu=\sum_{j=1}^kG^j_\mu X_j+\sum_{\nu=1}^lH_{\mu\nu}X^\nu\ \ (\mu=1,...,l)\ \ \ \text{and}\ \ \ -\hat{X}_\mu=\xi_\mu\ \ (\mu=l+1,....,n),$$
that is, if and only if $\hat{X}\in \Delta$ by (\ref{being in Delta}). Finally, the image of $(\tau_\cL)_{red}$ under the T-duality map is
\begin{equation*}
    T(\tau_\cL)_{red}
    =\left\{\begin{array}{lr}
X+\xi=\ \sum_{j=1}^kX_j \frac{\partial}{\partial y_j}+\sum_{\mu=1}^{l}\Big( G^j_\mu X_j+H_{\mu\nu}X^\nu\Big)\frac{\partial }{\partial \hat{p}_\mu} +\sum_{\mu=l+1}^n\hat{X}_\mu\frac{\partial}{\partial \hat{p}_\mu}+
\\ \ \ \ \ \ \ \ \ \ \ \ \ \ \  +\sum_{i=1}^m\xi^i dx_i-\sum_{\mu=1}^lX^\mu d\hat{p}_\mu -\sum_{\substack{j=1,...,k\\ \mu=l+1,...,n}}X_j\frac{\partial b^\mu}{\partial y_j}d\hat{p}_\mu,\\ 
\hat{X}_\mu,X_j,X^\mu,\xi^i\in \cC^\infty(\pi(S)),\ 
 \text{ satisfying for }j=1,...,k\\\ \sum_{i=1}^kF^{ij}X_i-\sum_{\mu=1}^lG^j_\mu X^\mu=\sum_{i=1}^m\xi^i\frac{\partial \phi_i}{\partial y_j}-\sum_{\mu=l+1}^n\hat{X}_\mu\frac{\partial b^\mu}{\partial y_j} 
    \end{array}\right\}.
\end{equation*}

As in the calculations before we assume that $H_{\mu\nu}$ is rank $2r\leq l$ and is in standard form (\ref{H in coordinates}). Let us calculate $(\tau_{\hat{\cL}})_{red}$, similarly to $(\tau_\cL)_{red}$. The tangent bundle to $\hat{S}$ in $T\hat{M}|_{\hat{S}}$ in the coordinates (\ref{hat S in hat M})  is
\begin{equation}\label{That(S)}
        T\hat{S}=span_{\cC^\infty(\hat{S})} \left\{\begin{array}{lr}
        \frac{\partial}{\partial y_j}=\sum_{i=1}^m\frac{\partial \phi_i}{\partial y_j}\frac{\partial}{\partial x_i}-\sum_{\mu=1}^l\frac{\partial G_\mu}{\partial y_j}\frac{\partial}{\partial \hat{p}_\mu}\ &j=1,...,k\\
         \frac{\partial }{\partial s_\mu}=\frac{\partial }{\partial \hat{p}_\mu}\ &\mu=1,...,2r\\
         \frac{\partial }{\partial \hat{q}_\mu}=\frac{\partial }{\partial \hat{p}_\mu} \ &\mu=l+1,...,n 
         \end{array}\right\}  . 
\end{equation}
Writing $G^j_\mu=\partial_{y_j}G_\mu$, the map $T^*M|_{\hat{S}}\ra T^*\hat{S}$ is given by
\begin{equation}
\begin{aligned}
&dx_i\mapsto \sum_{i=1}^k\frac{\partial \phi_i}{\partial y_j}dy_j\ (i=1,...,m);\ \ \ \ d\hat{p}_\mu\mapsto ds_\mu-\sum_{j=1}^kG^j_\mu dy_j\ (\mu=1,...,2r);\\
&d\hat{p}_\mu\mapsto -\sum_{j=1}^kG^j_\mu dy_j\ (\mu=2r+1,...,l);\ \ \  \ d\hat{p}_\mu\mapsto d\hat{q}_\mu\ (\mu=l+1,...n).
\end{aligned}
\end{equation}
Then, using the coordinate expression (\ref{hat(F) in coordinates}) for $\hat{F}$ we have
\begin{equation*}
    \begin{aligned}
       (\tau_{\hat{\cL}})_{red}=\Big\{\sum_{j=1}^k X_j\frac{\partial}{\partial y_j}+\sum_{\mu=1}^{2r}X_\mu\frac{\partial}{\partial s_\mu}+\sum_{\mu=l+1}^n\hat{X}_\mu\frac{\partial}{\partial \hat{q}_\mu}+\sum_{i=1}^m\xi^idx_i+\sum_{\mu=1}^n\xi^\mu d\hat{p}_\mu \Big |\ \ \ \ \ \ \ \ \\
       \Big| \ X_j,X_\mu,\hat{X}_\mu,\xi^i,\xi^\mu\in \cC^\infty(\pi(S))\ \text{satisfying}\ \ref{iota X hat(F) is xi}\Big\},
    \end{aligned}
\end{equation*}
where (\ref{iota X hat(F) is xi}) is the coordinate expression for $\iota_X\hat{F}=\xi|_{\hat{S}}$, that is,
\begin{equation}\label{iota X hat(F) is xi}
\begin{aligned}
    \sum_{j=1}^kF^{ij}X_i+\sum_{\mu=l+1}^n\frac{\partial b^\mu}{\partial y_j}\hat{X}_\mu&=\sum_{i=1}^m\xi^i\frac{\partial \phi_i}{\partial y_j}-\sum_{\mu=1}^l\xi^\mu G^j_\mu,\ \ \ (j=1,...,k);\\
     -\frac{\partial b^\mu}{\partial y_j}X_j&=\xi^\mu \ \ \ (\mu=l+1,...,n);\ \\
    -\frac{m_\mu}{n_\mu}X_\mu=\xi^{r+\mu},\ \ &\frac{m_\mu}{n_\mu}X_{r+\mu}=\xi^{\mu},\ \ \ (\mu=1,...,r).
\end{aligned}
\end{equation}
Using (\ref{That(S)}) to express the vectors tangent to $\hat{S}$ as vectors on $\hat{M}$ we indeed see that $$T(\tau_\cL)_{red}=(\tau_{\hat{\cL}})_{red}.$$
\end{proof}
Notice that the T-duals $(\hat{S},\hat{F})$ of $(S,F)$ are all diffeomorphic and differ by a constant translation along the fibers above $\pi(S)$. To arrive at a more precise picture we will have to upgrade the current method to include bundles with connections. This will solve the problem of multiple T-duals and we will see that roughly when an integral $F$ has non-integral T-dual it means that the T-dual bundle has higher rank.

Since our construction of T-duals is compatible with T-duality of generalized complex structures the following theorem is clear. 
\begin{theorem}
On T-dual integrable systems as in Theorem \ref{semiflat tdual} $BAA$-branes map to $BBB$-branes.
\end{theorem}
\begin{proof}
Indeed, if a generalized brane is compatible with a generalized complex structure its T-dual will be compatible with the T-dual generalized complex structure.
\end{proof}

\subsection{T-dual of generalized branes over general base}

To prove the global version of T-duality we first have to investigate the relationship between the Chern class and monodromy of an affine torus subbundle $S$ in a non-trivial affine torus bundle $\pi:M\ra B$. It turns out that the Chern class of the T-dual imposes a non-trivial topological constraint on the T-dualizability of generalized branes on $M$ (Theorem \ref{global thm1}).

Let us denote by $\Gamma_M\ra B$ and $c_M\in\HH^2(B,\Gamma_M)$ the monodromy local system and Chern class of $M$ and let $S\subset M$ be an affine torus subbundle. The monodromy local system of $S$ can be constructed functorially as the sheaf $(R^1(\pi_S)_*\dZ_S)^*$, where $\dZ_S$ is the constant $\dZ$ sheaf on $S$. We have the commutative diagram
\[
\begin{tikzcd}
S \arrow[hook]{rr}{i} \arrow{dr}[swap]{\pi_S} & & M|_{\pi(S)} \arrow{dl}{\pi}\\
& \pi(S) &
\end{tikzcd}
\]
and the constant sheaf $\dZ_S$ on $S$ can be written as $i^*\dZ_M$ for the constant sheaf $\dZ_M$ on $M|_{\pi(S)}$. In particular, there is a map $\dZ_M\ra i_*i^*\dZ_M $ and since $\pi_S=\pi\circ i$ and $i_*$ is exact also 
$$\Gamma_M^*|_{\pi(S)} = (R^1\pi_*\dZ)|_{\pi(S)} \ra R^1(\pi_S)_*\dZ_S = R^1\pi_*\circ i_*(i^*\dZ_M)=\Gamma_S^*.$$
The dual map
$$\Gamma_S\ra \Gamma_M|_{\pi(S)}$$
can be understood as the fiberwise inclusion $H_1(S_b,\dZ)\ra H_1(M_b,\dZ)$ for any $b\in \pi(S).$

Let us assume that $B=\pi(S)$ for ease of notation, as we may restrict $M$ to $\pi(S)$ otherwise. The monodromy and Chern class of $M=M|_{\pi(S)}$ can be explicitly constructed from the transition functions. Let $\{U_\alpha\}$ be a good cover of $\pi(S)=B$ and  $(y^\alpha,p_{\alpha})$ local coordinates where $p_{\alpha}$ are 1-periodic. The coordinates $p_{\alpha}$ can be taken such that $S$ is given by
$$S|_{U_\mu}=\{(y^\alpha,p_\alpha^1,...,p_\alpha^n): \ p_\alpha^{\mu}=b_\alpha^{\mu}\in \dR/\dZ\ \mu=l,...,n\}.$$
Let us write $p_{\alpha}=(p_\alpha^1,p_\alpha^2)$ such that 
$$p_\alpha^1=(p_\alpha^1,...,p_{\alpha}^l)\ \ \text{and}\ \ p_\alpha^2=(p_\alpha^{l+1},...,p_\alpha^n).$$
Over $U_{\alpha \beta}=U_\alpha\cap U_\beta$ the transition in the periodic coordinates is given by
$$p_{\beta}=A_{\alpha\beta}p_{\alpha}+c_{\alpha\beta}.$$
When $M$ has torsion Chern class we may assume that the $c_{\alpha\beta}$ are constant. Then, since $S$ has to be preserved $A$ must be upper block-diagonal in the local decompositions $p_{\alpha}=(p_\alpha^1,p_\alpha^2)$ and $p_{\beta}=(p_\beta^1,p_\beta^2)$. We write
\begin{equation}\label{transition of S in M}A_{\alpha\beta}=\begin{pmatrix} B_{\alpha\beta} & C_{\alpha\beta}\\
0 & D_{\alpha \beta}\end{pmatrix}\ \ \ \text{and}\ \ \ c_{\alpha\beta}=\begin{pmatrix}c^1_{\alpha\beta}\\ c_{\alpha\beta}^2 \end{pmatrix}.\end{equation}
 The cocycle conditions of $\{A_{\alpha\beta}\}$ in the above decomposition read as
\begin{align*}
    B_{\gamma\alpha}B_{\beta\gamma}B_{\alpha\beta}=Id,\ \ \ D_{\gamma\beta}D_{\beta\gamma}D_{\alpha\beta}=Id,\\
    B_{\gamma\alpha}B_{\beta\gamma}C_{\alpha\beta}+B_{\gamma\alpha}C_{\beta\gamma}D_{\alpha\beta}+C_{\gamma\alpha}D_{\beta\gamma}D_{\alpha\beta}=0.
\end{align*}
If $(y^\alpha,q_{\alpha})$ are coordinates on $S|_{U_\alpha}$ with $q^\alpha$ 1-periodic the inclusion $S\ra M$ is of the usual form
$$[y^\alpha,q_\alpha]\mapsto [y^\alpha,q_\alpha,b_\alpha].$$
Then, the transition functions for $S$ are given by
$$q_\beta= B_{\alpha\beta}q_\alpha+C_{\alpha\beta}b_\alpha+c_{\alpha\beta}^1  $$
with the condition that
$$b_{\beta}=D_{\alpha\beta}b_\alpha+c_{\alpha\beta}^2.$$
In particular, the monodromy of $S$ is given by $\{B_{\alpha\beta}\}$ and the Chern class by $\{C_{\alpha\beta}b_\alpha+c_{\alpha\beta}^1 \}$. More precisely, the Chern class $c_S\in H^2(\pi(S),\Gamma_S)$ of $S$ is
\begin{align*}
    \delta\{C_{\alpha\beta}b_\alpha+c_{\alpha\beta}^1\}&=B_{\gamma\alpha}B_{\beta\gamma}C_{\alpha\beta}b_\alpha+B_{\gamma\alpha}B_{\beta\gamma}c^1_{\alpha\beta} + B_{\gamma\alpha}C_{\beta\gamma}b_\beta+B_{\gamma\alpha}c^1_{\beta\gamma} +C_{\gamma\alpha}b_\gamma+c^1_{\gamma\alpha}\\
    &=(B_{\gamma\alpha}B_{\beta\gamma}C_{\alpha\beta}+B_{\gamma\alpha}C_{\beta\gamma}D_{\alpha\beta}+C_{\gamma\alpha}D_{\beta\gamma}D_{\alpha\beta})b_\alpha+\\
    &+B_{\gamma\alpha}C_{\beta\gamma}c^2_{\alpha\beta}+C_{\gamma\alpha}D_{\beta\gamma}c^2_{\alpha\beta}+C_{\gamma\alpha}c^2_{\beta\gamma}+B_{\gamma\alpha}B_{\beta\gamma}c^1_{\alpha\beta} + B_{\gamma\alpha}c^1_{\beta\gamma} + c^1_{\gamma\alpha}\\
    &=B_{\gamma\alpha}C_{\beta\gamma}c^2_{\alpha\beta}+C_{\gamma\alpha}D_{\beta\gamma}c^2_{\alpha\beta}+C_{\gamma\alpha}c^2_{\beta\gamma}+B_{\gamma\alpha}B_{\beta\gamma}c^1_{\alpha\beta} + B_{\gamma\alpha}c^1_{\beta\gamma} + c^1_{\gamma\alpha}\\
    &=n^1_{\alpha\beta\gamma}
\end{align*}
if we write the Chern class $\{n_{\alpha\beta\gamma}\}\in H^2(\pi(S),\Gamma_M)$  of $M$ as $\{(n^1_{\alpha\beta\gamma},n^2_{\alpha\beta\gamma})\}$ in the frame of $\Gamma_M$ induced by the coordinates $(p_\alpha^1,p_\alpha^2)$.
In particular, if $c_M$ is torsion, so is $c_S$ and there exists constant representatives $\{m_{\alpha\beta}\}$ of $\{C_{\alpha\beta}b_\alpha+c^1_{\alpha\beta}\}$. That is,
there exists $\xi_\alpha\in \Gamma(U_\alpha,V_S/\Gamma_S)$ such that $$B_{\alpha\beta}\xi_\alpha-\xi_\beta=C_{\alpha\beta}b_\alpha+c^1_{\alpha\beta}- m_{\alpha\beta} .$$ Then, changing coordinates to
$$\tilde{q}_\alpha=q_\alpha+\xi_\alpha$$
changes the transition functions along $S$ as
$$\begin{pmatrix}B_{\alpha \beta} & 0 \\ 0 & D_{\alpha\beta} \end{pmatrix}\begin{pmatrix} \tilde{q}_\alpha \\ b_\alpha \end{pmatrix}+\begin{pmatrix} m_{\alpha\beta} \\ c^2_{\alpha\beta} \end{pmatrix}=\begin{pmatrix}\tilde{q}_\beta \\ b_\beta \end{pmatrix} .$$
In general, if $\tilde{p}^1_\alpha=p^1_\alpha+\xi_\alpha$ we have 
$$\tilde{p}^1_\beta=p^1_\beta+\xi_\beta=B_{\alpha\beta}p^1_{\alpha}+C_{\alpha\beta}p^2_\alpha + c_{\alpha\beta}^1 +\xi_\beta=B_{\alpha\beta}\tilde{p}^1_\alpha+C_{\alpha\beta}p^2_\alpha -C_{\alpha\beta}b_\alpha+m_{\alpha\beta}$$
and if we also set $\tilde{p}^2_\alpha=p^2_\alpha-b_\alpha$  using that $b_\beta=D_{\alpha\beta}b_\alpha+c^2_{\alpha\beta}$ the new transition functions are
$$\begin{pmatrix}B_{\alpha\beta} & C_{\alpha\beta} \\ 0 & D_{\alpha\beta} \end{pmatrix}\begin{pmatrix}\tilde{p}_\alpha\\ \tilde{p}_\beta \end{pmatrix}+\begin{pmatrix}m_{\alpha\beta} \\ 0\end{pmatrix}=\begin{pmatrix} \tilde{p}_\beta^1\\ \tilde{p}_\beta^2 \end{pmatrix}.$$
The following lemma is clear from this description.

\begin{lemma}\label{chern class of s}
Let $S\subset M$ be an affine torus subbundle. Then, the Chern class $c_M$ of $M|_{\pi(S)}$ is the image of the Chern class $c_S$ of $S$ under the map
$$\HH^2(\pi(S),\Gamma_S)\ra \HH^2(\pi(S),\Gamma_M|_{\pi(S)})$$
induced by the inclusion $\Gamma_S\ra \Gamma_M|_{\pi(S)}$.
\end{lemma}
Let us assume that $M$ has torsion Chern class and let $\cL=(S,F)$ be a locally T-dualizable generalized brane and let $\hat{M}$ be a T-dual in the sense of Example \ref{Tdual torsion coordinates}.  Let $\{U_\alpha\}$ be a good cover of $\pi(S)$. The brane $\cL$ locally fits into the framework of Theorem \ref{local gg thm}, so there exist T-duals $(\hat{S}_\alpha,\hat{F}_\alpha)$ of $(S_\alpha,F_\alpha)=(\pi_S^{-1}(U_\alpha),F|_{\pi_S^{-1}(U_\alpha)})$ as generalized branes in $\hat{M}_\alpha=\hat{M}|_{U_\alpha}$. On the other hand, the local T-duals may not glue together, since, as we saw before, the existence of a global submanifold $\hat{S}\subset \hat{M}|_{\pi(S)}$ restricts the Chern class and monodromy of $\hat{M}|_{\pi(S)}$. The following theorem is the first step toward explaining the exact conditions we must impose on $(S,F)$ to obtain a global T-dual in the sense of generalized geometry. 
\begin{lemma}
Let $(S,F)$ be as above and let $\{U_\alpha\}$ be a good cover of $\pi(S)=B$. Suppose that there exists a set 
 $(\hat{S}_\alpha,\hat{F}_\alpha)$ of local T-duals in $\hat{M}_\alpha$ such that the $\hat{S}_\alpha$ glue to a cover of an affine torus subbundle $\hat{S}$ of $\hat{M}$. Then, the local two-forms $\hat{F}_\alpha$ also glue together and $(\hat{S},\hat{F})$ is a global generalized T-dual of $S$.
\end{lemma}
\begin{proof}
Let $E=TM+T^*M$ and $\tau_\cL\subset E|_S$ be the generalized tangent bundle of the generalized brane $\cL=(S,F)$. Then, $(\tau_\cL)_{red}$ is a well defined distribution in $E_{red}|_{\pi(S)}$. The generalized tangent bundles $\tau_{\hat{\cL}_\alpha}$ of the local T-duals  $\hat{\cL}_\alpha=(\hat{S}_\alpha,\hat{F}_\alpha)$ satisfy over $\pi(S)$
$$T(\tau_\cL)_{red}=(\tau_{\hat{\cL}_\alpha})_{red}.$$
That is, if $\hat{\pi}:\hat{M}|_{\pi(S)}\ra \pi(S)$ is the projection, then
$$\tau_{\hat{\cL}_\alpha}=\Big(\hat{\pi}^*T(\tau_\cL)_{red}\Big)\Big|_{\hat{S}_\alpha}.$$
The $\tau_{\hat{\cL}_\alpha}$  fit into the short exact sequence
\[
\begin{tikzcd}
    0 \arrow{r} & N^*\hat{S}_\alpha \arrow{r} & \tau_{\hat{\cL}_\alpha} \arrow{r}{\hat{\rho}} & T\hat{S}_\alpha \arrow{r} & 0,
\end{tikzcd}
\]
where $\hat{\rho}:T\hat{M}+T^*\hat{M}\ra T\hat{M}$ is the anchor map. In particular, the T-duals $\hat{S}_\alpha$ are integral submanifolds of the distribution $\hat{\rho}(\tau_{\hat{\cL}_\alpha})$.

The existence of a global $\hat{S}\subset \hat{M}$ means that $T\hat{S}_\alpha=T\hat{S}$ and over $\hat{S}$ the $\tau_{\hat{\cL}_\alpha}$ are given by the maximal isotropic subbundle $(\hat{\pi}^*T(\tau_\cL)_{red})|_{\hat{S}}$ of $\hat{E}=T\hat{M}\oplus T^*\hat{M}$. By \cite[Proposition 2.6]{MGthesis} every involutive maximal isotropic $L\subset \hat{E}$ can be expressed as a generalized tangent bundle over the integral submanifolds of $\hat{\rho}(L)$. Indeed, $\tau_\cL$ is involutive, therefore so is its T-dual, and there exists $\hat{F}\in \Omega^2(\hat{S})$ such that $(\hat{\pi}^*T(\tau_\cL)_{red})|_{\hat{S}}=\{X+\xi\in T\hat{S}+T^*\hat{M}|_{\hat{S}}\ |\ \iota_X\hat{F}=\xi|_{\hat{S}}\}.$
The $\tau_{\hat{\cL}_\alpha}$ are given by
$$\tau_{\hat{\cL}_\alpha}=\{X+\xi \in T\hat{S}|_{\hat{S}_{\alpha}}\oplus T^*\hat{M}|_{\hat{S}_{\alpha}}:\ \iota_X\hat{F}_{\alpha}=\xi|_{\hat{S}_\alpha} \}.$$
That is in an  all $X\in T\hat{S}$ locally over a cover of $\hat{S}$ we have
$$\iota_X\hat{F}=\iota_X\hat{F}_\alpha$$
 and therefore 
$$\hat{F}=\hat{F}_\alpha.$$
\end{proof}
This Lemma can also be proven explicitly by using local coordinates as in Theorem \ref{local gg thm} and the transition functions (\ref{transition of S in M}).

Let again $(M,0)$ and $(\hat{M},0)$ be a T-dual pair in the sense of Example \ref{Tdual torsion coordinates} and  $\cL=(S,F)$ a locally T-dualizable generalized brane in $M$. Let $\{U_\alpha\}$ be a good cover of $\pi(S)$. The local T-duals $\hat{S}_\alpha$ are affine torus subbundles of $\hat{M}|_{U_\alpha}$ and their local systems are determined by the short exact sequence (\ref{hat(S) lattice})
\[
 \begin{tikzcd}
            0 \arrow{r} & Ann(\Gamma_S) \arrow{r} & \Gamma_{\hat{S}_\alpha} \arrow{r} &  \Gamma_S^\vee \cap H_\alpha(V_{S_\alpha}) \arrow{r} & 0.
\end{tikzcd}
\]
Here, $H_\alpha$ are locally the images of $[F|_{S_\alpha}]$ under the projection $\HH^2(S|_{U_\alpha},\dR)\ra \HH^2(U_\alpha,\wedge^2V_{S_\alpha}^*)$. Since we assume that $M$ has torsion Chern class, so does   $\pi_S:S\ra \pi(S)$. In particular, the Leray spectral sequence $E^{p,q}_r(\pi_S,\dR)$ degenerates on the second page and there is a global map
\begin{equation}\label{H global}\HH^2(S,\dR)\ra \HH^2(\pi(S),\wedge^2 V_S^*),\ \ \ [F]\mapsto H. \end{equation}
That is, $H: V_S\ra V_S^*$ is well defined over $\pi(S)$ and consequently $\Gamma_{\hat{S}}\subset \Gamma_{\hat{M}}|_{\pi(S)}$ is also globally defined by the short exact sequence
\[
 \begin{tikzcd}
            0 \arrow{r} & Ann(\Gamma_S) \arrow{r} & \Gamma_{\hat{S}} \arrow{r} &  \Gamma_S^\vee \cap H(V_{S}) \arrow{r} & 0.
\end{tikzcd}
\]
In conclusion, the existence of a global T-dual to $(S,F)$ is only obstructed by the Chern class $c_{\hat{M}}\in \HH^2(\pi(S),\Gamma_{\hat{M}})$ of $\hat{M}$ in accordance with Lemma \ref{chern class of s}.

The global T-dual exists if and only if $c_{\hat{M}}$ over $\pi(S)$ is in the image \begin{equation}\label{chern1}
    \HH^2(\pi(S),\Gamma_{\hat{S}})\ra \HH^2(\pi(S),\Gamma_{\hat{M}}).
\end{equation}
Since the projection $\Gamma_{\hat{S}}\ra \Gamma_S^\vee\cap H(V_S)$ is the restriction of the projection $q:\Gamma_{\hat{M}}\ra \Gamma_S^\vee$ to the preimage of $H(V_S)$ we may rephrase (\ref{chern1}) in terms of $\Gamma_S^\vee$. Denote by $q$ also the map induced on cohomology.  The following theorem is now clear.
\begin{theorem}\label{global thm1}
    Let $(M,0)$ and $(\hat{M},0)$ be a T-dual pair and $(S,F)$ a locally T-dualizable brane of $M$. Then, $(S,F)$ has a T-dual if and only if
$$q(c_{\hat{M}})\in \HH^2(\pi(S),\Gamma_S^\vee\cap H(V_S)).$$    
\end{theorem}
\begin{example} Let $B$ be a manifold such that $\HH^2(B,\dZ)$ has a torsion element $\beta$. Then, the principal circle bundle $M=B\times S^1$ is T-dual to the circle bundle $\hat{M}=(\hat{\Gamma}=\dZ,\hat{c}=\beta)$. The brane $\cL=(M,0)$ on $M$ is locally T-dualizable but its T-dual should be an integral submanifold of the horizontal distribution of $\hat{M}$ which is also a section. This does not exist because $\hat{c}\neq 0$ and indeed the condition of Theorem \ref{global thm1} is not satisfied. On the other hand, there exists a multisection of $\hat{M}$.
\end{example}

\begin{example} Let $B=\dC$ with complex coordinate $z=x+iy$ and $\Gamma=span_\dZ\{dx,dy\}\subset T^*B$ the trivial lattice. Let  $\{z,w\}$ be dual complex coordinates on $M=T^*B/\Gamma$. The group $\dZ_2$ acts on $M$ via $-1: (z,w)\mapsto (-z,-w)$ with fixed points over $0\in\dC$. The quotient is a torus fibration over $\dC\cong \dC/\dZ_2$ with a singular fiber over $0\in \dC$. This space can also be seen as the moduli space of parabolic rank one Higgs bundles over $\dC\dP^1$ with four marked points \cite[Section 4.]{hauselToy}.

Let us denote by $M'$ the fibers over $\dC-\{0\}$. This is an affine torus bundle with monodromy in $SL(2,\dZ)$ and no Chern class. Therefore the fibers of $M'$ are principally polarized and the K\"ahler structure on the base induces a semi-flat hyperk\"ahler structure on $M'$. This is self T-dual in the sense of Theorem \ref{semiflat tdual principal pol}. Let $\cL=(M,0)$ be a $BBB$-brane. Then its T-dual $(BAA)$-branes are sections of $M$ which are integral submanifolds of the horizontal distribution. These correspond to the invariant sections of $T^*(\dC-\{0\})/\Gamma$ under the action of $\dZ_2$. There are four of these corresponding to the 2-torsion points of the square torus.
\end{example}

\begin{example}\label{BAA-BBB space filling tduals} Let $M$ be an algebraic integrable system of real dimension $4n$ endowed with a flat connection and the corresponding semi-flat hyperk\"ahler structure. Let $\hat{M}$ be a T-dual as in Example \ref{Tdual torsion coordinates} endowed with a flat connection and with the T-dual semi-flat hyperk\"ahler structure as in Theorem \ref{semiflat tdual}. 

Let $\cL=(M,\omega_\dI)$ be the space filling coisotropic $BAA$-brane on $M$ (Example \ref{canonical coisotropic}). That is,
$$\omega_\dI=\begin{pmatrix}\omega & 0 \\ 0 & \omega^{-1} \end{pmatrix}$$
Then $\cL$ is locally T-dualizable since $\omega_\dI$ restricts to the fibers as minus the polarization which is an integral cohomology class. Moreover, there exists a global T-dual since $H(V_M)=V_M^*$. 

Let $\hat{\cL}=(\hat{M},\hat{F})$ be the space filling $BBB$-brane of Example \ref{space filling BBB}. That is in the coordinates of Theorem \ref{semiflat tdual} we have
$$\hat{F}=\begin{pmatrix}\omega & 0 \\ 0 & \omega\end{pmatrix}.$$
It is once again a T-dualizable brane but now $\hat{F}$ represents a rational cohomology class on the fibers.
\begin{lemma}
    The T-dual of the space filling $BAA$ brane $\cL$ is the space filling $BBB$-brane $\hat{\cL}$
\end{lemma}
\begin{proof}
Indeed, in flat Darboux coordinates $\{x^i,y_i\}_{i=1}^n$ and dual fiber coordinates $\{p_i,q^i\}_{i=1}^n$ we can write $\omega$ as (\ref{omega in flat darboux coords}) and we have
$$\omega_\dI=\sum_{i=1}^n\frac{1}{d_i}dy_i\wedge dx^i +\sum_{i=1}^nd_idp_i \wedge dq^i.$$
Let $\{\hat{p}^i,\hat{q}_i\}_{i=1}^n$ be dual fiberwise coordinates on $\hat{M}\ra B$. Then, following the proof of Theorem \ref{local gg thm} the T-dual brane is space filling $\hat{S}=\hat{M}$ and the two-form is given by
$$\hat{F}=\sum_{i=1}^n\frac{1}{d_i}dy_i\wedge dx^i-\sum_{i=1}^n\frac{1}{d_i}d\hat{p}^i\wedge d\hat{q}_i$$
which is precisely $\hat{\cL}$.
\end{proof}
\end{example}

In the setting of topological T-duality, there is a gerbe $\cG$ on $M$ corresponding to a degree three integral cohomology class $h\in \HH^3(M,\dZ)$, which lies in the second filtered piece $F^2\HH^3(M,\dZ)$ of the Leray filtration. Via the isomorphisms
$$E^{2,1}_\infty(\pi,\dZ)\cong  \frac{Ker(d_2:E^{2,1}_2\ra E^{4,0}_2)}{Im(d_2:E^{0,2}_2\ra E^{2,1}_2)}\cong \frac{F^2\HH^3(M,\dZ)}{F^3\HH^3(M,\dZ)}$$
$h$ defines a class $[h]\in E^{2,1}_\infty$. If $(\hat{M},\hat{\cG})$ is a topological T-dual to $(M,\cG)$ then the Chern class $c_{\hat{M}}$ is an element in $\HH^2(B,\Gamma_M^\vee)\cong E^{2,1}_2(\pi,\dZ)$. A consequence of topological T-duality is \cite[Proposition 3.5]{B2} that $c_{\hat{M}}$ is actually in the kernel of $d_2$ so it also defines a class $[c_{\hat{M}}]$ in $E^{2,1}_\infty$ and we have
$$[h]=[c_{\hat{M}}]\in E^{2,1}_\infty(\pi,\dZ).$$
Then, Theorem \ref{chern1} can be viewed as a statement about the restriction of the gerbe $\cG$ to $S$. The restriction $\cG|_S$ of the gerbe has characteristic class $h|_S$, the image of $h$ pulled back along the inclusion $S\ra M$
$$\HH^3(M,\dZ)\ra \HH^3(S,\dZ).$$
Since the Leray spectral sequence is functorial, the class $h|_S$ lies in $F^2\HH^3(S,\dZ)$ and determines a class in $E^{2,1}_\infty(\pi_S,\dZ)$. The Chern class $c_{\hat{M}}$ of $\hat{M}$ restricted to $\pi(S)$ lies in $\HH^2(\pi(S),\Gamma_M^\vee)$. By \cite[Proposition 3.3]{B1} the $d_2$ differential of the Leray spectral sequence is given by cupping with the Chern class. Moreover, by Lemma \ref{chern class of s} the Chern class of $M|_{\pi(S)}$ is the image of the Chern class of $S$. That is, if $c_{\hat{M}}$ is $d_2$-closed in  $\HH^2(\pi(S),\Gamma_M^\vee)$, then $q(c_{\hat{M}})$ is $d_2$-closed in $\HH^2(\pi(S),\Gamma_S^\vee)=E^{2,1}_2(\pi_S,\dZ)$ and we have
\begin{align}[h|_S]=[q(c_{\hat{M}})]\in E^{2,1}_\infty(\pi_S,\dZ).\end{align}

\begin{remark}
    Let $S\subset M$ be an affine torus subbundle such that $\cG|_S$ is trivial. Then, for any $F\in \Omega^2(S)$ invariant there exist T-duals $(\hat{S},\hat{F})$. On the other hand, $\hat{\cG}|_{\hat{S}}$ may not be trivial as it is governed by the Chern class of $M$.  
\end{remark}


In the examples before we have seen that T-duality in generalized geometry is not one-to-one, we have to upgrade generalized branes to physical branes to get such a correspondence. The following theorem is a step towards understanding T-duality of unitary bundles and we will refer back to it in Section \ref{last chapter section general base}.
\begin{theorem}\label{gluing of Z}
Let $(S,F)$ be a generalized brane such that it has a global T-dual $(\hat{S},\hat{F})$. Then, the local leaves of the distribution $\Delta$ glue together to affine torus subbundles of $S\times_{\pi(S)}\hat{S}$ if and only if the following two conditions hold
\begin{enumerate}
    \item $c\in H^2(\pi(S),\Gamma_H)$,
    \item $-H(c)=q(\hat{c})\in H^2(\pi(S),q(\Gamma_{\hat{S}})).$
\end{enumerate}
Where by $H$ we mean the map induced on cohomology by $H:\Gamma_H\ra q(\Gamma_{\hat{S}})$.
\end{theorem}
\begin{proof}
    The local leaves of $\Delta$ glue together if and only if the Chern class of $S\times_{\pi(S)}\hat{S}$ which equals to $p_0^*c+\hat{p}_0^*\hat{c}\in \HH^2(\pi(S),\Gamma_S+\Gamma_{\hat{S}})$  lies in the image of $\HH^2(\pi(S),\Gamma_Z)$. Recall, that the maps $p_0:\Gamma_Z\ra \Gamma_S$ and $\hat{p}_0\ra \Gamma_{\hat{S}}$ are given via the following diagram.
\[
\begin{tikzcd}
    0 \arrow{r} & Ann(\Gamma_S)\arrow{d}{\cong} \arrow{r} & \Gamma_Z \arrow{d}{\hat{p}_0} \arrow{r}{p_0} & \Gamma_H \arrow{r} \arrow{d}{-H} & 0\\
    0 \arrow{r} & Ann(\Gamma_S) \arrow{r} & \Gamma_{\hat{S}} \arrow{r}{q} & H(V_S)\cap \Gamma_M^\vee  \arrow{r} & 0
\end{tikzcd}
\]
In particular, if $p_0^*c+\hat{p}_0^*\hat{c}=c_Z\in \HH^2(\pi(S),\Gamma_Z)$ then 
$p_0(c_Z)=c$ so $c\in \HH^2(\pi(S),\Gamma_H)$ and $q(\hat{c})=q(\hat{p}_0(c_Z))=-H(p_0 (c_Z))=-H(c)$. Where by every map we mean the induced map on cohomology.
\end{proof}
\begin{example} In the setting of Example \ref{BAA-BBB space filling tduals} let $H\in \HH^0(B,\wedge^2\Gamma^\vee)$ be the fiberwise component of $\omega_\dI$. Then, $\Gamma^\vee$ and the Chern class $\hat{c}=-H(c)\in \HH^2(B,\Gamma^\vee)$ define a T-dual of $M$ and the space filling $BAA$-brane $\cL=(M,\omega_\dI)$ satisfies the conditions of Theorem \ref{gluing of Z}. Moreover, if the fibers of $M$ are principally polarized, $M\cong \hat{M}$ via $H$ and the two-form of the T-dual $\hat{\cL}=(M,\hat{F})$ from example \ref{BAA-BBB space filling tduals} restricts to the same cohomology class on the fiber as $\omega_\dI$.
\end{example}




\chapter{Complex tori, Fourier-Mukai transform and factors of automorphy}
In the last chapter, we have seen  T-duality of generalized branes. The rest of this thesis is dedicated to T-duality of physical branes of supersymmetric sigma models. These can be viewed as submanifolds endowed with either Hermitian vector bundles and connections or principal $U(d)$-bundles with connections. Unfortunately, our method from Chapter \ref{chapter gen geom tdual} only applies to branes which are related to generalized branes. Following the argument of \cite{Kap1} generalized branes with integral two-forms correspond to rank-1 branes. Therefore, we start with $U(1)$-bundles supported on affine torus subbundles.

When a physical brane is compatible with the "$B$-type" topological twist, the Hermitian connection must be compatible with a complex structure on the manifold and induce a holomorphic structure. That is, physical $B$-branes are holomorphic vector bundles supported on submanifolds. It can be argued that on a complex manifold "stacking branes" and other physical processes induce the entire derived category of coherent sheaves as the category of $B$-branes. Moreover, it is postulated in \cite{KW} that in the setting of T-duality when a complex structure is mapped to a complex structure (\ref{BBI}) T-duality of $B$-branes should be given by Fourier-Mukai transform.

In this chapter, we focus on the mathematical formalism of those $B$-branes on a complex torus which are given by holomorphic line bundles supported on affine subtori. The results of this chapter will motivate our treatment of general branes on affine torus bundles. 

In the first section, we give an introduction to line bundles on complex tori based on \cite[Chapter 2]{BL}. We introduce the \emph{factors of automorphy} description of holomorphic line bundles which will be our main tool in the following chapters. We present the Appel-Humbert theorem which states that every holomorphic line bundle has a canonical factor of automorphy. Finally, we define the Poincar\'e line bundle.

In the next section, we define the Fourier-Mukai transform and give a full description of the Fourier-Mukai transform of line bundles supported on affine subtori. This section aims to derive an alternative description of the transform for these special sheaves as the direct summand of another sheaf. This result for non-degenerate line bundles is \cite[Corollary 14.3.6]{BL} and the generalizations follow from basic properties of the Fourier-Mukai transform. 

The last section is dedicated to developing the machinery of factors of automorphy to reprove the result (\ref{main equation nondegen}) of the previous section. We define the factor of automorphy of a holomorphic vector bundle which is the pushforward of a line bundle along an isogeny. Such factors of automorphy have been studied before by Matsushima \cite{matsushima} who identified the class of vector bundles which can be obtained by such pushforwards. We prove some lemmas with factors of automorphy and finally reprove (\ref{main equation for line bdle on affine subtorus}) without using properties of the Fourier-Mukai transform.

\section{Line bundles on complex tori}
In this section, we introduce factors of automorphy, a method of describing holomorphic line bundles on a complex torus. We define the canonical factor associated to a line bundle. We also mention how one can describe any holomorphic vector bundle using factors of automorphy on a torus. This description can be generalized to $U(d)$-bundles bundles with connections as well, which will be the topic of the next chapter. We finally introduce the Poincar\'e line bundle.

Let $X=V/\Gamma$ be a complex torus and $L\ra X$ a holomorphic line bundle. Let $\cO_X$ be the sheaf of holomorphic functions on $X$ and $\cO_X^*$ be the sheaf of non-vanishing holomorphic functions.  We have the short exact sequence
\begin{equation}\label{exponential SES}
\begin{tikzcd}[column sep = large]
0 \arrow{r} & \dZ \arrow{r} & \cO_X \arrow{r}{exp(2\pi i \cdot\ )} & \cO_X^*\arrow{r} & 0
\end{tikzcd}
\end{equation}
and $L$ is represented by a class in $\HH^1(X,\cO_X^*)$. Since the line bundle $L$ is trivial when pulled back to the universal cover $V$ we may construct $L$ as the quotient of $V\times\dC$ by an action of the fundamental group $\pi_1(X)=\Gamma$. This action is governed by a \emph{factor of automorphy}.
\begin{definition}
A \emph{holomorphic factor of automorphy} on a complex torus $X=V/\Gamma$ is a map
$$a: V\times \Gamma \ra \dC$$
holomorphic in $V$ which satisfies
$$a(v,\lambda+\mu)=a(v+\lambda,\mu)\cdot a(v,\lambda).$$
Two factors of automorphy $a$ and $a'$ are equivalent if there exists a holomorphic function $h:V\ra \dC$ such that
$$a'(v,\gamma)=h(v+\gamma)a(v,\gamma)h(v)^{-1}.$$
\end{definition}
A factor of automorphy defines an action of $\Gamma$ on $V\times \dC$ by
$$\gamma.(v,t)=(v+\gamma,a(v,\gamma)t)\ \ \ (v,t)\in V\times \dC,\ \gamma\in \Gamma,$$
and therefore a line bundle 
$$L= V\times_{\Gamma}\dC:=(V\times \dC)/\Gamma.$$
By \cite[Proposition B1.]{BL} every holomorphic line bundle can be constructed this way.   

The presentation of a line bundle using a factor of automorphy is not unique but the Appel-Humbert theorem asserts that there is a canonical choice. Via the boundary morphism in the long exact sequence corresponding to (\ref{exponential SES})
\[
\begin{tikzcd}
    ... \arrow{r} & \HH^1(X,\cO) \arrow{r} & \HH^1(X,\cO_X^*) \arrow{r}{c_1} & \HH^2(X,\dZ) \arrow{r} & ...
\end{tikzcd}
\]
the class of $L\in \HH^1(X,\cO_X^*)$ maps to the \emph{first Chern class} $c_1(L)\in \HH^2(X,\dZ)$ of $L$.  There is an isomorphism from $\HH^2(X,\dZ)$ to the space of alternating bilinear forms  $Alt^2(\Gamma,\dZ)$ on $\Gamma$ taking integer values. Indeed, there is a unique invariant representative $\omega$ of the class $c_1(L)$ in $\Omega^2(X)$. We may identify $V$ with the tangent space at any $x\in X$ and define $E\in Alt^2(V,\dR)$ as $\omega_x$. Since $\omega$ is invariant, $E$ does not depend on $x\in X$. Moreover, since $\omega$ represents an integral cohomology class $E$ takes integer values on $\Gamma$. Furthermore, if $\omega$ is the Chern class of a holomorphic line bundle it must be of Hodge type $(1,1)$. In particular, 
\begin{align}\label{E of L}E(Iv,Iw)=E(v,w)\ \ v,w\in V,\end{align}
and $E$ defines a \emph{Hermitian pairing} 
\begin{align}\label{H of L}
H(v,w)=E(Iv,w)+iE(v,w).
\end{align}
Hermitian forms on $V$ whose imaginary part takes integer values on $\Gamma$ form a group under addition called the \emph{Neron-Severi group} $NS(X)$ of $X$. The group of holomorphic line bundles on $X$ is called the \emph{Picard group} and denoted by $Pic(X)$. Via the construction above, we have a surjection $Pic(X)\ra NS(X)$.

The Hermitian form corresponding to a line bundle $L$ describes completely its topological type but not its holomorphic structure. That is $Pic(X)\ra NS(X)$ is not injective. The following lemma describes the kernel, which we denote by $Pic^0(X)$.
\begin{lemma}\emph{(\cite[Proposition 2.2.2]{BL})}
Let $L\ra X$ be a flat line bundle, that is $c_1(L)=0$. Then there exists a factor of automorphy $a_L$ for $L$ which is constant on $V\times e$ and takes value in $U(1)\subset \dC^\times$.
\end{lemma}
Any such factor of automorphy is given by a homomorphism 
$$\chi: \Gamma\ra U(1)$$
and two such factors represent the same line bundle if and only if they are conjugate to each other by an element of $U(1)$. In particular, if we denote by $Hom(\Gamma,U(1))$ the space of characters of $\Gamma$ we have an isomorphism
\begin{align*}
    Hom(\Gamma,U(1)) \cong Pic^0(X),
\end{align*}
between the space of characters and the space of flat holomorphic line bundles. Even more, the space of characters is precisely the dual torus (\ref{dual torus}) under the identification
\begin{equation}
    \begin{aligned}
        \hat{X}\ \ &\ra\ \  Hom(\Gamma,U(1))\\
        f \ \ &\mapsto\ \  \Big(\gamma \mapsto exp(2\pi i Im(f(\gamma)))\Big),
    \end{aligned}
\end{equation}
where we lift $f$ to an element in $V^*$. 

The difference between holomorphic line bundles which have the same Chern class is not a character when the Chern class is not zero, but a semi-character. 
\begin{definition}
    Let $H\in NS(X)$ be a Hermitian pairing on $X$. A \emph{semicharacter for $H$} is a map
    $$\chi: \Gamma \ra U(1)$$
    which satisfies for all $\lambda,\mu\in \Gamma$
    $$\chi(\lambda+\mu)=\chi(\lambda)\chi(\mu)exp(i\pi\cdot  Im(H(\lambda,\mu)))$$
\end{definition}
Let us define the group
\begin{equation}\label{P(Gamma) complex}
\begin{aligned}
    \cP(\Gamma):=\{ (H,\chi)\ | \ H&\in NS(X),\   \chi \text{ a semicharacter for $H$}.\},\\
    \text{with group operation: }&(H_1,\chi_1)\cdot (H_2,\chi_2)=(H_1+H_2,\chi_1\cdot \chi_2).
\end{aligned}
\end{equation}
Now we are ready to state the following theorem.
\begin{theorem}\label{appel humbert holo}\emph{(Appel-Humbert Theorem  \cite[2.2.3]{BL})}
There is an isomorphism of groups
\begin{equation}\label{canonical factor}
\begin{aligned}
   \cP(\Gamma)\ &\ra\ Pic(X)\\
   (H,\chi)\ &\mapsto\ a_{(H,\chi)}(v,\lambda)=\chi(\lambda)exp\Big(\pi H(v,\lambda)+\frac{\pi}{2}H(\lambda,\lambda)\Big).
\end{aligned}
\end{equation}
which fits into the isomorphism of short exact sequences
\[
\begin{tikzcd}
    0 \arrow{r} & Hom(\Gamma,U(1))\arrow{r}\arrow{d}{\cong} & \cP(\Gamma) \arrow{r}\arrow{d} & NS(X) \arrow{r}\arrow{d}{=} & 0\\
    0 \arrow{r} & Pic^0(X) \arrow{r} & Pic(X) \arrow{r} & NS(X) \arrow{r} & 0.
\end{tikzcd}
\]
We call (\ref{canonical factor}) the \emph{canonical factor of automorphy} of the line bundle $L$. 
    
\end{theorem}
To a line bundle $L$ on $X$ we can associate a homomorphism 
\begin{equation}\label{phi_L for holo}
    \begin{aligned}
        \phi_L:X&\ra \hat{X}\\
        x & \mapsto t_x^*L\otimes L^{-1}
    \end{aligned}
\end{equation}
whose analytic representation is $H:V\ra V^*$. This is clear by the canonical factor (\ref{canonical factor}) representation of $L$.  We denote the kernel of $\phi_L$ by $K(L)$ and the connected component of the identity in $K(L)$ by $K(L)_0$. We say that the line bundle $L$ is \emph{non-degenerate} if and only if $K(L)$ is finite and $\phi_L$ is an isogeny. Clearly, $L$ is non-degenerate if and only if $H$ is a non-degenerate Hermitian pairing on $V$.

\paragraph{Vector bundles via factors of automorphy.} Higher rank vector bundles on a complex torus can also be described using factors of automorphy \cite{iena}. Let $X=V/\Gamma$ be a complex torus.
\begin{definition}\label{vector bundle factor}
    An \emph{$r$-dimensional holomorphic factor of automorphy} is a map
    $$a: V\times \Gamma \ra GL(n,\dC)$$
    holomorphic in $V$ and satisfying $a(v,\lambda+\mu)=a(v+\lambda,\mu)\cdot a(v,\lambda)$ for all $v\in V$ and $\lambda,\mu\in \Gamma$. Two such factors of automorphy $a$ and $a'$ are equivalent if there exists a holomorphic function $h:V\ra GL(n,\dC)$ such that
    $$a'(v,\lambda)=h(v+\lambda)\cdot a(v,\lambda)\cdot h(v)^{-1}.$$
\end{definition}
An $r$-dimensional holomorphic factor of automorphy defines a rank $r$ vector bundle on $X$ as
$$E:=V\times_{\Gamma}\dC^r=V\times\dC^r/\Gamma$$
where $\gamma\in \Gamma$ acts on $(v,t)\in V\times \dC^r$ by $\gamma(v,t)=(v+\gamma,a(v,\gamma)t)$.
We have the following theorem.
\begin{theorem}\label{iena}\emph{\cite[Theorem 3.2]{iena}}
    Let $p:V\ra X$ be the projection. Then every rank $r$ holomorphic vector bundle $E$ on $X$ such that $p^*E$
 is trivial, can be represented by an $r$-dimensional holomorphic factor of automorphy.
 \end{theorem}
 Since the complex vector space $V$ is a Stein manifold, the classification of holomorphic and complex topological vector bundles coincide. In particular, every holomorphic vector bundle is 
 trivial on $V$ so every vector bundle on a complex torus can be represented by a factor of automorphy.

\paragraph{Poincar\'e line bundle.} We have seen that the dual torus $\hat{X}$ parametrizes the flat line bundles on $X$. The Poincar\'e bundle is the universal object over $X\times \hat{X}$, that is a holomorphic line bundle $\cP\ra X\times \hat{X}$ such that

(1) $\cP|_{X\times \{L\}}\cong L$,

(2) $\cP|_{\{0\}\times \hat{X}}$ is trivial.\\
The second condition is called the normalization. 

The Poincar\'e line bundle can be defined by a canonical factor of automorphy on $X\times \hat{X}$ due to Theorem \ref{appel humbert holo}. Recall that the dual torus is defined as $\hat{X}=\overline{\Omega}/\Gamma^\vee_\dC$, where $\overline{\Omega}$ is the space of $\dC$ valued $\dC$-antilinear functions on $V$. Let us first define the Hermitian pairing 
\begin{equation}
\begin{aligned}
    H_0: (V\times \overline{\Omega})\times (V\times \overline{\Omega})\ra \dC\\
    H_0(v+\hat{v},w+\hat{w})=\hat{v}(w)+\overline{\hat{w}(v)}.
\end{aligned}
\end{equation}
We have $Im H_0(\Gamma+\Gamma^\vee_\dC,\Gamma+\Gamma^\vee_\dC)\subset \dZ$ by definition of $\Gamma^\vee_\dC$ (\ref{dual cplx lattice}). It remains to assign a semicharacter to $H_0$. Let
\begin{equation}
\begin{aligned}
    \chi_0: \Gamma+\Gamma_\dC^\vee &\ra U(1),\\
    \lambda+\hat{\lambda} &\mapsto exp(i\pi\cdot Im\hat{\lambda}(\lambda)).
\end{aligned}
\end{equation}
Then, the Poincar\'e line bundle is defined by the factor of automorphy $a_{(H_0,\chi_0)}$, that is
\begin{equation}\label{Poincare line bundle}
    \begin{aligned}
    a_\cP(v+\hat{v},\lambda+\hat{\lambda})=exp\Big(\pi \hat{v}(\lambda)+\pi \overline{\hat{\lambda}(v)}+\pi \hat{\lambda}(\lambda)\Big).
    \end{aligned}
\end{equation}
For the proof of the properties and uniqueness see \cite[Theorem 2.5.1]{BL}.

\section{Fourier-Mukai transform of line bundles supported on affine subtori}\label{FM transform section}

Let $X$ be a $g$-dimensional complex abelian variety, that is a complex torus $X=V/\Gamma$ which admits a polarization. Let $\hat{X}$ be the dual torus and $\cP$ the Poincar\'e line bundle on $X\times \hat{X}$. Denote the projections from $X\times \hat{X}$ to $X$ and $\hat{X}$ by $p$ and $\hat{p}$ respectively. 
\[
\begin{tikzcd}
  &   X\times \hat{X} \arrow{dl}[swap]{p} \arrow{dr}{\hat{p}} & \\
  X & & \hat{X}  
\end{tikzcd}
\]
Let us denote by $\cD^b(X)$ and $\cD^b(\hat{X})$ the bounded derived categories of coherent sheaves on $X$ and $\hat{X}$ respectively. The functor $FM_X=R\hat{p}_*(p^*(\ ) \otimes \cP) $ induces a derived equivalence between $\cD^b(X)$ and $\cD^b(\hat{X})$ \cite{mukai}. We say that a sheaf $\cF$ is $W.I.T.$ of index $k$ if $FM_X(\cF)=R\hat{p}_*(p^*\cF \otimes \cP) $ is a complex concentrated in degree $k$. For a $W.I.T.$ sheaf $\cF$ of index $k$  we define the \emph{Fourier-Mukai transform of $\cF$} as 
\begin{equation}\label{FM transform of WIT sheaf} \hat{\cF}:=R^k\hat{p}_*(p^*\cF \otimes \cP). \end{equation}
In this section, we will show that any sheaf which is a line bundle supported on an affine subtorus $S\subset X$ is $W.I.T$. Moreover, we derive an expression which describes the Fourier-Mukai transform of such sheaves via non-derived pushforwards along morphisms of abelian varieties. These results are easy applications of well-known properties of the Fourier-Mukai transform (see  \cite{mukai}\cite{schnell}\cite{BL}).

\subsection{Non-degenerate line bundles}
Let $L$ be a non-degenerate line bundle on X, that is the Chern class of $L$ can be regarded as a non-degenerate Hermitian pairing $H$ on $V$. Denote by $Im(H)=E$ its imaginary part, which is a non-degenerate alternating bilinear pairing on $V$ taking integer values on $\Gamma$. Therefore, we may choose \emph{symplectic basis} $\{\mu_1,...,\mu_g,\lambda_1,..,\lambda_g\}$ of $\Gamma$ with respect to $E$, That is, is
\begin{align}\label{symplectic basis}E(\mu_i,\lambda_j)=d_i\delta_{ij},\ \ \ E(\mu_i,\mu_j)=E(\delta_i,\delta_j)=0,\ \ i,j=1,...,g.\end{align}
where $d_i\in \dZ_{>0}$. We say that $L$ is of \emph{type} $(d_1,...,d_g)$ and that $d=d_1\cdot ...\cdot d_g$ is the \emph{degree} of $L$. The number $d$ is independent of the choice of symplectic basis since it is the square root of the determinant, called the \emph{Pfaffian}, of $E$ denoted by $\text{Pf}(E)$. If $H$ has $r$ positive and $s$ negative eigenvalues with $r+s=g$ we say that the \emph{index} of $L$ is $s$.  The name `index' is related to the fact that the line bundle $L$ is a W.I.T. sheaf of index $s$. The Fourier-Mukai transform $\hat{L}$ is a vector bundle of rank $d$ \cite[Theorem 3.5.5 and  Lemma 14.2.1]{BL}. 

There exists an isogeny $f:X\ra Y$ and a  non-degenerate line bundle $N$ of type $(1,...,1)$ on $Y$ such that $f^*N=L$. Indeed, consider a decomposition $\Gamma=\Gamma_1+\Gamma_2$ of $\Gamma$ into maximal isotropic sub-lattices with respect to $E$. For example in a symplectic basis (\ref{symplectic basis}) we can set $\Gamma_1=\{\mu_1,...,\mu_g\}$ and $\Gamma_2=\{\lambda_1,...,\lambda_g\}$. Then define 
\begin{equation}\label{isogeny for degree d line bundle}
    \Gamma_Y:=\Gamma_1\oplus E^{-1}(\Gamma_1^\vee)\supset \Gamma_X,
\end{equation}
and let us define = $Y:=V/\Gamma_Y$. We have an isogeny $f:X\ra Y$ induced by identity on $V$ and $Y$ is also an abelian variety. 

The line bundle $L$ is described by  $(H,\chi)$ where $\chi$ is a semicharacter for $H$. To define $N$ on $Y$ such that $f^*N=L$ we have to describe $(H_N,\chi_N)$. By definition of $\Gamma_Y$, we may take $H=H_N$ and we only have to describe $\chi_N$. For any $\lambda\in \Gamma_X$ the semicharacter $\chi$ of $L$ can be written as follows
\begin{align}\label{semicharacters wrt a decomposition}\chi(\lambda)=exp(i\pi E(\lambda_1,\lambda_2)+2\pi i \hat{v}(\lambda))\end{align}
where $\lambda=\lambda_1+\lambda_2$ is the decomposition of $\lambda $ with respect to $\Gamma=\Gamma_1+\Gamma_2$ and $\hat{v}\in V^*$. This description of $\chi$ readily extends to a semicharacter of $H_N=H$ on $Y$ which clearly pulls back to that of $L$ under $f$. In particular, $N$ is defined by the same data as $L$ in the specific decomposition $\Gamma=\Gamma_1+\Gamma_2$.

Note that the choice of $Y$ and $N$ is far from unique. Any decomposition of $\Gamma$ with respect to $H$ will induce a torus $Y$ and a line bundle $N$ and we may also twist $N$ by any character of $\Gamma_Y$ which restricts to the trivial character on $\Gamma$.

Given a choice of $(Y,N)$ we have the commutative diagram
\begin{equation}\label{isogeny to type 1111..}
\begin{tikzcd}
    X \arrow{r}{\phi_L}\arrow{d}[swap]{f} & \hat{X}\\
    Y  \arrow{r}{\phi_N} & \hat{Y} \arrow{u}[swap]{\hat{f}}
\end{tikzcd}
\end{equation}
where $\hat{f}:\hat{Y}\ra \hat{X}$ is the dual homomorphism of $f$ defined as the pullback of flat line bundles from $Y$ to $X$ via $f$. It is again an isogeny of the same degree as $f$ \cite[Proposition 2.4.3]{BL}.

It is shown in \cite[ Corollary 14.3.9]{BL} that $\hat{L}$ can be calculated as the pushforward of $N^{-1}$ under isogenies. In particular,
\begin{equation}\label{FM of nondegenerate line bundle}\hat{L}=\hat{f}_*(\phi_N)_*N^{-1}\end{equation}
This corollary shows that $\hat{f}_*(\phi_N)_*N^{-1}$ is independent of the choice of $Y$ and $N$, it depends only on $L$.

In addition, in \cite[Corollary 14.3.6]{BL} it is shown that
\begin{equation}\label{main equation nondegen}(\phi_L)_*L^{-1} \cong \hat{L} \otimes H^s(X,L)\cong \widehat{L}\otimes \cO_{\hat{X}}^{\oplus d}.\end{equation}
In the next sections, we will generalize (\ref{FM of nondegenerate line bundle}) and (\ref{main equation nondegen}) to non-degenerate line bundles and also to line bundles supported on affine subtori. The upshot of these calculations is that line bundles on subtori are special sheaves, in the sense that their Fourier-Mukai transform can be calculated without really relying on derived geometry. There is a bigger class of vector bundles that have this property, called \emph{semihomogeneous vector bundles}. These are vector bundles whose projectivization is homogeneous, that is translation invariant. Every line bundle is semihomogeneous. Mukai studied these objects in  \cite{mukaiSemihomog} and many of his observations can be paralleled with our construction of T-dual generalized branes. 

\subsection{Degenerate line bundles}
Let $L$ be a line bundle on $X$ with Chern class $H$, a Hermitian inner product with $r$ positive and $s$ negative eigenvalues where now $s+r<g$.  The homomorphism
$$\phi_L: X\ra \hat{X}$$
is not surjective anymore as its analytification is given by $H:V\ra V^*$. Hence, its image  $Im(\phi_L)=:\widehat{S}$ and the connected component of its kernel $K(L)_0$ are abelian subvarieties of $\hat{X}$ and $X$ respectively. The Stein factorisation of $\phi_L$ is a decomposition of $\phi_L$ into a projection and an isogeny to its image. We have
\[
\begin{tikzcd}
    X\arrow{rr}{\phi_L} \arrow{dr}[swap]{q_L}&  & \hat{S} \\
  &  X/K(L)_0 \arrow{ur}[swap]{f_L} &
\end{tikzcd}
\]
where $q_L$ is a projection and $f_L$ is an isogeny. There exists a symplectic basis $$\{\mu_1,...,\mu_r,\lambda_1,...,\lambda_r,\epsilon_{2r+1},...,\epsilon_{2g}\}$$ for $E=Im(H)$ on $\Gamma$ such that 
\begin{align*}E(\mu_i,\lambda_j)=d_i\delta_{ij},\ \ \ E(\mu_i,\mu_j)=E(\lambda_i,\lambda_j)=E(\mu_i,\epsilon_k)=E(\lambda_i,\epsilon_k)=0,\\ \ \ i,j=1,...,r,\ k=2r+1,...,2g\end{align*}
with $d_i\in \dZ_{>0}$. The number $d=d_1\cdot ... \cdot d_{2r}$ is again called the \emph{degree} of $L$ and it is also the degree of the isogeny $f_L$. It is again independent of the choice of symplectic basis and it is called the reduced Pfaffian of $E$ denoted by $\text{Pfr}(E)$. 

We have the dual short exact sequences
\[
\begin{tikzcd}
    0 \arrow{r} &  K(L)_0\arrow{r}{i} &  X \arrow{r}{q_L} & X/K(L)_0\arrow{r} & 0
\end{tikzcd}
\]
and
\[
\begin{tikzcd}
   0 \arrow{r} & \widehat{X/K(L)_0} \arrow{r}{\hat{q}_L} & \hat{X} \arrow{r}{\hat{i}}&  \widehat{K(L)_0}\arrow{r} & 0,
\end{tikzcd}
\]
and $\widehat{S}=im(\phi_L)=ker(\hat{i})$ (see \cite{BL} Lemma 2.4.5.). Let us denote by $S$ the torus $X/K(L)_0$. Then, $\widehat{S}$ is the complex torus dual to $S$.

We have the following two lemmas.
\begin{lemma}\label{flatkernel}
Let $L\ra X$ be a line bundle and $i: K(L)_0\ra X$ the inclusion. Then, $i^*L$ is flat.
\end{lemma}
\begin{proof}
    Let $\bar{x}\in K(L)_0$ be a point which we may also consider as a point $x$ in $X$ via the inclusion $i$. We have a commuting diagram 
    \[\begin{tikzcd}
        X \arrow{r}{t_x}  & X \\
        K(L)_0 \arrow{u}{i} \arrow{r}{t_{\bar{x}}} & K(L)_0 \arrow{u}[swap]{i}
    \end{tikzcd}\]
    and therefore
    $$t_{\bar{x}}^*i^*L\otimes (i^*L)^{-1}=i^*(t_x^*L \otimes L^{-1})=i^*\cO_X=\cO_{K(L)_0}$$
    since $x\in K(L)_0$ which is in the kernel of $\phi_L$. That is, $i^*L$ is invariant under all the translations in $K(L)_0$ and therefore flat.
\end{proof}
\begin{lemma}\label{samekernel}
    Let $L$ be a line bundle on $X$, and $L_0$ a flat bundle. Then,
    $$\phi_{L\otimes L_0}=\phi_L,$$
    in particular,
    $$K(L\otimes L_0)_0=K(L)_0$$
\end{lemma}
\begin{proof}
    $t^*_x(L\otimes L_0)\otimes (L^{-1}\otimes L_0^{-1})=t^*_xL \otimes t^*_x L_0 \otimes L^{-1} \otimes L_0 = t^*_xL \otimes L_0 \otimes L^{-1} \otimes L_0^{-1} =t^*_xL \otimes L^{-1}, $
    where we use that flat line bundles are translation invariant.
\end{proof}
A degenerate line bundle is flat when restricted to $K(L)_0$. Moreover, if it is trivial on $K(L)_0$, then it is trivial on all the fibers of $q_L:X\ra X/K(L)_0$. Indeed, with the notation of Lemma \ref{flatkernel} for any $b\in X$
$$L|_{t_bK(L)_0}=i^*t_b^*L=t^*_{q_L(b)}i^*L=i^*L.$$
We prove our result in two steps. First, we assume that $L|_{K(L)_0}$ is trivial, then in the second step, we relax this assumption.

\paragraph{Step I.} Suppose, that $L|_{K(L)_0}$ is trivial.
Then, there exists a non-degenerate line bundle $L_0$ of index $s$ on $X/K(L)_0$ such that $L\cong (q_L)^*L_0$.  The isogeny $f_L$ is then given by the isogeny $\phi_{L_0}$ under the identification 
$$\widehat{S}\cong \widehat{X/K(L)_0} .$$ 
Using (\ref{main equation nondegen}) for $L_0$ and the Stein factorization of $\phi_L$ we have
\begin{align*}
    (\phi_L)_*L^{-1}&=(\phi_{L_0})_*(q_L)_*(q_L^*L_0^{-1})\\
    &=(\phi_{L_0})_*(L_0^{-1} \otimes (q_L)_*\cO_X) \ \ \ \text{(projection formula)}\\
    &=(\phi_{L_0})_*L_0^{-1} \ \ \ \text{(*)}\\
    &=\widehat{L_0}\otimes H^s(S,L_0) \ \ \ \ \text{(by \ref{main equation nondegen})}.
\end{align*}
For $(*)$ we use that $q_L$ is projective with connected fibers, $X$  and $Y$ are smooth so $(q_L)_*\cO_X=\cO_{S}$. To calculate the Fourier-Mukai transform of $L$ we use the following theorem. 
\begin{theorem}\label{homomFM}\emph{(\cite[Proposition 2.3]{CJ})}
    If $f:X \ra Y$ is a quotient of abelian varieties we have equivalences of functors 
$$FM_X\circ f^* \circ [dim X] \cong \hat{f}_* \circ FM_Y \circ [dim Y]\ \ \ \text{and}\ \ \ f^*\circ FM_Y\cong FM_X\circ \hat{f}_*,$$
where $FM_X=R\hat{p}_*(p^*(\ )\otimes \cP_X)$ for any abelian variety $X$.
\end{theorem}
We have  
\begin{align*}
    FM_X(L) = FM_X((q_L)^*L_0)= (\hat{q}_L)_*FM_S(L_0)[dim S - dim X].
\end{align*}
Since $L_0$ is non-degenerate and its Chern class has $s$ negative eigenvalues it is W.I.T. of index $s$ so
$$FM_Y(L_0)=\widehat{L_0}[-s],$$
where $\widehat{L_0}$ is a rank $d=\text{Pfr}(E)$ vector bundle on $\widehat{S}$. The dimension of $S$ is $r+s$ so we have
$$FM_X(L)=(\hat{q}_L)_*\widehat{L_0}[r-g].$$
That is, $L$ is a W.I.T. sheaf of index $g-r$ and its  Fourier-Mukai transform is 
$$\widehat{L}=(\hat{q}_L)_*\widehat{L_0}.$$
Since $L_0$ is non-degenerate there exists an abelian variety $S_0$ an isogeny $f:S\ra S_0$ and a type $(1,...,1)$ non-degenerate line bundle $N_0$ on $S_0$ such that $L_0=f^*N_0$ and $\widehat{L_0}=\hat{f}_*(\phi_{N_0})_*N_0^{-1}$. We have the following commutative diagram.
\begin{equation}\label{isogeny to type 111 degenerate}
\begin{tikzcd}
    X \arrow{r}{\phi_L} \arrow{d}[swap]{q_L} & \widehat{X}\\
    S \arrow{d}[swap]{f} \arrow{r}{\phi_{L_0}} & \widehat{S} \arrow{u}[swap]{\hat{q}_L}\\
    S_0 \arrow{r}{\phi_{N_0}} & \widehat{S_0} \arrow{u}[swap]{\hat{f}}
\end{tikzcd}
\end{equation}
In particular, 
\begin{align}\label{FM transform degen line bundle}\widehat{L}=(\hat{q}_L)_*\hat{f}_*(\phi_{N_0})_*N_0^{-1}. \end{align}
Recall that we showed on $\hat{S}$ that  $(\phi_L)_*L^{-1}=\widehat{L_0}\otimes H^s(S,L_0)$ so on $\hat{X}$ we have
\begin{equation}\label{main eqn degen line bdle}
\begin{aligned}(\phi_L)_*L^{-1}&=(\hat{q}_L)_*(\widehat{L_0}\otimes H^s(S,L_0))\\
&=(\hat{q}_L)_*(\widehat{L_0})\otimes H^s(S,L_0)\ \ \ \ \text{(projection formula)}\\
&=\widehat{L}\otimes \cO_{\hat{X}}^{\oplus d}.
\end{aligned}
\end{equation}
Where we use that $H^s(X,L)\cong H^s(S,L_0) \cong \dC^d$, see \cite[Theorem 3.5.5]{BL}.


\paragraph{Step II.} Suppose now that $L|_{K(L)_0}$ is not trivial. 

Due to Lemma \ref{flatkernel}, there exists a point $\hat{x}\in \hat{X}$ such that $L\otimes \cP_{\hat{x}}$ is trivial restricted to $K(L)_0$, moreover $K(L)_0=K(L\otimes \cP_{\hat{x}})_0$ by Lemma \ref{samekernel}. Therefore, $L\otimes \cP_{\hat{x}}$ is a W.I.T. sheaf of index $g_0-r$, where $r$ is the number of positive eigenvalues of the first Chern class of $L$. 

By \cite[Proposition 14.3.1]{BL} if $L\otimes \cP_{\hat{x}}$ is a W.I.T. sheaf of index $g-r$ then so is $L$ and 
\begin{align}\label{FM transform of degen line bdle 2}\widehat{L}=t_{-\hat{x}}^*(\widehat{L\otimes \cP_{\hat{x}}}).\end{align}
Moreover, since $\phi_L=\phi_{L\otimes \cP_{\hat{x}}}$ we know
$$(\phi_{L})_*(L\otimes \cP_{\hat{x}})^{-1}=\widehat{L\otimes \cP_{\hat{x}}}\otimes H^s(X,L\otimes \cP_{\hat{x}})$$
that is,
\begin{align}\label{main eqn degen lin bdle 2}t^*_{-\hat{x}}(\phi_L)_*(L^{-1}\otimes \cP_{-\hat{x}})\cong \widehat{L}\otimes \cO_{\hat{X}}^{\oplus d}.\end{align}
In particular, $\widehat{L}$ is a rank $d=\text{Pfr}(E)$ vector bundle supported on $\hat{S}+\hat{x}$. The result is independent of the choice of $\hat{x}$.


\subsection{Line bundles supported on affine subtori}

Let $S\subset X$ be a $g_0$-dimesnional affine subtorus in $X$, that is it is the translation of an abelian subvariety of $X$. In particular, for any $x\in S$ the subtorus $S_0=S-x$ is an abelian subvariety of $X$.  Let $L\ra S$ be a line bundle. Then,  $L_0:=t_x^*L$ is a line bundle supported on the abelian subvariety $S_0$.

Let $i:S_0\ra X$ be the inclusion homomorphism and its dual $\hat{i}:\hat{X}\ra \hat{S}_0$ is a surjection. We may use the second equation in  Theorem \ref{homomFM} to relate the $FM_X(L)$ to $FM_{S_0}(L_0)$. 
\begin{equation*}
\begin{aligned}
FM_X(L)&=FM_X(t_{-x}^*L_0)\\
&=FM_X(L_0) \otimes \cP_x\ \ \ \ \ \ \  \text{\cite[Proposition 14.7.8]{BL}}\\
&=FM_X(i_*L_0)\otimes \cP_x\\
&=\hat{i}^* FM_{S_0}(L_0) \otimes \cP_x\\
&=\hat{i}^*\widehat{L_0}[r-g_0] \otimes \cP_x.
\end{aligned}
\end{equation*}
That is, $L$ is a W.I.T. sheaf of index $g_0-r$ and
\begin{align}\label{FM of line bundle on subtorus}\widehat{L}=\hat{i}^*\widehat{L_0}\otimes \cP_x.\end{align}
Note that the result does not depend on the choice of $x\in S$. 

By (\ref{FM transform of degen line bdle 2}) the support of $\widehat{L}$ is the preimage of an
an affine subtorus under the projection $\hat{i}:\hat{X}\ra \hat{S}_0$. Therefore, there exists $\hat{s}\in \hat{S}_0$ such that
$$t^*_{-\hat{s}}(\phi_{L_0})_*(L_0^{-1}\otimes \cP^S_{-\hat{s}})\cong \widehat{L_0}\otimes H^s(S,L_0\otimes \cP^S_{\hat{s}})$$
so
\begin{align}\label{main equation for line bdle on affine subtorus}\hat{i}^*t^*_{-\hat{s}}(\phi_{L_0})_*(L_0^{-1}\otimes \cP^S_{-\hat{s}})\otimes \cP_{x}\cong \widehat{L}\otimes \cO_{\hat{X}}^{\oplus d}.
\end{align}

\section{Fourier-Mukai transform via factors of automorphy}
The goal of this section is to prove (\ref{main equation nondegen}) for non-degenerate line bundles using factors of automorphy. In the next chapter, we will generalize this to $U(1)$-bundles with connections. Eventually, these results will be used to prove a statement analogous to the final result (\ref{main equation for line bdle on affine subtorus}) of the previous section about the T-dual of $U(1)$-bundles.

In this section, we will first describe the pushforward of a line bundle under an isogeny using factors of automorphy and show that the pushforward of a line bundle under an isogeny is semihomogeneous. We describe explicitly the pushforward of the structure sheaf along an isogeny. We also prove a lemma which substitutes the property of the Fourier-Mukai transform that translation is mapped to tensoring with a flat line bundle. This section will conclude in Proposition \ref{FM transform and phi L holo} with a proof of (\ref{main equation nondegen}) without using properties of the Fourier-Mukai transform.

Let $f:X\ra Y$ be a degree $n$ isogeny of abelian varieties and $L\ra X$ be a holomorphic line bundle on $X$. Since $f$ is a covering map it is finite flat and $f_*L$ is a rank $n$ vector bundle on $Y$. 

\begin{remark}
    Let $K(f)\subset X$ be the kernel of $f:X\ra Y$. Then,
    $$f^*f_*L\cong \bigoplus_{x\in K(f)}t_x^*L.$$
\end{remark}

\begin{remark}
The dual map $f^*=\hat{f}$ is also an isogeny of degree $n$, so there are line bundles $\{L_i\}_{i=1,...,n}$ such that $f^*L_i\cong \cO_X$, that is $\{L_i\}=Ker(\hat{f})$. On the other hand, this means, that if $L$ is a line bundle on $X$ then
$$L\cong L\otimes_{\cO_X} f^*L_i,\ \ \ \ i=1,...,n.$$
In particular, by the projection formula, we have on $Y$
\begin{align}\label{invariance under tensoring with line bundle}f_*L\cong f_*(L\otimes_{\cO_X}  f^*L_i)=f_*L\otimes_{\cO_Y} L_i\ \ \ \forall i=1,...,n.\end{align}
\end{remark}

We want to understand the direct image $f_*L$ using factors of automorphy. Let us write $X=V/\Gamma_X$ and $Y=W/\Gamma_Y$. As we have discussed in Section \ref{section tori}, the analytification $F$ of $f$ is an isomorphism so we may assume that $W=V$ and $F=id$. Then, $\Gamma_X\subset \Gamma_Y$. 

The line bundle $L$ on $X=V/\Gamma_X$ can be constructed as an associated bundle 
$$L\cong V\times_{\Gamma_X} \dC$$
where $\lambda \in \Gamma_X$ acts as $\lambda.(v,t)=(v+\lambda, a_L(v,\lambda)t)$ via the factor of automorphy $a_L$. By Theorem \ref{iena} the vector bundle $f_*L$ can also be described via a factor of automorphy on $Y=V/\Gamma_Y$
$$a_{f_*L}: V\times \Gamma_Y \ra GL(n,\dC).$$ 
Even though the factor of automorphy is not a representation, we mimic the construction of an \emph{induced representation}. We consider $a_L$ roughly as a representation of $V\times \Gamma_X$ on $\dC$ with $(v,\lambda)\in V\times \Gamma_X$ acting on $t\in \dC$ as
$$(v,\lambda):\ \ \ t\ \ \mapsto\ \ a_L(v,\lambda)t.$$
Let $\{\lambda_1,...,\lambda_n\}$ be a full set of representatives of the cosets $[\Gamma_Y:\Gamma_X]$.  That is,
$$\Gamma_Y=\coprod_{i=1}^n \lambda_i \Gamma_X.$$
We define a representation of $V\times \Gamma_Y$ on 
$$\dC^n=\bigoplus_{i=1}^n\lambda_i\dC.$$
For any $\lambda\in \Gamma_Y$ we write
$$\lambda+\lambda_i=\lambda_{\lambda(i)}+\Lambda^i_\lambda$$
where $\lambda(i)\in\{1,...,n\}$ and $\Lambda^i_\lambda\in \Gamma_X$. That is, we write $\lambda+\lambda_i$ as an element of the coset $\lambda_{\lambda(i)}\Gamma_X$. Note that this way $\lambda$ induces a permutation of $\{1,...,n\}$. Indeed, if $\lambda+\lambda_i$ and $\lambda+\lambda_j$ were in the same coset, then $\lambda_i-\lambda_j$ would be in $\Gamma_X$ which contradicts that they form a set of representatives.

We define the action of $(v,\lambda)\in V\times \Gamma_Y$ as follows
$$(v,\lambda):\ \sum_{i=1}^n\lambda_it_i\ \ \ \mapsto\ \ \ \sum_{i=1}^n \lambda_{\lambda(i)} a_L(v+\lambda_i,\Lambda^i_\lambda)t_i.$$
As a matrix we represent this action as
\begin{equation}\label{holo factor pushforward under isogeny}
\begin{aligned}
    a_{f_*L}:\ V\times \Gamma_Y \ \ \ &\ra\ \ \ GL_n(\dC)\\
     a_{f_*L}(v,\lambda)&=\Big(a_L(v+\lambda_{\lambda(i)},\Lambda^i_\lambda) \delta^{\lambda(i)}_j \Big)_{ij}.
     \end{aligned}
\end{equation}
\begin{proposition}
The map $a_{f_*L}$ is a factor of automorphy, independent of the choice of representatives and represents $f_*L$.
\end{proposition}
\begin{proof}
First we show that $a_{f_*L}$ satisfies the product rule 
\begin{align*}
a_{f_*L}(v,\lambda+\mu)=a_{f_*L}(v+\lambda, \mu)a_{f_*L}(v,\mu)
\end{align*}
for factors of automorphy. Let $\lambda+\lambda_i=\lambda_{\lambda(i)}+\Lambda^i_\lambda$ and $\mu+\lambda_i=\lambda_{\mu(i)}+\Lambda_\mu^i$ with $\Lambda^i_\lambda,\Lambda^i_\mu\in \Gamma_X$. Then,
\begin{align*}
a_{f_*L}(v+\lambda, \mu)a_{f_*L}(v,\mu)&=\Big (\sum_{j=1}^n a_L(v+\lambda+\lambda_{\mu(i)},\Lambda^i_\mu) \delta^{\mu(i)}_j a_L(v+\lambda_{\lambda(j)},\Lambda^j_\lambda) \delta^{\lambda(j)}_k \Big)_{ik} \\
[\delta^{\mu(i)}_j \delta^{\lambda(j)}_k\neq 0\  \text{ if and only if }&\ \mu(i)=j\ \ \text{and} \ \ \lambda(j)=\lambda(\mu(i))=(\lambda+\mu)(i)=k] \\
&=\Big ( a_L(v+\lambda+\lambda_{\mu(i)},\Lambda^i_\mu)  a_L(v+\lambda_{\lambda(\mu(i))},\Lambda^{\mu(i)}_\lambda)  \Big)_{i \lambda(\mu(i))}\\
&=\Big ( a_L(v+\lambda_{\lambda(\mu(i))}+\Lambda^{\mu(i)}_\lambda,\Lambda^i_\mu)  a_L(v+\lambda_{\lambda(\mu(i))},\Lambda^{\mu(i)}_\lambda)  \Big)_{i \lambda(\mu(i))}\\
&=\Big (a_L(v+\lambda_{\lambda(\mu(i))},\Lambda^{\mu(i)}_\lambda+\Lambda^i_\mu)  \Big)_{i \lambda(\mu(i))}\\
&=\Big (a_L(v+\lambda_{(\lambda+\mu)(i)},\Lambda^{i}_{\lambda+\mu})  \Big)_{i (\lambda+\mu)(i)}\\
&=a_{f_*L}(v,\lambda+\mu),
\end{align*}
since $\lambda+\mu+\lambda_i=\lambda+\lambda_{\mu(i)}+\Lambda^i_\mu=\lambda_{\lambda(\mu(i))}+\Lambda^{\mu(i)}_\lambda+\Lambda^i_\mu$.

To show that the description is independent of the choice of generators let $\{\lambda'_1,...,\lambda'_n\}$ be another set of generators such that $\lambda'_i-\lambda_i=\mu_i\in \Gamma_X$. Then, for any $\lambda\in \Gamma_Y$ we have
$$\lambda+\lambda'_i=\lambda+\lambda_i+\mu_i=\lambda_{\lambda(i)}+\Lambda^i_{\lambda}+\mu_i=\lambda'_{\lambda(i)}-\mu_{\lambda(i)}+\Lambda^i_{\lambda}+\mu_i.$$
That is, $\Lambda^{\prime i}_\lambda=\Lambda^i_\lambda+\mu_i-\mu_{\lambda(i)}$ and we have
\begin{align*}
    a_L(v+\lambda'_{\lambda(i)},\Lambda^{\prime i}_\lambda)&=a_L(v+\lambda_{\lambda(i)}+\mu_{\lambda(i)},\Lambda^i_\lambda+\mu_i-\mu_{\lambda(i)})\\
    &=a_L(v+\lambda_{\lambda(i)},\Lambda^i_\lambda+\mu_i)a_{L}(v+\lambda_{\lambda(i)},\mu_{\lambda(i)})^{-1}\\
    &=a_L(v+\lambda_{\lambda(i)}+\Lambda^i_\lambda,\mu_i)a_L(v+\lambda_{\lambda(i)},\Lambda^i_\lambda)a_L(v+\lambda_{\lambda(i)},\mu_{\lambda(i)})^{-1}\\
    &=a_L(v+\lambda_i+\lambda,\mu_i)a_L(v+\lambda_{\lambda(i)},\Lambda^i_\lambda)a_L(v+\lambda_{\lambda(i)},\mu_{\lambda(i)})^{-1}.
\end{align*}
Define the holomorphic function
\begin{align*}
&\phi:\ \  V\ \ \ra\ \  GL_n(\dC)\\
&\phi(v)=diag(a_L(v+\lambda_1,\mu_1),...,a_L(v+\lambda_n,\mu_n)).
\end{align*}
If we denote by $a_{f_*L}'$ the factor of automorphy defined using the representatives $\{\lambda'_1,...,\lambda'_n\}$, then we have
$$a_{f_*L}'(v,\lambda)=\phi(v+\lambda) a_{f_*L}(v,\lambda)\phi^{-1}(v),$$
that is, $a_{f_*L}$ and $a_{f_*L}'$ are equivalent.

Finally, if $E$ is the holomorphic vector bundle corresponding to $a_{f_*L}$ we need to show for any $U\subset Y$ connected open
    $$\Gamma(f^{-1}(U),L)\cong \Gamma(U,E).$$
    If $X\cong V/\Gamma_X$ and $Y\cong V/\Gamma_Y$ let $\bar{U}$ be the preimage of $U$ and also of $f^{-1}(U)$ in $V$. A section $s\in  \Gamma(U,E)$ is given by a map
    $$s: \bar{U}\ra \dC^n$$
    satisfying $s(v+\lambda)=a_{f_*L}(v,\lambda)s$ for all $\lambda \in \pi_1(U)\leq \Gamma_Y $, where $\pi_1$ is the fundamental group with any choice of base point.

    To write $a_{f_*L}$ we chose a set of representatives of the cosets $\Gamma_Y/\Gamma_X$. We also have $\pi_1(f^{-1}(U))\leq \pi_1(U)$. From the isomorphism theorem
    $$\pi_1(U)/\pi_1(f^{-1}(U))= \pi_1(U)/(\pi_1(U)\cap \Gamma_X)=\pi_1(U)\Gamma_X/\Gamma_X$$
    with $\pi_1(U)\Gamma_X\leq \Gamma_Y$ a subgroup. So we may choose representatives of the cosets $\Gamma_Y/\Gamma_X$ such that the first $[\pi_1(U):\pi_1(f^{-1}(U)] $ of them represent the cosets $\pi_1(U)/\pi_1(f^{-1}(U))$ and such that $\lambda_1\in \Gamma_X$.
    
    If we write $s=(s_1,...,s_n)$ with $s_i:\bar{U}\ra \dC$, the $s_i$ satisfy
    $$s_i(v+\lambda)=a_L(v+\lambda_{\lambda(i)},\Lambda^i_\lambda)s_{\lambda(i)}(v).$$
    In particular, $a_L(v,\lambda_1)^{-1}s_1$ can be seen as a section of $L$ over $f^{-1}(U)$ since for all $\lambda\in \pi_1(U)\cap \Gamma_X=\pi_1(f^{-1}(U))$ it satisfies
    \begin{align*}
        a_L(v+\lambda,\lambda_1)^{-1}s_1(v+\lambda)&=a_L(v+\lambda,\lambda_1)a_L(v+\lambda_1,\lambda)s_1(v)\\
        &=a_L(v,\lambda)a_L(v,\lambda+\lambda_1)^{-1}a_L(v,\lambda+\lambda_1)a_L(v,\lambda_1)^{-1}s_1(v)\\
        &=a_L(v,\lambda)a_L(v,\lambda_1)^{-1}s_1(v).
    \end{align*}
    So we have a map $\Gamma(U,E)\ra \Gamma(f^{-1}(U),L)$. The inverse map $\Gamma(f^{-1}(U),L)$ is given as follows. A section $s\in \Gamma(f^{-1}(U),L)$ is a map 
    $$s: \bar{U} \ra \dC$$
    satisfying $s(v+\lambda)=a_L(v,\lambda)s(v)$ for all $\lambda\in \pi_1(f^{-1}(U))$. We may construct a section
    $$(s_1,...,s_n): \bar{U}\ra \dC^n$$
    as $s_i(v):=s(v+\lambda_i)$. Then for any $\lambda\in \pi_1(U)$ we have
    \begin{align*}
    s_i(v+\lambda)&=s(v+\lambda+\lambda_i)=s(v+\lambda_{\lambda(i)}+\Lambda^i_\lambda)=a_L(v+\lambda_{\lambda(i)},\Lambda^i_\lambda)s_i(v+\lambda_{\lambda(i)})\\
    &=a_L(v+\lambda_{\lambda(i)},\Lambda^i_\lambda)s_{\lambda(i)}(v)
    \end{align*}

\end{proof}
In \cite[Proposition 3.1]{matsushima} Matsushima explained how one can view the factors of automorphy of line bundles as holomorphic representations of the \emph{Heisenberg group} $G_H(\Gamma)$ associated to the lattice $\Gamma$ and to a Hermitian pairing $H\in NS(X)$. The Heisenberg group is $V\times \Gamma_X$ with a modified product rule such that factors of automorphy become representations. From this point of view, the pushforward is really the induced representation. 

He also studied representations of the group $G_H(\Gamma)$ for $H\in NS(X)\otimes \dQ$. That is, the imaginary part of the Hermitian pairing is only required to take rational values on the lattice.  Matsushima showed that the vector bundles which can be constructed from irreducible representations of $G_H(\Gamma)$ are the ones whose factor of automorphy can be brought to the form.
$$a_{E}(v,\gamma)=exp\Big(\pi H(v,\gamma)+\frac{\pi}{2}H(\gamma,\gamma)\Big)\cdot U(\gamma) $$
where $U$ is a \emph{semi-representation} of $\Gamma$ in $U(n)$ with respect to $H$.
\begin{definition}
    A \emph{semi-representation} of $\Gamma$ in $U(n)$ with respect to $H\in NS(X)\otimes \dQ$ is a map
    $$U:\Gamma\ra U(n)$$
    which satisfies
    $$U(\lambda+\mu)=U(\lambda)U(\mu)exp(i\pi ImH(\lambda,\mu)).$$
\end{definition}
These factors of automorphy describe semihomogeneous vector bundles, holomorphic vector bundles whose projectivization is \emph{homogeneous}. It is easy to see that $f_*L$ is semihomogeneous.
\small
\begin{equation}\label{f_*L as semiomog factor}
\begin{aligned}
    a_{f_*L}&(v,\lambda)=\\
    =&\Big(\chi(\Lambda^i_\lambda)exp(\pi H(v+\lambda_{\lambda(i)},\Lambda^i_\lambda)+\frac{\pi}{2} H(\Lambda^i_\lambda,\Lambda^i_\lambda)) \delta^j_{\lambda(i)} \Big)^i_j\\
    =&\Big(\chi(\Lambda^i_\lambda)exp(\pi H(v+\lambda_{\lambda(i)}, \lambda+\lambda_i-\lambda_{\lambda(i)})+\frac{\pi}{2} H(\lambda+\lambda_i-\lambda_{\lambda(i)},\lambda+\lambda_i-\lambda_{\lambda(i)})) \delta^j_{\lambda(i)} )\Big)^i_j\\
    =&\Big(\chi(\Lambda^i_\lambda)exp\Big(\pi H(v,\lambda)+\frac{\pi}{2}H(\lambda,\lambda)+\pi H(v,\lambda_i-\lambda_{\lambda(i)})+\pi H(\lambda_{\lambda(i)},\lambda+\lambda_i-\lambda_{\lambda(i)})\Big)\times\\
    &\times exp\Big( \frac{\pi}{2}H(\lambda,\lambda_i-\lambda_{\lambda(i)})+\frac{\pi}{2}H(\lambda_i-\lambda_{\lambda(i)},\lambda)+\frac{\pi}{2}H(\lambda_i-\lambda_{\lambda(i)},\lambda_i-\lambda_{\lambda(i)}) \delta^j_{\lambda(i)} \Big)^i_j\\
    =&\phi(v+\lambda)exp\Big(\pi H(v,\lambda) +\frac{\pi}{2}H(\lambda,\lambda)\Big)\cdot \Big( \chi(\Lambda^i_\lambda) exp\Big( i\pi Im H(\lambda+\lambda_i,\lambda-\lambda_{\lambda(i)})\Big)\delta^j_{\lambda(i)} \Big)^i_j\phi(v)^{-1}
\end{aligned}
\end{equation}
\normalsize
for $\phi(v)=diag\Big(exp\Big(\pi H(v,\lambda_i)+\frac{\pi}{2}H(\lambda_i,\lambda_i)\Big)\Big)$.

It turns out that the converse is also true. A holomorphic vector bundle $E$ on $X$ is simple and semihomogeneous if and only if there exists an isogeny $f:Y\ra X$ whose degree is the rank of $E$ and a line bundle $L$ on $X$ such that $f_*L=E$ (\cite[Theorem 8.1]{matsushima}, \cite[Proposition 7.3]{mukaiSemihomog}). Such vector bundles are also called \emph{projectively flat} as they admit Hermitian connections whose curvature $F\in \Omega^2(X,End(E))$ is a two-form times the identity \cite[Theorem 4.7.54]{kobayashi}.

\begin{lemma}\label{lemma: pushforward of trivial line bundle}
    Let $f:X\ra Y$ be a degree $n$ isogeny between $g$-dimensional  complex tori. Let $\{L_1,...,L_n\}$ be the flat line bundles forming the kernel of $\hat{f}:\hat{Y}\ra \hat{X}$. Then,
    $$f_*\cO_X=\bigoplus_{i=1}^nL_i.$$
\end{lemma}
\begin{proof}
    Since $\cO_X$ is flat, its factor of automorphy  is independent of $V$ and it is actually a representation $a_{\cO_X}:\Gamma_X\ra U(1)$. Since $\cO_X$ is the trivial holomorphic line bundle $a_{\cO_X}$ is the trivial representation.
    Then, $a_{f_*\cO_X}$ is again independent of $V$ and is given by the induced representation $\Gamma_Y\ra GL_n(\dC)$. The quotient $\Gamma_Y/\Gamma_X=G$ is a finite abelian group so we may write
    $$G=\dZ/d_1\dZ\oplus...\oplus \dZ/d_g\dZ$$
    for some $d_i\in \dZ_{> 0}$ such that $d_1\cdot...\cdot d_g=n$. Let $(\lambda_1,...,\lambda_g)$ be a generating subset of $\Gamma_Y$ which project to a set of generators of $G$. Then,
    $$S=\Big\{\sum_{i=1}^k m_i\lambda_i\ | \ m_i\in \{0,...,d_i-1\}\Big \}$$
    is a full set of representatives of the $\Gamma_X$-cosets in $\Gamma_Y$.

    Since $a_{f_*\cO_X}$ is a representation it is sufficient to describe $a_{f_*\cO_X}(\lambda_i)$ for $i=1,...,g$.  It is a permutation matrix of order $d_i$ acting on $S$ by sending
    $$\sum_{j=1}^k m_j\lambda_j\  \mapsto\  \sum_{j=1}^k m_j\lambda_j+\lambda_i.$$
    Moreover, since $$\dC^n=\bigotimes_{i=1}^g\dC^{d_i}=\bigotimes_{i=1}^g\bigoplus_{m=0}^{d_i-1}e_{m\lambda_i}\dC$$
    and  $a_{f_*\cO_X}(\lambda_j)$ only acts on the $j^{\text{th}}$ tensor factor we may write it as a tensor product representation
    $$a_{f_*\cO_X}(\lambda_j)=Id_{d_1\times d_1}\otimes ...\otimes Id_{d_{j-1}\times d_{j-1}}\otimes 
    \begin{pmatrix}
    0 & 0 & \hdots & 0 & 1\\
    1 & 0 & 0 & \hdots & 0 \\
    0 & 1 & 0 &   \hdots & 0 \\
    \vdots &  & &  & \vdots\\
    0 &\hdots  & 0  & 1 & 0
    \end{pmatrix}\otimes Id_{d_{j+1}\times d_{j+1}} \otimes ... \otimes Id_{d_g\times d_g}$$
    We may diagonalize this representation by changing the basis of $\dC^{d_i}$ for all $i=1,...,g$ as
    \begin{align}\label{change of basis}
    \begin{pmatrix}
        e_{0\lambda_i} \\ e_{1\lambda_i} \\ \vdots \\ e_{(d_i-1)\lambda_i}
    \end{pmatrix}\mapsto \begin{pmatrix}
    e_{0\lambda_i} + e_{1\lambda_i} + ...+  e_{(d_i-1)\lambda_i}\\
        e_{0\lambda_i} +\xi_i e_{1\lambda_i} + ...+ \xi_i^{d_i-1} e_{(d_i-1)\lambda_i}\\
        \vdots \\
        e_{0\lambda_i} +\xi_i^{d_i-1} e_{1\lambda_i} + ...+ \xi_i^{(d_i-1)^2} e_{(d_i-1)\lambda_i}
    \end{pmatrix}
    \end{align}
    where $\xi_i$ is a primitive $d_i^{\text{th}}$ root of unity. Changing the basis can be seen as an equivalence relation between the factors of automorphy. If $S$ is the change of basis it can be seen as a constant, hence holomorphic, function $\phi: V\ra GL_n(\dC)$, $\phi(v)=S$. We may also divide out by the determinant of $S$ to induce a unitary change of basis.
    The new (diagonal) factor of automorphy is then given by $\phi(v+\lambda)a_{f_*L}(v,\lambda)\phi(v)^{-1}$.
    
    In this new basis, our representation is given by
    $$a_{f_*\cO_X}(\lambda_j)=Id_{d_1\times d_1}\otimes ...\otimes Id_{d_{j-1}\times d_{j-1}}\otimes 
    \begin{pmatrix}
    1 & 0 & \hdots & 0 \\
    0 & \xi_i & \hdots & 0 \\
    \vdots & & & \vdots\\
    0 & \hdots & 0 & \xi_i^{d_i-1}
    \end{pmatrix}\otimes Id_{d_{j+1}\times d_{j+1}} \otimes ... \otimes Id_{d_g\times d_g}.$$
    That is,
    $$f_*\cO_X=\bigotimes_{i=1}^g \bigoplus_{j=1}^{d_i} L_{\xi_i}^j ,$$
        where $L_{\xi_i}$ is the line bundle given by the factor of automorphy
        $$a_{L_{\xi_i}}(\lambda_j)=\xi_i\delta_{ij}.$$
        Exchanging addition and multiplication we find
    $$f_*\cO_X=\bigoplus_{m_i\in [0,...,d_{i-1}]} \bigotimes_{i=1}^{g} L_{\xi_i}^{m_j},$$
    and the line bundles $\bigotimes_{i=1}^{g} L_{\xi_i}^{m_i}$ for all $m_i\in[0,...,d_i-1] $ and all $i=1,...,g$ are precisely the ones in the kernel of $f$.
\end{proof}
\begin{remark}
Any flat line bundle $L\ra X$ has a factor of automorphy which is a representation $\Gamma_X\ra U(1)$. Therefore, the induced representation $\Gamma_Y\ra GL_n(\dC)$ is again unitary. More precisely, in the notation of the previous lemma
 \begin{align*}&a_{f_*L}(\lambda_j)=\\
 &=Id_{d_1\times d_1}\otimes ...\otimes Id_{d_{j-1}\times d_{j-1}}\otimes 
    \begin{pmatrix}
    0 & 0 & \hdots & 0 & a_L(d_j\lambda_j)\\
    1 & 0 & 0 & \hdots & 0 \\
    0 & 1 & 0 &   \hdots & 0 \\
    \vdots &  & &  & \vdots\\
    0 &\hdots  & 0  & 1 & 0
    \end{pmatrix}\otimes Id_{d_{j+1}\times d_{j+1}} \otimes ... \otimes Id_{d_g\times d_g}.\end{align*}
Again this matrix is unitary, hence it is unitarily diagonalizable. Moreover all $a_{f_*L}(\lambda_j)$ commute therefore they are simultaneously diagonalizable as in the case of $a_{f_*\cO_X}$. If we denote the dual isogeny by $\hat{f}:\hat{Y}\ra \hat{X}$ the same proof as before yields
$$f_*L\cong \bigoplus_{N\in \hat{f}^{-1}(L)}N.$$
\end{remark}

Let $L\ra X$ be a non-degenerate line bundle of type $(d_1,...,d_g)$. Then, the corresponding isogeny
$$\phi_L: X\ra \hat{X}$$
is of degree $d^2=d^2_1\cdot ...\cdot d^2_g$ and 
$$\Gamma_{\hat{X}}/\Gamma_X=\Big(\dZ/d_1\dZ\Big)^{\oplus 2}\oplus...\oplus \Big(\dZ/d_g\dZ\Big)^{\oplus 2}.$$
We have seen before that there exists an isogeny $f:X\ra Y$ of degree $d$  and a non-degenerate line bundle $N\ra Y$ of type $(1,...,1)$ such that $f^*N=L$ and that we have the following commutative diagram.
\begin{equation}\label{komcsi negyzet}
\begin{tikzcd}
    X \arrow{r}{\phi_L} \arrow{d}[swap]{f} & \hat{X} \\
    Y \arrow{r}{\phi_N} & \hat{Y} \arrow{u}[swap]{\hat{f}}
\end{tikzcd}
\end{equation}

\begin{lemma}\label{Lemma: translation and tensor}
    Let $N\ra Y$ be a non-degenerate type $(1,...,1)$ line bundle on the complex torus $Y$. Then for any flat line bundle $L_{\hat{y}}\in Pic^0(Y)$ corresponding to $\hat{y}\in \hat{Y}$ we have
    $$(\phi_N)_*(N^{-1}\otimes L_{\hat{y}})\cong t_{-\hat{y}}^*(\phi_{N})_*N^{-1}.$$
\end{lemma}
\begin{proof}
The canonical factor of automorphy for $N^{-1}$ is given by
$$a_{N^{-1}}(v,\lambda)=\chi(\lambda)^{-1}exp(-\pi H(v,\lambda) -\frac{\pi}{2} H(\lambda,\lambda))$$
if $N$ corresponds to the Chern class $H$ and semicharacter $\chi$. The analytification of $\phi_N$ is given by
$$\phi_N: V\ra V^*,\ \ \ v\mapsto H(v),$$
and $H(\Gamma_Y)\subset \Gamma_Y^*=\Gamma_{\hat{Y}}.$ Hence,
\begin{align*}
    a_{(\phi_N)_*N^{-1}}(\hat{v},\hat{\lambda})&=\chi(H^{-1}(\hat{\lambda}))^{-1}exp(-\pi H(H^{-1}(\hat{v}),H^{-1}(\hat{\lambda})) -\frac{\pi}{2} H(H^{-1}(\hat{\lambda}),H^{-1}(\hat{\lambda})))\\
    &=\chi(H^{-1}(\hat{\lambda}))^{-1}exp(-\pi H^{-1}(\hat{\lambda},\hat{v})-\frac{\pi}{2}H^{-1}(\hat{\lambda},\hat{\lambda})),
\end{align*}
where we use that
\begin{align}
    H(H^{-1}(\hat{v}),H^{-1}(\hat{\lambda}))=\Big(H\circ H^{-1}(\hat{v})\Big )(H^{-1}(\hat{\lambda}))=\hat{v}(H^{-1}(\hat{\lambda}))=H^{-1}(\hat{\lambda},\hat{v}).
\end{align}
The flat line bundle $L_{\hat{y}}$ is given by the factor of automorphy
$$a_{L_{\hat{y}}}(v,\lambda)=exp(2\pi i\cdot \text{Im} \hat{y}(\lambda)).$$
Then,
\begin{align*}
a&_{(\phi_N)_*(N^{-1}\otimes L_{\hat{y}})}(\hat{v},\hat{\lambda})=\\
&=\chi(H^{-1}(\hat{\lambda}))^{-1}exp(-\pi H^{-1}(\hat{\lambda},\hat{v})-\frac{\pi}{2}H^{-1}(\hat{\lambda},\hat{\lambda}))\cdot exp(2\pi i\cdot \text{Im} \hat{y}(H^{-1}(\hat{\lambda})))\\
&=\chi(H^{-1}(\hat{\lambda}))^{-1}exp(-\pi H^{-1}(\hat{\lambda},\hat{v})-\frac{\pi}{2}H^{-1}(\hat{\lambda},\hat{\lambda}))\cdot exp(\pi \cdot  (H^{-1}(\hat{\lambda},\hat{y})-H^{-1}(\hat{y},\hat{\lambda})))\\
&=\chi(H^{-1}(\hat{\lambda}))^{-1}exp(-\pi H^{-1}(\hat{\lambda},\hat{v}-\hat{y})-\frac{\pi}{2}H^{-1}(\hat{\lambda},\hat{\lambda}))\cdot exp(-\pi \cdot  H^{-1}(\hat{y},\hat{\lambda}))
\end{align*}
The function $\phi: V^*\ra \dC$ given by $\phi(\hat{v})=exp(\pi \cdot  H^{-1}(\hat{y},\hat{v}))$ is holomorphic in $\hat{v}$ as $H^{-1}(\hat{y},\hat{v})=\hat{v}(H^{-1}(\hat{y}))$. Therefore, via $\phi$ the factor of automorphy $a_{(\phi_N)_*(N^{-1}\otimes L_{\hat{y}})}$ is equivalent to the factor of automorphy 
\begin{align*}a_{t_{-\hat{y}}^*(\phi_{N})_*N^{-1}}(\hat{v},\hat{\lambda})&=a_{(\phi_{N})_*N^{-1}}(\hat{v}-\hat{y},\hat{\lambda})\\
&=\chi(H^{-1}(\hat{\lambda}))^{-1}exp(-\pi H^{-1}(\hat{\lambda},\hat{v}-\hat{y})-\frac{\pi}{2}H^{-1}(\hat{\lambda},\hat{\lambda})).\end{align*}
\end{proof}

\begin{remark}
    Note that for the Fourier-Mukai transform we have \cite[Proposition 14.3.1]{BL} 
    $$\widehat{N\otimes L_{\hat{y}}}=(\phi_{N})_*(N\otimes L_{\hat{y}})^{-1}=(\phi_{N})_*(N^{-1}\otimes L_{-\hat{y}})=t_{\hat{y}}^*(\phi_{N})_*N^{-1}=t_{\hat{y}}^*\hat{N} .$$
\end{remark}

\begin{lemma}\label{Lemma: invariance under kernel}
Let $f:X\ra Y$ be a degree $n$ isogeny of $g$ dimensional complex abelian varieties. Let $x\in Ker(f)$ be a point in $X$ and $\cF$ a sheaf on $X$. Then,
$$f_*t_x^*\cF\cong f_*\cF.$$
\end{lemma}
\begin{proof}
    The sheaf $f_*t_x^*\cF$ is given over $U\subset Y$ as
    $$\Gamma(U,f_*t_x^*\cF)=\Gamma(f^{-1}(U),t_x^*\cF)=\Gamma(t_xf^{-1}(U),\cF)=\Gamma(f^{-1}(U),\cF).$$
 Indeed, since $f$ is an isogeny, it is a homomorphism so $f(x+y)=f(x)+f(y)=f(y)$ for any $x\in K(f)$ and $y\in X$.    
\end{proof}

\begin{proposition}\label{FM transform and phi L holo}
We have
$$(\phi_L)_*L^{-1}\cong \hat{f}_*(\phi_N)_*N^{-1} \otimes H^s(X,L)$$
where $s$ is the index of $L$.
\end{proposition}
\begin{proof}
For a non-degenerate line bundle $L$ of index $s$ and degree $d$ we have $H^s(X,L)\cong \dC^d$ \cite[Theorem 3.5.5]{BL}. Then, 
\begin{align*}
(\phi_L)_*L^{-1}&=\hat{f}_*(\phi_N)_*f_*f^*N &(\ref{komcsi negyzet})\\
&=\hat{f}_*(\phi_N)_*f_*(f^*N\otimes_{\cO_X} \cO_X)\\
&=\hat{f}_*(\phi_N)_*(N\otimes_{\cO_Y} f_*O_X) &(\text{projection formula})\\
&=\hat{f}_*(\phi_N)_*(N\otimes_{\cO_Y} \oplus_{\hat{y}\in Ker(\hat{f})} L_{\hat{y}})  &(\text{Lemma }\ref{lemma: pushforward of trivial line bundle})\\
&=\hat{f}_*(\phi_N)_*(\oplus_{\hat{y}\in Ker(\hat{f})} N^{-1}\otimes L_{\hat{y}})\\
&=\bigoplus_{\hat{y}\in Ker(\hat{f})}\hat{f}_*(\phi_N)_*( N^{-1}\otimes L_{\hat{y}})\\
&=\bigoplus_{\hat{y}\in Ker(\hat{f})} \hat{f}_*(\phi_{N})_*(N^{-1}\otimes L_{\hat{y}}) &(\text{Lemma }\ref{Lemma: translation and tensor})\\
&=\bigoplus_{\hat{y}\in Ker(\hat{f})} \hat{f}_*t_{-\hat{y}}^*(\phi_N)_*N^{-1}\ \ \ \ &(\text{Lemma }\ref{Lemma: invariance under kernel})\\
&=\bigoplus_{\hat{y}\in Ker(\hat{f})} \hat{f}_*(\phi_N)_*N^{-1}.
\end{align*}
The isogeny $f:X\ra Y$ is also of degree $d$ so $|Ker(\hat{f})|=d$ and we are done.
\end{proof}

\chapter{Factors of automorphy on real tori}
In this chapter, we describe how the factor of automorphy construction generalizes to principal $U(n)$-bundles with connections. This description for $U(1)$-bundles was used by Bruzzo, Marelli and Pioli in \cite{BMP1, BMP2} where they described the T-dual of flat $U(1)$-bundles supported on affine torus subbundles of affine torus bundles with trivial Chern classes. Their work greatly influenced ours as they have shown that using factors of automorphy one can translate techniques from the theory of complex tori to the real case. 

In the first section, we define factors of automorphy for principal $U(d)$-bundles. These classify principal $U(d)$-bundles topologically. Then, we describe how one can encode a connection on a principal bundle to a pair of a factor of automorphy and a one-form. We prove a version of the Appel-Humbert Theorem for $U(1)$-bundles with connections. However, in the $U(1)$-case we cannot find a canonical factor for all the bundles, we have to restrict to those whose connection has invariant curvature. This aligns with your idea of enhancing T-duality of generalized branes to T-duality of physical branes. Indeed, if $S$ is an affine torus subbundle of an affine torus bundle endowed with a $U(1)$-bundle $L\ra S$ with a connection whose curvature $2\pi i F$ is invariant, then $(S,F)$ is a locally T-dualizable brane.

In the next section, we prove the statement analogous to Proposition \ref{FM transform and phi L holo} for $U(1)$-bundles with connections. We spend most of the section showing that our construction is independent of the choices we make.  In the holomorphic case, the uniqueness of the Fourier-Mukai transform takes care of this problem.

Then, we describe the pushforward of a $U(1)$-bundle with connection along a general homomorphism of real tori. We then prove the analogue of (\ref{main eqn degen line bdle}) for $U(1)$-bundles with connections.

Finally, we show how the canonical factor of a holomorphic line bundle on a complex torus induces a canonical factor of the underlying $U(1)$-bundle with a connection. We show that pushing forward the holomorphic line bundle or the $U(1)$-bundle yields the same result and define the $U(1)$-version of the Poincar\'e line bundle for real tori.

\section{General construction}
In this section, we define factors of automorphy for principal $U(n)$-bundles on a real torus and explain how one can represent a connection on it. We then define the pushforward of a $U(1)$-bundle along an isogeny analogously to the holomorphic case. We define an analogue of the map $\phi_L$ for those $U(1)$-bundles with connections whose curvature is an invariant two-form. We prove a version of the Appel-Humbert theorem for these bundles as well. Finally, we show that once again, the pushforward of such a $U(1)$-bundle along an isogeny is projectively flat.

Let $X=V/\Gamma$ be a real torus with quotient map $\pi:V\ra X$. 
\begin{definition}
A \emph{$U(n)$ factor of automorphy} is a smooth function
$$a_P: V\times \Gamma_X \ra U(n),$$
satisfying 
$$a_P(v,\lambda+\mu)=a_P(v+\lambda,\mu)a_P(v,\lambda).$$ 
Two factors of automorphy $a_P$ and $a_P'$ are equivalent if they are related by a map $\phi: V\ra U(n)$ via
$$a_P'(v,\lambda)= \phi(v+\lambda)a_P(v,\lambda)\phi(v)^{-1}.$$
\end{definition}
The equivalence classes of $U(n)$ factors of automorphy are the first cohomology of the group $\Gamma$ in the ring $\cC^\infty(V,U(n))$ of $U(n)$-valued smooth functions on $V$. We denote it by $\HH^1(\Gamma,\cC^\infty(V,U(n))).$
We can mimic the proof of  \cite[Proposition B1]{BL} or that of \cite[Theorem 3.2]{iena} to prove the following.
\begin{theorem}
    There is a bijection of sets
$$\phi_1:\HH^1(\Gamma, \cC^\infty(V,U(n))) \ra Prin_{U(n)}(X),$$
where $Prin_{U(n)}(X)$ denotes the isomorphism classes of principal $U(n)$-bundles on $X$.
\end{theorem}
 The map $\phi_1$ is the associated bundle construction. The transition functions of the associated bundle can be described explicitly as follows. Let $\{U_i\}$ be a good cover of $X$ and let $W_i\subset V$ be a lift of $U_i$ for all $i$ with $\pi_i:=\pi|_{W_i}: W_i\ra U_i$ diffeomorphisms. Then, for any pair $(i,j)$ there exists a unique $\lambda_{ij}\in\Gamma$ such that $\pi_i^{-1}(x)+\lambda_{ij}=\pi_j^{-1}(x)$ for all $x\in U_{ij}$.  For a factor of automorphy $f$ define $g_{ij}=f(\pi_{i}^{-1}(x),\lambda_{ij})\in \Gamma(U_{ij},U(n))$. Then, $(g_{ij})\in H^1(X,\cC^\infty_{U(n)})$ due to the factor-of automorphy condition and that $\lambda_{ij}+\lambda_{jk}=\lambda_{ik}.$

We omit the proof as it is a straightforward adaptation of the proofs in the referenced works. Let $a_P$ be a factor of automorphy. We define the associated principal bundle as the quotient bundle
$$P:=V\times_{\Gamma} U(n)$$
under the action $\lambda(v,A)=(v+\lambda,a_P(v,\lambda)A)$ for $\lambda\in \Gamma$ and $(v,A)\in V\times U(n)$. A principal bundle connection on $V\times U(n)$ descends to one on $P$ if and only if is invariant under the action of $\Gamma$. Denote the Lie algebra of $U(n)$ by $\gu(n)$. We can prove the following.
\begin{proposition}\label{factor of automoprhy and connection proof}
There is a one-to-one correspondence between equivalence classes of principal bundle connections on $P\ra X$ and equivalence classes of pairs $(f,A)$ where $f$ is a factor of automorphy representing $P$ and $A\in \Omega^1(V,\gu(n))$ is a one-form satisfying 
$$A(v+\gamma)=f(v,\gamma)A(v)f(v,\gamma)^{-1}-df(v,\gamma)\cdot f^{-1}(v,\gamma).$$
Two pairs $(f,A)$ and $(f',A')$ are equivalent if for a smooth function $\phi: V\ra U(n)$ we have
$$f'(v,\lambda)=\phi(v+\lambda)f(v,\lambda)\phi(v)^{-1}$$
$$A'(v)=\phi(v)A(v)\phi^{-1}(v)-d\phi(v)\cdot \phi^{-1}(v).$$
\end{proposition}
\begin{proof}
The proof is standard in gauge theory. There is a one-to-one correspondence between principal bundle connections on $P\ra X$ and those on $V\times U(n)\ra V$ invariant under the action of $\Gamma=\pi_1(X)$. Two connections $\omega_1$ and $\omega_2$ on $P$ are equivalent if there is an automorphism $\varphi$ of $P$ such that $\omega_1=\varphi^*\omega_2$. Automorphisms of $P$ also descend from automorphisms of $V\times U(n)$ generated by global functions $\phi:V\ra U(n)$. We will show that invariant connections are in one-to-one correspondence with one-forms $A$ as in the statement and that the equivalence relation corresponds to pulling back under automorphisms.

   Let $\pi: V\times U(n)\ra P$ be the projection. Given a trivialization, that is a section $s: V\ra \pi^*P$ we get an isomorphism $\pi^*P\cong V\times U(n)$. Let $\omega\in \Omega^1(P,\gu(n))$ be a connection on $P$.     Since $\pi$ is $U(n)$-equivariant, that is, a map of principal bundles,   $\pi^*\omega \in \Omega^1(V\times U(n),\gu)$ is a connection on $V\times U(n)$. Using the section $s$ we define
   $$A:=s^*\pi^*\omega.$$
   An automorphism of $\pi^*P$ can be regarded as changing the section $s$ to $s'(v)=s(v)\cdot \phi(v)$ for some $\phi: V\ra U(n)$. The new section changes the connection form as 
   $$A'(v)=(s')^*\pi^*\omega(v)=s^*R_{\phi(v)}^*\pi^*\omega(v).$$
   Consider real coordinates $(v^\alpha, u^{ij})$ on $V\times U(n)$, where the coordinates on $U(n)$ are entries of a unitary matrix. The fundamental vector fields of the $U(n)$-action are given by
   $$X_\xi=\frac{d}{dt}\Big|_{t=0}exp(t\xi).(v,U)=\frac{d}{dt}\Big|_{t=0}(v,Uexp(t\xi))=(U\xi)_{ij}\frac{\partial}{\partial u^{ij}},\ \ \ \xi\in \gu(n).$$
   Since $\pi^*\omega$ is a connection, for any $\xi\in \gu(n)$ we have $\pi^*\omega(X_\xi)=\xi.$ In coordinates, if $\theta^{ij}$ is a basis of $\gu(n)$ we can write
   $$\pi^*\omega=\omega_\alpha dv^\alpha +\omega_{ij}\theta^{ij},$$
   where $v_\alpha$ are $\gu(n)$-valued functions on $V\times U(n)$ and $\omega_{ij}=(\omega_{ij})_{\alpha\beta}du^{\alpha\beta}$ for some functions $(\omega_{ij})_{\alpha\beta}$ on $V\times U(n)$.
   Then, from $\pi^*\omega(X_\xi)=\xi$ we find that $(\omega_{ij})_{\alpha\beta}(v,U)=(U^{-1})_{ij}\delta^{\alpha\beta}_{ij}$, that is we can write schematically
    $$\pi^*\omega=\omega_\alpha dv^\alpha +U^{-1}dU.$$
   The equivariance condition implies that $g\omega_\alpha(v,Ug)g^{-1}=\omega_\alpha(v,U)$ as functions with values in $\gu(n)$ as $U^{-1}dU$ is already equivariant. Finally, 
   \begin{align*}
       R_{\phi(v)}^*\pi^*\omega(v,U)&=\omega_\alpha(v,U\phi(v))dv^\alpha +(U\phi(v))^{-1}d(U\phi(v))\\
       &=\phi(v)^{-1}\omega_{\alpha}(v,U)\phi(v) +\phi(v)^{-1}U^{-1}(dU \phi(v) +U d\phi(v))\\
       &=\phi(v)^{-1}\omega_{\alpha}(v,U)\phi(v)+\phi(v)^{-1}U^{-1}dU\phi(v) +\phi(v)^{-1}d\phi(v)\\
       &=\phi(v)^{-1}\pi^*\omega(v,U)\phi(v) + \phi(v)^{-1}d\phi(v),
   \end{align*}
   and hence,   
   $$A'(v)=s^*R_{\phi(v)}^*\pi^*\omega(v)=\phi(v)^{-1}A(v)\phi(v) + \phi^{-1}(v) d\phi(v) . $$
The automorphism $s'(v)=s(v)\cdot \phi(v)$ changes the factor of automorphy as well. The action of $\Gamma$ in a trivialization coming from $s$ is given by
   $$\phi_\gamma(s(v)\cdot U)=\phi_\gamma(v,U)=(v+\gamma,f(v,\gamma)U)=s(v+\gamma)\cdot f(v,\gamma)U.$$
   Since $s'(v)=s(v)\cdot \phi(v)$, we have
   $$\phi_\gamma(s'(v)\cdot U)=\phi_\gamma(s(v)\cdot\phi(v)U)=s(v+\gamma)\cdot f(v,\gamma)\phi(v)U=s'(v+\gamma)\cdot \phi(v+\gamma)^{-1}f(v,\gamma)\phi(v)U$$
   therefore, 
   $$f'(v,\gamma)=\phi(v+\gamma)^{-1}f(v,\gamma)\phi(v).$$
   In addition, $\pi^*\omega$ is $\Gamma$-invariant, that is
    $$\phi_\gamma^*\pi^*\omega=\pi^*\omega,$$
 therefore,
   \begin{align*}
       A(v)&=s^*\pi^*\omega(v)\\
       &=s^*\phi_\gamma^*\pi^*\omega(v)\\
       &= s^*(R^*_{f(v,\gamma)}\pi^*\omega(v+\gamma,U))\\
       &=s^*(f(v,\gamma)^{-1}\omega(v+\gamma,U)f(v,\gamma)+f(v,\gamma)^{-1}df(v,\gamma))\\
       &=f(v,\gamma)^{-1}A(v+\gamma)f(v,\gamma)+f(v,\gamma)^{-1}df(v,\gamma).
   \end{align*}
Note that the relation $\phi_\gamma^*\pi^*\omega(v,U)=R_{f(v,\gamma)}^*\pi^*\omega(v+\gamma,U)$ only holds along the section $s(V)$.
\end{proof}

\paragraph{Pushforward of $U(1)$-bundles under isogeny.} Let $X=V/\Gamma_X$ and $Y=W/\Gamma_Y$ be real tori and let $f:X\ra Y$ be a degree $n$ isogeny. Recall, that the analytification of $f$ is its lift of it to a linear isomorphism $F: V\ra W$. That is, we may assume that $V=W$,  $F=id$ and $\Gamma_X\subset \Gamma_Y$. Let $\{\lambda_1,...,\lambda_n\}$ be a full set of representatives of the cosets $\Gamma_Y/\Gamma_X$. 

Let $L\ra X$ be a $U(1)$-bundle with a connection defined by a pair  $(a_L,A_L)$. Then, the pushforward $f_*L$ of $L$ along the isogeny $f$  is given by the factor of automorphy 
\begin{equation}\label{u(n) factor pushforward under isogeny}
\begin{aligned}
    a_{f_*L}&: V\times \Gamma_Y \ra U(n)\\
    a_{f_*L}&(v,\lambda)=\Big(a_L(v+\lambda_{\lambda(i)},\Lambda^i_\lambda) \delta^{\lambda(i)}_j\Big)^i_j
    \end{aligned}
\end{equation}
analogously to the holomorphic definition (\ref{holo factor pushforward under isogeny}). Note that  $a_{f_*L}$ is really unitary, since
$$a_{f_*L}(v,\lambda)^*=\Big(a_L(v+\lambda_{\lambda(i)},\Lambda^i_\lambda)^{-1} \delta^{\lambda(i)}_j\Big)^j_i=a_{f_*L}(v,\lambda)^{-1}.$$
We also have to describe the connection, which is a diagonal matrix
\begin{align}\label{u(n) connection pushforward under isogeny}A_{f_*L}(v)=\Big(A_L(v+\lambda_i)\delta^i_j\Big)^i_j.\end{align}
To show that this is a connection on $f_*L$ we have to show that
$$A_{f_*L}(v+\lambda)=a_{f_*L}(v,\lambda)A_{f_*L}(v)a_{f_*L}(v,\lambda)^{-1}-da_{f_*L}(v,\lambda)\cdot a_{f_*L}(v,\lambda)^{-1}$$
for all $\lambda\in \Gamma_Y$. Indeed, for all $i$ we have
$$A_L(v+\lambda+\lambda_i)=A_L(v+\lambda_{\lambda(i)})-da_L(v+\lambda_{\lambda(i)},\Lambda^i_{\lambda})\cdot a_L(v+\lambda_{\lambda(i)},\Lambda^i_{\lambda})^{-1} $$
which is precisely the condition on $A_L$, hence $A_{f_*L}$ is a well-defined connection on $f_*L$.

\begin{lemma}
    The $U(n)$-bundle with connection $(a_{f_*L},A_{f_*L})$ is independent of the choice of representatives $(\lambda_1,...,\lambda_n)$.
\end{lemma}
\begin{proof}
    We do the same calculation as in the holomorphic case. If $(\lambda'_1,...,\lambda'_n)$ is the other set of representatives defining $(a_{f_*L}',A_{f_*L}')$ and we set $\lambda'_i-\lambda_i=\mu_i$ we may define the function
    $$\phi: V\ra U(n),\ \ \ \phi(v)=diag(a_L(v+\lambda_1,\mu_1),...,a_L(v+\lambda_n,\mu_n))$$
    and we have
    $$a_{f_*L}'(v,\lambda)=\phi(v+\lambda)a_E(V,\lambda)\phi(v)^{-1}.$$
    It is left to show
    $$A_{f_*L}'=\phi A_E \phi^{-1} - d\phi \cdot \phi^{-1}=A_E-d\phi \cdot \phi^{-1}.$$
    because diagonal matrices commute. Indeed, 
    \begin{align*}
        A_L'(v)_{ii}&=A_L(v+\lambda_i')=A_L(v+\lambda_i+\mu_i)\\
        &=A_L(v+\lambda_i)-da_L(v+\lambda_i,\mu_i) a_L(v+\lambda_i,\mu_i)^{-1}\\
        &=A_L(v)_{ii}-d\phi(v)_{ii}\phi(v)_{ii}^{-1}.
    \end{align*}
\end{proof}
In later calculations, we will use the projection formula repeatedly. Although it is quite clear that it holds in this setting we give a quick proof here.
\begin{proposition}[Projection formula]\label{projection formula prop}
    Let $f:X\ra Y$ be a degree $n$ isogeny whose analytification is given by the identity map $V\ra V$ and $\Gamma_X\subset \Gamma_Y$. Let $N$ be a $U(1)$-bundle with connection on $Y$ and $L$ a $U(1)$-bundle with connection on $X$. Then,
    \begin{align}\label{projection formula}
    f_*(f^*N\otimes L) \cong N\otimes f_*L.
    \end{align}
\end{proposition}
\begin{proof}
    For any $v\in V$ and $\lambda\in \Gamma_Y$ we have
    \begin{align*}
        a_{f_*(f^*N\otimes L)}(v,\lambda)&=\Big(a_{f*N}(v+\lambda_{\lambda(i)},\Lambda^i_\lambda)a_L(v+\lambda_{\lambda(i)},\Lambda^i_\lambda)\delta^{\lambda(i)}_j \Big)_{ij}\\
        &=\Big(a_{N}(v+\lambda_{\lambda(i)},\Lambda^i_\lambda)a_L(v+\lambda_{\lambda(i)},\Lambda^i_\lambda)\delta^{\lambda(i)}_j \Big)_{ij}\\
        &=\Big(a_N(v,\lambda_{\lambda(i)}+\Lambda^i_\lambda)a_N(v,\lambda_{\lambda(i)})^{-1}a_L(v+\lambda_{\lambda(i)},\Lambda^i_\lambda)\delta^{\lambda(i)}_j \Big)_{ij}\\
        &=\Big(a_N(v,\lambda+\lambda_i)a_N(v,\lambda_{\lambda(i)})^{-1}a_L(v+\lambda_{\lambda(i)},\Lambda^i_\lambda)\delta^{\lambda(i)}_j \Big)_{ij}\\
        &=\Big(a_N(v+\lambda,\lambda_i)a_N(v,\lambda)a_N(v,\lambda_{\lambda(i)})^{-1}a_L(v+\lambda_{\lambda(i)},\Lambda^i_\lambda)\delta^{\lambda(i)}_j \Big)_{ij}\\
        &=a_N(v,\lambda) \cdot\Big(a_N(v+\lambda,\lambda_i)a_N(v,\lambda_{\lambda(i)})^{-1}a_L(v+\lambda_{\lambda(i)},\Lambda^i_\lambda)\delta^{\lambda(i)}_j \Big)_{ij}\\
        &=a_N(v,\lambda)\phi(v+\lambda)\Big(a_L(v+\lambda_{\lambda(i)},\Lambda^i_\lambda)\delta^{\lambda(i)}_j \Big)_{ij}\phi(v)^{-1}
    \end{align*}
for the function $\phi: V\ra U(n)$ given by 
$$\phi(v)=diag(a_N(v,\lambda_1),...,a_N(v,\lambda_n)).$$
Moreover, 
$$(f^*A_N)(v+\lambda_i)+A_L(v+\lambda_i)-dlog\phi(v)_{ii}=A_N(v)+A_{L}(v+\lambda_i)$$
so the connection on $f_*(F^*N\otimes L)$ is
$$A_{f_*(f^*N\otimes L)}(v)=A_{N}(v)\cdot Id_{n\times n} + A_{f_*L}(v)=A_{N\otimes f_*L}(v).$$
\end{proof}

\paragraph{Analogue of the map $\phi_L$.}
The map $\phi_L: X\ra \hat{X}$ in the case of holomorphic bundles was defined as
$$x\mapsto t^*_xL\otimes L^{-1}$$
mapping to the dual torus $\hat{X}\cong Pic^0(X)$. For $U(1)$-bundles with connections $\hat{X}\cong Hom(\pi_1(X),U(1))$ is the space of characters. It can be also viewed as the space of flat $U(1)$-bundles with connection due to the following lemma.

\begin{lemma}\label{character for flat U(1) bundles}
    Every $U(1)$-bundle with a flat connection can be represented by a constant factor of automorphy and $A_L=0$.
\end{lemma}
\begin{proof}
    Let $a_L:V\times \Gamma \ra U(1)$ be the factor of automorphy of $L$ on $X$ and $A_L$ the corresponding connection one-form. Then, since $dA_L=0$ there exists a real function $f_L$ on $V$ such that $A_L=2\pi i df$ and if we write $\phi=exp(2\pi i f)$ we see that $A_L=d\phi \cdot \phi^{-1}$. Therefore, the factor of automorphy $a_L'(v,\lambda)=\phi(v+\lambda)a_L(v,\lambda)\phi(v)^{-1}$ corresponds to $A_L'=0.$ 

    Then, $A_L'(v+\lambda)-A_L'(v)=0=-da_L'\cdot (a_L')^{-1}$ and we find that $da_L'=0$, therefore it is constant. In particular, $a_L:\Gamma\ra U(1)$ is a representation of $\Gamma$ and the connection on $L$ has monodromy $-a_L(\lambda)$ around the loop $\lambda\in \pi_1(X).$ Therefore the dual torus  $\hat{X}=Hom(\pi_1(X),U(1))$ can be identified with the space of flat $U(1)$-bundles.
\end{proof}

For line bundles, the equivalence relation $(a_L,A_L)\sim_\phi (a_L',A_L')$ does not change the curvature. Indeed, the curvature of the connection induced by $A_L$ is simply given by $dA_{L}$. It is a priori only a two-form on $V$, but also
$$dA_L(v+\lambda)=dA_L(v)-d(dlog (a_L(v,\lambda))
)=dA(v)$$
therefore $dA_L$ is a well-defined closed two-form on $X$. Hence, if $A^\phi_L=A_L-d\phi \phi^{-1}$ then
$$dA^\phi_L=dA_L-d(dlog\phi)=dA_L.$$
The curvature of the connection on $t_x^*L\otimes L^{-1}$ is given by
$$d(t_x^*A_L)-dA_L=t_x^*F_L-F_L$$
if $F_L\in \Omega^2(X,\dR)$ is the curvature of the connection on $L$. In particular, the map $\phi_L$ is only well defined as a map to the dual torus, if
$$t_x^*F=F\ \ \ \forall x\in X$$
that is, if $F$ is an invariant two-from.

When it is defined, $\phi_L: X\ra \hat{X}$ is a homomorphism and its kernel $K(L)$ is a disjoint union of affine subtori. We denote the connected component of the identity by $K(L)_0$. It is a subgroup which fits into the short exact sequence of groups
\begin{equation}\label{K(L)_0 SES}
\begin{tikzcd}
0 \arrow{r} & K(L)_0 \arrow{r} & X \arrow{r} & X/K(L)_0 \arrow{r} & 0.
\end{tikzcd}
\end{equation}

\paragraph{Appel-Humbert theorem.}
 For $U(1)$-bundles with invariant curvatures, we can mimic the proof of the Appel-Humbert Theorem \cite[Theorem 2.2.3]{BL}  to show that there is a canonical choice of factor of automorphy.

Let $F\in \Omega^2(X)$ be an invariant two-form on the real torus $X=V/\Gamma$, that is $t_x^*F=F$ for all $x\in X$. Every invariant form on a torus is closed and represents a cohomology class in $\HH^2(X,\dR)$. Once again, a two-from can be seen as an alternating bilinear form acting point-wise on the tangent spaces of $X$. We have isomorphisms $T_xX\cong V$ so for any $x\in X$ the restriction of the two-form $F_x$ to $T_xX$ induces an alternating bilinear pairing on $V$. Whenever $F$ is invariant, this pairing is independent of $x\in X$ and we denote it again by $F$. Whenever $F$ represents an integral cohomology class, the induced pairing is integral on $\Gamma$ which gives the isomorphism $H^2(X,\dZ)\cong Alt^2(\Gamma,\dZ)$.

Let us define the real analogue of group $\cP(\Gamma)$ (\ref{P(Gamma) complex}). In the complex case, we took Hermitian inner products from $NS(X)$ which is a subgroup of $\HH^2(X,\dZ)$. In the real case, we do not have to restrict ourselves to elements which are compatible with a complex structure. Moreover, the space of semicharacters only depends on the imaginary part of the Hermitian forms so the definition can remain unchanged.
\begin{definition}\label{semicharacter U(1)}
    Let $F\in Alt^2(\Gamma,\dZ)$ be an alternating bilinear pairing on $\Gamma$ taking integer values. Then, a \emph{semicharacter for $F$} is a map
    $$\chi: \Gamma\ra U(1)$$
    satisfying
    $$\chi(\lambda+\mu)=\chi(\lambda)\chi(\mu)exp(i\pi F(\lambda,\mu)).$$
\end{definition}
Finally, we can define the group:
\begin{equation}\label{P(Gamma) real}
    \begin{aligned}
        \cP_\dR(\Gamma)&:=\{ (F,\chi)\ | \ F\in \Omega^2(X)\text{ invariant},\ [F]\in \HH^2(X,\dZ),\ \chi\ \text{a semicharacter for $F$}\},\\
        &(F_1,\chi_1)\cdot (F_2,\chi_2)=(F_1+F_2,\chi_1\cdot \chi_2).
    \end{aligned}
\end{equation}
Let us denote by $\cA$ the group of line bundles on $X$ with invariant curvatures under the tensor product.  The map 
$$\cP_\dR(\Gamma)\mapsto \cA$$ 
is defined as follows:
\begin{align}\label{canonical factor u(1)}(\chi,F)\ \ \ \  \mapsto\ \ \ \  a_L(v,\lambda)=\chi(\lambda) exp(i\pi F(v,\lambda)),\ A_L=i\pi F(v,dv).\end{align}
Clearly, $A_L$ satisfies $$A_L(v+\lambda)=A_L(v)-a_L(v,\lambda)^{-1}da_L(v,\lambda).$$
Let $dim(V)=n$ and $\{v^1,...,v^n\}$ be linear coordinates on $V$ inducing 1-periodic coordinates on $X$. Then we can write $F$ as
$$F=\sum_{i<j} F_{ij}dv^i\wedge dv^j,$$
and the connection
$$A_L(v)=i\pi\sum_{i<j}F_{ij}(v^idv^j-v^jdv^i).$$
This description of $A_L$ is independent of the choice of 1-periodic coordinates. Indeed, changing coordinates is a constant linear transformation $v'=Av$ with $A\in SL(n,\dZ)$. Then $dv'=Adv$ and $\frac{\partial }{\partial v'}=(A^{-1})^T\frac{\partial }{\partial v}$. In the new coordinates $F=2\pi i F'_{ij}dv^{\prime i}\wedge dv^{\prime j}$ with $F'_{ij}=F_{kl}(A^{-1})^k_i(A^{-1})^l_j$. The new connection 1-form is
\begin{align*}
 A_L'&=i \pi  \sum_{i<j}F'_{ij}(v^{\prime i}dv^{\prime j}-v^{\prime j}dv^{\prime i})\\
 &=i\pi   \sum_{i<j}F_{kl}(A^{-1})^k_i(A^{-1})^l_j(v^{\prime i}dv^{\prime j}-v^{\prime j}dv^{\prime i})\\
 &=i\pi \sum_{i<j}F_{kl}((A^{-1})^k_iv^{\prime i}(A^{-1})^l_jdv^{\prime j}-(A^{-1})^l_jv^{\prime j}(A^{-1})^k_idv^{\prime i})\\
 &=i\pi \sum_{i<j}F_{ij}(v^idv^j-v^jdv^i)\\
 &=A_L.
\end{align*}
We can finally prove the following theorem.
\begin{theorem}[Modification of the Appel-Humbert theorem]\label{appel-humbert for u(1)} There is an isomorphism of short exact sequences
\[\begin{tikzcd}
    1 \arrow{r} & Hom(\Gamma,U(1)) \arrow{d}{\cong} \arrow{r} & \cP_\dR(\Gamma) \arrow{d}{\cong} \arrow{r} & Alt^2(\Gamma,\dZ) \arrow{r}\arrow{d}{=} & 1 \\
    1\arrow{r} & \hat{X} \arrow{r} & \cA \arrow{r}{c_1} & H^2(X,\dZ) \arrow{r} & 1
\end{tikzcd}\]
    
\end{theorem}
\begin{proof}
We know that the map $\cA\ra  H^2(X,\dZ)$ is given by $(f,A)\mapsto [dA]$ since $dA$ is a well-defined two-form on $X$. Therefore, the second square in the diagram commutes. The kernel of the map $(f,A)\mapsto [dA]$ is the space of line bundles with zero curvature. These by Lemma \ref{character for flat U(1) bundles} are in one-to-one correspondence with $Hom(\Gamma,U(1))$, the space of characters. Therefore the first square commutes as well. By the five lemma, we have an isomorphism in the middle.

We call the representations of line bundles (\ref{canonical factor u(1)}) under the map $\cP_\dR(\Gamma)\ra \cA$ the \emph{canonical factor}.
\end{proof}
Let $L$ be a $U(1)$-bundle with invariant curvature $F$. From the canonical representation, it is clear that the analytification of $\phi_L$ is given by $F:V\ra V^*$.

Moreover, given an isogeny $f:X\ra Y$ the pushforward $f_*L$ is a $U(n)$-bundle with connection which defines a Hermitian vector bundle $E$ with a Hermitian connection via the standard representation. From (\ref{u(n) connection pushforward under isogeny}) the curvature of this connection is 
$$dA_{f_*L}=F\cdot Id\in \Omega^2(Y,End(E)),$$
that is $E$ is projectively flat. We will say that $f_*L$ is \emph{projectively flat} when the induced Hermitian vector bundle is. As we have mentioned in the holomorphic case projectively flat Hermitian vector bundles have special factors of automorphy \cite[Theorem 4.7.54]{kobayashi}. This is true even if the Hermitian vector bundle does not admit a holomorphic structure, that is when $F\in \HH^2(X,\dQ)$ does not lie in $NS(X)\otimes \dQ$. We define semi-representations as before.
\begin{definition}
    Let $F\in \Omega^2(X)$ be an invariant two-form on $X=V/\Gamma$ representing a rational cohomology class $[F]\in \HH^2(X,\dQ)$. A \emph{semi-representation} for $F$ in $U(n)$ is a map
    $$U: \Gamma\ra U(n)$$
    satisfying for all $\lambda,\mu\in \Gamma$
    $$U(\lambda+\mu)=exp(i\pi F(\lambda,\mu))U(\lambda)U(\mu)$$
\end{definition}
We have,
\begin{equation}\label{f_*L as projectively flat factor}
\begin{aligned}
    a&_{f_*L}(v,\lambda)=\\
    &=\Big(\chi(\Lambda^i_\lambda)exp(i\pi F(v+\lambda_{\lambda(i)},\Lambda^i_\lambda)  \delta^j_{\lambda(i)}\Big)^i_j\\
    &=\Big(\chi(\Lambda^i_\lambda)exp(i\pi F(v+\lambda_{\lambda(i)}, \lambda+\lambda_i-\lambda_{\lambda(i)})  \delta^j_{\lambda(i)}\Big)^i_j\ \ \ \ \ \ \ \Lambda^i_\lambda=\lambda+\lambda_i-\lambda_{\lambda(i)}\\
    &=\Big(\chi(\Lambda^i_\lambda)exp(i\pi F(v, \lambda)+ i \pi F(v+\lambda, \lambda_i)-\\
    &\hspace{1cm}-i\pi F(\lambda,\lambda_i)-i\pi F(v,\lambda_{\lambda(i)}) +i \pi F(\lambda_{\lambda_i}, \lambda+\lambda_i) )  \delta^j_{\lambda(i)}\Big)^i_j\\
    &=\phi(v+\lambda)exp(i\pi F(v,\lambda))\Big(\chi(\Lambda^i_\lambda) exp(i\pi F(\lambda+\lambda_i,\lambda-\lambda_{\lambda(i)}) \delta^j_{\lambda(i)}\Big)^i_j \phi(v)^{-1}
\end{aligned}
\end{equation}
for $\phi(v)=diag\Big(exp(i\pi F(v,\lambda_i)\Big)$ so $f_*L$ is a projectively flat bundle with
\begin{equation}\label{f_*L as projectively flat connection}
\begin{aligned}
    A_{f_*L}(v)&=\Big(i\pi F(v+\lambda_i, dv) \delta^i_j\Big)^i_j\\
    &=\Big(i\pi F(v,dv) \delta^i_j\Big)^i_j - \phi^{-1} d\phi.
\end{aligned}
\end{equation}

\section{Non-degenerate U(1)-bundles}
 From now on when we say "$U(1)$-bundle with connection" we assume that the curvature of the connection is invariant. Finally, we are ready to show an analogue of Proposition \ref{FM transform and phi L holo} for $U(1)$-bundles. More precisely, if $L\ra X$ is a $U(1)$-bundle with connection such that $\phi_L$ is a degree $d^2$ isogeny, we will show that the $U(d^2)$-bundle $(\phi_L)_*L^{-1}$ can be decomposed into the product of $U(d)$-bundles.

Our strategy is the same as in the holomorphic case. We first show that there exists an isogeny $f:Y\ra X$ and a $U(1)$-bundle with connection $N\ra Y$ such that $f^*N=L$ and such that $\phi_N:Y\ra \hat{Y}$ is an isomorphism. Then, we use the commutative diagram
\[
\begin{tikzcd}
    X\arrow{r}{\phi_L} & \hat{X}\\
    Y\arrow{u}{f} \arrow{r}{\phi_N}[swap]{\cong} & \hat{Y} \arrow{u}[swap]{\hat{f}}
\end{tikzcd}
\]
 to show
 $$(\phi_L)_*L^{-1}\cong (\hat{f}_*(\phi_N)_*N^{-1})\otimes U(d).$$
 We explain in detail what we mean by the tensor product in Theorem \ref{pushforward tensor U(1) nondegen}.

\begin{definition}
    A $U(1)$-bundle with connection on a torus is called \emph{non-degenerate} if the curvature of the connection is a non-degenerate two-form. It is called \emph{degenerate} otherwise.
\end{definition}
Let $X=V/\Gamma$ be areal torus and let $L\ra X$ be a non-degenerate $U(1)$-bundle with connection. Then the dimension of $X$ must be even, let $dim_\dR(X)=2r$. The bundle $L$ can be represented  by the canonical factor as 
$$a_L(v,\gamma)=\chi(\gamma)exp(i\pi F(v,\gamma)),\ \ \ \ A_L(v)=i\pi F(v,dx),$$
where we denote by $\{x^1,...,x^n\}$ the coordinates on $V$ which are 1-periodic on $X$. In particular we can choose $\{x^i\}$ such that the basis $\{(1,0,...,0),(0,1,0,...,0),...,(0,...,0,1)\}$ of $\Gamma$ forms a \emph{symplectic basis for $F$}, that is
\begin{align}\label{F} F=\sum_{i=1}^r d_i dx^i\wedge dx^{r+i}\end{align}
with $d_i\in \dZ_{>0}$. Once again, we call $(d_1,...,d_r)$ the \emph{type} of $L$ and $d=d_1\cdot ...\cdot d_r$ the \emph{degree} of $L$.

Consider now a decomposition $$\Gamma=\Gamma_1\oplus \Gamma_2$$ into maximal isotropic subspaces with respect to $F$.  Let $$V=V_1\oplus V_2=(\Gamma_1\otimes \dR)\oplus (\Gamma_2\otimes \dR)$$
the decomposition into maximal isotropics induced on $V$. For example with respect to the symplectic frame $\{x^i\}$ we can choose $\Gamma_1$ and $\Gamma_2$ as the integer span of the first and second $r$ basis vectors respectively, then we have $V_1=\{x^{r+1}=...=x^{2r}=0\}$ and $V_2=\{x^1=...=x^r=0\}$.

The dual vector space $V^*$ and dual lattice $\Gamma^\vee$ inherit the decomposition as well. We have $V^*=V_1^*\oplus V_2^*$ where
$$F(V_1)=V_2^*\ \ \ \text{and}\ \ \ F(V_2)=V_1^*$$
and $\Gamma^\vee=\Gamma_1^\vee\oplus \Gamma_2^\vee$ where $\Gamma^\vee_i=\{ v^*\in V_i^*\ |\ v^*(\Gamma_i)\subset \dZ\}.$ Clearly,
$$F(\Gamma_1)\subset \Gamma_2^\vee\ \ \ \text{and}\ \ \ F(\Gamma_2)\subset \Gamma_1^\vee$$
and the decomposition of $V^*$ and $\Gamma^\vee$ are decompositions into maximal isotropics with respect to $F^{-1}$. Indeed, let $\{\hat{x}_1,...,\hat{x}_{2r}\}$ be coordinates on $\hat{X}$ dual to $\{x^i\}$ in the sense that the induce dual bases on $\Gamma$ and $\Gamma^\vee$. Then, 
\begin{align}\label{F inverse}F^{-1}=-2\pi i \sum_{i=1}^{r}\frac{1}{d_i}d\hat{x}^i\wedge d\hat{x}^{r+i},\end{align}
moreover $V_1^*=\{\hat{x}_{r+1}=...=\hat{x}_{2r}=0\}$ and $\{\hat{x}_1=...=\hat{x}_r=0\}$.

Let 
\begin{align}\label{Y lattice}\Gamma_Y:=\Gamma_1\oplus F^{-1}(\Gamma_1^\vee)\end{align} 
and define 
\begin{align}\label{Y torus}Y:=V/\Gamma_Y.\end{align}
Since $\Gamma_X\subset \Gamma_Y$ we have a degree $d$ isogeny $f: X\ra Y$.  

Restricting $F$ to $\Gamma_Y$ is a type $(1,...,1)$ pairing and we may extend $\chi$ to a semicharacter of $F$ over $\Gamma_Y$. Therefore, we can define a $U(1)$-bundle $N$ on $Y$ via the pair $(a_N,A_N)$ 
\begin{align}a_N(v,\gamma)=\chi(\gamma)exp(i\pi F(v,\gamma)),\ \ \ A_N(v)=i\pi F(v,dy)\end{align}
where now $\gamma\in \Gamma_Y$ and $\{y^1,...,y^{2r}\}$ are coordinates on $V$ which induce 1-periodic coordinates on $Y$. We have, $f^*N\cong L$. 

For example if $\{x^i\}$ are symplectic coordinates as before, we can set 
$$y^i=x^i,\ \ i=1,...,r;\ \ \ y^{r+i}=\frac{1}{d_i}x^{r+i},\ \ i=1,...r.$$
and we have
$$F=2\pi i \sum_{i=1}^r dy^i\wedge dy^{r+i}.$$

Let us construct $\hat{f}_*(\phi_N)_*N^{-1}$. It is the pushforward of a line bundle along an isogeny so it is a projectively flat $U(d)$-bundle with a connection. Let us first calculate $(\phi_N)_*N^{-1}$. We have $$\Gamma_Y^\vee=F(\Gamma_Y)=\Gamma_1^\vee\oplus F(\Gamma_1)$$ since $N$ is a type $(1,...,1)$ line bundle and $\phi_N$ is an isomorphism. Therefore, for all $\hat{v}\in V^*$ and $\hat{\gamma}\in \Gamma_Y^\vee$ we have
\begin{align*}
    a_{(\phi_N)_*N^{-1}}(\hat{v},\hat{\gamma})&=\chi(F^{-1}(\hat{\gamma}))^{-1}exp(i\pi F^{-1}(\hat{v},\hat{\gamma})),\\
    A_{(\phi_N)_*N^{-1}}(\hat{v})&=-i\pi F^{-1}(\hat{y},d\hat{y}),
\end{align*}
where we use that $-F(F^{-1}(\hat{v}),F^{-1}(\hat{\gamma}))=-\hat{v}(F^{-1}(\hat{\gamma}))=-F^{-1}(\hat{\gamma},\hat{v})=F^{-1}(\hat{v},\hat{\gamma}).$

The degree $d$ isogeny $\hat{f}:\hat{Y}\ra \hat{X}$ is induced by the inclusion $\Gamma_Y^\vee=\Gamma_1^\vee\oplus F(\Gamma_1) \subset \Gamma^\vee=\Gamma_1^\vee\oplus \Gamma_2^\vee$. To define $\hat{f}_*(\phi_N)_*N^{-1}$ we use a set of representatives of $$\Gamma^\vee:\Gamma_Y^\vee=\Gamma_2^\vee:F(\Gamma_1).$$ Choose these representatives $\{\lambda_1,...,\lambda_d\}$ in $\Gamma_2^\vee\in V_2^*$. The kernel of $\hat{f}$ is given by the points of $\Gamma_2^\vee/F(\Gamma_1)$ so any $\hat{y}\in ker(\hat{f})$ can be represented by one of the $\lambda_i$s.

We write  $\hat{\lambda}+\lambda_i=\hat{\lambda}_1+\hat{\lambda}_2+\lambda_i=\hat{\lambda}_1+\lambda_{\hat{\lambda}(i)}+\Lambda_{\hat{\lambda}}^i$ so  $\hat{f}_*(\phi_N)_*N^{-1}$ is given as follows.
\begin{align*}
    a_{\hat{f}_*(\phi_N)_*N^{-1}}(\hat{v},\hat{\lambda})&=\Big( a_{(\phi_N)_*N^{-1}}(\hat{v}+\lambda_{\hat{\lambda}(i)},\hat{\lambda}_1+ \Lambda_{\hat{\lambda}}^i )\delta^j_{\hat{\lambda}(i)} \Big)^i_j\\
    &=\Big( \chi(F^{-1}(\hat{\lambda}_1+\Lambda^i_{\hat{\lambda}}))^{-1}exp(i\pi F^{-1}(\hat{v}+\lambda_{\hat{\lambda}(i)},\hat{\lambda}_1+ \Lambda_{\hat{\lambda}}^i ))\delta^j_{\hat{\lambda}(i)} \Big)^i_j.
\end{align*}
Using that $F^{-1}(\lambda_i,\lambda_j)=0$ and changing the representative by $\phi(\hat{v})=\Big(exp(-i\pi F^{-1}(\hat{v},\lambda_i))\Big)^i_i$ we get 
\begin{align*}
a&_{\hat{f}_*(\phi_N)_*N^{-1}}(\hat{v},\hat{\lambda})=\\
    &=exp(i\pi F^{-1}(\hat{v},\hat{\lambda}))\cdot \Big(\chi(F^{-1}(\hat{\lambda}_1+\Lambda^i_{\hat{\lambda}}))^{-1} exp(i \pi F^{-1}(\lambda_{\hat{\lambda}(i)},\hat{\lambda})-i\pi F^{-1}(\hat{\lambda},\lambda_i))\delta^j_{\hat{\lambda}(j)}  \Big)^i_j\\
    &=exp(i\pi F^{-1}(\hat{v},\hat{\lambda}))\cdot U_{\hat{f}_*(\phi_N)_*N^{-1}}(\hat{\lambda}),
\end{align*}
where $U_{\hat{f}_*(\phi_N)_*N^{-1}}(\hat{\lambda}):\ \Gamma^\vee \ra U(d)$ is a semi-representation for $F^{-1}$. We further rewrite it as
\begin{equation}
\begin{aligned}
    U&_{\hat{f}_*(\phi_N)_*N^{-1}}(\hat{\lambda})=\\
    &=\Big(\chi(F^{-1}(\hat{\lambda}_1+\Lambda^i_{\hat{\lambda}}))^{-1} exp(i \pi F^{-1}(\lambda_{\hat{\lambda}(i)},\hat{\lambda}_1)-i\pi F^{-1}(\hat{\lambda}_1,\lambda_i))\delta^j_{\hat{\lambda}(j)}  \Big)^i_j\\
    &=\Big(\chi(F^{-1}(\hat{\lambda}_1))^{-1}\chi(F^{-1}(\Lambda^i_{\hat{\lambda}}))^{-1} \times \\    
    &\times exp(-i \pi F^{-1}(\hat{\lambda}_2+\lambda_i-\lambda_{\hat{\lambda}(i)},\hat{\lambda}_1)+i \pi F^{-1}(\lambda_{\hat{\lambda}(i)},\hat{\lambda}_1)-i\pi F^{-1}(\hat{\lambda}_1,\lambda_i))\delta^j_{\hat{\lambda}(j)}  \Big)^i_j\\
    &=\Big(\chi(F^{-1}(\hat{\lambda}_1))^{-1}\chi(F^{-1}(\Lambda^i_{\hat{\lambda}}))^{-1} exp(-i \pi F^{-1}(\hat{\lambda}_2,\hat{\lambda}_1)+2 \pi i F^{-1}(\lambda_{\hat{\lambda}(i)},\hat{\lambda}_1))\delta^j_{\hat{\lambda}(j)}  \Big)^i_j.
\end{aligned}
\end{equation}
The corresponding connection is given by
\begin{align*}
    A_{\hat{f}_*(\phi_N)_*N^{-1}}(\hat{v})&=\Big(A_{(\phi_N)_*N^{-1}}(\hat{v}+\lambda_i)\delta_{ij}\Big)^i_j,\\
    &=\Big(i\pi F^{-1}(\hat{v},d\hat{v})+i\pi F^{-1}(\lambda_i,d\hat{v}) \Big)^i_i,
\end{align*}
so changing the representative by $\phi(\hat{v})=\Big(exp(-i\pi F^{-1}(\hat{v},\lambda_i))\Big)^i_i$ we get
\begin{equation}
     A_{\hat{f}_*(\phi_N)_*N^{-1}}(\hat{v})=i\pi F^{-1}(\hat{v},d\hat{v})\cdot Id.
\end{equation}
The analogue of Proposition \ref{FM transform and phi L holo} for principal bundles with connections can be phrased in terms of the Hermitian vector bundles they define via the standard representation. Let $E_N$ be the rank $d$ Hermitian vector bundle corresponding to $\hat{f}_*(\phi_N)_*N^{-1}$ and $E_L$ the rank $d^2$ Hermitian vector bundle corresponding to $(\phi_L)_*L^{-1}$. We claim, that
$$E_L\cong E_N^{\oplus d},$$
as Hermitian vector bundles with connections. In terms of the principal bundles we can say that the principal $U(d^2)$-bundle admits a reduction of the structure group to the subgroup $U(d)$ of $U(d^2)$ via the homomorphism
\begin{equation}\label{embedding u(d) to u(d2)}
    \begin{aligned}
        \rho:\ \  U(d) \ \  &\ra\ \  U(d^2),\ \ \ \ \ \ 
        A \ \  &\mapsto \ \ \begin{pmatrix}A & 0 & ... & 0 \\
        0 & A & ... & 0\\
        \vdots &  & \ddots & \vdots\\
        0 & 0 & ... & A\end{pmatrix}.
    \end{aligned}
\end{equation}
In terms of principal bundles, we will write $P_L=P_N\otimes U(d)$ when such reduction of structure group is possible.

\begin{theorem}\label{pushforward tensor U(1) nondegen}
    Let $L\ra X$ be a degree $d$ non-degenerate $U(1)$-bundle with invariant curvature $F$, $f:X\ra Y$ an isogeny, $N\ra Y$ a type (1,...,1) line bundle such that $f^*N=L$. Then,
    $$(\phi_L)_*L^{-1}=\Big(\hat{f}_*(\phi_N)_*N^{-1}\Big)\otimes U(d).$$
    More precisely, we can find representatives of $(\phi_L)_*L^{-1}$ of the form
\begin{align*}a_{(\phi_L)_*L^{-1}}(\hat{v},\hat{\lambda})&=exp(i\pi F^{-1}(\hat{v},\hat{\lambda}))\cdot U_{(\phi_L)_*L^{-1}}(\hat{\lambda}),\\ A_{(\phi_L)_*L^{-1}}(\hat{v})&=i\pi F^{-1}(\hat{v},d\hat{v})\cdot Id_{d^2\times d^2},\end{align*}
and of $\hat{f}_*(\phi_N)_*N^{-1}$ 
\begin{align*}a_{\hat{f}_*(\phi_N)_*N^{-1}}(\hat{v},\hat{\lambda})&=exp(i\pi F^{-1}(\hat{v},\hat{\lambda}))\cdot U_{\hat{f}_*(\phi_N)_*N^{-1}}(\hat{\lambda}),\\ A_{\hat{f}_*(\phi_N)_*N^{-1}}(\hat{v})&=i\pi F^{-1}(\hat{v},d\hat{v})\cdot Id_{d\times d},
\end{align*}
such that the semi-representations $U_{(\phi_L)_*L^{-1}}$ and $U_{\hat{f}_*(\phi_N)_*N^{-1}}$ of $\Gamma^\vee$ satisfy
$$U_{(\phi_L)_*L^{-1}}=\rho \circ U_{\hat{f}_*(\phi_N)_*N^{-1}} : \Gamma^\vee \ra U(d^2).$$
\end{theorem}
\begin{proof}
Let us calculate the semi-representation of $(\phi_L)_*L^{-1}$. We take representatives of $$\Gamma^\vee:F(\Gamma)=(\Gamma_1^\vee\oplus\Gamma_2^\vee):(F(\Gamma_2)\oplus F(\Gamma_1))$$ as follows. Let $\{\lambda_1,...,\lambda_d\}$ be the full set of representatives of $\Gamma_2^\vee:F(\Gamma_1)$ in $V_2^*$ as in the construction of $\hat{f}_*(\phi_N)_*N^{-1}$ and let $\{\mu_1,...,\mu_d\}$ be a full set of representatives of $\Gamma_1^\vee:F(\Gamma_2)$ in $V_1^*$. Then, $\epsilon_{ij}=\mu_i+\lambda_j$ is a full set of representatives of $\Gamma^\vee:F(\Gamma)$.

We write $\hat{\lambda}+\epsilon_{ij}=\hat{\lambda}_1+\mu_i+\hat{\lambda}_2+\lambda_j=\mu_{\hat{\lambda}(i)}+M^i_{\hat{\lambda}}+\lambda_{\hat{\lambda}(j)}+\Lambda^j_{\hat{\lambda}}$
\begin{align*}
a&_{(\phi_L)_*L^{-1}}(\hat{v},\hat{\lambda})=\\
&=\Big( a_{L^{-1}}(F^{-1}(\hat{v}+\mu_{\hat{\lambda}(i)}+\lambda_{\hat{\lambda}(j)}),F^{-1}(M^i_{\hat{\lambda}}+\Lambda^j_{\hat{\lambda}}))\delta^k_{\hat{\lambda}(i)}\delta^l_{\hat{\lambda}(j)}
\Big)^{ij}_{kl}\\
&=\Big(\chi(F^{-1}(M^i_{\hat{\lambda}}+\Lambda^j_{\hat{\lambda}}))^{-1} exp(i\pi F^{-1}(\hat{v}+\mu_{\hat{\lambda}(i)}+\lambda_{\hat{\lambda}(j)},M^i_{\hat{\lambda}}+\Lambda^j_{\hat{\lambda}}))\delta^k_{\hat{\lambda}(i)}\delta^l_{\hat{\lambda}(j)}\Big)^{ij}_{kl}.
\end{align*}
Changing the representative by $\phi(\hat{v})=\Big(exp(-i\pi F^{-1}(\hat{v},\mu_i+\lambda_j) )\Big)^{ij}_{ij}$ yields,
\begin{align*}
a&_{(\phi_L)_*L^{-1}}(\hat{v},\hat{\lambda})= exp(i\pi F^{-1}(\hat{v},\hat{\lambda}))\cdot U_{(\phi_L)_*L^{-1}}(\hat{\lambda}),
\end{align*}
where $U_{(\phi_L)_*L^{-1}}:\Gamma^\vee\ra U(d^2)$  is a semi-representation. The connection one-form, after changing the representative by $\phi(\hat{v})=exp(-i\pi F^{-1}(\hat{v},\mu_i+\lambda_j))^{ij}_{ij}$ is given by
\begin{align*}
    A_{(\phi_L)_*L^{-1}}(\hat{v})&=i\pi F^{-1}(\hat{v},d\hat{v})\cdot Id.
\end{align*}
We may further change the semi-representation by conjugating with constant matrices which does not change the connection.
\small
\begin{equation*}
\begin{aligned}
&U_{(\phi_L)_*L^{-1}}(\hat{\lambda})=\\
&=\Big(\chi(F^{-1}(M^i_{\hat{\lambda}}+\Lambda^j_{\hat{\lambda}}))^{-1} exp(-i\pi F^{-1}(\hat{\lambda},\mu_i+\lambda_j)+i\pi F^{-1}(\mu_{\hat{\lambda}(i)}+\lambda_{\hat{\lambda}(j)},M^i_{\hat{\lambda}}+\Lambda^j_{\hat{\lambda}}))\delta^k_{\hat{\lambda}(i)}\delta^l_{\hat{\lambda}(j)}\Big)^{ij}_{kl}\\
&=\Big(\chi(F^{-1}(M^i_{\hat{\lambda}}))^{-1}\chi(F^{-1}(\Lambda^j_{\hat{\lambda}}))^{-1}exp(i\pi F^{-1}(M^i_{\hat{\lambda}},\Lambda^j_{\hat{\lambda}}))\times\\
&\ \ \ \times  exp(-i\pi F^{-1}(\hat{\lambda}_1,\lambda_j)-i\pi F^{-1}(\hat{\lambda}_2,\mu_i)+i\pi F^{-1}(\mu_{\hat{\lambda}(i)},\Lambda^j_{\hat{\lambda}})+i \pi F^{-1}(\lambda_{\hat{\lambda}(j)},M^i_{\hat{\lambda}}))\delta^k_{\hat{\lambda}(i)}\delta^l_{\hat{\lambda}(j)}\Big)^{ij}_{kl}.
\end{aligned}
\end{equation*}
\normalsize
Using the extension of $\chi$ to $\Gamma_Y$ and that $F^{-1}(\mu_i,\Lambda^j_{\hat{\lambda}})\in \dZ$ we have
\small
\begin{align*}
U&_{(\phi_L)_*L^{-1}}(\hat{\lambda})=\\
=&\Big(\chi(F^{-1}(\hat{\lambda}_1))^{-1}\chi(F^{-1}(\mu_i))^{-1}\chi(F^{-1}(\mu_{\hat{\lambda}(i)}))\chi(F^{-1}(\Lambda^j_{\hat{\lambda}}))^{-1}
exp(i\pi F^{-1}(\hat{\lambda}_1+\mu_i,\Lambda^j_{\hat{\lambda}}))\times\\
&\times  exp(-i\pi F^{-1}(\hat{\lambda}_1,\lambda_j)-i\pi F^{-1}(\hat{\lambda}_2,\mu_i)+i \pi F^{-1}(\lambda_{\hat{\lambda}(j)},M^i_{\hat{\lambda}}))\delta^k_{\hat{\lambda}(i)}\delta^l_{\hat{\lambda}(j)}\Big)^{ij}_{kl}\\
=&\Big(\chi(F^{-1}(\hat{\lambda}_1))^{-1}\chi(F^{-1}(\Lambda^j_{\hat{\lambda}}))^{-1}
exp(i\pi F^{-1}(\hat{\lambda}_1,\Lambda^j_{\hat{\lambda}})-i\pi F^{-1}(\hat{\lambda}_1,\lambda_j))+i\pi F^{-1}(\lambda_{\hat{\lambda}(j)},\hat{\lambda}_1))\times\\
&\times  \chi(F^{-1}(\mu_i))^{-1}\chi(F^{-1}(\mu_{\hat{\lambda}(i)}))exp(-i\pi F^{-1}(\hat{\lambda}_2,\mu_i)+i \pi F^{-1}(\lambda_{\hat{\lambda}(j)},\mu_i-\mu_{\hat{\lambda}(i)}))\times \\
&\times exp(-i\pi F^{-1}(\mu_i,\Lambda^j_{\hat{\lambda}}))\delta^k_{\hat{\lambda}(i)}\delta^l_{\hat{\lambda}(j)}\Big)^{ij}_{kl}.
\end{align*}
\normalsize
Changing the representatives by $\phi(\hat{v})^{ij}_{ij}=\chi(F^{-1}(\mu_i))exp(-i\pi F^{-1}(\mu_i,\lambda_j))$ yields
\small
\begin{align*}
U&_{(\phi_L)_*L^{-1}}(\hat{\lambda})=\\
=&\Big(\chi(F^{-1}(\hat{\lambda}_1))^{-1}\chi(F^{-1}(\Lambda^j_{\hat{\lambda}})
)^{-1}\times\\
&\times exp(i\pi F^{-1}(\hat{\lambda}_1,\Lambda^j_{\hat{\lambda}})-i\pi F^{-1}(\hat{\lambda}_1,\lambda_j))+i\pi F^{-1}(\lambda_{\hat{\lambda}(j)},\hat{\lambda}_1))\delta^k_{\hat{\lambda}(i)}\delta^l_{\hat{\lambda}(j)}\Big)^{ij}_{kl}\\
=&\Big(\chi(F^{-1}(\hat{\lambda}_1))^{-1}\chi(F^{-1}(\Lambda^j_{\hat{\lambda}})
)^{-1}exp(i\pi F^{-1}(\hat{\lambda}_1,\hat{\lambda}_2)-2\pi i F^{-1}(\hat{\lambda}_1,\lambda_{\hat{\lambda}(j)}))\delta^k_{\hat{\lambda}(i)}\delta^l_{\hat{\lambda}(j)}\Big)^{ij}_{kl}.
\end{align*}
\normalsize
That is for any $\hat{\lambda}\in \Gamma^\vee$
$$U_{(\phi_L)_*L^{-1}}(\hat{\lambda})\cong U_{\hat{f}_*(\phi_N)_*N^{-1}}(\hat{\lambda})\otimes E(\hat{\lambda})$$
where $E$ is a representation of $\Gamma^\vee$ in $U(d)$ given by 
$$E: \Gamma^\vee \ra U(d),\ \ \ \hat{\lambda}\mapsto \Big(\delta^i_{\hat{\lambda}(i)} \Big).$$
The permutation $i\mapsto \hat{\lambda}(i)$ is from the change $\hat{\lambda}+\mu_i=\mu_{\hat{\lambda}(i)}+M^i_{\hat{\lambda}}+\hat{\lambda}_2$, in particular $E$ is trivial on $F(\Gamma_2)+\Gamma_2^\vee$. 

Analogously to the proof of $f_*\dC\cong \oplus_{L_i\in Ker(\hat{f})} L_i$ (Lemma \ref{lemma: pushforward of trivial line bundle}) one can show that the vector bundle defined by $E$ is the sum of all the line bundles corresponding to characters of $\Gamma^\vee$ which are trivial on $F(\Gamma_2)\oplus \Gamma_2^\vee$. These characters can be represented by elements in $(F(\Gamma_2)\oplus \Gamma^\vee_2)^\vee\cong F^{-1}(\Gamma_2^\vee)\oplus \Gamma_2 $ which do not belong to $(\Gamma^\vee)^\vee\cong \Gamma_1\oplus \Gamma_2$. 

The cosets $F^{-1}(\Gamma_2^\vee):\Gamma_2$ are represented exactly by $F^{-1}(\lambda_1),...,F^{-1}(\lambda_n)$ where $\lambda_i$ represent $\Gamma_2^\vee:F(\Gamma_1)$. That is,
$$E\cong \bigoplus_{i=1}^d L_{\lambda_i}$$
where $L_{\lambda_i}$ are given by
$$a_{L_{\lambda_i}}:\Gamma^\vee \ra U(1),\ \ \ a_{L_{\lambda_i}}(\hat{\lambda})=exp(2\pi i F^{-1}(\lambda_i,\hat{\lambda})).$$

Rewriting $E$ into its diagonal form, using the Vandermond-type change of basis (\ref{change of basis}), gives us the following.
\begin{align*}
&U_{(\phi_L)_*L^{-1}}(\hat{\lambda})=\\
&=\Big(U_{\hat{f}_*(\phi_N)_*N^{-1}}(\hat{\lambda})\Big)^j_l\otimes \Big(exp(2\pi i F^{-1}(\lambda_i,\hat{\lambda}))\delta^k_{i}\Big)^i_k\\
&=\bigoplus_{i=1}^d\Big(U_{\hat{f}_*(\phi_N)_*N^{-1}}(\hat{\lambda})\cdot  exp(2\pi i F^{-1}(\lambda_i,\hat{\lambda}_1))\Big)^j_l\\
&=\bigoplus_{i=1}^d\Big(\chi(F^{-1}(\hat{\lambda}_1))^{-1}\chi(F^{-1}(\Lambda^j_{\hat{\lambda}}))^{-1}\times \\
&\hspace{1cm} \times exp(i\pi F^{-1}(\hat{\lambda}_1,\hat{\lambda}_2)-2\pi i F^{-1}(\hat{\lambda}_1,\lambda_{\hat{\lambda}(j)}))\cdot \delta^l_{\hat{\lambda}(j)}\cdot  exp(2\pi i F^{-1}(\lambda_i,\hat{\lambda}_1))\Big)^j_l\\
&=\bigoplus_{i=1}^d\Big(\chi(F^{-1}(\hat{\lambda}_1))^{-1}\chi(F^{-1}(\Lambda^j_{\hat{\lambda}}))^{-1}exp(i\pi F^{-1}(\hat{\lambda}_1,\hat{\lambda}_2)-2\pi i F^{-1}(\hat{\lambda}_1,\lambda_{\hat{\lambda}(j)}+\lambda_i))\cdot \delta^l_{\hat{\lambda}(j)}\Big)^j_l\\
&=\bigoplus_{i=1}^d U_{\hat{f}_*(\phi_N)_*N^{-1}}(\hat{\lambda})
\end{align*}
after we change the set of representatives from $\{\lambda_1,...,\lambda_d\}$ to $\{\lambda_1+\lambda_i,...,\lambda_d+\lambda_i\}$ for each $i$. 
\end{proof}

\subsection{Independence of the choice of $f$ and $N$}
For a non-degenerate holomorphic line bundle $L\ra X$ on a complex torus, the vector bundle $\hat{f}_*(\phi_N)_*N^{-1}$ agrees with the Fourier-Mukai transform of $L$, therefore it is independent of the choice of isogeny $f:Y\ra X$ and line bundle $N\ra Y$. In the $U(1)$-case we want to define a T-dual which agrees with the Fourier-Mukai transform for $B$-branes. To achieve this we must show, without relying on the Fourier-Mukai transform, that the bundle $\hat{f}_*(\phi_N)_*N^{-1}$ is independent of the choices of isogeny $f:Y\ra X$ and $U(1)$-bundle $N\ra Y$.

First, we show that given an isogeny $f:Y\ra X$ constructed as (\ref{Y lattice}) and (\ref{Y torus}) the bundle $\hat{f}_*(\phi_N)_*N^{-1}$ is independent of the choice of line bundle $N\ra Y$ satisfying $f^*N=L$. We construct $Y$ by choosing a decomposition $\Gamma=\Gamma_1+\Gamma_2$ with respect to the curvature $F$ of the connection on $L$. Different choices of decompositions yield different tori $Y$ and different isogenies $f:Y\ra X$. In Part II. we show that the result is also independent of the choice of decomposition.

\paragraph{Part I. Choice of line bundle $N$.} We may multiply $N$ by any line bundle which pulls back to the trivial bundle on $X$. These line bundles are precisely the ones corresponding to $\hat{y}\in Ker(\hat{f})$. We need the following two lemmas which are real analogues of Lemmas \ref{Lemma: translation and tensor} and \ref{Lemma: invariance under kernel}.
\begin{lemma}
Let $N\ra Y$ be a type $(1,...,1)$ non-degenerate line bundle. Then, for any $\hat{y}\in \hat{Y}$ we have
$$(\phi_N)_*(N^{-1}\otimes L_{\hat{y}})\cong t_{-\hat{y}}^*(\phi_N)_*N^{-1}.$$
\end{lemma}
\begin{proof}
We represent $N$ by the canonical factor as
$$a_N(v,\lambda)=\chi(\lambda)exp(i\pi F(v,\lambda)),\ \ \ A_N(v)=i\pi F(v,dv).$$
The line bundle $L_{\hat{y}}$ is represented by 
$$a_{L_{\hat{y}}}(v,\lambda)=exp(2\pi i \hat{y}(\lambda)), \ \ A_{L_{\hat{y}}}=0$$
if we view $\hat{y}\in V^*$ as a lift of $\hat{y}\in \hat{Y}$. Then,
  \begin{align*}
     a_{(\phi_N)_*(N^{-1}\otimes L_{\hat{y}})}(\hat{v},\hat{\lambda})&=\chi(F^{-1}(\hat{\lambda}))^{-1}exp(-i\pi F(F^{-1}(\hat{v}),F^{-1}(\hat{\lambda}))exp(2\pi i \hat{y}(F^{-1}(\hat{\lambda})))\\
     &=\chi(F^{-1}(\hat{\lambda}))^{-1}exp(i\pi F^{-1}(\hat{v},\hat{\lambda})exp(2\pi i F^{-1}(\hat{\lambda},\hat{y}))\\
     &=\chi(F^{-1}(\hat{\lambda}))^{-1}exp(i\pi F^{-1}(\hat{v}-\hat{y},\hat{\lambda}))exp(-\pi i F^{-1}(\hat{y},\hat{\lambda}))\\
     &=a_{(\phi_N)_*N^{-1}}(\hat{v}-\hat{y},\hat{\lambda})exp(-\pi i F^{-1}(\hat{y},\hat{\lambda}))
  \end{align*}
  meanwhile,
  \begin{align*}
      A_{(\phi_N)_*(N^{-1}\otimes L_{\hat{y}})}(\hat{v})=A_{(\phi_N)_*N^{-1}}(\hat{v})= i\pi F^{-1}(\hat{v},d\hat{v}).
  \end{align*}
  Changing the factor of automorphy and connection by $\phi(\hat{v})=exp(\pi i F^{-1}(\hat{y},\hat{v}))$ results in
  $$a_{(\phi_N)_*(N^{-1}\otimes L_{\hat{y}})}(\hat{v},\hat{\lambda})=a_{(\phi_N)_*N^{-1}}(\hat{v}-\hat{y},\hat{\lambda}),\ \ \       A_{(\phi_N)_*(N^{-1}\otimes L_{\hat{y}})}(\hat{v})=A_{(\phi_N)_*N^{-1}}(\hat{v}-\hat{y}).$$
    using the specific form of $A_{N^{-1}}$.
\end{proof}
\begin{lemma}\label{Lemma4}
Let $f:X\ra Y$ be an isogeny and $L$ a line bundle on $X$. Then, for any $x\in Ker(f)$ we have
$$f_*t_x^*L\cong f_*L.$$
\end{lemma}
\begin{proof}
A point $x\in X$ lies in $Ker(f)$ if it is in $\Gamma_Y/\Gamma_X$. Let $\lambda_x\in \Gamma_Y$ be a preimage of $x\in X$ in the universal cover. Then, $t_x^*L$ is represented by
$$a_{t_x^*L}(v,\lambda)=a_L(v+\lambda_x,\lambda),\ \ \ A_{t_x^*L}(v)=A_L(v+\lambda_x).$$
Therefore,
\begin{align*}
    a_{f_*t_x^*L}(v,\lambda)&=\Big(a_L(v+\lambda_x+\lambda_{\lambda(i)},\Lambda^i_\lambda) \delta^{\lambda(i)}_j\Big)^i_j\\
    A_{f_*t_x^*L}(v)&=\Big(A_L(v+\lambda_x+\lambda_i)\delta^i_j\Big)^i_j.
\end{align*}
Since $\{\lambda_1,...,\lambda_n\}$ is a full set of representatives of $\Gamma_Y/\Gamma_X$ and $\lambda_x\in \Gamma_Y$, $\{\lambda_x+\lambda_1,...,\lambda_x+\lambda_n\}$ is again a full set of representatives. As we have proved before, the pushforward does not depend on the choice of representatives. Note that
$$\lambda+\lambda_i=\lambda_{\lambda(i)}+\Lambda^i_\lambda$$
so
$$\lambda+\lambda_x+\lambda_i=\lambda_x+\lambda_{\lambda(i)}+\Lambda^i_\lambda.$$
\end{proof}
Since $\phi_N=\phi_{N\otimes L_{\hat{y}}}$ for any $\hat{y}\in \hat{Y}$, we have for any $\hat{y}\in Ker(\hat{f})$ 
\begin{align*}
\hat{f}_*(\phi_N)_*(N^{-1}\otimes L_{\hat{y}})=\hat{f}_*t_{-\hat{y}}^*(\phi_N)_*N^{-1}=\hat{f}_*(\phi_N)_*N^{-1}.
\end{align*}
Therefore, $\hat{f}_*(\phi_N)_*N^{-1}$ is independent of the choice of $N$ given the isogeny $f:X\ra Y$.

\paragraph{Part II. Choice of decomposition.} Our strategy is as follows. We first calculate how the lattice $\Gamma_Y$ changes if we change the decomposition. We can choose symplectic frames with respect to both decompositions, then the change of basis between the two frames is an integral symplectic linear transformation. We can view the changing of lattices as changing the symplectic frames.  We follow the work of Hua and Reiner \cite{huaReiner} to find generators of the integral symplectic group transforming symplectic frames into each other. We finally show that changing the frame by any of these generators does not change the resulting $U(d)$-bundle $\hat{f}_*(\phi_N)_*N^{-1}$.

Let $\Gamma=\Gamma_1\oplus \Gamma_2$ be a decomposition and $\{x^i,y^i\}_{i=1}^r$ one-periodic coordinates on $X$ with respect to this decomposition. That is, under the identification $V\cong T_xX$ for any point $x\in X$ the coordinate basis $\{\frac{\partial }{\partial x_i}\}_{i=1,..,r}$ and $\{ \frac{\partial }{\partial y_i}\}_{i=1,...,r}$ is a basis of $\Gamma_1$ and $\Gamma_2$ respectively. We may also assume that the induced frame is symplectic so we can write
$$F=\sum_{i=1}^rd_idx^i\wedge dy^i.$$
for positive integers $d_i$. 

Suppose $\Gamma=\Gamma_1'\oplus \Gamma_2'$ and let $\{u^i,v^i\}$ be 1-periodic coordinates with respect to this decomposition which also form a symplectic frame for $F$.

Let us represent $F$ in matrix form as
\begin{equation}\label{F in matrix form}\cF=\begin{pmatrix}0 & F \\ -F & 0\end{pmatrix}\in GL(2r,\dZ) ,\end{equation}
with $F=diag(d_1,...,d_r)$. Then the change of basis between the coordinates $\{x^i,y^i\}$ and $\{u^i,v^i\}$ can be represented by a  matrix \begin{equation}\label{transition matrix}\cS=\begin{pmatrix} A & B \\ C & D \end{pmatrix}\in SL(2r,\dZ)\end{equation} 
such that 
\begin{equation}\label{F-symplectic}\cS\begin{pmatrix}0 & F\\ -F & 0 \end{pmatrix}\cS^T=\begin{pmatrix}0 & F\\ -F & 0 \end{pmatrix} \end{equation}
and the coordinates transform as 
$$\begin{pmatrix} A^T & C^T \\ B^T & D^T \end{pmatrix}\begin{pmatrix} x \\ y  \end{pmatrix}=\begin{pmatrix} u \\ v \end{pmatrix}.$$
We defined $\Gamma_Y$ as $\Gamma_1\oplus F^{-1}(\Gamma_1^\vee)$ so $\Gamma_{Y'}=\Gamma_1'\oplus F^{-1}(\Gamma_1^{\prime \vee})$. The new 1-periodic coordinates on $Y$ are $\{x^i,d_{i}y^i\}$  so $\Gamma_Y$ is spanned by $\{\frac{\partial}{\partial x^i},d^{-1}_{i}\frac{\partial}{\partial y^i} \}$ and on $Y'$ are $\{u^i,d_{i}v^i\}$. The transition between them is given by
$$\begin{pmatrix}A^T & C^T F^{-1}\\ FB^T & FD^TF^{-1}\end{pmatrix}\begin{pmatrix}x \\ F y \end{pmatrix}=\begin{pmatrix} u \\ F v\end{pmatrix} .$$
The two lattices $\Gamma_Y$ and $\Gamma_{Y'}$ coincide if and only if this new transition is integral. From the symplectic property of $\cS$ we know that
$$\cS^{-1}=\begin{pmatrix}0 & F \\ -F & 0 \end{pmatrix}\cS^T\begin{pmatrix}0 & -F^{-1} \\ F^{-1} & 0  \end{pmatrix}=\begin{pmatrix} FD^TF^{-1} & -FB^TF^{-1} \\ -FC^TF^{-1} & FA^TF^{-1} \end{pmatrix}$$
which is again in $SL(2r,\dZ)$ so $FD^TF^{-1}$ is integral. That is if $CF^{-1}$ is integral $\Gamma_Y=\Gamma_{Y'}$.

\paragraph{Generators of $Sp(\cF,\dZ)$.} We will denote $2r\times 2r $ matrices by mathcal letters and $r\times r$ matrices by latin letters. We denote the $r\times r$ identity matrix by $I$. Let $F$, the alternating non-degenerate bilinear pairing on $\Gamma$, be given in normal form as 
$$\cF=\begin{pmatrix}0 & F  \\ -F & 0 \end{pmatrix}$$
where $F=diag(d_1,...,d_r)$ with $d_i\in \dZ_{>0}$ and  $d_{i}| d_{i+1}$. Let $Sp(\cF,\dZ)$ be the space of $2r\times 2r $ matrices $\cS$ lying in $SL(2r,\dZ)$ satisfying
$$\cS\cF \cS^T=\cF.$$
In \cite[Page 1]{huaReiner} Hua and Reiner gave the full set of representatives of the group $Sp(2r,\dZ)$, that is of the group $Sp(\cF,\dZ)$, where $\cF$ is the standard symplectic matrix with $d_1=...=d_r=1$. We can slightly modify their proof to find a full set of representatives of $Sp(\cF,\dZ)$ with general $\cF$.
\begin{theorem}\label{generators}
    The group $Sp(\cF,\dZ)$ is generated by the following types of elements

    (I) Translations: $$\cS=\begin{pmatrix}  I & S \\ 0 & I\end{pmatrix}$$
    where $S$ satisfies $SF=FS^T$.

    (II) Rotations:
    $$\cS=\begin{pmatrix} A & 0 \\ 0 & D \end{pmatrix}$$
    where $D=F(A^{-1})^TF^{-1}$.

    (III) Semi-involutions:
    $$\cS=\begin{pmatrix}J & I-J \\ J-I & J \end{pmatrix}$$
    where $J$ is a diagonal matrix whose diagonal elements are 0's and 1's, so that $J^2=J$ and $(I-J)^2=I-J$.    
\end{theorem}
Theorem \ref{generators} differs from the result of \cite{huaReiner} in two points. In (II) we require $D=F(A^{-1})^TF^{-1}$ instead of $D=(A^{-1})^T$ and  in $(I)$ we require that $S$ satisfies $SF=FS^T$ instead of $S$ being symmetric. We give a name to this last property.
\begin{definition}
We say an $r\times r$ matrix $S$ is \emph{$F$-symmetric} if it satisfies $SF=FS^T$.
\end{definition}
The proof of Theorem \ref{generators} relies on two lemmas which we also modify slightly from \cite{huaReiner}.
\begin{lemma} \label{lemma1}
Let $m$ be a nonzero integer and $T$ an $r\times r$ $F$-symmetrix matrix (not necessarily integral)  and suppose that $m$ does not divide at least one of the elements of $T$. Then there exists an $F$-symmetric matrix matrix $S$ with integral elements such that
$$0< |det(T-mS)|< \frac{d_1}{d_r} |m|^r.$$
\end{lemma}
We modified \cite[Lemma 1]{huaReiner} by changing "symmetric" to "$F$-symmetric" and adding the scaling factor $d_1/d_r$ to the right-hand side of the inequality. The proof is also a word-for-word retelling of the relevant proof of Hua and Reiner with "symmetric" exchanged by "$F$-symmetric" and the right scaling factors added.
\begin{proof}[Proof of Lemma \ref{lemma1}]
    For $r=1$ we have $T=t$ and we may write $t=s\cdot m +r +\{t\}$ for $s,r\in \dZ$, $|r|<m$ and $\{t\}$ the fractional part of $t$. $S=[s]$ is the matrix we are looking for.

    For $r=2$ let $T=(t_{ij})$ and $S=(s_{ij})$. The $F$-symmetry condition requires that $d_2t_{12}=d_1t_{21}$. Since $d_1|d_2$ this means $t_{12}=\frac{d_1}{d_2}t_{21}$ and $s_{12}=\frac{d_1}{d_2}s_{21}.$ We have
    $$|det(T-mS)|=|(t_{11}-ms_{11})(t_{22}-ms_{22})-\frac{d_1}{d_2}(t_{21}-m s_{21})^2|.$$
    If $m$ divides both $t_{11}$ and $t_{22}$ it cannot divide $t_{21}$ by our assumption. Let then $s_{11}=t_{11}/m$ and $s_{22}=t_{22}/m$ and let $s_{21}$ such that $|t_{21}-m s_{21}|<m$ and we are done. 

    If $m$ does not divide at least one of the diagonal elements, say $t_{11}$ let $s_{21}$ be an arbitrary integer and choose $s_{11}$ such that $|t_{11}-ms_{11}|<m$. Then we have
    $$|det(T-mS)|=|-m(t_11-ms_{11})s_{22}+...|,$$
    where $...$ denotes the already fixed terms not involving $s_{22}$. Using the Euclidean algorithm on $...$ there exists an $s_{22}$ such that
    $$|det(T-mS)|\leq |m(t_11-ms_11)| <|m|^2.$$

    We proceed by induction on $r$. Suppose that the result has been established for $r-1$ with $r\geq 3$. Let $T$ be an $r\times r$ matrix and suppose that $t_{ij}$ is not divisible by $m$. Let $t_{kk}$ be a diagonal element not in the same row or column as $t_{ij}$. Let $T_1$ be the $(r-1)\times (r-1)$ matrix obtained from $T$ by removing the $k^{th}$ row and column and let $S_1$ be obtained the same way from $S$. By the induction hypothesis, we may choose an $F$-symmetric $S_1$ such that
    \begin{align*}
    0 <|det (T_1-mS_1)|< \begin{cases}
         \frac{d_1}{d_{r-1}} |m|^{r-1}\\
         \frac{d_1}{d_r} |m|^{r-1}\\
        \frac{d_2}{d_{r}}|m|^{r-1}
    \end{cases}
    \leq \frac{d_1}{d_r} |m|^{r-1},
    \end{align*}
    since $d_1\geq d_2$ and $d_{r-1}\geq d_r$. Fix now $s_{lk}$ and $s_{kp}$ for $l=k+1,..., n$ and $p=1,...,k-1$ arbitrarily. Then, $s_{kl}=\frac{d_l}{d_k}s_{lk}$ and $s_{pk}=\frac{d_p}{d_k}$ and we have
    $$|det(T-mS)|=|-ms_{kk}det(T_1-mS_1)+...|,$$
    where the terms $...$ do not involve $s_{kk}$. Using again the Euclidean algorithm we can choose $s_{kk}$ such that
    $$|det(T-mS)|\leq |m| \cdot |det(T_1-mS_1)| < \frac{d_1}{d_r}|m|^r.$$
\end{proof}
\begin{lemma}\label{lemma2}
Let $A$ and $B$ be integer matrices satisfying $AFB^T=BFA^T$ and let $det(A)\neq 0$. Then, there exists an $F$-symmetric integer matrix $S$ such that either
$$B-AS=0$$
or 
$$0<|det(B-AS)|<\frac{d_1}{d_r}|det(A)|.$$
\end{lemma}
Once again we have modified \cite[Lemma 2]{huaReiner} by simply changing "symmetric" to "$F$-symmetric" and adding the right scaling factors. The proof is again the proof from the paper with the same changes.
\begin{proof}[Proof of Lemma \ref{lemma2}]
    Let $m=|det(A)|$.     From $AFB^T=BFA^T$ the matrix $A^{-1}B$ is $F$-symmetric but may be no longer integral.  Then either every element of $A^{-1}B$ is an integer, in which case there exists an $F$-symmetric $S$ with $A^{-1}B=mS$ and $B-AS=0$, or, by \ref{lemma1} there exists $F$-symmetric matrices $R$ and $S$ such that $S$ is integral
    $$mA^{-1}B=mS+R$$
    and $0<|R|<\frac{d_1}{d_r}|m|^r$. Then, 
    $B-AS=AR/m$ and
    $$0<|det(B-AS)|=\Big|det\Big(\frac{AR}{m}\Big) \Big|=\frac{|det(A)|\cdot |det(R)|}{|m|^r}=\frac{|det(R)|}{|m|^{r-1}}<\frac{d_1}{d_r}|m|.$$
    \end{proof}
    \begin{proof}[Proof of Theorem \ref{generators}] We follow word-for-word the proof from  \cite[Section 3]{huaReiner} just exchanging "symmetric" to "$F$-symmetric" and taking care of the scaling factors.
        Let $$\cS=\begin{pmatrix}
            A & B \\ C & D
        \end{pmatrix}.$$
        The relation $\cS \cF \cS^T=\cF$  implies that $AFD^T-BFC^T=F$ so not both $A$ and $B$ are $0$. Moreover,
        $$\begin{pmatrix} A & B \\ C& D \end{pmatrix}\begin{pmatrix}0 & I \\ -I & 0 \end{pmatrix}=\begin{pmatrix}-B & A \\ * & * \end{pmatrix}$$
        so we may assume that $A$ has rank $k>0$. Furthermore, 
        $$\begin{pmatrix}U_1 & 0 \\ 0 & U_2 \end{pmatrix}  \begin{pmatrix}A & B \\ C & D \end{pmatrix} \begin{pmatrix} V_1 & 0 \\ 0 & V_2\end{pmatrix} = \begin{pmatrix}U_1A V_1 & * \\ * & * \end{pmatrix} $$
        where $U_2= F(U_1^T)^{-1}F^{-1}$ and  $V_2=F(V_1^T)^{-1} F^{-1}$ so we may take $A$ to have the form
        $$A=\begin{pmatrix} A_1 & 0 \\ A_2 & 0 \end{pmatrix}$$
        where $A_1$ is an $k\times k $ non-degenerate matrix. We write $B$ as
        $$B=\begin{pmatrix} B_1 & * \\ * & * \end{pmatrix}$$
        where again $B_1$ is an $k\times k$ matrix. From $\cS \cF \cS^T =\cF$ we have $A_1F_1B_1^T=B_1F_1A_1^T$. By Lemma \ref{lemma2}, there exists an integral matrix $S_1$ such that either $A_1S_1+B_1=0$ or $A_1S_1+B_1=R_1$ with $0< |det(R_1)|< (d_1/d_k)\cdot|det(A_1)|$. Let 
        $$S=\begin{pmatrix} S_1 & 0 \\ 0 & 0 \end{pmatrix}.$$
    Then,
    $$\begin{pmatrix}A & B \\ C & D \end{pmatrix}\begin{pmatrix}I & S \\ 0 & I\end{pmatrix}=\begin{pmatrix}A & AS+B \\ * & * \end{pmatrix}$$
    so either $B_1$ becomes $0$ or it is replaced by $R_1$ while $A$ is unaltered. If the latter occurs we proceed as follows: let
    $$J=\begin{pmatrix}0 & 0 \\ 0 & I_{(r-k)\times (r-k)} \end{pmatrix}.$$
 
    Then,
    $$\begin{pmatrix}A & B \\ C& D\end{pmatrix}\begin{pmatrix} J & I-J\\ J-I & J  \end{pmatrix}=\begin{pmatrix}A' & * \\ * & * \end{pmatrix}$$
    where 
    $$A'=AJ+B(J-I)=\begin{pmatrix} -R_1 & 0 \\ * & 0 \end{pmatrix}.$$
    We now repeat the process as before. Since there are finitely many positive integers less than $(d_1/d_k)\cdot |det(A_1)|$, this process eventually terminates. Thus, by multiplying by matrices of the form $(I)-(III)$ we arrive to a matrix
    $$\cS=\begin{pmatrix}A_0 & B_0 \\ C_0 & D_0  \end{pmatrix}$$
    with
    $$A_0=\begin{pmatrix}R & 0 \\ * & * \end{pmatrix}, \ \ \ B_0 = \begin{pmatrix} 0 & * \\ * & *\end{pmatrix}$$
    where $det (R)\neq 0 $. Since $A_0FB_0^T=B_0FA_0^T$ we see that $B_0$ must be of the form
    $$B_0=\begin{pmatrix} 0 & *\\ 0 & * \end{pmatrix}.$$
    But then,
    $$\begin{pmatrix} 0  & I\\ -I & 0 \end{pmatrix}\begin{pmatrix} A_0 & B_0\\ C_0 & D_0 \end{pmatrix}\begin{pmatrix} J & I-J\\ J-I & J \end{pmatrix}=\begin{pmatrix} A^+ & B^+ \\ 0 & D^+ \end{pmatrix}$$
    where $J$ is as before. Finally, an upper block diagonal matrix in $Sp(\cF,\dZ)$ can be written as
   $$ \begin{pmatrix}A & B \\ 0 & D \end{pmatrix}= \begin{pmatrix}U & 0 \\ 0 & V \end{pmatrix} \begin{pmatrix}I & S \\ 0 & I \end{pmatrix}$$
    with $U=A$, $V=D$ and $S=BD^{-1}=BF(A^T)F^{-1}$.
   \end{proof}

Let $L\ra X=V/\Gamma$ be a non-degenerate $U(1)$-bundle with connection and let us construct an isogeny $f:Y\ra X$ as (\ref{Y lattice}) and (\ref{Y torus}) via a decomposition $\Gamma=\Gamma_1+\Gamma_2$. Let $N\ra Y$ be a $U(1)$-bundle with connection such that $f^*N=L$.
\begin{theorem}
    Let $(Y,N)$ be a pair as above together with a degree $d$ isogeny $f:X\ra Y$. Then, $\hat{f}_*(\phi_N)_*N^{-1}$ is independent of the choice of $Y$ and $N$ with the prescribed properties.
\end{theorem}
\begin{proof}
    It remains to show that $\hat{f}_*(\phi_N)_*N^{-1}$ is independent of the choice of decomposition $\Gamma=\Gamma_1+ \Gamma_2$, where we define $Y=V/\Gamma_Y$ as $V/(\Gamma_1+F^{-1}(\Gamma_1^\vee)$. Let $\Gamma=\Gamma_1'\oplus \Gamma_2'$ be another decomposition. Let $(x^i,y^i)$ and $(u^i,v^i)$ be 1-periodic coordinates for the two decompositions in which $F$ is in normal form, that is,
    $$F=\sum_{i=1}^r d_i dx^i\wedge dy^i=\sum_{i=1}^r d_i du^i\wedge dv^i$$
    where $d_i$ are positive integers and $d_i|d_{i+1}$. Therefore, we can write $F$ in matrix form $\cF$ (\ref{F in matrix form}) and the transition matrix between the coordinates is $\cS\in Sp(\cF,\dZ)$ (\ref{transition matrix}).     
    
    By Theorem \ref{generators} such an $\cS$ can be written as the product of two types of matrices: 
    (I.) upper block-diagonal matrices, a combination of translations and rotations, and (II.) semi-involutions. We prove separately that our construction is independent of base change by these generators to conclude the result.

    \textit{(Step I.)} $\cS$ is upper block-diagonal 
    $$\cS=\begin{pmatrix} A & B \\ 0 & D\end{pmatrix}.$$
    We have to investigate the difference between $F(\lambda_1,\lambda_2)$ and $F(\lambda_1',\lambda_2')$ for $\lambda=\lambda_1+\lambda_2$ with respect to $\Gamma=\Gamma_1+\Gamma_2$ and $\lambda=\lambda_1'+\lambda_2'$ with respect to $\Gamma=\Gamma_1'+\Gamma_2'$. Using the bases corresponding to the 1-periodic coordinates we have
    $$\begin{pmatrix}A^T & 0 \\ B^T & D^T \end{pmatrix}\begin{pmatrix} \lambda_1 \\ \lambda_2 \end{pmatrix}=\begin{pmatrix} \lambda_1' \\ \lambda_2' \end{pmatrix}.$$
    and $F(\lambda_1,\lambda_2)=\lambda_1^T\cF \lambda_2$ so 
    $$F(\lambda_1',\lambda_2')=(\lambda_1')^T\cF \lambda_2'=(A^T\lambda_1)^T\cF (B^T\lambda_1+D^T\lambda_2)=\lambda_1^TAFD^T\lambda_2+\lambda_1^TAFB^T\lambda_1$$
    Since $\cS\cF\cS^T=\cF$ we have
    $$F(\lambda_1',\lambda_2')=F(\lambda_1,\lambda_2)+\lambda_1^T(AFB^T)\lambda_1$$
    where $AFB^T$ is symmetric. In particular, for any $\lambda\in \Gamma$ we have
    $$exp(i\pi \lambda_1^T(AFB^T)\lambda_1)=exp\Big(i\pi \sum_{i=1}^r(AFB^T)_{ii}(\lambda_1)_i\Big).$$
    If we express the semi-character $\chi$ for $L$ either as
    $$\chi_L(\lambda)=exp(i\pi F(\lambda_1,\lambda_2)+2\pi i \hat{x}(\lambda)) $$
    or as
    $$\chi_L(\lambda)=exp(i\pi F(\lambda_1',\lambda_2')+2\pi i \hat{x}'(\lambda))$$
    then, $\hat{x}'=\hat{x}+\frac{1}{2}\sum_{i=1}^r(AFB^T)_{ii}(x^i)^*,$ where $(x^i)^*$ is the element of the dual basis on $V^*$ dual to $x^i$.

    Consider now $\Gamma_Y=\Gamma_1+F^{-1}(\Gamma_1^\vee)$ and $\Gamma_{Y'}=\Gamma_1'+F^{-1}(\Gamma_1^{\prime \vee})$. From the description
    $(x^i,d_i y^i)$ and $(u^i,d_i v^i)$ are 1-periodic coordinates for these lattices and we have
    $$\begin{pmatrix} A^T & 0\\ FB^T & FD^TF^{-1} \end{pmatrix} \begin{pmatrix} x \\ Fy \end{pmatrix}= \begin{pmatrix}u \\ Fv\end{pmatrix}.$$
    The matrix 
    $$\cS'=\begin{pmatrix} A & BF\\ 0 & F^{-1}DF \end{pmatrix}$$
    is integral, since $\cS^{-1}=\cF\cS^T\cF$.

    It remains to show that extending $\chi_L$ to $\Gamma_Y$ does not depend on the description we chose, that is
    $$exp(i\pi F(\lambda_1,\lambda_2)+2\pi i \hat{x}(\lambda))=exp(i\pi F(\lambda_1',\lambda_2')+2\pi i \hat{x}'(\lambda))$$
    for any $\lambda\in \Gamma_Y$. Note, that $\cS'\in Sp(\dZ)$ that is
    $$\cS'\begin{pmatrix} 0 & I \\ -I & 0 \end{pmatrix}(\cS')^T=\begin{pmatrix} 0 & I \\ -I & 0 \end{pmatrix} $$
    and $\cF$ acts as $\begin{pmatrix} 0 & I \\ -I & 0 \end{pmatrix}$ on $\Gamma_Y$. That is for $\lambda\in \Gamma_Y$ we have
    $$F(\lambda_1,\lambda_2)=\lambda_1^T\lambda_2$$
    and
    $$F(\lambda_1',\lambda_2')=(A^T\lambda_1)^T(FB^T\lambda_1+FD^TF^{-1}\lambda_2)=\lambda_1^T\lambda_2+\lambda_1^TAFB^T\lambda_1$$
    so the difference between $x$ and $x'$ still cancels out the difference between $F(\lambda_1,\lambda_2)$ and $F(\lambda_1',\lambda_2')$.

    \textit{(Step II.)} $\cS$ is a semi-involution
    $$\cS=\begin{pmatrix}J & I-J\\ J-I & J\end{pmatrix}$$
    for $J$ a diagonal matrix with entries $0$ and $1$.     We will show the proof for 
    $$J=\begin{pmatrix}0 & 0 \\
    0 & I_{(n-1)\times (n-1)}\end{pmatrix}$$
    and explain how it implies the general setting. Let $\{x^i,y^i\}$ be 1-periodic coordinates for $\Gamma=\Gamma_1+\Gamma_2$ and $\{-y_1,x_2,...,x_r,x_1,y_2,...,y_r\}$ 1-periodic coordinates for $\Gamma=\Gamma_1'+\Gamma_2'$. In both cases
    $$F=\sum_{i=1}^r d_i dx^i\wedge dy^i$$
    and we define
    \begin{align*}\Gamma_Y&=span\Big\{ \frac{\partial}{\partial x_i}, -\frac{1}{d_i}\frac{\partial}{\partial y_i}\Big \},\ \ \\
    \Gamma_{Y'}&=span\Big\{ -\frac{\partial}{\partial y_1},\frac{\partial}{\partial x_2},...,\frac{\partial}{\partial x_r}, -\frac{1}{d_1}\frac{\partial}{\partial x_1},  -\frac{1}{d_2}\frac{\partial}{\partial y_2},...,-\frac{1}{d_r}\frac{\partial}{\partial y_r}\Big \}\end{align*}
    moreover,
    $$F(\lambda_1,\lambda_2)=F(\lambda_1',\lambda_2').$$
    so the semicharacter 
    $\chi_L(\lambda)=exp(i\pi F(\lambda_1,\lambda_2)+2\pi i \hat{x}(\lambda))$
    which extends to both $\Gamma_Y$ and $\Gamma_{Y'}$ defining $N$ and $N'$.
    The dual lattices are given by
    $$\Gamma_Y^\vee=\{dx^i,d_i dy_i\},\ \ \ \Gamma_{Y'}^\vee=\{ -dy_1, dx_2,...,dx_r,d_1 dx_1, d_2 dy_2,...,d_r dy_r\}.$$

    Let us now calculate $\hat{L}_1=\hat{f}_*(\phi_N)_*N^{-1}$ and $\hat{L}_2=\hat{f}'_*(\phi_{N'})_*(N')^{-1}$. We need to choose representatives for $\Gamma^\vee:\Gamma_Y^\vee$ and $\Gamma^\vee:\Gamma_{Y'}^\vee$, so let them be
    $$\lambda_{\underline{m}}=\sum_{i=1}^r m_i dy_i,\ \ \ m_i\in \{0,...,d_i-1\}$$
    and
    $$\mu_{\underline{m}}=m_1dx_i +\sum_{i=2}^r m_i dy_i,\ \ \ m_i\in \{0,...,d_i-1\}$$
    respectively.
    
    Both $\hat{L}_1$ and $\hat{L}_2$ are projectively flat vector bundles on $\hat{X}$ of the same rank and same curvature. In particular, their factor of automorphy is of the form
    $$a_{\hat{L}_1}(\hat{v},\hat{\lambda})=exp(i\pi F^{-1}(\hat{v},\hat{\lambda}))U_1(\hat{\lambda})\ \ \ \text{and}\ \ \  a_{\hat{L}_2}(\hat{v},\hat{\lambda})=exp(i\pi F^{-1}(\hat{v},\hat{\lambda}))U_2(\hat{\lambda})$$
    where $U_1,U_2: \Gamma^\vee \ra U(d)$ are semi-representations. Both vector bundles are endowed 
    with the same connection
    $$A_{\hat{L}_1}(\hat{v})=A_{\hat{L}_2}(\hat{v})=\Big(i\pi F^{-1}(\hat{v},d\hat{v})\Big)^i_i.$$
    Let us investigate the semi-representations $U_1, U_2: \Gamma^\vee\ra U(d)$. We have
    \begin{align*}
    U_1&(\hat{\lambda})=\\
    &=\big(\chi_{N}(F^{-1}(\Lambda^{\underline{m}}_{\hat{\lambda}}))^{-1}exp(-i\pi (F^{-1}(\hat{\lambda},\hat{\lambda}_{\underline{m}})+F^{-1}(\hat{\lambda}_{\hat{\lambda}({\underline{m}})}, \hat{\lambda})+F^{-1}(\hat{\lambda}_{\hat{\lambda}({\underline{m}})},\hat{\lambda}_{\underline{m}})))\delta^{\underline{p}}_{\hat{\lambda}({\underline{m}})} \Big)^{\underline{m}}_{\underline{p}}\\
    &=\big(\chi_{N}(F^{-1}(\Lambda^{\underline{m}}_{\hat{\lambda}}))^{-1}exp(-i\pi (F^{-1}(\hat{\lambda},\hat{\lambda}_{\underline{m}})+F^{-1}(\hat{\lambda}_{\hat{\lambda}({\underline{m}})}, \hat{\lambda})))\delta^{\underline{p}}_{\hat{\lambda}({\underline{m}})} \Big)^{\underline{m}}_{\underline{p}}.
    \end{align*}
    Since
    $$\chi_{N}(F^{-1}(\Lambda^{\underline{m}}_{\hat{\lambda}}))^{-1})=exp\Big(-i\pi F((\Lambda^{\underline{m}}_{\hat{\lambda}})_1,(\Lambda^{\underline{m}}_{\hat{\lambda}})_2)-2\pi i F^{-1}(\hat{\lambda}+\hat{\lambda}_{\underline{m}}-\hat{\lambda}_{\hat{\lambda}(\underline{m})},\hat{x})\Big),$$
    changing the factor of automorphy by the constant $\phi(v)=diag(exp(2\pi i F^{-1}(\hat{\lambda}_{\underline{m}},\hat{x})))$
    yields
    \begin{align*}
    U_1(\hat{\lambda})&=exp(-2\pi i F^{-1}(\hat{\lambda},\hat{x}))\times\\
    &\times \Big(exp(-i\pi F((\Lambda^{\underline{m}}_{\hat{\lambda}})_1,(\Lambda^{\underline{m}}_{\hat{\lambda}})_2)exp(-i\pi (F^{-1}(\hat{\lambda},\hat{\lambda}_{\underline{m}})+F^{-1}(\hat{\lambda}_{\hat{\lambda}({\underline{m}})}, \hat{\lambda})))\delta^{\underline{p}}_{\hat{\lambda}({\underline{m}})} \Big)^{\underline{m}}_{\underline{p}}.
    \end{align*}
    Note that a constant change in the factor of automorphy does not change the connection. Similarly,
    \begin{align*}
    U_2(\hat{\lambda})&=exp(-2\pi i F^{-1}(\hat{\lambda},\hat{x}))\times\\
    &\times \Big(exp(-i\pi F((M^{\underline{m}}_{\hat{\lambda}})_1,(M^{\underline{m}}_{\hat{\lambda}})_2)exp(-i\pi (F^{-1}(\hat{\lambda},\hat{\mu}_{\underline{m}})+F^{-1}(\hat{\mu}_{\hat{\lambda}({\underline{m}})}, \hat{\lambda})))\delta^{\underline{p}}_{\hat{\lambda}({\underline{m}})} \Big)^{\underline{m}}_{\underline{p}}.
    \end{align*}
A semi-representation $\Gamma^\vee$ is defined by its value on the generators $\{dx_1,...,dx_r,dy_1,...,dy_r\}$. On $\{dx_2,...,dx_r,dy_2,...,dy_r\}$ the two semi-representations agree, in particular for $i=2,...,r$
\begin{align*}U_1(dx^i)=U_2(dx^i)&=exp(-2\pi i F^{-1}(dx^i,\hat{x}))\Big( exp(-2i\pi F^{-1}(dx_i, m_i dy_i) \Big)^{\underline{m}}_{\underline{m}}\\ 
&=exp(-2\pi i F^{-1}(dx^i,\hat{x}))\Big( exp\Big(\frac{2\pi i m_i}{d_i}\Big)  \Big)^{\underline{m}}_{\underline{m}}
\end{align*}
and
$$U_1(dy^i)=U_2(dy^i)=exp(-2\pi i F^{-1}(dx^i,\hat{x}))\Big( \delta^{p_i}_{m_i+1}\Big)  \Big)^{\underline{m}}_{\underline{p}}$$
Meanwhile, for $\hat{\lambda}=dx^1$ or   $\hat{\lambda}=dy^1$ we have
\begin{equation}\label{U_1 and U_2 on generators}
\begin{aligned}
    U_1(dx^1)&=exp(-2\pi i F^{-1}(dx^1,\hat{x}))\Big( exp\Big(\frac{2\pi i m_1}{d_1}\Big)  \Big)^{\underline{m}}_{\underline{m}}, \\
    U_2(dx^1)&=exp(-2\pi i F^{-1}(dx^1,\hat{x}))\Big(\delta^{p_1}_{m_1+1} \Big)^{\underline{m}}_{\underline{p}}\\
    U_1(dy_1)&=exp(-2\pi i F^{-1}(dy^1,\hat{x}))\Big( \delta^{p_1}_{m_1+1}\Big)  \Big)^{\underline{m}}_{\underline{p}},\\
    U_2(dy^1)&=exp(-2\pi i F^{-1}(dy^1,\hat{x})\Big( exp(-2\pi i F^{-1}(dy^1, m_1dx^1)) \Big)^{\underline{m}}_{\underline{m}}\\
    &=exp(-2\pi i F^{-1}(dy^1,\hat{x})\Big( exp\Big(-\frac{2\pi i m_1}{d_1}\Big) \Big)^{\underline{m}}_{\underline{m}}
\end{aligned}
\end{equation}
These two semi-representations are conjugate to each other via a constant Vandermond-type matrix
\begin{align}\label{V} V=\frac{1}{\sqrt{d_1}}\Big(exp\Big(\frac{2\pi i (m_1-1)(p_1-1) }{d_1} \Big)\Big)^{\underline{m}}_{\underline{p}}\end{align}
that is,
$$V^{-1}U_1V=U_2$$
and therefore $\hat{L}_1\cong \hat{L}_2$. 

Changing the basis using a more general $J$ would mean that we flip $x^i$ and $y^i$ for a subset of $\{1,...,n\}$. This would mean that $U_1$ and $U_2$ would be different on all of the flipped generators. On the other hand, the difference is again conjugating by a rank $d$ matrix $V$ which can be decomposed into the tensor product $V_1\otimes ...\otimes V_n$ where $V_i$ is of rank $d_i$. Each $V_i$ is either the identity if $x^i$ and $y^i$ were not flipped or a Vandermond matrix as above if  $x^i$ and $y^i$ were flipped. In conclusion $\hat{L}_1\cong \hat{L}_2$ again.
\end{proof}

\begin{example} In two dimensions, that is $r=1$, and in coordinates $\{x,y\}$ the matrices  (\ref{U_1 and U_2 on generators}) are given by
$$U_1(dy)=U_2(dx)=(\delta^j_{i+1})^i_j=
\begin{pmatrix} 
0 & 0 & ... & 0 & 1\\
1 & 0 & ... & 0 & 0\\
0 & 1 & ... & 0 & 0\\
\vdots & & & & \vdots\\
0 & 0 & ... & 1 & 0
\end{pmatrix}$$
and 
$$U_1(dx)=\begin{pmatrix} 1 & 0 & \hdots & 0 \\
0 & \xi & \hdots & 0 \\
\vdots & & & \vdots\\
0 & 0 & \hdots & \xi^{d-1}\end{pmatrix},\  \text{and}\ \ \  U_2(dy)=\begin{pmatrix} 1 & 0 & \hdots & 0 \\
0 & \xi^{-1} & \hdots & 0 \\
\vdots & & & \vdots\\
0 & 0 & \hdots & \xi^{-d+1}\end{pmatrix}.$$
so
$$V^{-1}U_1(dy)V=U_2(dy).$$
Therefore,
$$U_2(dx)=U_1(dy)=VU_2(dy)V^{-1}=V^{-1}U_1(dx)V=V^{-1}U_2(dy)^{-1}V$$
if and only if 
$$V^2U_2(dy)=U_2(dy)^{-1}V^2.$$
Indeed,
\begin{align*}(V^2)^i_j=\sum_k V^i_k V^k_j &= \frac{1}{d}\sum_k exp\Big(2\pi i \frac{(i-1)(k-1)+(j-1)(k-1)}{d}\Big)\\
&=\begin{cases}
    1 & \text{if $i+j-2=0,d$}\\
    0 & \text{otherwise.}
\end{cases}\end{align*}
Writing out in matrix form
$$V^2=\begin{pmatrix}
     1 & 0 & ... & 0 & 0\\
     0 & 0 & ... & 0 & 1\\
     0 & 0 & ... & 1 & 0 \\
     \vdots & & & & \vdots\\
     0 & 1 & ... & 0 & 0
\end{pmatrix}.$$
So
$$V^2U_2(dy)=\begin{pmatrix} 
1 & 0 & ... & 0 & 0\\
     0 & 0 & ... & 0 & \xi^{-d+1}\\
     0 & 0 & ... & \xi^{-d+2} & 0 \\
     \vdots & & & & \vdots\\
     0 & \xi^{-1} & ... & 0 & 0
\end{pmatrix},\ \text{and}\ \ \ U_2(dy)^{-1}V^2=\begin{pmatrix}
1 & 0 & ... & 0 & 0\\
     0 & 0 & ... & 0 & \xi^1\\
     0 & 0 & ... & \xi^2 & 0 \\
     \vdots & & & & \vdots\\
     0 & \xi^{d-1} & ... & 0 & 0
\end{pmatrix}.$$
\end{example}

\section{Degenerate U(1)-bundles}
Let $X=V/\Gamma$ be a real $n$-dimensional torus and let $L\ra X$ be a degenerate $U(1)$-bundle. Then, if the curvature of the connection is invariant we can represent $L$ by the canonical factor of automorphy
$$a_L(v,\lambda)=\chi(\lambda)exp(i\pi F(v,\lambda)), \ \ \ A_L(v)=i\pi F(v,dv).$$
The analytification of $\phi_L:X\ra \hat{X}$ is given by $x\mapsto F(x)$ whose kernel is the subspace $Ker(F)$. Therefore, the kernel of $\phi_L$ is a disjoint union of subtori $K(L)$ of $X$. Moreover, $\phi_L$ is not surjective anymore, its image is given by $F(V)/F(\Gamma)$.

In this section, we generalize Theorem \ref{pushforward tensor U(1) nondegen} to degenerate $U(1)$-bundles analogously to the holomorphic situation (\ref{main eqn degen line bdle}). To do this we have to define the pushforward of a $U(1)$-bundle with connection along a general homomorphism of real tori. Any homomorphism decomposes as the composition of an isogeny and a projection (Stein factorization) so it suffices to define the pushforward along a projection.

Let $q:X\ra Y$ be a projection between real tori with connected torus fibers and let $L\ra X$ be a $U(1)$-bundle with a connection on $X$. 
Let $E$ be the Hermitian line bundle with a Hermitian connection $\nabla$ on $X$ associated to $L$ via the standard representation. 
Then, using the short exact sequence 
$$0\ra q^*\Omega^1(Y)\ra \Omega^1(X)\ra \Omega^1_{X/Y} \ra 0$$
for complex-valued differential forms, we may take the relative fiber-wise component 
$$\nabla^1:\Gamma(L)\ra \Gamma(L\otimes \Omega^1_{X/Y})$$
of $\nabla$ on $L$. Finally, we give the following definition.
\begin{definition}\label{definition: pushforward along projection}
    Let $q:X\ra Y$ be a projection between real tori with connected torus fibers and let $L\ra X$ be a $U(1)$-bundle with a connection. Then, if the sheaf
    $$q_*E=q_*Ker(\nabla^1)$$
    is again a Hermitian line bundle with a connection, then $q_*L$ is defined to be the frame bundle of $q_*E$. 
\end{definition}
If there exists a Hermitian line bundle with a connection $E_0\ra Y$ such that $q^*E_0=E$, then $q_*E=q_*Ker(\nabla^1)=E_0$. If $E$ is flat on the fibers of $q$ but not trivial, then $q_*E=q_*Ker(\nabla^1)=0$ as a flat but non-trivial Hermitian line bundle on a torus has no global flat sections.

This point of view extends to isogenies as well, since any principal bundle is trivial over a discrete set of points. In particular, our definition of the pushforward of a $U(1)$-bundle along an isogeny is recovered if we consider instead pushing forward the sheaf $Ker(\nabla^1)$ and the induced connection on the direct image sheaf.



Let again $L\ra X$ be a degenerate $U(1)$-bundle with connection on $X$ and consider the corresponding homomorphism $\phi_L:X\ra \hat{X}$. The connected component of the identity $K(L)_0\subset X$ in $Ker(\phi_L)$ of the kernel of $\phi_L$ is given by $Ker(F)/(Ker(F)\cap \Gamma)$. From the representation of $L$ by the canonical factor of automorphy it is easy to see that $L|_{K(L)_0}$ is flat. Note, that $\phi_L$ depends only on $F$ so for any flat $U(1)$-bundle $L_0$ we have $\phi_{L\otimes L_0}=\phi_L$.

Suppose now that $L|_{K(L)_0}$ is trivial. Then,  $L$ is also trivial on any translates of $K(L)_0$. This is once again easy to see from the canonical factor. Therefore, there exists a $U(1)$-bundle $N$ on $X/K(L)_0$ such that $L=q^*N$ where $q:X\ra X/K(L)_0$ is the projection. We may now decompose $\phi_L$ into a projection with connected fibers and an isogeny onto its image via the commutative square
\begin{equation}
\begin{tikzcd}
X\arrow{d}{q} \arrow{r}{\phi_L} & \hat{X}\\
X/K(L)_0 \arrow{r}{\phi_{N}} & \widehat{X/K(L)_0} \arrow{u}{\hat{q}}
\end{tikzcd}
\end{equation}
In particular,
\begin{align}\label{equation for lemma}(\phi_L)_*L^{-1}=\hat{q}_*(\phi_N)_*q_*L^{-1}=\hat{q}_*(\phi_N)_*N^{-1}.\end{align}
Note that $\hat{q}$ is a closed embedding so pushing forward along it is just the same sheaf supported on the image. We can finally prove the following theorem.

\begin{theorem}\label{pushforward U(1) degen}
    Let $L\ra X$ be a $U(1)$-bundle with connection and invariant curvature $2\pi i F$ and let $d:=\text{Pfr}(F)$ be the reduced Pfaffian of $F$. Suppose that $L|_{K(L)_0}$ is trivial. Then, $(\phi_L)_*L^{-1}$ is a $U(d^2)$-bundle with a connection supported on $Im(\phi_L)$ and on $Im(\phi_L)$ using the same notation as in Theorem \ref{pushforward tensor U(1) nondegen} we have
    $$(\phi_L)_*L^{-1}\cong \hat{L}\otimes U(d)$$
     for a projectively flat $U(d)$-bundle $\hat{L}$ on $Im(\phi_L)$. 
\end{theorem}
\begin{proof}
From (\ref{equation for lemma}) and the discussion before we know that there exists a non-degenerate $U(1)$-bundle $N$ on $X/K(L)_0$ such that $q^*N=L$ and $q_*L-N$. The analytification of  $q: X\ra X/K(L)_0$ is given by the projection $q:V\ra V/Ker(F)$. The map on lattices is given by $\Gamma\ra \Gamma/(\Gamma\cap Ker(F))$.

The line bundle $N$ can described via the following factor of automorphy and connection one-form. For any $v\in V/Ker(F)$ and $\lambda\in \Gamma/(\Gamma\cap Ker(F))$
$$a_{N}(v,\lambda)=a_{L}(\bar{v},\bar{\lambda}),\ \ \ A_{N}(v)=A_{L}(\bar{v}),$$
where $\bar{v}\in V$ and $\bar{\lambda}\in \Gamma$ are lifts of $v$ and $\lambda$,  and $(a_{L},A_{L})$ is the canonical factor of automorphy and one-form (\ref{canonical factor u(1)}) for $L$. This definition does not depend on the lift because $L$ is trivial on $K(L)|_0$.

Then $N^{-1}=q_*L^{-1} $ and $N^{-1}$ is a degree $d$ non-degenerate $U(1)$-bundle on $X/K(L)_0$. Therefore, by Theorem \ref{pushforward tensor U(1) nondegen} there exists a projectively flat $U(d)$-bundle $\hat{L}$ on $\widehat{X/K(L)_0}$ such that $(\phi_N)_*N^{-1}\cong \hat{L}\otimes U(d)$.
\end{proof}

\section{Transition between holomorphic and unitary description}
Let $X=V/\Gamma$ be a complex torus and $L\ra X$ a holomorphic line bundle. Then choosing a Hermitian metric on the fibers $L$ we can associate to $L$ a $U(1)$-bundle endowed with a connection. In this section, we show that on a complex torus, there is a canonical way of doing so and that the corresponding connections on the $U(1)$-bundles have invariant curvatures. Moreover, in this setting the homomorphism $\phi_L$ is independent of whether we view $L$ as a holomorphic line bundle or a $U(1)$-bundle with connection. We also show that pushing forward along an isogeny as a holomorphic line bundle or as a $U(1)$-bundle coincide. Finally, we explain how one can define the Poincar\'e line bundle of a real torus.

Let $X$ be torus defined as $X=V/\Gamma$ and suppose that there is a complex structure $I\in End(V)$  endowing $X$ with the structure of a complex torus. The complex structure on $X$ can also be understood via the complex vector space $V^{1,0}$ with the complex structure acting as $i$ via the map 
\begin{align*}
    V\ \ &\ra\ \  V^{1,0}\\
    v\ \ &\mapsto\ \ \frac{1}{2}(v-iIv). 
\end{align*}
Let $E$ be an alternating bilinear form on $V$ taking integer values on $\Gamma$ and satisfying $E(I\ ,I\ )=E(\ ,\ )$. Then, the hermitian form
$H(v,w)=E(Iv,w)+iE(v,w)$
lies in $NS(X)$. Let $L$ be a $U(1)$-bundle with connection $A_L$ such that $E=dA_L$. Then, the Hermitian structure $H$ is compatible with the complex structure on $X$ and there is a holomorphic structure on $L$. Denote by $\cL$ be the holomorphic line bundle associated to $L$. 

Representing the line bundles by the canonical factors of automorphy we have
$$a_L(v,\lambda)=\chi(\lambda)exp(i\pi E(v,\lambda)),\ \ \ A_L(v)=i\pi E(v,dv),$$
and
$$a_\cL(v,\lambda)=\chi(\lambda)exp\Big(\pi H(v,\lambda)+\frac{\pi}{2}H(\lambda,\lambda)\Big).$$

\begin{proposition}
    We have a commutative diagram 
    \begin{equation}\label{phi_L holo or U(1)}
\begin{tikzcd}
    V^{1,0} \arrow{r}{H} & (V^*)^{0,1} \\
    V \arrow{u}{\frac12(v-iIv)} \arrow{r}{E} & V^* \arrow{u}[swap]{\hat{v}+iI^*\hat{v}}
\end{tikzcd}
\end{equation}
that is, $\phi_L=\phi_\cL$.
\end{proposition}
\begin{proof}
    Indeed,
    \begin{align*}
    H\Big(\frac12 (v-iIv)\Big)&=\frac12\Big( E(Iv)+iE(v)+iE(v)+E(Iv)\Big)\\
    &=iE(v)-(I^*E)(v)\\
    &=i(E(v)+iI^*E(v)).
\end{align*}
\end{proof}
Let $f:X\ra Y$ be a degree $d$ isogeny of complex tori. We have seen that the pushforward can be written in terms of a semi-representation both in the holomorphic and in the $U(1)$-case. Indeed, by (\ref{f_*L as semiomog factor}) and (\ref{f_*L as projectively flat factor}) (\ref{f_*L as projectively flat connection}) we have
\begin{align*}
    a_{f_*L}(v,\lambda)&=exp(i\pi E(v,\lambda))\cdot U(\lambda),\ \ \ A_L(v)=i\pi E(v,dv)\cdot Id_{d\times d},\\
    a_{f_*\cL}&=exp\Big(\pi H(v,\lambda)+\frac{\pi}{2}H(\lambda,\lambda)\Big)\cdot U(\lambda)
\end{align*}
for the same semi-representation
$$U : \Gamma_Y\ra U(d).$$

\paragraph{The Poincar\'e bundle.} We have defined the Poincar\'e line bundle on $X\times \hat{X}$ via a canonical factor of automorphy. The underlying $U(1)$-bundle with connection can be generalized as the Poincar\'e bundle on $X\times \hat{X}$ where $X$ is a real torus.  This was also the starting point of the work of Bruzzo, Marelli and Pioli (cf. \cite[Section 2.1]{BMP1}).
With the conventions for the complex structure of the dual torus established in Chapter 1, we have the commutative diagram similar to (\ref{phi_L holo or U(1)})
\[
\begin{tikzcd}
    V^{1,0}+(V^*)^{0,1} \arrow{r}{H} & V^{1,0}+(V^*)^{0,1}\\
    V+V^* \arrow{u}{\frac12 (v-iIv)+i(\hat{v}+iI^*\hat{v})} \arrow{r}{E} & V+V^* \arrow{u}[swap]{{\frac12 (v-iIv)+i(\hat{v}+iI^*\hat{v})}}
\end{tikzcd}
\]
where $H(v+\hat{v},w+\hat{w})=\overline{\hat{w}(v)}+\hat{v}(w)$. The map $E$ is the imaginary part of $H$, so using that 
$$\frac{i}{2}(\hat{v}+iI^*\hat{v})(v+iIv)=\frac{i}{2}(\hat{v}(v)+i\hat{v}(Iv)+i^*I^*\hat{v}(v)-I^*\hat{v}(Iv))=-\hat{v}(Iv)+i\hat{v}(v),$$
we have
\begin{align}\label{u(1) curvature of Poincare line bundle}
E(v+\hat{v},w+\hat{w})=\hat{v}(w)-\hat{w}(v)
\end{align}
the natural non-degenerate anti-symmetric bilinear pairing on $V+V^*$.

The holomorphic Poincar\'e bundle (\ref{Poincare line bundle}) is defined by the factor of automorphy 
\begin{align*}
   a(v+\hat{v},\lambda+\hat{\lambda})=exp\Big( \pi\cdot  \hat{v}(\lambda)+\pi\cdot \overline{\hat{\lambda}(v)}+\pi \cdot \hat{\lambda}(\lambda)\Big), \ \\ \ v+\hat{v}\in V^{1,0}+(V^*)^{0,1},\ \lambda+\hat{\lambda}\in \Gamma_\dC+\Gamma_\dC^\vee.
\end{align*}
Therefore, in real coordinates we have 
\begin{equation}\label{u(1) poincare bundle}
\begin{aligned}
    a_{\cP}(v+\hat{v},\lambda+\hat{\lambda})&=exp(i\pi (\hat{\lambda}(\lambda)+\hat{v}(\lambda)-\hat{\lambda}(v)))\ \ \\ 
    A_\cP(v+\hat{v})&=i\pi (\hat{v}dv-vd\hat{v})\\
    &v+\hat{v}\in V+V^*,\ \lambda+\hat{\lambda}\in\Gamma+\Gamma^\vee,\\
    \text{and the curvature is } \ 2\pi i P&=2\pi i\ d\hat{v}\wedge dv.
\end{aligned}
\end{equation}
We can view the dual torus $\hat{X}$ as the space of flat $U(1)$-bundles (cf. Lemma \ref{character for flat U(1) bundles}). Then it is easy to see that the Poincar\'e bundle (\ref{u(1) poincare bundle}) $\cP$ on $X\times \hat{X}$ satisfies 

(1) $\cP|_{X\times\{L}\cong L$,

(2) $\cP|_{\{0\}\times \hat{X}}$ is trivial,

analogously to the holomorphic case.

 \chapter{T-duality for U(1)-bundles with connections}\label{last chapter}
 In this Chapter, we finally upgrade T-duality of generalized branes to T-duality of physical branes. Here, we do not restrict ourselves to $A$ or $B$-branes but we do account for the fact that the T-dual of a $B$-brane should be given by the Fourier-Mukai transform. 

 The history of T-duality for $A$-branes started with the work of Arinkin and Polishchuk \cite{AP} who T-dualized local systems on Lagrangian sections of an affine torus bundle. Later, Bruzzo, Marelli and Pioli \cite{BMP1,BMP2} extended the ideas of \cite{AP} to local systems supported on Lagrangian submanifolds which are affine torus subbundles. They represented these local systems as flat connections on a $U(d)$-bundle using factors of automorphy. Finally, Glazebook Jardim and Kamber defined a T-dual for $U(d)$-bundles with connections on an affine torus bundle when the connection is flat on the fiber. The only work, that the author is aware of, which deals with $U(d)$-bundles that are not flat on the fiber is \cite{CLZ} by Chan, Leung and Zhang. They have given a different construction that we will explain briefly below. When a brane is $BAA$ on a hyperk\"ahler manifold, one can apply Fourier-Mukai transform to the underlying $B$-brane. The resulting holomorphic object in the context of \cite{KW} and Theorem \ref{semiflat tdual} is supposed to be a $BBB$-brane but finding the right hyperholomorphic structure is not possible in general. 
 
 Let us first recall our plan from Section \ref{section T-duality of generalized branes}. Let  $(M, H=0)$ and $(\hat{M},\hat{H}=0)$ be a T-dual pair of affine torus bundles with torsion Chern classes in the sense of generalized geometry.  Let $S\subset X$ be an affine torus subbundle and $L\ra S$ a $U(1)$-bundle with connection such that the curvature $2\pi i F\in \Omega^2(S)$ of the connection is invariant. Then, $(S,F)$ is a locally T-dualizable brane and we can locally construct the following diagram (\ref{big diagram}).
\begin{equation*}
\begin{tikzcd}
& Z \arrow{dl}[swap]{p_Z} \arrow{dr}{\hat{p}_Z} \arrow[hook]{rrr}{i_Z} &  & & M\times_{\pi(S)}\hat{M} \arrow{dl}[swap]{p} \arrow{dr}{\hat{p}} \\
S\arrow[hook, bend right =30]{rrr}{i_S} & & \hat{S} \arrow[hook, bend right =20]{rrr}{i_{\hat{S}}} & \pi^{-1}(\pi(S)) & & \hat{\pi}^{-1}(\pi(S))
\end{tikzcd}
\end{equation*}
In Part 1 of Theorems \ref{tdual u(1) bundles on a torus}, \ref{tdual U(1) bundles trivial base} and \ref{tdual u(1) bundle general base} we show that one can find leaves $Z\subset S\times_{\pi(S)}\hat{M}$ on which the $U(1)$-bundle with connection
$$\hat{L}_Z:=p_Z^*L\otimes \cP|_Z$$
is trivial on the fibers of $\hat{p}_Z:Z\ra \hat{S}$. Here $\cP\ra M\times_B\hat{M}$ is the Poincar\'e bundle which is either a $U(1)$-bundle with connection or a gerbe trivialization. Then we naively define the T-dual as
$$\hat{E}:=\hat{p}_*\hat{L}_Z.$$
Glazebook, Jardim and Kamber also used this method in \cite{GJK} to T-dualize $U(d)$-bundles with connections but those bundles were all flat on the fibers. The structure of the leaves (cf. Lemma \ref{local leaves coordfree}) depends on the restrictions of $F$ to the fibers. Moreover, the projection $\hat{p}_Z: Z\ra \hat{S}$ is determined by a fiberwise homomorphism of tori. If the curvature $F$ is zero when restricted to the fibers of $S$ this homomorphism is just a projection but when $F$ is non-zero it is the composition of a projection and an isogeny. The analytification of this isogeny is given by $-F$ restricted to the fiber. 

Therefore, when $L$ restricted to the fibers of $S$ is flat, $q_*\hat{L}_Z$ determines the genuine T-dual of $(S,L)$ in accordance with \cite{GJK}. On the other hand, when $L$ is not flat on the fiber, $\hat{p}_Z$ is the composition of a projection and a degree $d^2$ isogeny, where $d$ is the degree of $L$ on the fibers. Then we show, analogously to Theorem \ref{pushforward U(1) degen}, that there exists a projectively flat $U(d)$-bundle $\hat{L}\ra \hat{S}$ such that
$$\hat{E}=\hat{L}\otimes U(d).$$
This is Part 3. of Theorems \ref{tdual u(1) bundles on a torus}, \ref{tdual U(1) bundles trivial base} and \ref{tdual u(1) bundle general base}.

In \cite{CLZ} Chan, Leung and Zhang gave a different construction. They start with an affine torus subbundle $S\subset X$ in an affine torus bundle together with a $U(1)$-bundle $L\ra S$ with a connection whose curvature is invariant. They associate a spinor bundle $\cS\ra S$ to the pair $(S,L)$ which depends on the curvature. This spinor bundle is trivial when the curvature of $L$ is trivial on the fibers of $S$. Then, they T-dualize $L\otimes \cS$ together with a Dirac operator with the method of Arinkin and Polishchuk.  They come to the same conclusion as we do but their method avoids the extra factor in the T-dual.

This chapter is organized as follows. In Section \ref{last chapter section on a torus}, we carry out our program for a single real torus and show that the result coincides with (\ref{main equation for line bdle on affine subtorus}). That is, on a complex torus the Fourier-Mukai transform of a holomorphic line bundle supported on an affine subtorus is the same as its T-dual as a $U(1)$-bundle with connection. The main result of this section is Theorem \ref{tdual u(1) bundles on a torus}.

In Section \ref{last chapter section contr base}, we T-dualize branes on a trivial affine torus bundle. Here we first extend the factor of automorphy description to $U(1)$-bundles with connections on trivial affine torus bundles. We prove a relative version of the Appel-Humbert theorem. Then we generalize the main result of Section \ref{last chapter section on a torus}.

Finally, in Section \ref{last chapter section general base}, we define T-duality for branes on affine torus bundles with torsion Chern classes. To achieve this we have to refine the T-duality relation from the generalized geometry setting to topological T-duality. Following \cite{B2} we recall the relevant background on gerbes, gerbe connections and gerbe modules and the definition of topological T-duality. We define the Poincar\'e bundle (\ref{u(1) poincare bundle}) as a gerbe and generalize Theorem \ref{tdual u(1) bundles on a torus}.

 \section{On a torus}\label{last chapter section on a torus}
In this section, we T-dualize $U(1)$-bundles with connections supported on affine subtori of a real torus.  Let $X\cong V/\Gamma$ be a real torus, $\hat{X}\cong V^*/\Gamma^\vee$ the dual torus and  $\cP$ the Poincar\'e bundle (\ref{u(1) poincare bundle}) on $X\times \hat{X}$. Then, the tori $X$ and $\hat{X}$, viewed as trivial torus bundles over a point, are T-dual in the sense of generalized geometry (Definition \ref{gen geom tdual}) with $H=0$, $\hat{H}=0$ and $P\in \Omega^2(X\times \hat{X})$ the curvature of Poincar\'e bundle divided by $2\pi i $.

Let $S\subset X$ be an affine subtorus and $L\ra S$ a $U(1)$-bundle with a connection whose curvature is $2\pi i F$ with $F\in \Omega^2(S)$ invariant. In this section, we construct the T-dual of $(S,L)$ and show that our construction recovers the Fourier-Mukai transform when the $U(1)$-bundle and the connection are compatible with a complex structure on the torus.

The pair $\cL=(S, F)$ is a T-dualizable generalized brane on $X$ and it has T-duals in the sense of generalized geometry since the base is contractible (cf. Theorem \ref{local gg thm}). The T-duals are constructed via the diagram
 \begin{equation}\label{big diagram 2}
\begin{tikzcd}
& Z \arrow{dl}[swap]{p_Z} \arrow{dr}{\hat{p}_Z} \arrow[hook]{rrr}{i_Z} &  & & X\times \hat{X} \arrow{dl}[swap]{p} \arrow{dr}{\hat{p}} \\
S\arrow[hook, bend right =30]{rrr}{i_S} & & \hat{S} \arrow[hook, bend right =20]{rrr}{i_{\hat{S}}} & X & & \hat{X}
\end{tikzcd}
\end{equation}
where $Z$ is a leaf of the integrable distribution $\Delta=(F+P)^1$ (\ref{F+P}).  On $Z$ we define the $U(1)$-bundle $\hat{L}_Z$ with connection via the equation
\begin{align}\label{bundles on Z} \hat{L}_Z \cong p_Z^*L\otimes i_Z^*\cP^*.\end{align}
Then we have the following theorem.
\begin{theorem}\label{tdual u(1) bundles on a torus}
    In the setting above, the following hold.
    \begin{enumerate}
        \item\label{1 absolute} There exists a single generalized T-dual $\hat{\cL}=(\hat{S},\hat{F})$  in $\hat{X}$ of $\cL=(S,F)$ such that for any $Z$ with $\hat{p}_Z(Z)=\hat{S}$ the $U(1)$-bundle with connection $\hat{L}_Z$ is trivial on the fibers of $\hat{p}_Z:Z\ra \hat{S}$.
        \item\label{2 absolute} For any leaf $Z$ mapping onto $\hat{S}$ as in \ref{1 absolute}., the pushforward
        $$\hat{E}=(\hat{p}_Z)_*\hat{L}_Z$$
        is a projectively flat $U(d^2)$-bundle independent of $Z$. The curvature of the connection on $\hat{E}$ is given by $2\pi i\hat{F}\cdot Id \in \Omega^2(\hat{S},End(\hat{E}))$.
        \item\label{3 absolute} There exists a projectively flat $U(d)$-bundle $\hat{L}$ on $\hat{S}$ which satisfies
        $$\hat{E}\cong \hat{L}\otimes U(d)$$
        in the sense of Theorems \ref{pushforward tensor U(1) nondegen} and \ref{pushforward U(1) degen}. Moreover, the curvature of the connection on $\hat{L}$ is $2\pi i \hat{F}\cdot Id \in \Omega^2(\hat{S},End(\hat{L})).$
    \end{enumerate}   
\end{theorem}
\begin{definition}
    The T-dual of $(S,L)$ is $(\hat{S},\hat{L})$ as in theorem \ref{tdual u(1) bundles on a torus}.
\end{definition}
\begin{remark}
    The last part of our theorem shows that in general taking fiberwise flat sections (the \emph{naive T-dual}), does not give the right answer, but a vector bundle with the action of $\dZ/d\dZ$. Only after we "divide out" by this action, in part \ref{3 absolute}., we find a T-dual with the properties we want. In particular, if $L$ has a holomorphic structure so does $\hat{L}$ and $L$ and $\hat{L}$ are Fourier-Mukai partners.
\end{remark}
\begin{proof}[Proof of Theorem \ref{tdual u(1) bundles on a torus}]
The submanifold $S\subset X$ is an affine subtorus, that is there exists $b\in X$ such that 
\begin{align}\label{S=t_bS_0}S=t_bS_0=S_0+b,\end{align}
where $S_0$ is a subgroup of $X$.  Then, $S_0$ is given by $V_S/\Gamma_S$ where $V_S\subset V$ is a subspace and $\Gamma_S=\Gamma\cap V$ is a primitive sublattice.   We can write \begin{align}\label{L=t_b*L_0}L\cong (t_b)_*L_0\end{align}  
where $L_0\ra S_0$ is a $U(1)$-bundle with a connection whose curvature is invariant on $S_0$. We denote this curvature also with $2\pi i F$ and we describe $L_0$ by a canonical factor of automorphy
 \begin{align}\label{L0 on S0} a_{L_0}(s,\gamma)=\chi_0(\gamma)exp(i\pi F(s,\gamma)+2\pi i \hat{c}(\gamma)),\ \ \ A_{L_0}(s)=i\pi F(s,ds),\end{align}
 for $s\in V_S$ and $\gamma\in \Gamma_S$. Here, $\hat{c}\in V_S^*$ is a character of $S_0$ lifted from $V_S^*/\Gamma_S^\vee$ and  $\chi_0$ is a semi-character for $F$ on $\Gamma_S$ such that it vanishes on $Ker(F)\cap \Gamma_S$. More precisely, $F$ is non-degenerate and integral on $\Gamma_S/(Ker(F)\cap \Gamma_S)$ so may decompose this lattice into maximal isotropics with respect to $F$
 $$\Gamma_S/(Ker(F)\cap \Gamma_S)=\Gamma_1\oplus\Gamma_2.$$
 Then if we denote the projection $r:\Gamma_S \ra 
 \Gamma_1\oplus\Gamma_2$ a semi-character $\chi_0$ can be given as
 \begin{align}\label{chi0} \chi_0(\gamma)=exp(i\pi F(r(\gamma)_1,r(\gamma)_2)),\ \ \ \text{where\ }r(\gamma)=r(\gamma)_1+r(\gamma)_2,\ \ \  r(\gamma)_i\in \Gamma_i. \end{align}
Suppose that $b'\in X$ is another element such that $S_0+b=S_0+b'$. Then $b-b'\in S_0$ and $L\cong (t_b)_*L_0\cong (t_{b'})_*L_0'$ that is $L_0'\cong (t_{b-b'})_*L_0\cong t_{b'-b}^*L_0$. In particular, after changing representatives
\begin{align}\label{b to b'} a_{L_0'}(s,\gamma)=a_{L_0}(s+b'-b,\gamma)=a_{L_0}(s,\gamma)\cdot exp(2\pi i F(b'-b,\gamma)),\ \ \ A_{L_0'}(s)=A_{L_0}(s).\end{align}
That is, we change $\hat{c} \mapsto \hat{c}+F(b'-b)$ by changing $b$ to $b'$.
 
 We identify $S\times \hat{X}$ with a translate of $S_0\times \hat{X}=(V_S\times V^*)/(\Gamma_S+\Gamma^\vee)$ as
 \begin{align}S\times \hat{X}=t_{(b,0)}(S_0\times \hat{X})\subset X\times \hat{S}\end{align}
  so we have
  $$\cP|_{S\times \hat{X}}=\cP|_{t_{(b,0)}(S_0\times \hat{X})}=(t_{(b,0)})_*\Big((t_{(b,0)}^*\cP)|_{S_0\times \hat{X}}\Big)$$
 we can represent $(t_{(b,0)}^*\cP)|_{S_0\times \hat{X}}$ by the factor of automorphy
  \begin{align*}
    a_\cP(s,\hat{v};\gamma,\hat{\lambda})&=exp\Big(i\pi (\hat{v}(\gamma)-\hat{\lambda}(s+b)+\hat{\lambda}(\gamma))\Big),\ \ \ A_\cP(s,\hat{v})=i\pi (\hat{v}\cdot ds-(s+b)\cdot d\hat{v}),
\end{align*}
with $s+\hat{v}\in V_S+V^*$ and $\gamma+\hat{\lambda}\in \Gamma_S+\Gamma^\vee$. Identifying $S$ with $S_0+b'$ would again change the representative of $(t_{(b,0)}^*\cP)|_{S_0\times \hat{X}}$ by exchanging $b$ with $b'$.

By Lemma \ref{local leaves coordfree} the leaves of the distribution $\Delta$ are affine subtori of $S\times \hat{X}$ modelled on $V_Z/\Gamma_Z$, where $V_Z\subset V_S+V^*$ and $\Gamma_Z\subset V_S+V^*$ fit into the short exact sequences (\ref{V_Z SES}) and (\ref{Gamma_Z SES}). By Proposition \ref{space of leaves of Delta} the space of leaves is parametrized by points of the torus $V_S^*/\Gamma_S^\vee$, the elements of $(V_S+V^*)/V_Z$ up to elements of $(\Gamma_S+\Gamma^\vee)/\Gamma_Z$. Denote elements of $V^*_S/\Gamma_S^\vee$ by $c$, and the corresponding leaf by $Z_c$. Let $\bar{c}\in V_S+V^*$ be a lift of $c$, so we have
\begin{align}\label{Z_c in Sxhat(X)} Z_c\cong t_{(b,0)}(t_{\bar{c}} Z_0)\subset t_{(b,0)}(S_0\times \hat{X})=S\times \hat{X}.\end{align}
Note that we can choose $\bar{c}\in V_X^*$ because $F$ is integral. 
The image of $Z_c$ under the projection $\hat{p}:S\times \hat{X}\ra \hat{X}$ is a subtorus of $\hat{X}$ modelled on $V_{\hat{S}}/\Gamma_{\hat{S}}$ where $V_{\hat{S}}\subset V^*$ and $\Gamma_{\hat{S}}\subset \Gamma^\vee$ fit into the shirt exact sequences (\ref{hat(S) vertical}) and (\ref{hat(S) lattice}). By Lemma \ref{space of T-duals} the space of images is parametrized by the torus $coker(F)/coker(F,\Gamma_S^\vee)$, that is by elements of $V_S^*/F(V_S)$ up to elements of the lattice $\Gamma_S^\vee/(\Gamma_S^\vee\cap F(V_S))$. 

Let us denote the image of $Z_c$ under the projection $\hat{p}$ by $\hat{S}_{\hat{p}(c)}$. Then, $\hat{p}_{Z_c}:Z_c\ra \hat{S}_{\hat{p}(c)}$ is the composition of a homomorphism and a translation. Moreover, the projection $p_{Z_c}:Z_c\ra S$ can also be written as the composition of a homomorphism and a translation. The homomorphism parts $p_0$ and $\hat{p}_0$ of $p_{Z_c}$ and $\hat{p}_{Z_c}$ respectively are given by the following diagrams.
\begin{equation}\label{p_0 and hat(p)_0 on vectors}
\begin{tikzcd}
    0 \arrow{r} & Ann(V_S) \arrow{r}\arrow{d}{\cong} & V_Z \arrow{r}{p_0}\arrow{d}{\hat{p}_0} & V_S \arrow{r}\arrow{d}{-F} & 0 \\
    0 \arrow{r} & Ann(V_S) \arrow{r} & V_{\hat{S}} \arrow{r}{q} & F(V_S) \arrow{r} & 0,\\
\end{tikzcd}
\end{equation}
\begin{equation}\label{p_0 and hat(p)_0 on lattices}
\begin{tikzcd}
     0 \arrow{r} & Ann(\Gamma_S) \arrow{r}\arrow{d}{\cong} & \Gamma_Z \arrow{r}{p_0}\arrow{d}{\hat{p}_0} & \Gamma_S \arrow{r}\arrow{d}{-F} & 0\\
     0 \arrow{r} & Ann(\Gamma_S) \arrow{r} & \Gamma_{\hat{S}} \arrow{r}{q} & F(V_S)\cap \Gamma_S^\vee \arrow{r} & 0.
\end{tikzcd}
\end{equation}
Indeed $p_0$ and $\hat{p}_0$ are the restrictions of $V+V^*\ra V$ and $V+V^*\ra V^*$ to $V_Z$ and $\Gamma_Z$. Then we have the commutative diagram
\begin{equation}\label{hat(p) as homom and translation}
\begin{tikzcd}
    Z_0\arrow{d}{\hat{p}_0} \arrow{r}{t_{(b,0)}\circ t_{\bar{c}}} &     Z_c \arrow{d}{\hat{p}_{Z_c}}\\
    S_0 \arrow{r}{ t_{c} } & S_{\hat{p}(c)}
\end{tikzcd}
\end{equation}
and we see that $\hat{p}(c)$ in $coker(F)/coker(F,\Gamma_S^\vee)$ is represented by $c=\hat{p}_0(\bar{c})\in V_S^*$. Here we abuse notation, and we also denote by $c$ the lift of $c\in V_S^*/\Gamma_S^\vee$ to $V_X^*$.

Let $\hat{L}_{Z_c}=(p^*L\otimes \cP)|_{Z_c}$. We have on the correspondence space
$$p^*L\otimes \cP\cong p^*(t_b)_*L_0\otimes \cP\cong (t_{(b,0)})_*p^*L_0\otimes \cP\cong (t_{(b,0)})_*\Big(p^*L_0\otimes t_{(b,0)}^*\cP\Big)$$
so we have
\begin{align}\label{hat(L)_c}\hat{L}_{Z_c}=(t_{(b,0)})_*\Big(p^*L_0\otimes t_{(b,0)}^*\cP\Big)\Big|_{t_{\bar{c}}Z_0}.\end{align}
Let us denote by $L_c$ the $U(1)$-bundle  
\begin{align}\label{L_c}L_{c}\cong t_{\bar{c}}^*\Big(p^*L_0\otimes t_{(b,0)}^*\cP)\Big|_{t_{\bar{c}}Z_0}\cong \Big(t_{\bar{c}}^*p^*L_0 \otimes t_{\bar{c}}^*t_{(b,0)}^*\cP\Big)\Big|_{Z_0},\end{align} 
 so we have $\hat{L}_{Z_c}\cong (t_{(b,0)}\circ t_{\bar{c}})_*L_c$. It is a $U(1)$-bundle on $Z_0$, therefore it can be represented by a canonical factor of automorphy. For $v\in V_Z$ and $\gamma\in \Gamma_Z$, with notation as in the diagrams (\ref{p_0 and hat(p)_0 on vectors}) and (\ref{p_0 and hat(p)_0 on lattices}) and using that $\hat{p}_0(\bar{c})=c$ and $p_0(\hat{c})=0$ we have
 \begin{align*}
a_{L_{c}}&(v,\lambda)=\\
=&exp\Big(  i\pi ((\hat{p}_0(v)+c)(p_0(\lambda))-(\hat{p}_0(\lambda))(p_0(v)+b)+i \pi \hat{p}_0(\lambda)(p_0(\lambda))\Big)\times\\
&\hspace{3cm} \times \chi_0(p_0(\lambda))exp\Big(i\pi F(p_0(v),p_0(\lambda))+2\pi i \hat{c}(p_0(\lambda)) \Big)\\
=&\chi_0(p_0(\lambda))exp\Big(i\pi (-F(p_0(v),p_0(\lambda))+c(p_0(\lambda))+F(p_0(\lambda),p_0(v))-\hat{p}_0(v)(b))\Big)\times\\
&\hspace{2cm} \times exp\Big(-i\pi F(p_0(\lambda),p_0(\lambda))+i \pi F(p_0(v),p_0(\lambda))+2\pi i \hat{c}(p_0(\lambda))  \Big)\\
=&\chi_0(p_0(\lambda))exp\Big(-i\pi F(p_0(v),p_0(\lambda))+i \pi c(p_0(\lambda))-i \pi \hat{p}_0(v)(b) +2\pi i \hat{c}(p_0(\lambda))  \Big),\\
A_{L_{c}}&(v)=i\pi F(p_0(v)+b,dp_0(v))+ i\pi (\hat{p}_0(v)+c) \cdot dp_0(v)-(p_0(v)+b)\cdot d\hat{p}_0(v)\\
=&i \pi F(p_0(v),dp_0(v)) - i \pi F(p_0(v),dp_0(v)) + i \pi c \cdot dp_0(v) +\\
&+i \pi F(dp_0(v), p_0(v)) - b\cdot d\hat{p}_0(v)\\
=&-i\pi F(p_0(v),dp_0(v))+ i \pi c \cdot dp_0(v)  -i \pi b\cdot d\hat{p}_0(v).\end{align*}
Changing the representatives by $\phi(v)=exp(-\pi i b\cdot \hat{p}_0(v)   + i \pi c \cdot p_0(v))$ we have
\begin{equation}\label{L_c factor}
\begin{aligned}
a_{L_{c}}(v,\lambda)&=\chi_0(p_0(\lambda))\times \\
&\ \ \ \times exp\Big(-i\pi F(p_0(v),p_0(\lambda))- 2 \pi i \hat{p}_0(\lambda)(b) +2\pi i (\hat{c}+c)(p_0(\lambda))  \Big)\\
A_{L_{c}}(v)&=-i\pi F(p_0(v),dp_0(v)).
\end{aligned}
\end{equation}
The line bundle $L_{c}$ is independent of the choice of $b\in X$. Indeed, changing $b$ to $b'$ would add the term $-2\pi 
 i \hat{p}_0(\lambda)(b'-b)$ to the factor of automorphy but we would have to also change $\hat{c}$ to $\hat{c}+F(b'-b)$ (see (\ref{b to b'})). Using that $b'-b\in V_S$ we get that $-\hat{p}_0(\lambda)(b'-b)=(q\circ \hat{p}_0)(\lambda)(b'-b)=-(-F\circ p_0)(\lambda)(b'-b)=F(p_0(\lambda),b'-b)$. Meanwhile, $F(b'-b)(p_0(\lambda))=F(b'-b,p_0(\lambda))$. Therefore the resulting bundle and connection does not depend on the choice of $b\in X$.

 \textit{Proof of \ref{1 absolute}.:}  By (\ref{p_0 and hat(p)_0 on vectors}) the kernel of $\hat{p}_0:V_Z\ra V_{\hat{S}}$ surjects onto $Ker(-F)\subset V_S$ via $p_0$. Henceforth, the fibers of $\hat{p}_{Z_c}:Z_c\ra \hat{S}_{\hat{p}(c)}$ are disjoint union of affine tori  modelled on $Ker(-F)/(Ker(-F)\cap \Gamma_S)\cong K(L)_0$ (\ref{K(L)_0 SES}). On $K(L)_0$ the restriction of $L_{c}$ is flat and it is given by the image of $\hat{c}+c\in V_S^*/\Gamma_S^\vee$ along the surjection 
 \begin{align*}\hat{j}:V_S^*/\Gamma_S^\vee\ra \widehat{K(L)_0}\end{align*}
dual to the inclusion $j:K(L)_0\ra S$. The other components of the fiber are translations of $K(L)_0$ by the preimage of a full set of representatives $\{\lambda_1,...,\lambda_{d^2}\}$  of $(H(V_S)\cap \Gamma^\vee_S): H(\Gamma_S)$. But $L_c$ is invariant under such translations as
$$a_{t^*_{F^{-1}(\lambda_i)}L_c}(v,\lambda)=exp\Big(-i\pi \lambda_i(p_0(\lambda)) \Big)a_{L_c}(v,\lambda),\ \ \ A_{t^*_{F^{-1}(\lambda_i)}L_c}(v)=A_{L_c}(v)-i\pi \lambda_i(dp_0(v)),$$
so changing the representative by $\phi(v)=exp(-i\pi \lambda_i(p_0(v)))$ gives the isomorphism.

In particular, $L_{c}$ is trivial on the fibers of $\hat{p}_{Z_c}$ if and only if $\hat{c}+c$ lies in the kernel of $\hat{j}$ or from the point of view of vector spaces if and only if $\hat{c}+c$ lies in $F(V_S)$. By Corollary \ref{space of leaves mapping to a T-dual} the space of leaves mapping to a single T-dual is indeed parametrized by elements of $F(V_S)/F(\Gamma_S)$ so there is a single T-dual $\hat{S}$ such that $L_c$ is trivial on all $Z_c$ mapping onto $\hat{S}$.

\textit{Proof of \ref{2 absolute}.:} Let us now calculate $(\hat{p}_c)_*L_{c}$ for $c$ such that $L_{c}$ is trivial on the fibers. We decompose $\hat{p}_{Z_c}$ into a translation and a homomorphism as in (\ref{hat(p) as homom and translation}) $$\hat{p}_{Z_c}=t_{c} \circ \hat{p}_0: Z_c \ra \hat{S}_{\hat{p}(c)}$$
where we choose a lift of $c\in V_S^*/\Gamma_S^\vee$ to an element $c$ of $\hat{X}$. 
Let 
$$\hat{E}_c:=(\hat{p}_0)_*L_{c}.$$
By (\ref{L_c factor}) we can write
$$L_c\cong (p_0^*L_c^S)\otimes \hat{p}^*_0\cP_{b}$$
where $L_c^S\ra S_0$ and $\cP_{b}\ra \hat{S}_0$ are $U(1)$-bundles given by
\begin{align*}a_{L_c^S}(s,\gamma)&=\chi_0(\gamma)exp(-i\pi F(s,\gamma)+2\pi i (\hat{c}+c)(\gamma)),\ \ \ A_{L_c^S}(s)=-i\pi F(s,ds),\\ a_{\cP_b}(\hat{s},\hat{\gamma})&=exp(-2\pi i \hat{\gamma}(b)),\ \ \ A_{\cP_b}=0.\end{align*}
In particular, using the projection formula
\begin{align*}
\hat{E}_c=(\hat{p}_0)_*(L_c)\cong \Big((\hat{p}_0)_*p_0^*L_c^S\Big)\otimes \cP_b,
\end{align*}
and by the commutative square in (\ref{p_0 and hat(p)_0 on vectors}) and (\ref{p_0 and hat(p)_0 on lattices}) we have
\begin{align*}\hat{E}_c\cong \Big(q^*(-\phi_{L_0})_*L_c^S\Big)\otimes \cP_b,\end{align*}
where $-\phi_{L_0}$ is the isogeny whose analytification is $-F$.  Notice that $\cP_b$ is $\cP|_{\{b\}\times \hat{X}}$ restricted to $\hat{S}_0$. 

That is, to determine $\hat{E}_c$ we only need to determine $(-\phi_{L_0})_*L_c^S$. Writing out the definition of the pushforwards (\ref{definition: pushforward along projection}) we find
\begin{align}\label{funky equation} 
(-\phi_{L_0})_*L_c^S\cong (\phi_{L_0})_*(L^{-1}_0\otimes \hat{\cP}_{-c}),
\end{align}
where $\hat{\cP}_{-c}=\cP|_{S\times \{-c\}}$. Therefore, $\hat{E}\cong (t_c)_*\hat{E}_c\cong t_{-c}^*\hat{E}_c$ is  
$$\hat{E}\cong t_{-c}^*\Big(q^*(\phi_{L_0})_*(L_0^{-1}\otimes \hat{\cP}_{-c})\Big)\otimes t_{-c}^*\cP_b$$
since $c$ is the lift of an element in $V_S^*/\Gamma_S^\vee$ to $\hat{X}$ we have $t_{-c}^*q^*=q^*t_{-c}^*$ and since $\cP_b$ is flat we have $t_{-c}^*\cP_b=\cP_b$. In conclusion,
\begin{align}\label{hat(E) FM type description}\hat{E}\cong \Big(q^*t_{-c}^*(\phi_{L_0})_*(L_0^{-1}\otimes \hat{\cP}_{-c})\Big)\otimes \cP_b.\end{align}
The curvature of the connection is indeed given by the two-form $q^*F^{-1}=\hat{F}$. 

Finally, to show that $\hat{E}$ is well defined, let $Z_{c'}$ with $c'\in V_S^*$ be another leaf projecting onto $\hat{S}_{\hat{p}(c)}$. Then again $\hat{c}+c'\in Im(F)$ and 
$$(t_c)_*\hat{E}_c\cong (t_{c'})_*\hat{E}_{c'}\ \ \ \text{if and only if}\ \ (t_{c-c'})_*\hat{E}_c\cong \hat{E}_{c'}.$$
We have $c-c'\in Im(F)$ as well so $(t_{c-c'})_*\hat{E}_c\cong \hat{E}_{c'}$ if and only if
$$(t_{c-c'})_*(\phi_{L_0})_*(L_0^{-1}\otimes \hat{\cP}_{-c})\cong (\phi_{L_0})_*(L_0^{-1}\otimes \hat{\cP}_{-c'})$$
Indeed, analogously to Lemma \ref{Lemma: translation and tensor}, we have
$$(t_{c'-c})^*(\phi_{L_0})_*(L_0^{-1}\otimes \hat{\cP}_{-c})\cong (\phi_{L_0})_*(\hat{\cP}_{c-c'}\otimes L_0^{-1}\otimes \hat{\cP}_{-c} )\cong (\phi_{L_0})_*(L_0^{-1}\otimes \hat{\cP}_{-c'}).$$

\textit{Proof of \ref{3 absolute}.:} By Theorem \ref{pushforward U(1) degen} there exist projectively flat $U(d)$-bundles $\hat{L}_c$ on $\hat{S}_0$ such that on $\hat{S}$
\begin{align*}\hat{E}\cong \Big(q^*t_{-c}^*(\hat{L}_c\otimes U(d))\Big) \otimes \cP_b\cong \Big( q^*t_{-c}^*\hat{L}_c\otimes \cP_b\Big)\otimes q^*U(d).\end{align*}
\end{proof}
Comparing to (\ref{main equation for line bdle on affine subtorus}), it is clear that starting with a holomorphic line bundle on a complex torus and taking its Fourier-Mukai transform is the same as starting with the $U(1)$-bundle with connection associated to the constant Hermitian form and taking its T-dual.

 \section{On affine torus bundles with contractible base}\label{last chapter section contr base}
We would like to generalize Theorem \ref{tdual u(1) bundles on a torus} to physical branes in affine torus bundles. In this section, we work locally, that is on trivial affine torus bundles over a contractible base. In the first subsection, we extend the factor of automorphy description for $U(1)$-bundles with connections on a trivial family of tori. We prove a version of the Appel-Humbert theorem for these bundles as well. Finally, in the second subsection, we prove the analogue of Theorem \ref{tdual u(1) bundles on a torus}. 

\subsection{Factors of automorphy in family}
Let $U$ be a simply connected open subset of $\dR^k$ and let $T^n=V/\Gamma=\dR^n/
\dZ^n$ be the standard torus. In this section, we prove a version of the Appel-Humbert theorem for $U(1)$-bundles with connection and invariant curvature on $M=U\times T^n$.

We can identify $M$ with $V/\Gamma$, where $V=U\times \dR^n$ is the trivial vector bundle and $\Gamma=U\times \dZ^n$ is the trivial lattice. The lattice $\Gamma$ is again isomorphic to the fundamental group $\pi(M)$ and it acts on the universal cover $U\times V$ of $M$ fiberwise. We define factors of automorphy for $U(d)$-bundles as follows. 
\begin{definition}\label{factor of automoprhy family}
    A $U(d)$-factor of automorphy on $M$ is a smooth map
    $$a_E:U\times V\times \Gamma \ra U(n)$$
    satisfying $$a_E(y,v;\lambda+\mu)=a_E(y,v+\lambda;\mu)a_E(y,v;\lambda).$$ Two such factors of automorphy $a_E$ and $a_E'$ are equivalent is there exists a smooth map $\phi(y,v):U\times V\ra U(n)$ such that  \begin{align}\label{equivalence of factors family}a_E'(y,v;\lambda)=\phi(y,v+\lambda)a_E(y,v;\lambda)\phi(y,v)^{-1}.\end{align}
\end{definition}
We may define a $U(d)$-bundle on $M$ as the induced bundle
$$E\cong (U\times V)\times_\Gamma U(d)$$
under the action $\lambda.(y,v,t)=(y,v+\lambda,a_E(y,v;\lambda)t)$. A connection on $E$ can be specified by a one form  $A_E\in \Omega^1(U\times V,\gu(n))$ satisfying
\begin{align}\label{connection on bundle family} A_E(y,v+\lambda)=a_E(y,v;\lambda)A_E(y,v)a_E(y,v;\lambda)^{-1}-da_E(y,v;\lambda)\cdot a_E(y,v;\lambda).\end{align}
Two pairs $(a_E,A_E)$ and $(a_E',A_E')$ are equivalent if there exists a a smooth map $\phi(y,v):U\times V\ra U(n)$ such that  $a_E'$ and $a_E$ are related by $\phi$ as in  (\ref{equivalence of factors family}) and 
$$A_E'(y,v)=\phi(y,v)A_E(y,v)\phi(y,v)^{-1}-d\phi(y,v)\cdot \phi^{-1}(y,v).$$
Following the proof of Proposition \ref{factor of automoprhy and connection proof} we can prove the following.
\begin{theorem}
There is a one-to-one correspondence between equivalence classes of principal $U(d)$-bundles with connections $E\ra M$ and equivalence classes of pairs $(a_E,A_E)$ as in Definition \ref{factor of automoprhy family} and (\ref{connection on bundle family}).
\end{theorem}
Let us now focus on $U(1)$-bundles with connections or the associated Hermitian line bundles. Again, as in the ``absolute case" (when $U$ is a single point) the first Chern class of a pair $(a_L,A_L)$ is given by the cohomology class $$\frac{1}{2\pi i}dA_L \in H^2(M,\dZ)\cong H^2(T^n,\dZ).$$

\begin{lemma}\label{character for flat u(1) bundles family}
Let $L\ra M$ be a $U(1)$-bundle with a flat connection on $M$. Then, $L$ can be uniquely represented by a pair $(A_L,a_L)$ with $a_L$ constant on $U\times V$ and $A_L=0$.
\end{lemma}
\begin{proof}
The proof is the same as in the absolute case (Lemma \ref{character for flat U(1) bundles}).  Let $(A_L,a_L)$ be a representative of $L$. Since $dA_L=F=0$ there exists a function $f: U\times V\ra \dR$ such that
 $A_L=2\pi i df.$ Changing the representatives by $\phi(y,v)=exp(2\pi i f)$ sets $A_L=0$ invariant under translation by $\Gamma$. Using the compatibility between $A_L$ and $a_L$ we find that $da_L=0$ and hence constant on $U\times V$.  
\end{proof}
Lemma \ref{character for flat u(1) bundles family} proves that a global flat bundle on $M$ is again given by a character, that is by an element of $Hom(\Gamma,U(1))\cong\hat{T^n}$. Let us denote by $\Omega^2_{cl}(U)$ the space of closed real two-forms on $U$ and by $\Gamma(U,\hat{T^n})$ the space of smooth sections from $U$ to the torus dual to $T^n$. The following lemma describes connections with invariant curvature on the topologically trivial $U(1)$-bundle on $M$.

\begin{lemma}
The equivalence classes of connections on the topologically trivial $U(1)$-bundle with invariant curvature are in bijection with the space $\Omega^2_{cl}(U)\times \Gamma(U,\hat{T}^n)$.
\end{lemma}
\begin{proof}Let $\{y^1,...,y^k\}$ be global coordinates on $U$ and $\{q^1,...,q^n\}$ one-periodic coordinates on $T^n$ such that $\{\frac{\partial}{\partial q^\mu}\}$ is a frame for $\Gamma\subset V\cong T_xT^n$ for any $x\in T^n$. Then, $\{y^i,q^\mu\}$ provide global periodic coordinates for $M$.

Any invariant $F=F_\nabla$ can be written in coordinates $\{y^i,q^\mu\}$ as 
$$F=2\pi i\sum_{i,j=1}^k F^0_{ij}dy^i\wedge dy^j + 2\pi i \sum_{\substack{i=1,...,k\\ \mu=1,..,n}}G_{i\mu}dy^i\wedge dq^\mu+\pi i \sum_{\mu,\nu=1}^nH_{\mu\nu} dq^\mu \wedge dq^\nu,$$
where $F^0_{ij}$ and $G_{i\mu}$ are independent of the $q$ coordinates and $H_{\mu\nu}$ are constant integers. The cohomology class of $F$ depends only on $H_{\mu\nu}$ and $[F]=0$ if and only if $H_{\mu \nu}=0$. In particular, 
$$F=2\pi i \sum_{i,j=1}^k F^0_{ij}dy^i\wedge dy^j + 2\pi i \sum_{\substack{i=1,...,k\\ \mu=1,..,n}}G_{i\mu}dy^i\wedge dq^\mu.$$
Since $dF=0$ there exists $G_\mu:U\ra \dR$ for $\mu=1,...,n$ such that $G_{i\mu}=\frac{\partial G_\mu}{\partial y^i}$ and $A^0=2 \pi i\sum_{i=1}^k A^0_idy^i$ such that $dA^0=2\pi i \sum_{i,j=1}^k F^0_{ij}dy^i\wedge dy^j$. Note that $A^0$ is well-defined up to exact forms on $U$ and $G_\mu$ are well-defined up to constant functions.

Let $L'$ be the line bundle on $M$ corresponding to the pair $(A_{L'}=-A^0-2\pi i\sum_{\mu=1}^n G_\mu dq^\mu,a_{L'}\equiv 1)$. Then, $L\otimes L'$ is a flat line bundle on $M$ and therefore there exists a representative $(a_{L\otimes L'},A_{L\otimes L'})$ such that $A_{L\otimes L'}=A_L+A_{L'}=0$ and $a_{L\otimes L'}=a_L\otimes a_L'=a_L$ is constant on $U\times V$. That is, there exists a global character $\hat{t}\in \hat{T}^n$ such that $a_L(y,q;\lambda)=exp(2\pi i \bar{t}(\lambda))$, where $\bar{t}\in V^*$ is a lift of $\hat{t}\in V^*/\Gamma^\vee$ to $V^*$. 

Changing the representative again by $\phi(y,q)=exp(2\pi i \bar{t}\cdot q)$ we have 
$$(A_L+A_{L'})(y,q)=2\pi i \sum_{\mu=1}^n\bar{t}_\mu dq^\mu,\ \ \ a_{L}\otimes a_{L'}\equiv 1.$$
In particular, we find a representative for $L$
\begin{align}\label{representative of connection on trivial bundle}
A_L(y,q)&=A^0+2\pi i \sum_{\mu=1}^n (G_\mu+\bar{t}_\mu)dq^\mu,\ \ \ \ a_L(y,q,\lambda)=1,
\end{align}
where we only made choices for $A^0$ such that $dA^0=F^0$, for $G_\mu$ such that $dG_\mu=\sum G_{i\mu} dy^i$ and for $\bar{t}\in V^*$ such that it projects to $\hat{t}\in \hat{T}^n$.

If we choose a different $(A^0)'$ then  $A^0-(A^0)'$ is closed and hence exact on $U$. So if $A^0-(A^0)'=2\pi i df$ then the two representatives of $L$ are equivalent via $\phi(y,q)=\phi(y)=exp(2\pi i f(y)).$

To show (\ref{representative of connection on trivial bundle}) is well defined it remains to show that $G_\mu+\bar{t}_\mu$ is a well-defined family of characters associated to $L$. Suppose that we chose $G_\mu'$. Then, there exists a representative for $L$ with 
\begin{align*}
A'_L(y,q)&=A^0+2\pi i \sum_{\mu=1}^n (G'_\mu+\bar{t}_\mu')dq^\mu,\ \ \ \ a'_L(y,q,\lambda)=1.
\end{align*}
Then, $$(A_L-A'_L,a_L\otimes (a'_L)^{-1})=\Big(2\pi i \sum_{\mu=1}^n(G_\mu+\bar{t}_\mu-G_\mu'-\bar{t}_\mu')dq^\mu,1\Big)$$ is a representative of the trivial connection on the topologically trivial $U(1)$-bundle. In particular, there must exists a function $\phi: U\times V\ra U(1)$ such that
\begin{align*}
    1=\phi(y,q+\lambda)\phi(y,q)^{-1},\ \ \ \sum_{\mu=1}^n(G_\mu+\bar{t}_\mu-G_\mu'-\bar{t}_\mu')dq^\mu = dlog \phi(y,q).
\end{align*}
From the second equation, $\phi(y,\mu)=exp(2\pi i (G_\mu+\bar{t_\mu}-G_\mu'-\bar{t}_\mu')q^\mu)$ up to constant and it is invariant under translation by elements of $\Gamma$ if and only if $G+\bar{t}-G'-\bar{t}'\in\Gamma^\vee$.
\end{proof}
\begin{remark}
    Let $L$ be a $U(1)$-bundle corresponding to the pair $(F^0,G)$ where $F^0\in \Omega^2_{cl}(U)$ and $G:U\ra \hat{T}^n$. Then, if $dA^0=F$ we may represent $L$ by the pair
    \begin{align}\label{repr of connection on triv bundle 2} A_L=2\pi i A^0-2\pi i dG(q),\ \ \ a_L(y,q;\lambda)=exp(2\pi i G(\lambda)).\end{align}
    Indeed, this differs from (\ref{representative of connection on trivial bundle}) by $\phi(y,q)=exp(2\pi i G(q)). $
\end{remark}

Finally, we can prove a version of the Appel-Humbert Theorem \ref{appel-humbert for u(1)} for $U(1)$-bundles with connections on a trivial family of tori. 

Let $\cP_\dR^U(\Gamma)$ be the group of triples $(F^0,H,\chi)$ where $F^0\in \Omega^2_{cl}(U)$ is a closed two-form on $U$, $H\in Alt^2(\Gamma,\dZ)\cong \HH^2(M,\dZ)\cong \HH^2(T^n,\dZ)$  and $\chi: U\times \Gamma\ra U(1)$ is a semicharacter for $H$. That is, for any two sections $\lambda,\mu: U\ra \Gamma$ we have
$$\chi(\lambda+\mu)=\chi(\lambda)\chi(\mu)exp(i\pi H(\lambda,\mu))$$
where we evaluate $H$ fiberwise. The group multiplication is given by $$(F^0,H,\chi)\cdot((F^0)',H',\chi')=(F^0+(F^0)',H+H',\chi\cdot \chi').$$

Let us denote by $\cA$ the group of $U(1)$-bundles on $U\times T^n$ with connections whose curvature is invariant and by $\cA_0$ the connections with invariant curvature on the trivial $U(1)$-bundle.
\begin{theorem}[Appel-Humbert Theorem in family]\label{appel-humbert family}
There exists an isomorphism
$$\cP_\dR^U(\Gamma)\ra \cA$$
such that 
$$(F^0,H,\chi)\mapsto (A_L,a_L=\chi(\lambda)exp(i\pi H(q,\lambda)).$$
\end{theorem}
\begin{proof}
    To define the map $\cP_\dR^U(\Gamma)\ra \cA$ we have to find a well-defined $A_L$ corresponding to $(F^0,H,\chi)$. Firstly, the space of semi-characters  $\chi: U\times \Gamma \ra U(1)$ for some $H$ is a torsor over the space of characters $\Gamma(U,\hat{T}^n)$. Indeed, for any two semi-characters $\chi, \chi'$ for $H$ we have 
    $$    \chi'(\lambda+\mu)\cdot\chi(\lambda+\mu)^{-1}=\chi'(\lambda)\chi'(\mu)\chi(\lambda)^{-1}\chi(\mu)^{-1} exp(i\pi H(\mu,\lambda))exp(-i\pi H(\mu,\lambda)).$$
    Let $\chi^0_H$  be a constant semi-character for $H$ on $U\times \Gamma$, that is a semi-character pulled back from $\Gamma$. Then, $\chi$ can be written as
    $$\chi(\lambda)=\chi^0_H(\lambda)exp(2\pi i G( \lambda))$$
    for some $G\in \Gamma(U,\hat{T}^n)$. Let us denote by $\bar{G}$ a lift of $G$ to $\Gamma(U,\dR^n)$ and let $A^0\in \Omega^1(U)$ be such that $dA^0=F^0$.     Let $L$ be the $U(1)$-bundle with connection associated to the following representatives
    \begin{align*}a_L(y,q;\lambda)&=\chi(\lambda)exp(i\pi H(q,\mu)),\\
    A_L(y,q)&=2\pi i A^0-2\pi i \sum_{\substack{\mu=1,...,n\\ i=1,...,k}} \partial_i\bar{G}_\mu q^\mu dy^i +i\pi\sum_{\mu\nu} H_{\mu\nu}q^\mu dq^\nu\\
    &=2\pi i A^0-2\pi i d\bar{G}(q)+i\pi H(q,dq).\end{align*}
    We have to show that this map is well-defined.

    Again, if $(A^0)'$ is such that $d(A^0)'=dA^0=F^0$ then there exists a function $f: U\ra \dR$ such that $A^0-(A^0)'=2\pi i f$ and the two pairs $(a_L,A_L)$ and $(a_L,A_L')$ are equivalent via $\phi(y)=exp(2\pi i f)$.

    Suppose now that we choose a different constant semi-character $\chi^1_H$ for $H$. Then, there exists a unique $\hat{t}\in \hat{T}$ such that 
    $$\chi^0_H(\lambda)=\chi^1_H(\lambda)\cdot exp(2\pi i \hat{t}\cdot \lambda)$$
    In particular, 
    $$\chi(\lambda)=\chi^1_H(\lambda)exp(2\pi i (G+\hat{t})\cdot \lambda)$$
    and given any lift $\bar{t}$ of $\hat{t}$ to $\dR^n$ we have
    $\partial_i \bar{G}_\mu=\partial_i(\bar{G}_\mu+\bar{t}_\mu)$ so the representative $(a_L,A_L)$ of $L$ is unchanged.

    Finally, suppose that $\bar{G}'$ is a different lift of $G$ to $\Gamma(U,\dR^n)$. Then, $\bar{G}'-\bar{G}: U\ra \Gamma^\vee$ so there exists a constant $\hat{\lambda}\in \Gamma^\vee$ such that $\bar{G}'_\mu(y)=\bar{G}_\mu(y)+\hat{\lambda}_\mu$. In particular, the expression for $(A_L,a_L)$ remains unchanged.
    
    Finally, the map defined above provides a morphism of short exact sequences,
    \[
    \begin{tikzcd}
        1 \arrow{r} & \Omega^2_{cl}(U)\times \Gamma(U,\hat{T^n})
    \arrow{r}\arrow{d}{\cong} &\cP_\dR^U(\Gamma) \arrow{d} \arrow{r} & Alt^2(\Gamma,\dZ) \arrow{d}{\cong}\arrow{r} & 1\\
    1 \arrow{r} & \cA_0 \arrow{r} & \cA \arrow{r} & H^2(M,\dZ) \arrow{r} & 1 
     \end{tikzcd}
    \]
    and by the five lemma, we have isomorphism in the middle.
\end{proof}

\subsection{Main theorem on contractible base}
We are finally ready to generalize Theorem \ref{tdual u(1) bundles on a torus} to branes in a trivial affine torus bundle. We will assume that the support of the brane $S\subset M$ is an affine torus subbundle in the affine torus bundle $\pi:M\ra B$ such that $\pi(S)$ is trivial. Then, we can represent $U(1)$-bundles with connections on $S$ by the factors of automorphy described in Theorem \ref{appel-humbert family}.

Let $\pi:M\ra B$ be an affine torus bundle with a section, that is $M\cong V/\Gamma$. Let $\pi:\hat{M}\ra B$ be the dual torus bundle $\hat{M}\cong V^*/\Gamma^\vee$. Let $U\subset B$ open and $\{y,v\}$ local coordinates on $M|_U$. Let $\{y,\hat{v}\}$ be local coordinates on $\hat{M}|_U$ such that $\{v\}$ and $\{\hat{v}\}$ include dual frames of $\Gamma$ and $\Gamma^\vee$ on the fibers. Then we can define the Poincar\'e bundle on $M\times_U\hat{M}$ as the pullback of the Poincar\'e bundle from one of the fibers
\begin{equation}\label{u(1) poincare line bundle in family}
    a_\cP(y,v,\hat{v};\lambda,\hat{\lambda})=exp(i\pi (\hat{v}(\lambda)-\hat{\lambda}(v)+\hat{\lambda}(\lambda))),\ \ \ A_\cP(y,v, \hat{v})=i\pi (\hat{v}dv-vd\hat{v}).
\end{equation}
Since the Chern class of $M$ and $\hat{M}$ are trivial, the transition functions between fiberwise coordinates are linear transformations which preserve $a_\cP$ and $A_\cP$. Therefore, $\cP$ is a well-defined $U(1)$-bundle with a connection whose curvature is given locally by $2\pi i d\hat{v}\wedge dv$. Via the two-form $P=d\hat{v}\wedge dv$ the torus bundles $(M,H=0)$ and $(\hat{M},\hat{H}=0)$ are T-dual in the sense of generalized geometry.

Let $S\subset M$ be an affine torus subbundle such that $\pi(S)$ is contractible. Let $L\ra S$ be a $U(1)$-bundle with a connection such that the curvature $2\pi i F$ of the connection is invariant. Then once again, $\cL=(S,F)$ is a locally T-dualizable generalized brane and we can construct the diagram:
\begin{equation*}
\begin{tikzcd}
& Z \arrow{dl}[swap]{p_Z} \arrow{dr}{\hat{p}_Z} \arrow[hook]{rrr}{i_Z} &  & & M\times_{\pi(S)}\hat{M} \arrow{dl}[swap]{p} \arrow{dr}{\hat{p}} \\
S\arrow[hook, bend right =30]{rrr}{i_S} & & \hat{S} \arrow[hook, bend right =20]{rrr}{i_{\hat{S}}} & \pi^{-1}(\pi(S)) & & \hat{\pi}^{-1}(\pi(S))
\end{tikzcd}
\end{equation*}
With this notation, we have the following theorem.
\begin{theorem}\label{tdual U(1) bundles trivial base}
Let $S\subset M$ be an affine torus subbundle such that $\pi(S)$ is contractible. Let $L\ra S$ be a $U(1)$-bundle with connection as such that the curvature $2\pi i F$ is invariant. Then, the following hold.
\begin{enumerate}
    \item\label{1 contr base} There exists a unique generalized T-dual $(\hat{S},\hat{F})$ of $(S,F)$ such that for any leaf $Z$ of $\Delta$ projecting onto $\hat{S}$ via  $\hat{p}$ the $U(1)$-bundle with connection
$$ \hat{L}_Z:= p^*L\otimes \cP|_Z$$
is trivial when restricted to the fibers of $\hat{p}_Z:Z\ra \hat{S}.$ 
\item \label{2 contr base} For any leaf $Z$ as in \ref{1 contr base}., the  pushforward 
$$\hat{E}:=(\hat{p}_Z)_*L_Z$$
is a projectively flat $U(d^2)$-bundle with connection independent of $Z$, and the curvature of the connection is given by $\hat{F} \cdot Id \in \Omega^2(\hat{S},\gu(d^2))$.
\item \label{3 contr base} Finally, there exists a projectively flat $U(d)$-bundle $\hat{L}\ra \hat{S}$ such that
$$\hat{E} \cong \hat{L}\otimes U(d)$$
in the sense of Theorem \ref{pushforward U(1) degen}. The curvature of the connection on $\hat{L}$ is  $\hat{F}\cdot Id\in \Omega^2(\hat{S},\gu(d))$.
\end{enumerate}
We say that $(\hat{L},\hat{S})$ is the T-dual of $(S,L)$.
\end{theorem}
\begin{proof} We begin analogously to the proof of Theorem \ref{tdual u(1) bundles on a torus}. Let us denote the base by $U$. The affine torus subbundle $S\subset M$ is a translate of $S_0\cong V_S/\Gamma_S$ by a section $b\in \Gamma(U,M)$ that is, 
$$S=t_b(S_0).$$
By Theorem \ref{appel-humbert family} we can represent $L_0=t_b^*L\ra S_0$ as a triple $(F^0,H,\chi)$ where $F^0\in \Omega^2_{cl}(U)$, $H\in \HH^0(U,\wedge^2\Gamma_S^\vee)$ and $\chi: U\times \Gamma_S\ra U(1)$ is a semi-character for $H$.

As in Theorem \ref{appel-humbert family} we may choose a decomposition $\Gamma_S/(Ker(H)\cap\Gamma_S)=\Gamma_1+\Gamma_2$ for $H$. Then the corresponding constant semicharacter is
$$\chi_0(\lambda)=exp(i\pi H(r(\lambda)_1,r(\lambda)_2)),\ \ \ \lambda\in \Gamma(U,\Gamma_S)$$
where $r:\Gamma_S\ra \Gamma_S/(Ker(H)\cap\Gamma_S)$ is the projection and $r(\lambda)=r(\lambda)_1+r(\lambda)_2$ with respect to the decomposition. A factor of automorphy and connection 1-form representing the triple $(F^0,H,\chi)$ are given by
\begin{align*}
    a_{L_0}:&\ \  U\times V_S\times \Gamma_S\ \ \ra\ \  U(1)\\
    a_{L_0}(y,v;\lambda)&=\chi_0(\lambda)exp(i\pi H(v,\lambda)+2\pi i G( \lambda))\\
    A_{L_0}(y,v)&=2\pi i A^0-2\pi i dG\cdot v+ i\pi H(v,dv),
\end{align*}
where $dG: TU \ra V_S^*$ means the derivative of $G:U\ra V^*_S/\Gamma_S^\vee$ and $\cdot$ denotes contraction between elements of $V_S$ and $V^*_S$.

 By Lemma \ref{local leaves coordfree} the leaves $Z$ of the distribution $\Delta$ in $S\times_{U}\hat{M}$ are affine torus subbundles modelled on $Z_0=V_Z/\Gamma_Z\subset S_0\times \hat{M}$. By Proposition \ref{space of leaves of Delta} the space of leaves is parametrized by flat sections of $V_S^*/\Gamma_S^\vee$. For $c\in V_S^*/\Gamma_S^\vee$ let
 $$Z_c:=t_{(b,-G+c)}Z_0\subset S\times \hat{M},$$
 where we choose a lift of $-G+c: U\ra V_S^*$ to a section $-G+c: U\ra V_M^*$, which we denote by the same letters. Let
$$L_Z=(p^*L\otimes \cP)|_Z \cong (t_{(b,-G+c)})_*L_c$$
for $L_c\ra Z_0$ given by
\begin{align}\label{L_c family} L_c:=\Big(t^*_{(b,-G+c)}(p^*L\otimes \cP)\Big)\Big|_{Z_0}=\Big(p_0^* L_0\otimes t_{(b,-G+c)}^*\cP\Big)\Big|_{Z_0}.\end{align}
By Lemma \ref{local images of Z} the image of any leaf $Z\subset S\times \hat{M}$ is an affine torus subbundle of $\hat{M}$ modelled on $\hat{S}_0=V_{\hat{S}}/\Gamma_{\hat{S}}$. The homomorphism $p_0:Z_0\ra S_0$ and $\hat{p}_0:Z_0\ra \hat{S}_0$ are induced by the restrictions of the projections $p_0:V+ V^*\ra V$ and $\hat{p}_0:V+V^*\ra V^*$. They fit into the following morphisms of short exact sequences.
 \begin{equation}\label{p_0 and hat(p)_0}
\begin{tikzcd}
0 \arrow{r} & Ann(V_S) \arrow{r} \arrow{d}{\cong} & V_Z \arrow{r}{p_0} \arrow{d}{\hat{p}_0} & V_S \arrow{r} \arrow{d}{-H} & 0 \\
0 \arrow{r} & Ann(V_S) \arrow{r} & V_{\hat{S}} \arrow{r}{q} & H(V_S) \arrow{r} & 0\\
0 \arrow{r} & Ann(\Gamma_S) \arrow{r}\arrow{d}{\cong} & \Gamma_Z \arrow{r}{p_0} \arrow{d}{\hat{p}_0} & \Gamma_S \arrow{r} \arrow{d}{-H} & 0 \\
0 \arrow{r} & Ann(\Gamma_S) \arrow{r} & \Gamma_{\hat{S}} \arrow{r}{q} & H(V_S)\cap \Gamma_S^\vee \arrow{r} & 0.
\end{tikzcd}
\end{equation}
Then a representative of $L_c$ is given as follows.
\begin{align*}
    a_{L_c}&(y,v;\lambda)=\\
    =&\chi_0(p_0(\lambda))exp(i\pi H(p_0(v),p_0(\lambda))+2\pi i G(p_0(\lambda)))\times \\
    &\times exp\Big( i\pi \hat{p}_0(\lambda)(p_0(\lambda))+i \pi (\hat{p}_0(v)-G+c)(p_0(\lambda))-i\pi \hat{p}_0(\lambda)(p_0(v)+b) \Big)\\
    =&\chi_0(p_0(\lambda))exp(i\pi H(p_0(v),p_0(\lambda))+2\pi i G(p_0(\lambda)))\times \\
    &\times exp\Big(  -i \pi H(p_0(\lambda),p_0(\lambda))-i \pi H(p_0(v),p_0(\lambda))+i \pi (-G+c)(p_0(\lambda))\Big) \times \\
    &\times exp\Big(i\pi H(p_0(\lambda),p_0(v))-i \pi \hat{p}_0(\lambda)(b) \Big)\\
    =&\chi_0(p_0(\lambda))exp(-i\pi H(p_0(v),p_0(\lambda))+i \pi (G+c)(p_0(\lambda))-i\pi \hat{p}_0(\lambda)(b)),\\
    A_{L_c}&(y,v)=\\
    =&2\pi i A^0-2\pi i dG\cdot p_0(v) +i \pi H(v,dv) + i \pi (\hat{p}_0(v)-G+c)\cdot d(p_0(v)+b)-\\
    &- i \pi (p_0(v)+b) \cdot d(\hat{p}_0(v)-G+c)\\
    =&2\pi i A^0-2\pi i dG\cdot p_0(v) + i \pi H(p_0(v),dp_0(v)) - H(p_0(v),dp_0(v)) +i\pi \hat{p}_0(v) \cdot db+\\
    &+i\pi(- G+c) \cdot dp_0(v)+ i \pi (-G+c) \cdot db + H(dp_0(v),p_0(v))-i\pi b \cdot d\hat{p}_0(v)+\\
    &+ i \pi p_0(v)\cdot dG + i \pi b \cdot dG\\
    =&2\pi i A^0 +i\pi(-G+c)\cdot db + i \pi b \cdot dG + i\pi \hat{p}_0(v)\cdot db - i\pi b \cdot d\hat{p}_0(v) - \\
    &-i \pi H(p_0(v),dp_0(v))+i \pi (-G+c)\cdot dp_0(v) -i\pi p_0(v) \cdot dG.
\end{align*}
Changing the representative by $\phi(v)=exp(-i\pi (G+c)\cdot p_0(v)- i \pi \hat{p}_0(v)\cdot b +i \pi (-G+c) \cdot b)$ we find
\begin{align}\label{Lc}
a_{L_c}(y,v;\lambda)=&\chi_0(p_0(\lambda))exp(-i\pi H(p_0(v),p_0(\lambda))-2\pi i\hat{p}_0(\lambda)(b))\\
A_{L_c}(y,v)=&2\pi i (A^0 + b \cdot dG) + 2\pi i \hat{p}_0(v)\cdot db+ 2\pi i c\cdot dp_0(v) - i \pi H(p_0(v),dp_0(v)).\nonumber
\end{align}
 The maps $p_{Z_c}:Z_c\ra \hat{S}_{c}$ can then be decomposed into homomorphisms and translations. We have the commutative square
 \begin{equation}\label{Zc as translate of Z0}
     \begin{tikzcd}
         Z_0 \arrow{rr}{t_{(b,-G+c)}} \arrow{d}{\hat{p}_0}  & & Z_c \arrow{d}{p_{Z_c}}\\
         \hat{S}_0 \arrow{rr}{t_{-G+c}} & & \hat{S}_c
     \end{tikzcd}.
 \end{equation}

\textit{Proof of \ref{1 contr base}.:} Same as the proof of Theorem \ref{tdual u(1) bundles on a torus} Part \ref{1 absolute}. The kernel of $\hat{p}_0$ is the disjoint union of affine tori. The connected component of the identity is $$Ker(\hat{p}_0)_0=Ker(H)/(Ker(H)\cap \Gamma_S),$$
and the other components are translates of it by elements $\{\lambda_1,...,\lambda_{d^2}\}$ such that \linebreak $\{H(\lambda_1),...,H(\lambda_{d^2})\}$ is a full set of representatives of $(H(V_S)\cap\Gamma_S^\vee):H(\Gamma_S)$. The $U(1)$-bundle $L_c$ is invariant under translations by $\lambda_i$ so it is trivial on the fibers if it is trivial when restricted to $K_0=Ker(\hat{p}_0)_0$. The restriction $L_c|_{K_0}$ is given by the image of $c\in V_S^*$ in $coker(H)$, that is $L_c|_{K_0}$ is trivial if and only if $c\in H(V_S)$. By Proposition \ref{space of leaves mapping to a T-dual} this defines a unique image $\hat{S}$.

\textit{Proof of \ref{2 contr base}.:} For $c\in H(V_S)$ let us define
$$\hat{E}_c=(\hat{p}_0)_*L_c.$$
It is a projectively flat $U(d)^2$-bundle on $\hat{S}_0$. We could follow the proof of Theorem \ref{tdual u(1) bundles on a torus} Part \ref{2 absolute}. here as well but instead we write out the full calculation to avoid proving analogues of the necessary lemmas.

We use the decomposition $\Gamma_S/(Ker(H)\cap \Gamma_S)=\Gamma_1+\Gamma_2$ to decompose $H(V_S)=V_1+V_2$ with $V_i=H(\Gamma_i)\otimes \dR$ and we take representatives $\{\delta_i\}_{i=1,...,d}$ and $\{\epsilon_j\}_{j=1,...,d}$ of $(V_i\cap \Gamma_S^\vee):H(\Gamma_i)$ for $i=1,2$ respectively. Let $E_i$ and $D_j$ for $i,j=1,...,d$ be lifts of the representatives to $\Gamma_{\hat{S}}$ so we have.
\begin{align*}
    a_{\hat{E}_c}(y,\hat{v};\lambda)=&\Big( \chi_0(-H^{-1}(q(\Lambda^{ij}_\lambda)))\times \\
    &\ \ \times exp(i \pi H^{-1}(q(\hat{v})+\delta_{\lambda(i)}+\epsilon_{\lambda(j)},q(\Lambda^{ij}_\lambda)) -2\pi i \Lambda^{ij}_{\lambda}(b)) \delta^k_{\lambda(i)}\delta^l_{\lambda(j)}\Big)^{ij}_{kl},\\
    A_{\hat{E}_c}(y,\hat{v})=&\Big(2\pi i (A^0+b\cdot dG)+ 2\pi i (\hat{v}+D_i+E_j)\cdot db + \\
    &\ \ +2\pi i H^{-1}(c , dq(\hat{v}))+ i \pi H^{-1}(q(\hat{v})+\delta_i+\epsilon_j,dq(v)) \Big)^{ij}_{ij}.
\end{align*}
Changing the representatives by the diagonal matrix valued function $$\phi(\hat{v})^{ij}_{ij}=exp\Big(2\pi i (D_i+E_j)\cdot b-i \pi H^{-1}(q(\hat{v}),\delta_i+\epsilon_j)\Big),$$ we get
\small
\begin{align}\label{hatEc}
    a_{\hat{E}_c}(y,\hat{v};\lambda)&=exp\Big(i \pi H^{-1}(q(\hat{v}),q(\lambda))-2\pi i \lambda(b)\Big)\cdot U_{\hat{E}_c}(q(\lambda)),\\
  U_{\hat{E}_c}(q(\lambda))&=\Big(\chi_0(-H^{-1}(q(\Lambda^{ij}_\lambda)))exp(i\pi H^{-1}(q(\lambda)+\delta_i+\epsilon_j,q(\lambda)-\delta_{\lambda(i)}-\epsilon_{\lambda(j)})) \delta^k_{\lambda(i)}\delta^l_{\lambda(j)}\Big)^{ij}_{kl},\nonumber\\
    A_{\hat{E}_c}(y,\hat{v})&=\Big(2\pi i (A^0+b\cdot dG)+ 2\pi i \hat{v}\cdot db + 2\pi i H^{-1}(c , dq(\hat{v}))+ i \pi H^{-1}(q(\hat{v}),dq(v))\Big) \cdot Id.\nonumber
\end{align}
\normalsize
It remains to show that $\hat{E}_c$ is independent of the choices we made. Namely, of $c\in H(V_S)$, of the lift of $G:U\ra V_S^*$ to $G:U\ra V_M^*$ and of the choice of $b:U\ra M$.

\textbf{Independence of $c\in H(V_S)$.}
We have
$$\hat{p}_*L_c\cong (t_{-G+c})_*\hat{E}_c$$
Let $c'\in H(V_S)\subset V_S^*$ indicate another leaf and its lift $c'\in V_M^*$. Then,
$$\hat{p}_*L_c\cong \hat{p}_*\hat{L}_{c'}\ \ \ \text{if and only if}\ \ \ t_{c'-c}^*\hat{E}_c\cong \hat{E}_{c'}.$$
Since $c$ only enters the description of $\hat{E}_c$ as an element of $H(V_S)\subset V_S^*$ it is clear that $\hat{E}_c$ does not depend on the choice of lift of $c$ along $V_M^*\ra V_S^*$.
\begin{align*}
    a_{t_{c'-c}^*\hat{E}_c}(y,\hat{v};\lambda)&=a_{\hat{E}_c}(y,\hat{v}+c'-c,\lambda)\\
    &=exp\Big(i \pi H^{-1}(q(\hat{v})+c'-c,q(\lambda))-2\pi i \lambda(b)\Big)\cdot U_{\hat{E}_c}(q(\lambda)), \\
\end{align*}
\begin{align*}
    A_{t_{c'-c}^*\hat{E}_c}(y,\hat{v})&=\Big(2\pi i (A^0+b\cdot dG)+ 2\pi i (\hat{v}+c'-c)\cdot db + 2\pi i H^{-1}(c , dq(\hat{v}))+\\
    &\ \ \ \ \ + i \pi H^{-1}(q(\hat{v})+c'-c,dq(v))\Big) \cdot Id.\\
    &=\Big(2\pi i (A^0+b\cdot dG)+ 2\pi i (\hat{v}+c'-c)\cdot db + i\pi  H^{-1}(c+c' , dq(\hat{v}))+\\
    &\ \ \ \ + i \pi H^{-1}(q(\hat{v}),dq(v))\Big) \cdot Id.
\end{align*}
\normalsize
Changing the representatives by $\phi(y,\hat{v})=exp\Big(2\pi i (c'-c)\cdot b - i\pi H^{-1}(c'-c,q(\hat{v}))\Big)$ yields the desired isomorphism.

\textbf{Independence of lifting $G$.} 
From the description of $L$ the section $G\in \Gamma(U,V_S^*)$ is well defined up to constant translation. In the construction of $\hat{E}$ we chose a lift of it to a section $G\in \Gamma(U,V_M^*)$. 

Suppose now that we chose a different lift $G'\in \Gamma(U,V_M^*)$.  Let 
$$\hat{E}_c'\cong (\hat{p}_0)_*\Big( (p_0^*L_0'\otimes t^*_{(b',-G'+c)}\cP)|_{Z_0})\Big),$$
so $(t_{-G+c})_*\hat{E}_c\cong (t_{-G'+c})_*\hat{E}_c'$ if and only if 
$$\hat{E}_c'\cong t_{G-G'}^*\hat{E}_c$$
Since $G$ only enters the description of $\hat{E}_c$ as $dG$ we readily see that it is independent of constant translations in $V_S^*$ and we may suppose that $G-G'\in \Gamma(U,Ann(V_S))$. Inspecting (\ref{hatEc}) and changing the representatives of $t_{G-G'}^*\hat{E}_c$ by  $\phi(\hat{v})=exp(2\pi i (G-G')\cdot b)$ yields the desired isomorphism.

\textbf{Independence of $b$.} Let $b'\in \Gamma(U,M)$ such that $S=t_bS_0=t_{b'}S_0$, then we have $b-b'\in \Gamma(U,S_0)$ and since 
$L\cong t_{-b}^*L_0\cong t_{-b'}^*L_0'$
we have $$L_0'\cong t_{b'-b}^*L_0.$$
In particular,
\begin{align*}
    a_{L_0'}(y,v;\lambda)&=\chi_0(\lambda)exp(i\pi H(v+b'-b,\lambda))exp(2\pi i G(\lambda)),\\
    A_{L_0'}(y,v)&=2\pi i A^0-2\pi i\partial_UG\cdot (v+b'-b)+i\pi H(v+b'-b,d(v+b'-b)),\\
    &=2\pi i \Big(A^0- dG\cdot (b'-b)+ \frac{1}{2} H(b'-b,d(b'-b))\Big)-2\pi i dG\cdot v+\\
    &\ \ +i\pi H(v,d(b-b'))+i\pi H(b'-b,dv)+ i \pi H(v,dv).
\end{align*}
Changing the representatives by $\phi(v)=exp(i\pi H(b'-b,v))$ we find
that $L_0'$ corresponds to the triple $((F^0)',H,\chi')$ with
\begin{align*}
\chi'(\lambda)&=\chi_0(\lambda) exp(2\pi i (G+H(b'-b))(\lambda)),\\
(F^0)'&=F^0+2\pi i dG \cdot d(b'-b) + i \pi H(d(b'-b)\wedge d(b'-b)),
\end{align*}
with $(A^0)'=A^0-(b'-b)\cdot dG + \frac{1}{2}H((b'-b),d(b'-b))$.

Let $\hat{E}_c'=(\hat{p}_0)_*\Big(p_0^*L_0'\otimes t_{(b',-G'+c)}^*\cP)|_{Z_0}\Big)$ so $(t_{-G'+c})_*\hat{E}_c'\cong (t_{-G+c})_*\hat{E}_c$ if and only if 
$$\hat{E}_c'\cong t_{G-G'}^*\hat{E}_c.$$
Since $G-G'=H(b-b')$ we have
\begin{align*}
    a_{t_{G-G'}^*\hat{E}_c}(y,\hat{v};\lambda)&=exp(i\pi H^{-1}(q(\hat{v})+H(b-b'),q(\lambda))-2\pi i  \lambda(b))\cdot U_{\hat{E}_c}(q(\lambda))\\
 &= exp(-i\pi q(\lambda)(b-b'))\cdot a_{\hat{E}_c'}(y,\hat{v};\lambda),
\end{align*}
\begin{align*}
    &A_{t_{G-G'}^*\hat{E}_c}(y,\hat{v})=\\
    &=\Big(2\pi i( A^0+ b\cdot dG)+ 2\pi i (\hat{v} +H(b-b'))\cdot db+2\pi i H^{-1}(c,dq(\hat{v})+dH(b-b'))+\\
    & \hspace{1cm} +i \pi H^{-1}(q(\hat{v})+H(b-b'),dq(\hat{v})+dH(b-b'))\Big) \cdot Id.\\
    &=\Big(2\pi i( A^0+ b\cdot dG)+ 2\pi i \hat{v}\cdot db' - 2\pi i \hat{v} \cdot(b'-b)+2\pi i H(b-b'))\cdot db+\\
    &\hspace{1cm} +2\pi i H^{-1}(c,dq(\hat{v}))-2\pi id(b-b')\cdot c
     +i \pi H^{-1}(q(\hat{v}),d\hat{q})+i \pi (b-b')\cdot dq(\hat{v})-\\
     &\hspace{1cm} -i \pi d(b-b')\cdot q(\hat{v})+ i \pi H(d(b-b'),b-b')\Big) \cdot Id,\\
     &\text{since $q(\hat{v})\cdot d(b'-b)=\hat{v}\cdot d(b'-b)$ we have}\\
     &=\Big(2\pi i (A^0+b\cdot dG +\frac{1}{2}H(b'-b,d(b'-b))-H(b'-b)\cdot db) - \\
     &\hspace{1cm} -2\pi i c\cdot (b-b')-\pi i q(\hat{v}) \cdot d(b'-b)-\pi i (b'-b)\cdot dq(\hat{v}) + 2\pi i \hat{v}\cdot db'+\\
     &\hspace{1cm} +2\pi i H^{-1}(c,dq(\hat{v}))+i\pi H^{-1}(q(\hat{v}),dq(\hat{v}))\Big) \cdot Id.
\end{align*}
Finally, changing the representatives by
$$\phi(y,\hat{v})=exp\Big(i\pi q(\hat{v})\cdot (b-b') +2\pi i H(b'-b)\cdot b' -2\pi i c\cdot (b'-b)\Big)$$
yields the desired isomorphism. Note also that
$$(A^0)'+b'\cdot dG'=A^0-(b'-b)\cdot dG+\frac{1}{2}H(b'-b,d(b'-b))+b'\cdot dG+b'\cdot dH(b'-b).$$

\textit{Proof of \ref{3 contr base}.:}
The semi-representation $U_{\hat{E}_c}$ is the pullback of a constant semi-representation from $H(V_S)\cap \Gamma_S^\vee$. Moreover,
$$\chi_0(-H^{-1}(q(\Lambda^{ij}_\lambda)))=\chi_0(H^{-1}(q(\Lambda^{ij}_\lambda)))^{-1}.$$
In particular, $U_{\hat{E}_c}$ is of the form $U_{(\phi_{L})_*L^{-1}}$ for a suitable $U(1)$-bundle on $S_0$. That is we can readily apply Theorem \ref{pushforward U(1) degen} to show
$$U_{\hat{E}_c}\cong \rho\circ U_{\hat{L}_c}$$
where $\rho:U(d)\ra U(d^2)$ is the diagonal embedding (\ref{embedding u(d) to u(d2)}) and
$$U_{\hat{L}_c}:U\times \Gamma_{\hat{S}}\ra U(d)$$
is a constant semi-representation  for $H^{-1}$ pulled back from $H(V_S)\cap \Gamma_S^\vee$.

\end{proof}
\begin{remark}
    As in the absolute case (Part \ref{2 absolute}. Theorem \ref{tdual u(1) bundles on a torus}),
    $$L_c\cong \hat{p}_0^*\hat{\cP}_{b} \otimes (p_0)^*L_c^S $$
    where $\hat{\cP}_b=\cP|_{M\times_U \{b\}}$. Moreover,
    $$(\hat{p}_0)_*L_c=\hat{\cP}_{b} \otimes (\hat{p}_0)_*(p_0)^*L_c^S$$
    and 
    $$(\hat{p}_0)_*(p_0)^*L_c^S=q^*(-\phi_{H})_*L_c^S= q^*(\phi_{H})_*(L_0^{-1}\otimes \cP_{-G+c})$$
    where $\phi_{H}$ is the fiberwise isogeny defined by $H:V_S\ra V_S^*$ and $\cP_{-G+c}=\cP|_{M\times_U \{-G+c\}}$. In particular, we can write (cf. (\ref{hat(E) FM type description}))
    \begin{align}\label{hat(E) as FM transform}\hat{E}=q^*t_{-G+c}^*(\phi_{H})_*\Big(L_0^{-1}\otimes \cP|_{-G+c}\Big)\otimes \hat{\cP}_{b}.\end{align}
\end{remark}

\begin{example}\label{BAA-BBB space filling tdual bundles} Let $\pi: M\ra B$ be an algebraic integrable system that admits a Lagrangian section over a contractible base. Let $\hat{M}\ra B$ be the affine torus bundle with zero Chern class and monodromy local system $\Gamma_M^\vee$. Then, by Theorem \ref{semiflat tdual} $\hat{M}$ is also an algebraic integrable system and $M$ and $\hat{M}$ carry T-dual semi-flat hyperk\"ahler structures. 

In Example \ref{BAA-BBB space filling tduals} we described a T-dual pair of space-filling generalized branes corresponding to the semi-flat structures. These pairs can be enhanced to T-dual pairs of bundles with connections.

Let us denote by $\omega_\dI$ one of the K\"ahler forms of the semi-flat hyperk\"ahler structure on $M$. In terms of the special K\"ahler structure $(g,I,\omega, \nabla)$ on $B$ it is given by
$$\omega_\dI=\begin{pmatrix}\omega & 0 \\ 0 & \omega^{-1}\end{pmatrix},$$
and $(M,\omega_\dI)$ is a space-filling $BAA$-brane. There exist $U(1)$-bundles with connections $L^A$ whose curvature is $2\pi i \omega_\dI$. These can be represented as
\begin{equation}
    \begin{aligned}
    a^A(x,v;\lambda)&=\chi_0(\lambda)exp(i\pi \omega^{-1}(v,\lambda)+2\pi i c(\lambda))\\
    A^A(x,v)&=i\pi \omega^{-1}(v,dv)+i\pi \omega(x,dx),
    \end{aligned}
\end{equation}
where $c: B\ra V_M^*/\Gamma_M^\vee$ is a constant character. The T-dual of this brane depends on the type of the polarization. 

If the polarization which defines $\omega_\dI$ is principal, then $M\cong \hat{M}$ and the T-dual of $L^A$ is again a $U(1)$-bundle $\hat{L}^B$ represented by
\begin{equation}
    \begin{aligned}
        a^B(x,\hat{v};\hat{\lambda})&=\chi_0(-\omega(\hat{\lambda}))exp(i\pi \omega(\hat{v},\hat{\lambda}))\\
        A^B(x,\hat{v})&=i \pi \omega(\hat{v},d\hat{v})+2\pi i \omega(c,d\hat{v})+i\pi \omega(x,dx).
    \end{aligned}
\end{equation}
Pulling back via the isomorphism $\phi:M\ra \hat{M}$ we find
\begin{align*}
    (\phi^* a^B)(x,v;\lambda)&=a^B(x,\omega^{-1}(v);\omega^{-1}(\lambda))\\
    &=\chi_0(-\omega(\omega^{-1}(\lambda)))exp(i\pi \omega(\omega^{-1}(v),\omega^{-1}(\lambda)))\\
    &=\chi_0(-\lambda)exp(i\pi \omega^{-1}(\lambda,v))\\
    &=\chi_0(-\lambda)exp(-i\pi \omega^{-1}(v,\lambda))\\
    (\phi^*A^B)(x,v)&=A^B(x,\omega^{-1}(v)) \\
    &=i \pi \omega(\omega^{-1}(v),d\omega^{-1}(v))+2\pi i \omega(c,d\omega^{-1}(v))+i\pi \omega(x,dx)\\
    &=i \pi \omega^{-1}(dv,v)+2\pi i c\cdot dv +i\pi \omega(x,dx)\\
    &=-i \pi \omega^{-1}(v,dv)+2\pi i c\cdot v +2\pi i  \omega(x,dx).
\end{align*}
That is, via the isomorphism induced by the polarization the T-dual of the brane $(M,L^A)$ restricts to each fiber as the $U(1)$-bundle dual to $L^A$. This is a well-known property of the Fourier-Mukai transform as well.

If the polarization is not principal, but of type $d$ the T-dual of $(M,L^A)$ is a projectively flat $U(d)$-bundle $\hat{L}^B$ represented as follows.
\begin{equation*}
    \begin{aligned}
        a^B(x,\hat{v};\hat{\lambda})&=exp(i\pi \omega(\hat{v},\hat{\lambda}))\Big(\chi_0(-\omega(\hat{\lambda}_1+\Lambda^i_{\hat{\lambda}}))exp(i\pi \omega(\hat{\lambda}+\lambda_i,\hat{\lambda}-\lambda_{\hat{\lambda}(i)})) \delta^j_{\hat{\lambda}(i)}\Big)\\
        A^B(x,\hat{v})&=(i \pi \omega(\hat{v},d\hat{v})+2\pi i \omega(c,d\hat{v})+i\pi \omega(x,dx))\cdot Id_{d\times d}.
    \end{aligned}
\end{equation*}
Using the specific form of $\chi_0$ we can write
\begin{align*}\chi_0(-\omega(\hat{\lambda}_1+\Lambda^i_{\hat{\lambda}}))&=exp\Big(i\pi \omega^{-1}(-\omega(\hat{\lambda}_1+\Lambda^i_{\hat{\lambda}})^1,-\omega(\hat{\lambda}_1+\Lambda^i_{\hat{\lambda}})^2)\Big)\\
&=exp\Big(i\pi \omega^{-1}(\omega(\hat{\lambda}_1),\omega(\Lambda^i_{\hat{\lambda}}))\Big)\\
&=exp\Big(i\pi \omega(\Lambda^i_{\hat{\lambda}},\hat{\lambda}_1)\Big)\\
&=exp\Big(i\pi \omega (\hat{\lambda}_1,\hat{\lambda}_2+\lambda_i-\lambda_{\hat{\lambda}(i)})\Big),
\end{align*}
where we use that $\omega(\Lambda^i_{\hat{\lambda}},\hat{\lambda}_1)\in \dZ$. Then, we have
\begin{align*}
    \omega& (\hat{\lambda}_1,\hat{\lambda}_2+\lambda_i-\lambda_{\hat{\lambda}(i)})+  \omega(\hat{\lambda}+\lambda_i,\hat{\lambda}-\lambda_{\hat{\lambda}(i)})=\\
    &=\omega(\hat{\lambda}_1,\hat{\lambda}_2)+\omega(\hat{\lambda}_1,\lambda_i)-\omega(\hat{\lambda}_1,\lambda_{\hat{\lambda}(i)})+\omega(\hat{\lambda}_1,\hat{\lambda}_2)-\omega(\hat{\lambda}_1,\lambda_{\hat{\lambda(i)}})+\omega(\hat{\lambda}_2,\hat{\lambda}_1)+\omega(\lambda_i,\hat{\lambda}_1)\\
    &=\omega(\hat{\lambda}_1,\hat{\lambda}_2)-2\omega(\hat{\lambda}_1,\lambda_{\hat{\lambda}(i)})
\end{align*}
and
\begin{equation}
     \begin{aligned}
        a^B(x,\hat{v};\hat{\lambda})&=exp(i\pi \omega(\hat{v},\hat{\lambda}))\Big(exp(i\pi \omega(\hat{\lambda}_1,\hat{\lambda}_2)+2\pi i \omega(\lambda_{\hat{\lambda}(i)},\hat{\lambda}_1)) \delta^j_{\hat{\lambda}(i)}\Big)\\
        A^B(x,\hat{v})&=\Big(i \pi \omega(\hat{v},d\hat{v})+2\pi i \omega(c,d\hat{v})+i\pi \omega(x,dx)\Big)\cdot Id_{d\times d}.
    \end{aligned}
\end{equation}
The underlying generalized brane in both cases is the $BBB$-brane  $(\hat{M},\hat{F})$ with 
$$\hat{F}=\begin{pmatrix}\omega & 0 \\ 0 & \omega\end{pmatrix}.$$
\end{example}
\begin{remark}
    Note that if $M$ and $\hat{M}$ as above admit a section but $B$ is not simply connected, then we may define $(M,L^A)$ and $(\hat{M},\hat{L}^B)$ locally using affine coordinates on the fibers. Since the local representatives are invariant under the change of coordinates (see Lemma \ref{global H semichar}), the local bundles glue together to a global T-dual pair.
\end{remark}

\section{On affine torus bundles over general base}\label{last chapter section general base}
We would like to generalize Theorem \ref{tdual U(1) bundles trivial base} to $U(1)$-bundles supported on affine torus subbundles $S$ of affine torus bundles $M\ra B$ with torsion Chern classes. Already in the generalized geometry setting, the existence of T-duals depends on a topological constraint (cf. Theorem \ref{global thm1}), which is related to the Chern class of the T-dual affine torus bundle $\hat{M}\ra B$. To upgrade to theorems about $U(1)$-bundles we have to impose even stronger constraints.

To state Theorem \ref{tdual U(1) bundles trivial base} on a general base we first have to define the Poincar\'e bundle on $M\times_B\hat{M}$ analogously to (\ref{u(1) poincare line bundle in family}). It is easy to see that when $M$ and $\hat{M}$ do not admit smooth sections such a $U(1)$-bundle cannot exist. Instead, we define $\cP$ as a twisted $U(1)$-bundle or equivalently as a gerbe trivialization. Moreover, we also have to relax the definition of a physical brane from a $U(1)$-bundle supported on $S$ to a twisted $U(1)$-bundle. The precise description of these twists is the content of topological T-duality.

In the first subsection, we discuss the theory of gerbes, gerbe connections and gerbe modules. Then, we give the precise definition of topological T-duality. In the following subsection, we define the Poincar\'e bundle as a gerbe on the product affine torus bundles with torsion Chern classes $M$ and $\hat{M}$ which are T-dual in the sense of generalized geometry. In Proposition \ref{top tdual pair} we explain when such $M$ and $\hat{M}$ are also topologically T-dual. Finally, in the last subsection, we give a version of Theorem \ref{tdual U(1) bundles trivial base} on a general base.

\subsection{Gerbes and topological T-duality}
T-duality in generalized geometry connects affine torus bundles fibered over the same base endowed with $H$-fluxes, that is closed differential three-forms. Even though the definition is phrased in terms of differential forms, the existence of a T-duality relation between two affine torus bundles depends only on the de Rham classes of the $H$-fluxes. Topological T-duality is the refinement of this relationship to integer cohomology classes. 

Let $M$ be a differentiable manifold. Degree three de Rham cohomology $\HH^3(M,\dR)$ classifies the equivalence classes of exact Courant algebroids on $M$. The degree three integral cohomology $\HH^3(M,\dZ)$ classifies \emph{$U(1)$-bundle gerbes} which replace Courant algebroids in topological T-duality. The $H$-flux is replaced by the curvature of a gerbe connection. In this section, we introduce a `working definition' of gerbes, following notes of Hitchin \cite{hitchinGerbes} and Sections 2.2. and 4.3. in \cite{B2}. For a more general discussion see \cite{stevenson}.

\begin{definition}\label{bundle gerbe}
     A $U(1)$-bundle gerbe $\cG$ on $M$ is a triple $(\cU,L,\phi)$, where $\cU=\{U_i\}$ is an open cover of $M$, $L=\{L_{ij}\}$ is a collection of $U(1)$-bundles $L_{ij}\ra U_{ij}$ on the double intersections, together with isomorphisms
     $$L_{ij}\cong L_{ji}^{-1},$$
   $\phi=\{\phi_{ijk}\}$ is a collection of trivializations of 
     $\phi_{ijk}:L_{jk}\otimes L_{ik}^{-1}\otimes L_{ij} $ on $U_{ijk}$, such that
     $$\delta\phi=\phi_{jkl}\cdot \phi_{ikl}^{-1}\cdot \phi_{ijl}\cdot \phi_{ijk}^{-1}=1$$
     on $U_{ijkl}$, where the $1$ on the left-hand is understood as the canonical trivialization of the product $L_{jk}L_{jl}^{-1}L_{kl}L_{kl}^{-1}L_{il}L_{ik}^{-1} L_{jl}L_{il}^{-1}L_{jk}L_{jk}^{-1}L_{ik}L_{ij}^{-1}$.
\end{definition}
A gerbe is defined via line bundles and open covers so it can be restricted to submanifolds or pulled back along smooth maps. We can define tensor products and duals of gerbes in the obvious way. 

Suppose that $\cG=(\cU,L,\phi)$ is a gerbe on $M$ and $\cU'$ is a refinement of $\cU$. That is for all $U'_j\in \cU'$ there exists an $U_{\rho(j)}\in \cU$ with $U'_j\subset U_{\rho(j)}$. The map $\rho$ is called the refinement map. We  define the \emph{refinement of $\cG$} as the gerbe $(\cU', L'=\{L_{\rho(i)\rho(j)}\},\phi'=\{\phi_{\rho(i)\rho(j)\rho(k)}\})$.

There are two different notions of isomorphisms of gerbes called \emph{strict} and \emph{stable} isomorphisms. A strict isomorphism $\mu:\cG\ra \cG'$ between two gerbes $\cG=(\cU,L,\phi)$ and $\cG'=(\cU',L',\phi')$ can exists if $\cU=\cU'$. In this case, it is a collection of isomorphism 
$$\mu_{ij}:L_{ij}\ra L'_{ij}$$
such that the diagram
\begin{equation*}
    \begin{tikzcd}[column sep = large]
        L_{ij}\otimes L_{jk}\arrow{d}{\phi_{ijk}} \arrow{r}{\mu_{ij}\otimes \mu_{jk}} & L'_{ij}\otimes L'_{jk} \arrow{d}{\phi'_{ijk}}\\
        L_{ik} \arrow{r}{\mu_{ik}} & L'_{ik}
    \end{tikzcd}
\end{equation*}
commutes.

To define stable isomorphisms we have to first define trivial gerbe. Let $\cU$ be an open cover of $M$ and $N=\{N_i\ra U_i\}$ be a collection of $U(1)$-bundles. The gerbe \emph{trivial gerbe} $\delta(N)$ is defined by the $U(1)$-bundles
$$L_{ij}=N_i\otimes N_j^{-1} \ \ \ \text{on }U_{ij}$$
and the canonical trivializations of 
$L_{ij}L_{ik}^{-1}L_{ij}=L_{i}L_{j}^{-1}L_{i}L_{k}^{-1}L_iL_j^{-1}.$
A gerbe $\cG$ is called \emph{trivializable} if there exists a refinement $\cG'$ of $\cG$, and a strict isomorphism $\mu:\cG'\ra \delta(N)$ to a trivial gerbe.

Two gerbes $\cG$ and $\cG$ on $M$ are called \emph{stably isomorphic} if after passing to a common refinement there exists a trivial gerbe $\delta(N)$ on the refinement and a strict isomorphism
$$\mu:\cG\otimes \delta(N)\ra \cG'.$$

A trivialization of a $U(1)$-bundle can be understood as a section. Suppose $\cG$ is a gerbe defined on a good cover $\cU$ so each $L_{ij}$ is trivializable. Let $s_{ij}$ be a choice of sections of $L_{ij}$. Then, on $L_{jk}L_{ik}^{-1}L_{ij}$ there are two trivializations, one coming from the $s_{ij}$ and one from $\phi_{ijk}$. We can write
$$ s_{ij}s_{ik}^*s_{ij}=g_{ijk}\phi_{ijk}$$
for $U(1)$-valued functions $g_{ijk}$. By $\delta\phi=0$, the collection $\{g_{ijk}\}$ is a \v{C}ech cocycle in $\check{C}^2(\cU,\cC_{U(1)})$, do it represents a cohomology class $h\in \HH^2(M,\cC_{U(1)})\cong \HH^3(M,\dZ)$. This class is called the \emph{Dixmier-Douday class} of the gerbe $\cG$. If $\cG$ is not defined over a good cover we define its Dixmier-Douady class as the class of a refinement of it to a good cover. The Dixmier-Douady class of a gerbe is trivial if and only if it is trivializable over any refinement of its cover. In particular, the Dixmier-Douady class characterizes the stable isomorphism classes of gerbes.

\begin{definition} \label{connection on a gerbe} 
A \emph{connection} on a gerbe $\cG=(\cU,L,\phi)$ is a collection of connections $\nabla_{ij}$ on $L_{ij}$ such that $$\nabla_{ijk}(\phi_{ijk})=0,$$ where $\nabla_{ijk}$ is the induced connection on $L_{jk}L_{ik}^{-1}L_{ij}$, together with a collection of two-forms $F_{i}\in \Omega^2(U_{i})$ satisfying 
$$F_{i}-F_j=F_{\nabla_{ij}}.$$
The collection of two forms $\{F_{i}\}$ is called the \emph{curving} of the connection. Since $dF_{\nabla_{ij}}=0$ the three-forms $\{dF_{i}\}$ glue together to a global closed three-form $H\in \Omega^3(M)$ called the \emph{curvature} of the gerbe connection and curving.
\end{definition}
The curvature of a gerbe connection represents the image of the Dixmier-Douady class of the gerbe in de Rham cohomology. When $H=0$ we say that the gerbe connection is \emph{flat}. The notion of strict isomorphism can be defined for gerbes with connections and curvings by requiring that $\mu_{ij}$ are isomorphism of $U(1)$-bundles with connections and that the curvings agree.

\paragraph{Gerbe modules.} We are interested in branes on manifolds endowed with gerbes. In generalized geometry, a generalized submanifold is a submanifold $S$ together with a two-form which satisfies $dF=H|_S$. This can be viewed as a trivialization of the exact Courant algebroid over $S$. Therefore, if the enhancement of an exact Courant algebroid to integral cohomology is a gerbe, the enhancement of a generalized brane must be a submanifold together with a gerbe trivialization. This definition is slightly too restrictive though. Indeed, a gerbe which is trivial in de Rham cohomology may not be trivial in integral cohomology. These gerbes do not admit trivializations but they admit higher-rank twisted bundles. We refer to these as \emph{gerbe modules}.

\begin{definition}\label{gerbe module}
Let $\cG=(\cU,L,\phi)$ be a gerbe on $M$. A \emph{rank $d$ gerbe module} is a collection of $U(d)$-bundles $E_i\ra U_i$ together with isomorphisms
$$\varphi_{ij}: E_{i} \ra E_{j}\otimes  L_{ij}$$
such that the diagram
\[
\begin{tikzcd}
    E_i \arrow{r}{\varphi_{ij}} \arrow{d}{\varphi_{ik}} & E_j\otimes L_{ij}\arrow{d}{\varphi_{jk}}\\
    E_k\otimes L_{ik} & E_k\otimes L_{ij}\otimes L_{jk} \arrow{l}{\phi_{ijk}} \\
   \end{tikzcd}
\]
commutes. 
\end{definition}
The tensor product in the definition is understood as follows. We can associate a rank $d$ hermitian vector bundle to $E_j$ and a hermitian line bundle to $L_{ij}$. Their tensor product is a rank $d$ hermitian vector bundle and the $U(d)$-bundle $E_{j}\otimes L_{ij}$ is its frame bundle. In terms of transition functions the transition functions of $E_{j}\otimes L_{ij}$ is the tensor product of the transition functions of $E_j$ and $L_{ij}$.

A rank 1 gerbe module is precisely a trivialization. Moreover, if $\cG$ admits a rank $d$ gerbe module $E$, the collection of $U(1)$-bundles $\{det(E_i)\}$ is a rank 1 module of $\cG^d$. In particular, if $\cG$ admits a rank $d$ gerbe module its Dixmier-Douady class must be $d$-torsion.

We can easily generalize a rank 1 gerbe module to a gerbe module with connection. 
\begin{definition}
    A rank 1 module of a gerbe $\cG$ with connection is a rank 1 gerbe module $E$ such that each $U(1)$-bundle $E_i$ is endowed with a connection $\nabla_i$, $\varphi_{ij}$ are isomorphisms of bundles with connections and the curvatures of $\nabla_i$ give the curving of the connection on $\cG$.
\end{definition} 
This definition can be extended to higher-rank modules with projectively flat connections. 
\begin{definition}
     A rank $d$ module of a gerbe $\cG$ with connection is a rank $d$ gerbe module $E$ such that each $U(d)$-bundle $E_i$ is endowed with a projectively flat connection $\nabla_i$, $\varphi_{ij}$ are isomorphisms of bundles with connections and the curvatures of $\nabla_i$ are $F_i\cdot Id$ where $\{F_i\}$ is the curving of the connection on $\cG$.
\end{definition}
Even though gerbe modules are defined with respect to a certain cover, it is clear that we can restrict them to any refinement and still get a gerbe module. Moreover, tensoring a gerbe module with a rank 1 gerbe module changes the twisting gerbe by a trivial gerbe. In particular, the category of gerbe modules only depends on the stable isomorphism class of the gerbe.

\paragraph{Topological T-duality.} We have seen that T-duality has a description in generalized geometry but T-duality first arose from string theory as a duality between torus bundles endowed with gerbes. This duality underlies the conjecture of Strominger-Yau-Zaslow \cite{SYZ} which states that mirror symmetry in the zeroth order should be T-duality. A mathematical formulation of topological T-duality was first laid out by Bouwknegt, Evslin and Mathai in \cite{BEM} where T-duality is described as duality between principal torus bundles. Here we follow the work of Baraglia \cite{B1, B2} who extended the definition to affine torus bundles endowed with gerbes.

We simplify Baraglia's work slightly as his definition uses graded gerbes. We assume that all our torus bundles are oriented so the grading does not play a role. Moreover, he includes the lifting gerbe of the vertical bundle in his definition which corresponds to the third integral Stiefel-Whitney class $W_3(V)$ of the vertical bundle. The inclusion of this has physical reasons. Our aim in this project was to consider $BBB$ and $BAA$ branes on algebraic integrable systems where the vertical bundle is a complex vector bundle, so $W_3(V)=0$. Therefore, our definition agrees with Baraglia's in the relevant cases.

Let us consider the usual T-duality diamond once again.
\[
\begin{tikzcd}
    & M\times_B\hat{M} \arrow{dd}{q} \arrow{ld}[swap]{p} \arrow{rd}{\hat{p}}\\
    M \arrow{rd}[swap]{\pi} & & \hat{M} \arrow{ld}{\hat{\pi}}\\
    & B & 
\end{tikzcd}
\]
\begin{definition}[Definition 3.1 of \cite{B2}] Let $\pi: M\ra B$ and $\hat{\pi}:\hat{M}\ra B$ be rank $n$ affine torus bundles on $B$. Let $\cG$ and $\hat{\cG}$ be gerbes on $M$ and $\hat{M}$ respectively. We say that $(M,\cG)$ and $(\hat{M},\hat{\cG})$ are topologically T-dual if the following holds:\\
(1) For all $b\in B$ the restrictions $\cG|_{\pi^{-1}(b)}$ and $\hat{\cG}|_{\hat{\pi}^{-1}(b)}$ are trivial.\\
(2) There exists a stable isomorphism $\mu: p^*\cG \ra \hat{p}^*\hat{\cG}$.\\
(3) The isomorphism $\mu$ satisfies the Poincar\'e property.
\end{definition}
We will not explain the Poincar\'e property in detail here (see \cite[page 14.]{B2}). It roughly states that if the strict isomorphism underlying $\mu$ is given by $\mu:p^*\cG\otimes \delta(\cP)\ra \hat{p}^*\hat{\cG}$, then the local line bundles $\cP$ generating $\delta(\cP)$ look like the Poincar\'e line bundle (\ref{u(1) poincare line bundle in family}). In the next section, we will explicitly construct a pair of gerbes on affine torus bundles with torsion Chern classes and we spell out the isomorphism $\mu$ as well. This will satisfy the Poincar\'e property and we will only work in this framework.

As in the differentiable situation, not every pair $(M,\cG)$ admits a T-dual, there is a topological restriction on the Dixmier-Douady class of $\cG$. To understand this constraint we have to view the class $h\in \HH^3(M,\dZ)$ as a class in $\HH^2(M,\cC_{U(1)})$, where $\cC_{U(1)}$ is the sheaf of $U(1)$-valued functions on $M$. The second page of the Leray spectral sequence on $\pi: M\ra B$ corresponding to the sheaf $\cC_{U(1)}$ is given by
$$E^{p,q}_2(\pi,\cC_{U(1)})=\HH^p(B,R^q\pi_*\cC_{U(1)}),$$
in particular, 
$$E^{2,0}_2(\pi,\cC_{U(1)})=E^{2,0}_\infty(\pi,\cC_{U(1)})=\HH^2(B,\pi_*\cC_{U(1)}).$$
The boundary morphism $E^{2,0}_\infty(\pi,\cC_{U(1)})=F^{2,2}(\pi,\cC_{U(1)})\ra \HH^2(M,\cC_{U(1)})$ is given by the pullback $\pi^*$ (see \cite[section 3.2]{B2}). 

\begin{theorem}\label{existence of tduals}\emph{\cite[Theorem 3.1]{B2}}
    Let $(M,\cG)$ be an affine torus bundle endowed with a gerbe and denote by $h\in \HH^2(M,\cC_{U(1)})$ the Dixmier-Douady class of $\cG$. Then there exists a T-dual of $(M,\cG)$ if and only if $h$ lies in the image of the pullback
    $$\pi^*: \HH^2(B,\pi_*\cC_{U(1)})\ra \HH^2(M,\cC_{U(1)}).$$
\end{theorem}
Consider the short exact sequence of sheaves 
$$0 \ra \dZ \ra \cC_\dR \ra \cC_{U(1)} \ra 0$$
where $\cC_\dR$ and $\cC_{U(1)}$ are the sheaves of smooth $\dR$ and $U(1)$ valued functions on $M$, respectively.  Via the boundary morphism in the long exact sequence of derived direct images, we have
$$R^{q}\pi_*\cC_{U(1)}\cong R^{q+1}\pi_*\dZ = \wedge^{q+1} \Gamma^*\ \ \ \forall q\geq 1$$
as $\cC_{\dR}$ is a fine sheaf so $R^q\pi_*\cC_{\dR}=0$ for $q\geq 1$.

By \cite[Theorem 3.2]{B2} there is a morphism of spectral sequences 
$$\delta_r: E^{p,q}_r(\pi,\cC_{U(1)})\ra E^{p,q+1}_r(\pi,\dZ)$$
which on the second page is induced by the boundary morphism $R^q\pi_*\cC_{U(1)}\ra R^{q+1}\pi_*\dZ$. Moreover, the boundary morphism $\delta: \HH^n(M,\cC_{U(1)})\ra \HH^{n+1}(M,\dZ)$ in the long exact sequence of cohomology respects the Leray filtrations (with a shift) and induces an isomorphism
$$\delta: F^{2,2}(\pi,\cC_{U(1)})\ra F^{2,3}(\pi,\dZ).$$
Together with the projection $F^{2,3}(\pi,\dZ)\ra F^{2,3}(\pi,\dZ)/F^{3,3}(\pi,\dZ)=E^{2,1}_\infty (\pi,\dZ)$ the class $h$ defines an element
\begin{align}\label{h in E21infty}
    [h]\in E^{2,1}_\infty(\pi,\dZ).
\end{align}
The following topological consequences of T-duality are important for our discussions later.
\begin{proposition}\label{top consequences}\emph{\cite[Proposition 3.1 and Proposition 3.5]{B2}}
    Let $(M,\cG)$ and $(\hat{M},\hat{\cG})$ be topological T-duals. Denote the monodromy local systems of $M$ and $\hat{M}$ by $\Gamma$ and $\hat{\Gamma}$.
    Then, we have
    \begin{align}\label{monodromy tduality}\Gamma^\vee \cong \hat{\Gamma}.\end{align}
    In particular, taking the cup product of the Chern classes $c\in \HH^2(B,\Gamma)$ of $M$ and $\hat{c}\in \HH^2(B,\Gamma^\vee)$ of $\hat{M}$ and contracting the coefficients yields a class $\langle c\cup \hat{c}\rangle \in \HH^4(B,\dZ)$. We have
    \begin{align}\label{chern class tduality} \langle c\cup \hat{c}\rangle =0.\end{align}
    Cupping with the Chern class is the second-page differential in the Leray spectral sequence so by (\ref{chern class tduality}) $\hat{c}$ defines a class $[c]$ in $E_\infty^{2,1}(\pi,\dZ)$ and $[c]$ in $E_\infty^{2,1}(\hat{\pi},\dZ)$. Let $h\in HH^3(M,\dZ)$ and $\hat{h}\in \HH^3(\hat{M},\dZ)$ be the Dixmier-Douady classes of $\cG$ and $\hat{\cG}$. By (\ref{h in E21infty}) these also define classes in $E_\infty^{2,1}(\pi,\dZ)$ and $E_\infty^{2,1}(\hat{\pi},\dZ)$ respectively. We have
    \begin{align}
        [h]=[\hat{c}]\in E^{2,1}(\pi,\dZ)\ \ \  \text{and}\ \ \ [\hat{h}]=[c]\in E^{2,1}_{\infty}(\hat{\pi},\dZ).
    \end{align}
\end{proposition}

\subsection{The Poincar\'e bundle as a gerbe}\label{poincare gerbe} Let $\pi:M\ra B$ be an affine torus bundle with monodromy local system $\Gamma\ra B$ and torsion Chern class $c\in \HH^2(B,\Gamma)$. Let $\hat{\pi}:\hat{M}\ra B$ be another affine torus bundle with $\hat{\Gamma}\cong \Gamma^\vee$ and torsion Chern class $\hat{c}\in \HH^2(B,\Gamma^\vee)$.  Then, $(M,0)$ and $(\hat{M},0)$ are T-dual in the sense of generalized geometry (see Example \ref{Tdual torsion coordinates}). In this section, we construct the gerbe $\delta(\cP)$ from local Poincar\'e line bundles and give a condition to when the generalized geometry T-duality can be enhanced to topological T-duality.

Let us choose flat connections $A$ and $\hat{A}$ of $M$ and $\hat{M}$ and let $P=\langle p^*\hat{A}\wedge\hat{p}^*\hat{A}\rangle$ be the closed two-form as before. Let $\cU=\{U_i\}$ be a good cover of $B$ and let us choose flat sections of $M|_{U_i}$ and $\hat{M}|_{U_i}$ with respect to $A$ and $\hat{A}$. Let us also choose dual coordinates on the fibers of $M$ and $\hat{M}$ so the transition functions are
\begin{equation}\label{psi on Mxhat(M)}
\begin{aligned}
    \psi_{ij}:\ \  M\times _{U_i}\hat{M}|_{q^{-1}(U_{ij})}\ \ &\ra \ \ M\times _{U_j}\hat{M}|_{q^{-1}(U_{ij})}\\
    (y,v,\hat{v})\ \ &\mapsto \ \ (\rho_{ij}(y), A_{ij}v+c_{ij},\hat{A}_{ij}\hat{v}+\hat{c}_{ij}),
\end{aligned}
\end{equation}
with $\hat{A}_{ij}=(A_{ij}^{-1})^T$ and $c_{ij}: U_{ij}\ra V$ and $\hat{c}_{ij}: U_{ij}\ra V^*$ constant. We may lift the transition functions to diffeomorphisms
\begin{equation}\label{psi on VxV*}
\begin{aligned}
    \psi_{ij}:\ \  V\times _{U_i}V^*|_{U_{ij}}\ \ &\ra \ \ V\times _{U_j}V^*|_{U_{ij}}\\
    (y,v,\hat{v})\ \ &\mapsto \ \ (\rho_{ij}(y), A_{ij}v+c_{ij},\hat{A}_{ij}\hat{v}+\hat{c}_{ij}),
\end{aligned}
\end{equation}
which then satisfy
\begin{equation}\label{twisted cocycle condition}
    \psi_{ki} \circ \psi_{jk}\circ \psi_{ij} = t_{(n_{ijk},\hat{n}_{ijk})} : (y,v,\hat{v})\mapsto (y,v+n_{ijk},\hat{v}+\hat{n}_{ijk}),\ \ \ n_{ijk}\in \Gamma,\ \hat{n}_{ijk}\in \Gamma^\vee,
\end{equation}
where $c=\{n_{ijk}\}\in \HH^2(B,\Gamma)$ and $\hat{c}=\{\hat{n}_{ijk}\}\in \HH^2(B,\Gamma^\vee)$ are \v{C}ech representatives of the Chern classes.

Via the local trivializations, we can define the Poincar\'e line bundle $\cP_i$ over $M\times_{U_i}\hat{M}$ as (\ref{u(1) poincare line bundle in family})
\begin{equation*}
    a_i^\cP(y,v,\hat{v};\lambda,\hat{\lambda})=exp(i\pi (\hat{v}(\lambda)-\hat{\lambda}(v)+\hat{\lambda}(\lambda))),\ \ \ A_i^\cP(y,v, \hat{v})=i\pi (\hat{v}dv-vd\hat{v}).
\end{equation*}
 That is, with $N=\coprod_i M\times_{U_i}\hat{M}$ and $\cP\ra N$ given by $\cP_i$ the triple $(\delta(\cP),N,M\times_B\hat{M})$ is a trivial $U(1)$-bundle gerbe with a flat connection and curving $2\pi iP\in \Omega^2(M\times_B\hat{M},\dC).$
It is represented over $q^{-1}(U_{ij})$ by the $U(1)$-bundle $\cP_{ij}=\cP_i\otimes \psi_{ij}^*\cP_j^{-1}$
 \begin{equation}
     \begin{aligned}
         a_{ij}^\cP(y,v,\hat{v};\lambda,\hat{\lambda})&=exp(-i\pi \hat{c}_{ij}(A_{ij}\lambda)+i\pi (\hat{A}_{ij}\hat{\lambda})(c_{ij}) ),\\
         A_{ij}^\cP&=-i\pi  \hat{c}_{ij}(A_{ij}dv)+i\pi (\hat{A}_{ij}d\hat{v})(c_{ij}) ).
     \end{aligned}
 \end{equation}
 The gerbe product is given by a trivialization  $\phi^\cP_{ijk}$ of $\cP_{ij}\otimes \psi_{ij}^*\cP_{jk}\otimes \cP_{ik}^{-1}\otimes \cP_{ij}$. This represents a cohomology class in $\HH^1(M\times_B\hat{M},\cC_{U(1)})$, the Dixmier-Douady class of the gerbe. The gerbe product of the trivializable gerbe $\delta(\cP)$ can be determined as explained in \cite[Proposition 3.2]{B2}, but here we use a different method which can be used to calculate the gerbe product of non-trivializable gerbes as well. 
  
Suppose that a gerbe $\cG$ is represented by $U(1)$ bundles $L_{\alpha\beta}$ on the double intersections of some open covers. If $L_{\alpha\beta}$ are trivializable on $U_{\alpha\beta}$ we may choose trivializing sections $s_{\alpha\beta}$  of them. Then, $s_{\beta\gamma}\otimes s_{\alpha\gamma}^*\otimes s_{\alpha\beta}$ is a section of the trivial bundle, that is a $U(1)$-valued function $g_{\alpha\beta\gamma}$. The product rule ensures that $\delta g=0$.
 
Let us determine the gerbe product of $\delta(\cP)$ now. The $U(1)$-bundles $\cP_{ij}$ are topologically trivial on $q^{-1}(U_{ij})$, only the connection is non-trivial. A trivializing section of $\cP_{ij}$ is a function $\theta_{ij}:V\times_{U_{ij}}V^*\ra U(1)$ such that 
$$a_{ij}^\cP(y,v,\hat{v};\lambda,\hat{\lambda})=\theta_{ij}(y,v+\lambda,\hat{v}+\hat{\lambda})\cdot \theta_{ij}(y,v,\hat{v})^{-1}.$$
Let
$$\theta_{ij}(y,v,\hat{v})=exp(-i\pi\hat{c}_{ij}(A_{ij}v) + i\pi (\hat{A}_{ij}\hat{v})(c_{ij})). $$
Then, the gerbe product $\phi_{ijk}^\cP$ is a function which satisfies
\begin{align}\label{P gerbe factor}
                \phi_{ijk}^\cP(y,v+\lambda,\hat{v}+\hat{\lambda})\cdot \phi_{ijk}^\cP(y,v,\hat{v})^{-1}&=1,\\ \label{P gerbe conn}
        \psi^*_{ij}A_{jk}-A_{ki}+A_{ij}+dlog\psi_{ij}^*\theta_{jk}-dlog\theta_{ik}+dlog\theta_{ij}&=-dlog \phi^\cP_{ijk},\\ \label{P gerbe cycle}
        (\delta g)_{ijk}&=1,
\end{align}
where $\delta$ denotes the \v{C}ech differential and the right-hand side of (\ref{P gerbe factor}) is given by
$$(\psi_{ij}^*a_{jk}^\cP\cdot (a_{ik}^\cP)^{-1}\cdot a_{ij})(y,v,\hat{v};\lambda,\hat{\lambda}) \cdot (\psi_{ij}^*\theta_{jk}^{-1}\cdot \theta_{ik}\cdot \theta_{ij}^{-1})(y,v+\lambda,\hat{v}+\hat{\lambda})\cdot (\psi_{ij}^*\theta_{jk}\cdot \theta_{ik}^{-1}\cdot \theta_{ij})(y,v,\hat{v})=1.$$
From (\ref{P gerbe conn}) we have
$$-2\pi i \hat{n}_{ijk}\cdot dv + 2\pi i d\hat{v}\cdot n_{ijk}=-dlog \phi^\cP_{ijk}(y,v,\hat{v})$$
so 
$$\phi^\cP_{ijk}(y,v,\hat{v})=exp(-2\pi i \hat{v}(n_{ijk})+ 2\pi i \hat{n}_{ijk}(v)) \cdot \varphi_{ijk}$$
where $\varphi_{ijk}$ is constant. Then, if (\ref{P gerbe cycle}) holds we have
\begin{align*}
\psi_{ij}^*\phi^\cP_{jkl}\cdot (\phi^\cP_{ikl})^{-1}\cdot \phi_{ijl}^\cP \cdot (\phi^\cP_{ijk})^{-1} = (\delta\varphi)_{ijkl}\cdot exp(-2\pi i \hat{c}_{ij}(n_{jkl})+2\pi i \hat{n}_{jkl}(c_{ij}))
\end{align*}
Let us denote by $\gamma$ the \v{C}ech cochain $\{c_{ij}\}\in \check{C}^1(\cU,V)$ and by $\hat{\gamma}$ the cochain $\{\hat{c}_{ij}\}\in \check{C}(U_{ij},V^*)$. These represent the translation parts of the transition functions so we have $\delta \gamma=\{n_{ijk}\}=c\in \check{C}^2(\cU,\Gamma)$ and $\delta\hat{\gamma}=\{\hat{n}_{ijk}\}=\hat{c}\in \check{C}^2(\cU,\Gamma^\vee)$. Then, 
$$\delta\langle \gamma \cup \hat{\gamma}\rangle = \langle \delta \gamma\cup \hat{\gamma}\rangle-\langle \gamma\cup \delta \hat{\gamma}\rangle=\langle n_{jkl}, \hat{c}_{ij}\rangle -\langle c_{ij},  \hat{n}_{jkl}\rangle ,$$
where we use that $\langle n_{jkl}, \hat{c}_{ij}\rangle=-\langle c\cup\hat{\gamma}\rangle_{klji}=\langle c\cup\hat{\gamma}\rangle_{ijkl}$. In conclusion
\begin{equation}\label{P gerbe product}
    \phi_{ijk}^\cP(y,v,\hat{v})=exp\Big(-2\pi i \hat{v}(n_{ijk})+2\pi i \hat{n}_{ijk}(v)+2\pi i \langle \gamma\cup\hat{\gamma}\rangle_{ijk}\Big).
\end{equation}
\begin{definition}
   In the above setting the trivial gerbe $\delta(\cP)$ or the collection of line bundles $\{\cP_i\}$ is called the \emph{twisted Poincar\'e bundle}.
\end{definition}
The twisted Poincar\'e bundle depends on the choice of flat connections but also on the choice of cover and sections as they provide the normalizations of $\cP_i$. 

\begin{proposition}\label{top tdual pair}
The affine torus bundles $M$ and $\hat{M}$ can be endowed with gerbes $\cG$ and $\hat{\cG}$ such that $(M,\cG)$ and $(\hat{M},\cG)$ are topologically T-dual if and only if
$$\langle c\cup \hat{c}\rangle =0 \in \HH^4(B,\dZ).$$
\end{proposition}
\begin{proof}
If $(M,\cG)$ and $(\hat{M},\hat{\cG})$ are topologically T-dual, then $\langle c\cup \hat{c}\rangle =0$ by \cite[Proposition 3.5]{B2}. The opposite indication is also a consequence of Baraglia's work, one can follow for example the proof of \cite[Theorem 3.1]{B2}. Here, we spell out roughly the same proof by constructing the gerbes $\cG$ and $\hat{\cG}$ explicitly, which we will use later.

Let's assume that $\langle c\cup \hat{c} \rangle =0$. Let $\delta(\cP)$ be the twisted Poincar\'e bundle associated to a choice of flat connections $A$, $\hat{A}$, cover $\cU=\{U_i\}$ and local flat sections of $M$ and $\hat{M}$.  We show that there exists bundle gerbes $\cG$ and $\hat{\cG}$ with connections and curvings on $M$ and $\hat{M}$ such that on $M\times_B\hat{M}$ we have a strict isomorphism
 $$\mu: p^*\cG\otimes \delta(\cP) \ra p^*\hat{\cG}$$
of gerbes with connections which induces a stable isomorphism $[\mu]: p^*\cG\ra \hat{p}^*\hat{\cG}$.

Over triple intersections of the cover $M\times_{U_i}\hat{M}$ the gerbe products $g$, $\hat{g}$ and $g^\cP$ of $\cG$, $\hat{\cG}$ and $\cP$ have to satisfy
$$p^*g\cdot g^\cP_{ijk}=\hat{p}^*\hat{g}_{ijk}\cdot \delta(\mu_{ij}),$$
where $\mu_{ij}$ represents the gerbe isomorphism. 

We set the gerbe product of $\cG$ to
\begin{align}\label{G gerbe product} g_{ijk}(y,v)=exp(2\pi i \hat{n}_{jkl}(v)-2\pi i f_{ijk})\end{align}
over $\pi^{-1}(U_{ijk})$ for some appropriate functions $f_{ijk}:U_{ijk}\ra \dR$ such that $\delta(g)=1$. Since $\langle c \cup \hat{c}\rangle =0\in \HH^4(B,\dZ)$ there exists some $m=\{m_{ijkl}\}\in \check{C}^3(\cU,\dZ)$ such that
$$\langle c \cup \hat{c}\rangle_{ijklm}=\langle n_{ijk}, \hat{n}_{klm}\rangle = (\delta m)_{ijklm}.$$
Then, we have 
$$\delta \langle c \cup \hat{\gamma}\rangle = \langle c \cup \delta\hat{\gamma}\rangle =\delta m$$
so $\langle c \cup \hat{\gamma}\rangle-m$ must be exact in $\check{C}^3(\cU,\cC^\infty_\dR)$. Therefore, we can write
$$\hat{n}_{jkl}(c_{ij})=m_{ijkl}+(\delta f)_{ijkl},$$
where $m_{ijkl}\in \dZ$ and $\{f_{ijk}\}\in \check
{C}^2(\cU,\cC^\infty_\dR)$.

We can represent this gerbe with local $U(1)$-bundles $L_{ij}$ on the double intersections and endow it with a connection. The $L_{ij}$ are represented by the factor of automorphy and connection 1-form



 \begin{equation}\label{G line bundles}
     \begin{aligned}
         a_{ij}^\cG(y,v;\lambda)&=exp(\pi i \hat{c}_{ij}(A_{ij}\lambda)),\\
         A_{ij}^\cG(y,v)&=i\pi \hat{c}_{ij}(A_{ij}dv)+ 2\pi i \epsilon_{ij},
     \end{aligned}
 \end{equation}
 where $\epsilon_{ij}\in \Omega^1(U_{ij},\dR)$. We can trivialize $L_{ij}$ as
 $$a_{ij}^\cG(y,v;\lambda)=\theta_{ij}(y,v+\lambda)\cdot \theta_{ij}^{-1}(y,v)$$
 with
 $$\theta_{ij}^\cG(y,v)=exp(i\pi \hat{c}_{ij}(A_{ij}v)).$$
 Then, by the analogue of (\ref{P gerbe conn}) we have
 $$2\pi i \hat{n}_{ijk}\cdot dv+2\pi i (\epsilon_{jk}-\epsilon_{ik}+\epsilon_{ij})=-dlog\ g_{ijk}.$$
 That is $\epsilon_{jk}-\epsilon_{ik}+\epsilon_{ij}=df_{ijk}$. Indeed, such $\epsilon_{ij}$ exist since the sheaf of differential forms is fine. 
 
 The curvature of this gerbe connection is determined by the curvature of the local bundles so it is basic. We can find a flat connection on $\cG$ if and only if we can find \emph{constant} $\{f_{ijk}\}$ as in \ref{G gerbe product}. This could happen if the gerbe product represented a torsion cohomology class in $\HH^3(M,\dZ)$.

Analogously we define the gerbe $\hat{\cG}$ on $\hat{M}$ via the gerbe product on the triple intersections of $\hat{\pi}^{-1}(U_{ijk})$ as
\begin{align}\label{hat(G) gerbe product}
    \hat{g}_{ijk}(y,\hat{v})=exp(-2\pi i \hat{v}(n_{ijk})-2\pi i \hat{f}_{ijk})
\end{align}
with $\hat{f}_{ijk}=f_{ijk}+\langle c_{ij}\cup \hat{c}_{jk} \rangle$. Then, $\langle \gamma \cup \hat{c}\rangle =m+ \delta \hat{f}$ and $\delta(g^{\hat{\cG}})=1$.

Local bundles $\hat{L}_{ij}$ representing $\hat{\cG}$ and a connection are given by
   \begin{equation}
     \begin{aligned}
         a_{ij}^{\hat{\cG}}(y,\hat{v};\hat{\lambda})&=exp(i\pi  (\hat{A}_{ij}\hat{\lambda})(c_{ij})),\\
         A_{ij}^{\hat{\cG}}(y,\hat{v})&=i\pi c_{ij}\cdot \hat{A}_{ij}d\hat{v}+2\pi i \epsilon_{ij},
     \end{aligned}
 \end{equation}
with local trivializations
$$\theta^{\hat{\cG}}_{ij}(y,\hat{v})=exp(i\pi  (\hat{A}_{ij}\hat{v})(c_{ij})).$$
Once again the curvature of this connection is a basic three-form.  

Finally, we need to find isomorphisms 
$$\mu_{ij}:p^*L_{ij}\otimes \cP_{ij}\ra \hat{p}^*\hat{L}_{ij} $$
which commute with the gerbe products. The local trivializations $\theta_{ij}$, $\theta^\cG_{ij}$ and $\theta^{\hat{\cG}}_{ij}$ the $\mu_{ij}$ can be represented by $U(1)$-valued functions on $V\times_{U_{ij}}V^*$,  satisfying
\begin{align}\label{gerbe iso factor}
    1&=\mu_{ij}(y,v+\lambda,\hat{v}+\lambda)\cdot \mu_{ij}(y,v,\hat{v})^{-1},\\ \label{gerbe iso conn}
    0&=-dlog \mu_{ij}\\ \label{gerbe iso cocycle}
    p^*g_{ijk}\cdot \phi_{ijk}^\cP\cdot \hat{p}^*\hat{g}_{ijk}^{-1}&= \psi_{ij}^*\mu_{jk}\cdot \mu_{ik}^{-1}\cdot \mu_{ij},
\end{align}
where the left-hand side of (\ref{gerbe iso conn}) is given by
$$p^*(A_{ij}^\cG+dlog \theta_{ij}^\cG)+A_{ij}^\cP+dlog \theta_{ij}-\hat{p}^*(A_{ij}^{\hat{\cG}}+dlog \theta_{ij}^{\hat{\cG}})  =0.$$
By construction $\mu_{ij}=1$ satisfies all the conditions. 

The equation (\ref{gerbe iso cocycle}) is equation 3.9. of \cite{B2} in the proof of Theorem 3.1. In particular, given such $\mu_{ij}$ the induced stable isomorphism $[\mu]: p^*\cG \ra \hat{p}^*\hat{\cG}$ satisfies the Poincar\'e property.
\end{proof} 
Note that the representatives $\{g_{ijk}\}$ and $\{\hat{g}_{ijk}\}$ of the gerbe products on $\cG$ and $\hat{\cG}$ actually lie in $\HH^2(B,\pi_*\cC_{U(1)})$ and $\HH^2(B,\hat{\pi}_*\cC_{U(1)})$. In particular, their Dixmier-Douady classes  $h$ and $\hat{h}$ lie in $F^{2,2}\HH^2(M,\cC_{U(1)})$ and $F^{2,2}\HH^2(\hat{M},\cC_{U(1)})$ in line with Theorem \ref{existence of tduals}.

\begin{remark}
The theorem shows that given $\langle c\cup \hat{c}\rangle=0$, there exist T-dual gerbes with connections on $M$ and $\hat{M}$ and that the curvatures of the gerbe connections are given by a basic 3-form pulled back from the base. The existence of a connection further restricts the topology of $M$ and $\hat{M}$ but if $\cG$ admits a flat connection so does $\hat{\cG}$.
\end{remark}

\subsection{Main theorem on general base}
In this section, we update Theorem \ref{tdual U(1) bundles trivial base} to physical branes on an affine torus bundle endowed with a gerbe. We first define branes in this context, which are given by gerbe modules with connections. Then, we prove a technical lemma and finally state and prove our final result Theorem \ref{tdual u(1) bundle general base}. 

Let $(M,\cG)$ and $(\hat{M},\hat{\cG})$ be a T-dual pair as in Theorem \ref{top tdual pair} such that $\cG$ and $\hat{\cG}$ are flat. 
\begin{definition}
A \emph{rank d brane} on $M$ is a pair $(S,E)$ of an affine torus subbundle $S\subset M$ and a rank $d$ module $E\ra S$ of the gerbe $\cG|_S$ after we shift the curving of $\cG|_S$ by the curvature of $E$.
\end{definition}
This definition makes sense since the line bundles defining $\cG$ are flat. That is, if $E=\{(E_i,\nabla_i)\}$ is a collection of projectively flat $U(d)$-bundles with isomorphisms $\varphi_{ij}: E_i\ra E_j\otimes L_{ij}$ then the two-froms describing the curvatures of $\nabla_i$ satisfy $F_{\nabla_i}=F_{\nabla_j}$ over $U_{ij}$. That is, they define a global closed two-form $F\in \Omega^2(S)$. This global two-form is then a curving for the flat connection on $\cG|_S$.

The curving of the gerbe connection corresponds to the \emph{B-field} in physics. That is, shifting the curving can be seen as ``turning on the B-field" without changing the connection. We do not restrict ourselves to a specific $B$-field as we did not constrain it in the generalized geometry case either. Indeed, with our current definition, we can associate to a rank $d$ brane $(S,E)$ a generalized brane $(S,F)$. Since our methods only apply to line bundles and generalized branes which have invariant curvatures we have the following definition.
\begin{definition}
    A rank $d$ brane $(S,E)$ is called \emph{T-dualizable} if the corresponding generalized brane $(S,F)$ is T-dualizable. That is, if and only if $F\in \Omega^2(S)$ is invariant.
\end{definition}
The following lemma will be necessary to describe T-dualizable branes and prove a global theorem.
\begin{lemma}\label{global H semichar}
Consider $H\in H^0(B,\wedge^2 \Gamma^*)$ as a family of alternating bilinear forms on the fibers of $\Gamma$ with values in $\dZ$. Then, there exists a global semicharacter $\chi_0$ for $H$. More precisely $\chi_0: \Gamma\ra U(1)$ such that for any $b\in B$, $\chi_0(b):\Gamma_b \ra U(1)$ is a semicharacter for $H|_{\Gamma_b}$.
\end{lemma}
\begin{proof}
Let $\{U_\alpha\}$ be a good cover of $B$ so we may represent $H$ as a collection of alternating bilinear forms $\{H_\alpha\}$ acting on $\Gamma|_{U_\alpha}\cong U_\alpha\times\dZ^{n}$ which satisfy $\psi_{\alpha\beta}^*H_{\beta}=H_{\alpha}$.

We define $\chi_0$ locally on $U_\alpha$ using a decomposition for $H_\alpha$. Let $\{\mu_1,...\mu_n\}$ with $\mu_i\in \Gamma(U,\dZ^n)\cong \dZ^n$ be a symplectic frame for $H_\alpha$. That is,
$$H_{\alpha}(\mu_i,\mu_{r+i})=h_i,\ \ \ i=1,...,r\ \ \ \text{and}\ \ \ H_{\alpha}(\mu_{l},\lambda )=0 \ \ \ l=2r+1,...,n,\ \lambda\in \Gamma_{\alpha}.$$
A decomposition of $\Gamma_{\alpha}$ with respect to $H_{\alpha}$ is given by $\Gamma_{\alpha}=\Gamma_{\alpha}^1+\Gamma_{\alpha}^2+\Gamma_{\alpha}^0$ where
$\Gamma_{\alpha}^1=span\{\mu_i\ |\ i=1,...r\},\ \ \Gamma_{\alpha}^2=span\{\mu_{r+i}\ |\ i=1,...,r\}\ \ \text{and}\ \ \Gamma_{\alpha}^0=span\{\mu_l\ |\ l=2r+1,...,n\}.$ 

We then decompose each $\lambda\in \Gamma|_{U_\alpha}$ as $\lambda=\lambda^1+\lambda^2+\lambda^0$ and we define
$$\chi_\alpha(\lambda):=exp(i\pi H_\alpha(\lambda^1,\lambda^2)).$$

We may choose the local decompositions to be compatible with the transition functions $\psi_{\alpha\beta}$ in the following sense. If $\{\mu_i\}$ is a symplectic frame for $H_{\alpha}$, then $\{\psi_{\alpha\beta}\mu_i\}=\{\nu_i\}$ is a symplectic frame for $H_{\beta}$ over $U_{\alpha\beta}$. But since the transition functions are constant and so are sections of $\Gamma_\beta$ there is no issue of extending $\nu_i$ to a symplectic frame of $H_\beta$ over $U_\beta$. Then, $\Gamma_\beta^1=span\{ \nu_i\ |\ i=1,...r\}$, $\Gamma_\beta^2=span\{ \nu_{r+i}\ |\ i=1,...r\}$ and $\Gamma_\beta^0=span\{\nu_{l}\ |\ l=2r+1,...,n\}$ is a decomposition with respect to $H_\beta$. 

Let $\lambda\in \Gamma(U_{\alpha,\beta},\dZ^n) $ be a local section of $\Gamma$. By the choice of decompositions, $\psi_{\alpha\beta}(\lambda^1)=(\psi_{\alpha\beta}\lambda)^1$ so we have
\begin{align*}
    (\psi_{\alpha\beta}^*\chi_\beta)(\lambda)&=\chi_\beta(\psi_{\alpha\beta}\lambda)=exp(i\pi H_\beta((\psi_{\alpha\beta}\lambda)^1,(\psi_{\alpha\beta}\lambda)^2)=exp(i\pi H_\beta(\psi_{\alpha\beta}(\lambda^1),\psi_{\alpha\beta}(\lambda^2))\\&=exp(i\pi H_{\alpha}(\lambda^1,\lambda^2))=\chi_\alpha(\lambda)
\end{align*}
so $\{\chi_\alpha\}$ glue together to a global $\chi_0$.
\end{proof}
\begin{theorem}\label{tdual u(1) bundle general base}
Let $(S,L)$ be a rank 1 T-dualizable brane on $(M,\cG)$. Let $F\in \Omega^2(S)$ be the curvature of $L$ and denote the fiberwise component of $F$ by $H\in \HH^0(\pi(S),\Gamma^\vee_S)$. Then the following holds.
\begin{enumerate}
    \item \label{1 global} $-H(c)=q(\hat{c})\in \HH^2(B,\Gamma_S^\vee),$ that is the distribution $\Delta$ on $M\times_B\hat{M}$ has closed leaves $Z$ which are affine torus subbundles over $\pi(S)$ (Theorem \ref{gluing of Z}).
    \item \label{2 global} There exists a leaf $Z\subset S\times_{\pi(S)}\hat{M}$  such that the collection of line bundles $p^*L_i\otimes \cP_i$ are trivial on the fibers of $\hat{p}_Z$. Moreover, over $Z$ we have an isomorphism between collections of $U(1)$-bundles
    $$p^*L\otimes \cP \cong \hat{L}$$
    such that $\delta(\hat{L})$ is a trivialization of $\hat{p}^*\hat{\cG}$.
    \item \label{3 global} The T-dual $(\hat{S},\hat{L})$ is a rank $d$ T-dualizable brane on $\hat{M}$.
\end{enumerate}
\end{theorem}
\begin{proof}[Proof of \ref{1 global}.] Let $V$, $V_{S}$ and $\Gamma$, $\Gamma_S$ be the vertical bundle and the monodromy local system of $M$ and $S$ respectively. Let $\cU=\{U_i\}$ be the good cover of $B$ defining $\cG$.  By passing to a refinement if necessary we can assume that  $\cU\cap \pi(S)=\{U_i\cap \pi(S)=V_i\}$ is also a good cover. Let us denote by $S_i$ the intersection $\pi^{-1}(V_i)\cap S$ and by $S_{ij}=S_i\cap S_j$. Then $L$ is represented by local bundles $L_i\ra S_i$ which are T-dualizable bundles in the sense of Theorem \ref{tdual U(1) bundles trivial base}. The gerbe $\cG|_S$ is represented by $L_{ij}^\cG|_{S_{ij}}$. If the $L=\{L_i\}$ form a gerbe module then we have a strict isomorphism between the trivial gerbe $\delta(L)$ and $\cG|_S$.

Since all the line bundles $L_{ij}^\cG|_{S_{ij}}$ and $L_i\otimes L_j^{-1}$ are trivializable we may chose sections $s_{ij}^\cG$ of $L_{ij}^\cG|_{S_{ij}}$ and $s^L_{ij}$ of $L_i\otimes L_j^{-1}$. Then the isomorphisms over $S_{ij}$
$$\mu_{ij}: L_i\otimes L_j^{-1}\ra L_{ij}^\cG|_{S_{ij}}$$
are represented by functions $m_{ij}:S_{ij}\ra U(1)$ such that
$$\mu_{ij}(s^L_{ij})=m_{ij}s_{ij}^\cG.$$
These must commute with the gerbe products. More precisely, if we denote by $\{g_{ijk}\}$ and $\{g_{ijk}^L\}$ the gerbe products of of $\cG$ and $\delta(L)$ we have
$$g_{ijk}\cdot (g^L_{ijk})^{-1}=\psi_{ij}^*m_{jk}\cdot m_{ik}^{-1}\cdot m_{ij}.$$
That is,
$$\{g_{ijk}\}=\{g_{ijk}^{L}\}\in \HH^2(\pi(S),(\pi_S)_*\cC_{U(1)})$$
and
$$-q(\hat{c})=\delta_2\{g_{ijk}^L\}\in \HH^2(\pi(S),\Gamma_S^\vee),$$
so it suffices to show that $\delta_2\{g_{ijk}^L\}=H(c)$.

 We denote by $M_0=V/\Gamma$ and $S_0=V_S/\Gamma_S$ the group bundles corresponding to $M$ and $S$. Since $S$ is a submanifold of $M|_{\pi(S)}$ the Chern class  $c$ of $M|_{\pi(S)}$ is given by Chern class $c_S$ of $S$ (Lemma \ref{chern class of s}). In particular, we have a projection $M|_{\pi(S)}\ra (M_0|_{\pi(S)})/S_0$ and $S$ is the preimage of a section $b$.  That is, there exists $b_i: V_i\ra M_0$ such that $S_i=t_{b_i}S_0$ and $b_i-b_j\in S_0$. 

By \ref{appel-humbert family} we have unique representatives for $L_i$, and now we would like to determine $L_i\otimes L_j^{-1}$. For this, we have to calculate the transition functions of $S$ as we did in the discussion before Lemma \ref{chern class of s}. In the notation of (\ref{transition of S in M}) we have
\begin{equation}\label{transition of S}
\begin{aligned}
\psi_{ij}: t_{-b_i}S_i\ \ &\ra\ \  t_{-b_i} S_j\\
[y,q]\ \ &\mapsto\ \  [\phi_{ij}(y),B_{ij}q+C_{ij}b_i+c_{ij}^1],
\end{aligned}
\end{equation}
with $b_j=D_{ij}b_i+c_{ij}^2$. Let us denote the translations by $\tilde{c}_{ij}=C_{ij}b_i+c_{ij}^1$. We can lift (\ref{transition of S}) to diffeomorphisms $\psi_{ij}: V_S|_{V_i}\ra V_S|_{V_j}$ so they satisfy
\begin{equation}
    \psi_{ik}\circ \psi_{jk}\circ \psi_{ij} = t_{\tilde{n}_{ijk}},
\end{equation}
where $\{\tilde{n}_{ijk}\}=c_S\in\HH^2(\pi(S),\Gamma_S^\vee)$.

Now we can represent $L_i$ (actually $t_{b_i}^*L_i$ as in the proof of Theorems \ref{tdual u(1) bundles on a torus} and \ref{tdual U(1) bundles trivial base}) with a canonical factor (Theorem \ref{appel-humbert family}) using the invariant constant semicharacter from Lemma \ref{global H semichar} and we can calculate the gerbe product of $\delta(L)$ as before. We have for $y\in V_i,\ v\in V_S,\ \lambda\in \Gamma_S$
\begin{align*}
    a_{L_i}\cdot (\psi_{ij}^*a_{L_j})^{-1}(y,v;\lambda)=&exp\Big(-i\pi H(\tilde{c}_{ij},B_{ij}\lambda)+2\pi i(G_i-B_{ij}^TG_j)(\lambda) \Big),\\
    (A_{L_i}-\psi_{ij}^*A_{L_j})(y,v)=&2\pi i (A_i^0-A_j^0)- 2\pi i (dG_i-B_{ij}^TdG_j)\cdot v+2\pi i dG_j\cdot \tilde{c}_{ij}- \\
    &-i\pi H(\tilde{c}_{ij},B_{ij}dv)-i\pi H(B_{ij}v, d\tilde{c}_{ij})-i\pi H(\tilde{c}_{ij},d\tilde{c}_{ij}).
\end{align*}
A trivializing section is given by the function
\begin{align}\label{trivi of L} \theta_{ij}(y,v)=exp(-i\pi H(\tilde{c}_{ij},B_{ij}v)+2\pi i (G_i-B_{ij}^TG_j)(v)), \end{align}
then, $$A_{ijk}=\delta(A_{ij}-dlog\theta_{ij})$$ is
\begin{align*}
A_{ijk}=&-2\pi i \Big( (G_j-B_{jk}^TG_k)\cdot d(B_{ij}v+\tilde{c}_{ij})+dG_k\cdot \tilde{c}_{jk}- (G_i-B_{ij}^TG_k)dv -dG_k\cdot \tilde{c}_{ik}+\\
&+(G_i-B_{ij}^TG_j)dv+dG_j\cdot \tilde{c}_{ij}\Big)-i\pi\Big( H(\tilde{c}_{jk},d\tilde{c}_{jk})-H(\tilde{c}_{ik},d\tilde{c}_{ik})+H(\tilde{c}_{ij},d\tilde{c}_{ij})\Big)\\
&-2\pi i H(\tilde{n}_{ijk},dv)-2\pi i H(\tilde{c}_{jk},B_{jk}d\tilde{c}_{ij}).
\end{align*}
Using several times that $\delta(\tilde{c}_{ij})=\tilde{n}_{ijk}$ we find
$$A_{ijk}(y,v)=-dlog\Big(exp(2\pi  i H(\tilde{n}_{ijk},v))\cdot \mu_{ijk}(y) \Big)$$
with 
\begin{align*}\mu_{ijk}=&exp\Big(i\pi H(B_{jk}\tilde{c}_{ij},\tilde{c}_{ik})-i\pi H(\tilde{n}_{kji}, \tilde{c}_{ik})-i\pi H(\tilde{n}_{kji},B_{jk}\tilde{c}_{ij})\Big)\times\\
&\ \times exp\Big(2\pi i (G_j-B_{jk}^TG_k)\tilde{c}_{ij}+2\pi i G_k\tilde{n}_{kij}+f_{ijk}\Big).\end{align*}
The map $\delta_2$ is given by
\begin{align*}
    \delta_2:\ \ \ \HH^2(\pi(S),(\pi_S)_*\cC_{U(1)})\ \ &\ra \ \ \HH^2(\pi(S),\Gamma_S^\vee)\\
    \{g_{ijk}\}\ \ &\mapsto \ \ \{\lambda\mapsto dlog(g_{ijk})(v+\lambda)-dlog(g_{ijk})(v)\},
\end{align*}
therefore it is clear that $\delta_2(\{g^L_{ijk}\})=H(c_S).$

\textit{Proof of \ref{2 global}.} By Part \ref{1 global}. the Chern class of $\hat{M}|_{\pi(S)}$ is in the image of 
$$\HH^2(\pi(S),\Gamma_{\hat{S}}^\vee) \ra \HH^2(\pi(S),\Gamma_M^\vee),$$
and the Chern class of $S\times_{\pi(S)}\hat{M}$ is in the image of 
$$\HH^2(\pi(S),\Gamma_Z)\ra \HH^2(\pi(S),\Gamma_S+\Gamma_{M}^\vee).$$
Therefore, for any global $\hat{S}$ there exists a global leaf $Z\subset S\times_{\pi(S)}\hat{S}$. That is, there exists a global leaf $Z$ such that $\{p^*L\otimes P_i\}$ are trivial on the fibers of $\hat{p}_Z$ if and only if the local \emph{unique} T-duals $(\hat{S}_i,\hat{L}_i)$ form a global submanifold of $\hat{M}|_{\pi(S)}$.

Similarly to the discussion in Part \ref{1 global}., by Proposition \ref{space of T-duals} any global $\hat{S}$ is the preimage of a section under the projection
$$\rho: \hat{M}\ra coker(H)/coker(H,\Gamma_S^\vee).$$
The image of the local T-duals $\hat{S}_i\subset \hat{M}$ under $\rho$ is a local section. By the diagram \ref{Zc as translate of Z0} these local sections are given by
$$[-G_i]: V_i\ra coker(H)/coker(H,\Gamma_S^\vee),$$
the images of $-G_i: V_i\ra V_S^*$ under the projection $V_S^*\ra V_S^*/Im(H)$. We, therefore, need to show that the $[-G_i]$ form a global section. 

Recall, that via the trivializing sections $s_{ij}^L$ and $s^\cG_{ij}$ the isomorphism $\mu_{ij}:L_i\otimes L_j^{-1}\ra L^\cG_{ij}|_{S_{ij}}$ are given by functions. These satisfy
$$A_{L_i}-\psi^*A_{L_j}+dlog\theta_{ij} -A_{ij}^\cG-dlog \theta^\cG_{ij}=-dlog \mu_{ij},$$
where $\theta_{ij}$ is (\ref{trivi of L}) and $\theta_{ij}^\cG$ trivializes $L_{ij}^\cG|_{S_{ij}}$. The $L_{ij}^\cG|_{S_{ij}}$ are represented by
\begin{align*}
    a_{ij}^\cG(y,v;\lambda)&=exp(i\pi \hat{c}_{ij}(B_{ij}\lambda))\\
    A_{ij}^\cG(y,v)&=i\pi \hat{c}_{ij}\cdot B_{ij}dv+i\pi \hat{c}_{ij}\cdot d\tilde{c}_{ij}
\end{align*}
so 
\begin{align}
    \theta_{ij}^\cG(y,v)=exp(i\pi \hat{c}_{ij}(B_{ij} v)),
\end{align}
and we have
\begin{align}\label{last one pls}
    -dlog\mu_{ij}=2\pi i (&A^0_i-A^0_j-\frac{1}{2}H(\tilde{c}_{ij}, d\tilde{c}_{ij})+dG_j\cdot \tilde{c}_{ij})-\\
    &-2\pi i (H(B^T_{ij}\tilde{c}_{ij})-G_i+B_{ij}^TG_j-B^T_{ij}\hat{c}_{ij})\cdot dv.
\end{align}
Since $F_i=dA_{L_i}=dA_{L_j}=F_j$ the first term in (\ref{last one pls}) is a closed one-form pulled back from $V_{ij}$. Therefore, there exists $f_{ij}:V_{ij}\ra U(1)$ such that
\begin{align}\label{lastlast} -dlog\mu_{ij}-dlog\pi_S^*f_{ij}=-2\pi i (H(B^T_{ij}\tilde{c}_{ij})-G_i+B_{ij}^TG_j-B^T_{ij}\hat{c}_{ij})\cdot dv.\end{align}
The expression $H(B^T_{ij}\tilde{c}_{ij})-G_i+B_{ij}^TG_j-B^T_{ij}\hat{c}_{ij}$ depends only on $y$ and the left-hand side of (\ref{lastlast}) is closed, therefore $H(B^T_{ij}\tilde{c}_{ij})-G_i+B_{ij}^TG_j-B^T_{ij}\hat{c}_{ij}$ must be constant $\hat{n}_{ij}$ and 
$$\mu_{ij}\cdot \pi_S^*f_{ij}=exp(-2\pi i\hat{n}_{ij}(v)+ (\text{constant})):S_{ij}\ra U(1).$$
That is $\hat{n}_{ij}(\lambda)\in \dZ$ and $\hat{n}_{ij}\in \Gamma_S\vee.$

Going back to the local sections $[-G_i]$, we have
$$[-G_i]+[B^T_{ij}G_j]=[H(B^T_{ij}\tilde{c}_{ij})-G_i+B_{ij}^TG_j-B^T_{ij}\hat{c}_{ij}]=[\hat{n}_{ij}]=0,$$
where we use that $q(\hat{c})\in \HH^2(\pi(S),Im(H)\cap \Gamma_S^\vee)$.

\textit{Proof of \ref{3 global}.} We have isomorphisms
$$\hat{L}_i\cong \hat{p}_Z^*L^\cG_{ij}\otimes L_j$$
so by the projection formula
$$\hat{E}_i=(\hat{p}_Z)_*L_i\cong (\hat{p}_Z)_*L_j\otimes L^{\cG}_{ij}=\hat{E}_j\otimes L^{\cG}_{ij}.$$
In particular, using that $\hat{E}_i=\oplus^d\hat{L}_i$
$$\hat{L}_i\cong \hat{L}_j\otimes L_{ij}^\cG.$$
\end{proof}

\begin{example} Let $M\ra B$ be an algebraic integrable system with principally polarized fibers. Then, $M$ carries a semi-flat hyperk\"ahler structure and by Theorem \ref{semiflat tdual principal pol} the affine torus bundle $M$ endowed with $H=0$ flux is self-T-dual in the sense of generalized geometry. It is easy to see that $M$ can also be endowed with a (trivializable) gerbe such that $(M,\cG)$ is topologically self-T-dual. 

Let $\omega$ be the K\"ahler form of the special K\"ahler structure on the base $B$. Then the restriction of $\omega_\dI$, one of the k\"ahler forms of the semi-flat metric (\ref{semiflat cotangent1}, \ref{semiflat cotangent2}), to the fibers is
$$\omega^{-1}\in \HH^0(B,\wedge^2\Gamma_M^\vee).$$ 
With our conventions for the k\"ahler form, the restriction is minus the polarization, therefore $\omega^{-1}$ induces an isomorphism $\omega^{-1}:\Gamma_M\ra \Gamma_M^\vee$. If we set
$$\hat{c}:=-\omega^{-1}(c)=-d_2(\omega^{-1}),$$
then $\langle c\cup \hat{c}\rangle =d_2\circ d_2(\omega^{-1})=0$. 

Let $\hat{M}\ra B$ be the affine torus bundle with local system $\Gamma_M=\Gamma_M^\vee$ and Chern class $\hat{c}$ and let $\cG$ and $\hat{\cG}$ be the gerbes defined in Proposition \ref{top tdual pair}.

Let now $\cU=\{U_i\}$ be a good cover of $B$. Then on $M|_{U_i}=M_i$ we can define the $U(1)$-bundle with connection $(M_i,L_i^A)$ corresponding to the generalized brane $(M,\omega_\dI)$ (Example \ref{BAA-BBB space filling tdual bundles}). We can calculate the gerbe product associated with the trivial gerbe $\delta(L^A)$ as in the construction of the Poincar\'e bundle as a gerbe. It is easy to see that $\delta(L^A)$ is a trivialization of $\cG$. Moreover, the local T-duals $(\hat{M}_i,\hat{L}^B_i)$ trivialize $\hat{\cG}$.

Therefore, $M$ endowed with the trivial gerbe corresponding to the class $h\in \HH^3(M,\dZ)$ is topologically self-T-dual. On the other hand, if we denote by $\phi:M\ra \hat{M}$ the isomorphism induced by polarization, the pullback of $\hat{L}^B_i$ via $\phi$ is not $L^A_i$ (cf. Example \ref{BAA-BBB space filling tdual bundles}). These are instead different $U(1)$-bundles which restrict to the fibre as the duals of $L^A_i$. In particular, the pullback $\phi^*\hat{\cG}$ is only stably isomorphic to $\cG$ not strictly.

If $M$ is not principally polarized, then we can still set $\hat{c}:=-\omega^{-1}(c)\in \HH^2(B,\Gamma_M^\vee)$. If we take $\hat{M}$ to be affine torus bundle with monodromy local system $\Gamma_M^\vee$ and Chern class $\hat{c}$, then there exist gerbes $\cG$ and $\hat{\cG}$ such that $(M,\cG)$ and $(\hat{M},\hat{\cG})$ are topologically T-dual. The gerbe $\cG$ on $M$ is still trivialized by $(M_i,L^A_i)$, but its T-dual is now a higher rank brane given locally by $(\hat{M}_i,\hat{L}^B_i)$ forming a gerbe module of $\hat{\cG}$.
\end{example}

\section{T-duality of higher rank branes}
In this section, we consider T-duality of higher-rank branes. More precisely, we will see that we can upgrade T-duality of generalized branes to T-duality of twisted $U(n)$-bundles with connections.

Let $M$ and $\hat{M}$ be a pair of topologically T-dual affine torus bundles in the sense of Proposition \ref{top tdual pair} endowed with flat connections and gerbes $\cG$ and $\hat{\cG}$. In the previous section, we assumed that $\cG$ and $\hat{\cG}$ admitted flat connections but this is not necessary for the following discussion.

Let $(S,F)$ be a generalized brane in $M$ which admits global generalized T-duals and it satisfies the conditions of Theorem \ref{gluing of Z}. Let $(\hat{S},\hat{F})$ be a generalized T-dual of $(S,F)$. Then, by Theorem \ref{gluing of Z} there exists an affine torus subbundle $Z\subset S\times_{\pi(S)}\hat{S}$ which fits into the following diagram.
\begin{equation}\label{final diagram}
\begin{tikzcd}
& (Z,\hat{p}_Z^*\hat{F}-p_Z^*F=P|_Z) \arrow{dl}[swap]{p_Z} \arrow{dr}{\hat{p}_Z} &\\
(S,F) & & (\hat{S},\hat{F})
\end{tikzcd}
\end{equation}
Suppose first that $\pi(S)$ is simply connected. Then we can prove the following local statement.
\begin{theorem}\label{last local}
    Let $H$ be the fiberwise component of $F$, and suppose that in a local symplectic frame $\{\mu\}$ of $\Gamma_S$ it can be written as
    $$H=\sum_{i=1}^r \frac{n_i}{m_i}\mu_i^*\wedge \mu_{i+r}^*. $$
    Let $m=\prod_{i=1}^rm_i$ and $n=\prod_{i+1}^rn_i$.
Then, for any $Z\subset S\times_{\pi(S)}\hat{S}$ there exist $U(1)$-bundles $L_Z$ and $\hat{L}_Z$ with connections which satisfy
\begin{align}\label{higher rank tduals 1}L_Z\otimes \cP|_Z \cong \hat{L}_Z.\end{align}
Moreover, there exists a projectively flat $U(m)$-bundle with connection $E\ra S$ and a projectively flat $U(n)$-bundle with connection $\hat{E}\ra \hat{S}$ such that for any $Z$ we have
\begin{align}\label{higher rank tduals 2} (p_Z)_*L_Z\cong E\otimes U(m),\ \ \ \ \ (\hat{p}_Z)_*\hat{L}_Z\cong \hat{E}\otimes U(n),\end{align}
and the curvatures of the connections on $E$ and $\hat{E}$ are given by 
$$F_E=2\pi i F\cdot Id \in \Omega^2(S,\gu(m))\ \ \ \text{and}\ \ \ F_{\hat{E}}=2\pi i \hat{F} \cdot Id \in \Omega^2(\hat{S},\gu(n)).$$
That is, for any pair of generalized T-duals, we can find higher-rank branes over $S$ and $\hat{S}$ which are T-dual in the sense of \ref{higher rank tduals 1} and \ref{higher rank tduals 2}. Moreover, the connections of these higher rank branes are projectively flat and their curvatures are determined by the underlying generalized branes.
\end{theorem}
\begin{remark}\label{semihomog remark}
        We can apply this theorem to dual complex tori $M$ and $\hat{M}$ and T-dualizable generalized $B$-branes $(S,F)$ and $(\hat{S},\hat{F})$. Then the resulting sheaves $E$ and $\hat{E}$ are Fourier-Mukai pairs of semihomogeneous vector bundles supported on affine subtori. Therefore, calling $(S,E)$ and $(\hat{S},\hat{E})$ a T-dual pair of higher rank branes is justified.
\end{remark}        
\begin{proof}[Proof of Theorem \ref{last local}]
    We first determine $\cP|_Z$ in terms of factors of automorpy for a choice of $(S,F)$ and $(\hat{S},\hat{F})$ and $Z\subset S\times_{\pi(S)}\hat{S}$. 

Locally we can write $S$ as $t_bS_0$ inside $M|_{\pi(S)}$ for $b: \pi(S)\ra M$.  We then have $Z=t_{b,-G}Z_0$ in $S\times_{\pi(S)}\hat{M}$, where $G:\pi(S)\ra \hat{M}$ is a lift of a section $\pi(S)\ra V_S^*/\Gamma_S^\vee$ determined up to a constant by $F\in \Omega^2(S)$. Then, $\hat{S}=t_{-G}\hat{S}_0$. We have the maps
\[
\begin{tikzcd}
    & Z_0\arrow{dl}[swap]{p_0}\arrow{dr}{\hat{p}_0} & \\
    S_0 & & \hat{S}_0
\end{tikzcd}
\]
modelled on the homomorphisms
\[
\begin{tikzcd}
    0 \arrow{r} & Ann(\Gamma_S) \arrow{d}{\cong} \arrow{r} &\Gamma_Z \arrow{d}{\hat{p}_0} \arrow{r}{p_0} & \Gamma_H \arrow{r} \arrow{d}{-H} & 0\\
    0 \arrow{r} & Ann(\Gamma_S) \arrow{r} & \Gamma_{\hat{S}} \arrow{r}{q} & \Gamma_{S}^\vee \cap H(V_S) \arrow{r} &  0.
\end{tikzcd}
\]
Then, $p_Z$ and $\hat{p}_Z$ are determined by the commutative diagrams
\[
\begin{tikzcd}
Z_0 \arrow{d}[swap]{p_0} \arrow{r}{t_{(b,-G)}} & Z \arrow{d}{p_Z} & Z_0 \arrow{d}[swap]{\hat{p}_0} \arrow{r}{t_{(b,-G)}} & Z \arrow{d}{\hat{p}_Z} \\
S_0 \arrow{r}{t_b} & S, & \hat{S}_0 \arrow{r}{t_{-G}} & \hat{S}
\end{tikzcd}.
\]
Let us write $\cP^0_Z\ra Z_0$ for $t_{(b,-G)}^*(\cP|_Z)$. We have 
\begin{align*}
\cP|_Z=(t_{(b,-G)})_*\Big((t_{(b,-G)}^*\cP)|_{Z_0} \Big)
\end{align*}
so $\cP_Z^0=(t_{(b,-G)}^*\cP)|_{Z_0} $. The bundle $(t_{(b,-G)}^*\cP)$ on $M\times_{\pi(S)}\hat{M}$ is represented by
\begin{align*}
a_{(t_{(b,-G)}^*\cP)}(v,\hat{v};\lambda,\hat{\lambda})&=a_\cP(v+b,\hat{v}-G;\lambda,\hat{\lambda})\\
&=exp(i\pi(\hat{\lambda}(\lambda)+\hat{v}(\lambda)-\hat{\lambda}(v)-G(\lambda)-\hat{\lambda}(b)))\\
A_{(t_{(b,-G)}^*\cP)}(v,\hat{v})&=i\pi (\hat{v}-G)\cdot d(v+b)-(v+b)\cdot d(\hat{v}-G))\\
&=i\pi (\hat{v}\cdot dv-\hat{v}\cdot dv)-i\pi G\cdot dv+i\pi v\cdot dG -i\pi b\cdot d\hat{v} + i\pi \hat{v}\cdot db\\
&\ \ \ -i\pi G\cdot db+i\pi b\cdot dG.
\end{align*}
Therefore,
\begin{align*}
    a_{\cP^0_Z}(v;\lambda)&=exp\Big(i\pi \Big(\hat{p}_0(\lambda)(p_0(\lambda))+\hat{p}_0(v)(p_0(\lambda))-\hat{p}_0(\lambda)(p_0(v))-G(p_0(\lambda))-\hat{p}_0(\lambda)(b)\Big)\Big)\\
    &=exp\Big( i\pi \Big( -H(p_0(\lambda),p_0(\lambda))-H(p_0(v),p_0(\lambda))+H(p_0(\lambda),p_0(v))\\
    &\hspace{9cm} -G(p_0(\lambda))-\hat{p}_0(\lambda)(b)\Big)\Big)\\
    &=exp\Big( -2\pi i H(p_0(b),p_0(\lambda))-i\pi G(p_0(\lambda))-i\pi \hat{p}_0(\lambda)(b) \Big),\\
    A_{\cP^0_Z}(v)&=i\pi(-H(p_0(v))\cdot dp_0(v)+p_0(v)\cdot dH(p_0(v)))-i\pi G dp_0(v)+i\pi p_0(v)\cdot dG\\
& \ \ \ \ -i\pi bd\hat{p}_0(v)+i\pi \hat{p}_0(v)\cdot db-i\pi G\cdot db +i\pi b \cdot dG\\
&=-2\pi i H(p_0(v),dp_0(v))-i\pi G dp_0(v)+i\pi p_0(v)\cdot dG-i\pi bd\hat{p}_0(v)+i\pi \hat{p}_0(v)\cdot db\\
&\ \ \ -i\pi G\cdot db +i\pi b \cdot dG.
\end{align*}
Let us change the representatives by
$$\phi(v)=exp(i\pi(G(p_0(v))+\hat{p}_0(v)(b)+i\pi G\cdot b) )$$
so we have
\begin{align*}
    a_{\cP^0_Z}(v;\lambda)&=exp(-2\pi i H(p_0(v),p_0(\lambda))),\\
    A_{\cP^0_Z}(v) &=-2\pi i H(p_0(v),dp_0(v))-2\pi i G\cdot dp_0(v)-2\pi i b\cdot d\hat{p}_0(v) - 2\pi i b \cdot dG.
\end{align*}
The two-form $F\in \Omega^2(S)$ pulls back to a two-form $t_b^*F\in \Omega^2(S_0)$ under the diffeomorphism $t_b: S_0\ra S$. We can write
$$p_0^*t_b^*F=2\pi i H(dp_0(v), dp_0(v))+2\pi i dG \wedge dp_0(v) +2\pi i F^0$$
Let $A^0\in \Omega^1(\pi(S))$ be a two-form such that $dA^0=F^0$ and define the $U(1)$-bundles with connection $L^0_Z$ and $\hat{L}^0_Z$ on $Z_0$ as
\begin{align*}
    a_{L^0_Z}(v;\lambda)&=\chi_0(\lambda)exp(i\pi H(p_0(v),p_0(\lambda)))\\
    A_{L^0_Z}(v)&=i\pi H(p_0(v),dp_0(v))+2\pi i G\cdot dp_0(v) + 2\pi i A^0,\\
    a_{\hat{L}^0_Z}(v;\lambda)&=\chi_0(\lambda)exp(-i\pi H(p_0(v),p_0(\lambda))),\\
    A_{\hat{L}^0_Z}(v)&=-i\pi H(p_0(v),dp_0(v))-2\pi i b\cdot d\hat{p}_0(v)-2\pi i b\cdot dG+2\pi i A^0.
\end{align*}
Then we have 
$$L^0_Z\otimes \cP^0_Z\cong \hat{L}^0_Z$$
on $Z_0$. 

We can now repeat the proof of Theorem \ref{tdual U(1) bundles trivial base} to show that $(\hat{p}_Z)_*\hat{L}_Z$ is independent of the choice of $Z$ and of the choice of representatives and that
$$(\hat{p}_Z)_*\hat{L}_Z=\hat{E}\otimes U(n).$$

To show that the analogous statement holds for $L_Z$ we have to view $(S,F)$ as the T-dual of $(\hat{S},\hat{F})$ and work backwards. If we view $M$ as the dual of $\hat{M}$, the corresponding Poincar\'e bundle is given by $\cP^{-1}$, so we have on $Z$
$$\hat{L}_Z\otimes \cP^{-1}|_Z= L_Z.$$
Moreover, on $\hat{S}$ we can determine the fiber-wise component 
$$\hat{H}\in \HH^0(\pi(S),\wedge^2\Gamma_{\hat{S}}^\vee\otimes \dQ)$$
of the two-form $\hat{F}$, which acts on $V_{\hat{S}}$ as
$$\hat{H}(\hat{v},\hat{v})=H^{-1}(q(\hat{v}),q(\hat{v})).$$
Then repeating all the arguments in Chapter \ref{chapter gen geom tdual} we can show that the vertical bundle $V_Z$ and monodromy local system $\Gamma_Z$ of $Z_0$ can be expressed also as
\[
\begin{tikzcd}
    0 \arrow{r} & Ann(\Gamma_{\hat{S}}) \arrow{r} & \Gamma_Z \arrow{r} & \Gamma_{\hat{H}} \arrow{r} & 0 \\
    0 \arrow{r} & Ann(V_{\hat{S}}) \arrow{r} & V_Z \arrow{r} & V_{\hat{S}} \arrow{r} & 0
\end{tikzcd}
\]
and we can express $p_0$ and $\hat{p}_0$ via the maps
\[
\begin{tikzcd}
     0 \arrow{r} & Ann(\Gamma_{\hat{S}}) \arrow{d}{\cong} \arrow{r} & \Gamma_Z \arrow{r}{\hat{p}_0} \arrow{d}{p_0} & \Gamma_{\hat{H}}\arrow{d}{-\hat{H}} \arrow{r} & 0 \\
     0  \arrow{r} & Ann(\Gamma_{\hat{S}}) \arrow{r} & \Gamma_S \arrow{r}{\hat{q}} & \hat{H}(V_{\hat{S}})\cap \Gamma_{\hat{S}}^\vee \arrow{r} & 0.
\end{tikzcd}
\]
Repeating the proofs of Theorem \ref{tdual u(1) bundle general base} shows that 
$$(p_Z)_*L_Z\cong E\otimes U(m).$$
\end{proof}
The resulting pair $(S,E)$ and $(\hat{S},\hat{E})$ are also shown to be T-dual in \cite{CLZ}.
\begin{remark}
    Looking more closely at the factors of automorphy representing $E$ and $\hat{E}$ it is easy to show that 
    $$p_Z^*E\cong L_Z\otimes U(n)\ \ \ \text{and}\ \ \ \hat{p}_Z^*\hat{E}\cong \hat{L}_Z\otimes U(m).$$
\end{remark}

We can extend our theorem to T-dual torus bundles over a non-contractible base. Theorem \ref{top tdual pair} shows that if $M$ and $\hat{M}$ satisfy $\langle c\cup \hat{c}\rangle=0\in \HH^4(M,\dZ)$ we can find gerbes with connections $\cG$ and $\hat{\cG}$ on $M$ and $\hat{M}$ respectively such that the pairs $(M,\cG)$ and $(\hat{M},\hat{\cG})$ are topologically T-dual. The gerbes $\cG$ and $\hat{\cG}$ may not be flat, but their connections' curvatures are basic. 

Let $(S,F)$ and $(\hat{S},\hat{F})$ be a pair of T-dual generalized branes for zero $H$-flux, that is $dF=0$ and $d\hat{F}=0$. Moreover, assume that there exists a global leaf $Z\subset S\times_{\pi(S)}\hat{S}$  fitting into the diagram,
\begin{equation}\label{final diagram 2}
\begin{tikzcd}
& (Z,\hat{p}_Z^*\hat{F}-p_Z^*F=P|_Z) \arrow{dl}[swap]{p_Z} \arrow{dr}{\hat{p}_Z} &\\
(S,F) & & (\hat{S},\hat{F}).
\end{tikzcd}
\end{equation}
The following theorem shows that even if $\cG$ and $\hat{\cG}$ are not flat, over $S$ and $\hat{S}$ there exist flat gerbes with gerbe modules which are locally T-dual in the sense of Theorem \ref{last local}.
\begin{theorem}\label{last global} In the above setting let $m$ and $n$ be the positive integers as in Theorem \ref{last local}. Then, there exist flat gerbes $\cG'$ and $\hat{\cG}'$ on $S$ and $\hat{S}$ such that the following holds.
    \begin{enumerate}
        \item\label{higher rank 1} Over any $Z\subset S\times_{\pi(S)}\hat{S}$ as in (\ref{final diagram 2}) the twisted Poincar\'e bundle is a trivialization of $\hat{p}_Z^*\hat{\cG}'\otimes (p_Z^*\cG')^{-1}$. Moreover,
        $$\cG'\otimes (\cG|_S)^{-1}\cong\pi^*\cG_0\ \ \ \text{and}\ \ \ \hat{\cG}'\otimes (\hat{\cG}|_{\hat{S}})^{-1}\cong\hat{\pi}^*\cG_0 $$
        where $\cG_0$ is a gerbe on $\pi(S)$.
        \item\label{higher rank 2} For any $Z\subset S\times_{\pi(S)}\hat{S}$ as in (\ref{final diagram 2}) there exists $U(1)$-bundles $L$ and $\hat{L}$ on $Z$ twisted by $p_Z^*\cG'$ and $\hat{p}_Z^*\hat{\cG}'$ such that we have an isomorphism of gerbe trivializations
    $$\delta(\hat{L})\otimes \delta(L)^{-1} \cong \delta(\cP).$$
        Moreover, $L$ is
        trivial on the fibers of $p_Z:Z\ra S$ and $\hat{L}$ is trivial on the fibers of $\hat{p}_Z:Z\ra \hat{S}$. 
        \item\label{higher rank 3} There exists a $\cG'$-module $E$ of rank $m$ on $S$ and a $\hat{\cG}'$-module $\hat{E}$ of rank $n$ on $\hat{S}$ such that
         $$E\otimes U(m)\cong (p_Z)_*L\ \ \text{and}\ \ \hat{E}\otimes U(n)\cong (\hat{p}_Z)_*\hat{L}_Z.$$
         Moreover, the connections on $E$ and $\hat{E}$ are projectively flat with curvatures given by 
         $$F_E=2\pi i F\cdot Id \in \Omega^2(S,\gu(m))\ \ \ \text{and}\ \ \ F_{\hat{E}}=2\pi i \hat{F} \cdot Id \in \Omega^2(\hat{S},\gu(n)).$$
    \end{enumerate}
    \end{theorem}
\begin{proof}
    We define $\cG'$ and $\hat{\cG}'$ via the local line bundles $L_Z$ and $\hat{L}_Z$ of Theorem \ref{last local}. 
    
    Let $\cV$ be a good cover of $B$ such that $\cU:=\{\pi(S)\cap V\ |\ V\in \cV\}$ is a good cover of $\pi(S)$. Then, for any $Z\subset S\times_{\pi(S)}\hat{S}$ we can define line bundles $L_i$ and $\hat{L}_i$ over $Z_i=q^{-1}(U_i)\cap Z$  for $U_i\in \cU$ as in Theorem \ref{last local}. 

    Analogously to (\ref{psi on Mxhat(M)}), the transition functions for $Z$ over the cover $\{Z_i\}$ can be written as
    \begin{align*}
        \psi_{ij}^Z:\ \ \ Z_i\ \ &\ra\ \ Z_j\\
        (y,v)\ \ &\mapsto\ \ (\rho_{ij}(y),A_{ij}^Zv+c^Z_{ij}).
    \end{align*}
    
    We define a trivializable gerbe $\delta(L)$ on $Z$ via the collection of $U(1)$-bundles $\{L_i\}$. Over $Z_i\cap Z_j$ we can represent $\delta(L)$ by the the local flat $U(1)$-bundles via the following factor of automoprhy and connection one forms.
    \begin{align*}
        a_{L_i\otimes (\psi_{ij}^Z)^*L_j^{-1}}(y,v;\lambda)&=exp(-i\pi H(p_0(c_{ij}^Z),p_0(A^Z_{ij}\lambda)) ),\\
        A_{L_i\otimes (\psi_{ij}^Z)^*L_j^{-1}}(y,v)&=-i\pi H(p_0(c^Z_{ij}), dp_0(A^Z_{ij}v))-i\pi H(p_0(A_{ij}^Zv), dp_0(c^Z_{ij}))+\\
        &\ \ \ +2\pi i G_i\cdot dp_0(v)-2\pi i G_j\cdot dp_0(A^Z_{ij} v)-2\pi i G_jdp_0(c^Z_{ij})-\\
        &\ \ \ -i\pi H(p_0(c^Z_{ij}),dp_0(c^Z_{ij}))+2\pi i (A^0_i-A^0_j).
    \end{align*}
    Since $dA_{L_i\otimes (\psi_{ij}^Z)^*L_j^{-1}}=F_i-F_j=0$ we have 
    $$2\pi i G_jdp_0(c^Z_{ij})-i\pi H(p_0(c^Z_{ij}),dp_0(c^Z_{ij}))+2\pi i (A^0_i-A^0_j)=2\pi i d\phi_{ij}(y)$$
    for some basic functions $\phi_{ij}:U_{ij}\ra \dR$. We can calculate the gerbe product as in Section \ref{poincare gerbe}, so we find
    \begin{align*}
    g_{ijk}'(y,v)=exp(-2\pi i H(p_0(n^Z_{ijk}),p_0(v))+h_{ijk}(y))  \in \HH^2(\pi(S),q_*\cC_{U(1)}),
    \end{align*}
    where $h_{ijk}$ are smooth functions pulled back from $U_{ijk}$. 

    Using that $Z_i=t_{(b_i,-G_i)}Z_0$ and that $Z$ is a global submanifold of $S\times_{\pi(S)}\hat{S}$ we find that for all $i,j$ we have in
    \begin{align*}-((A^S_{ij})^{T})^{-1}H(p_0(v))-((A&^S_{ij})^T)^{-1}G_i+q(\hat{c}_{ij})\\
   &=-H(p_0(A^Z_{ij}v))-H(p_0(c^Z_{ij}))-G_j\ \ \ \in \left(V_S^*/\Gamma_S^\vee\right)\Big|_{U_{ij}},\end{align*}
    where $q:V_M^*\ra V_S^*$ is the projection and $\hat{c}_{ij}\in \HH^1(\pi(S),V_M^*/\Gamma_M^\vee)$ represents the Chern class of $\hat{M}$.
    Since $(A^S_{ij})^THp_0(A^Z_{ij})=(A^S_{ij})^THA^S_{ij}=H$ we have
    \begin{align}\label{chern class of Z}(A^S_{ij})^Tq(\hat{c}_{ij})=-(A^S_{ij})^TH(p_0(c^Z_{ij}))+G_i-(A^S_{ij})^TG_j\ \ \ \in \left(V_S^*/\Gamma_S^\vee\right)\Big|_{U_{ij}}. \end{align}
    The consequence of (\ref{chern class of Z}) is of course the equality of Chern classes $q(\hat{c}|_{\pi(S)})=-H(c_Z)\in \HH^2(\pi(S),\Gamma_S^\vee\cap H(V_S))$.
    
    Using (\ref{chern class of Z}) it is clear that we can define a gerbe $\cG'$ over $S$ such that  $p_Z^*\cG'\cong\delta(L)$. The gerbe $\cG'$ can be represented over the cover $\{S_i:=\pi_S^{-1}(U_i)\}$ of $S$ via the flat line bundles on the double overlaps
    \begin{align*}
        a'_{ij}(y,v:\lambda)&=exp\left(i\pi q(\hat{c}_{ij})(A^S_{ij}\lambda)-i\pi (G_i- (A^S_{ij})^TG_j)(\lambda)\right),\\
        A'_{ij}(y,v)&=i\pi q(\hat{c}_{ij}) A^S_{ij}dv+i\pi (G_i-(A^S_{ij})^TG_j)dv+ i\pi (dG_i-(A^S_{ij})^TdG_j)( v)+2\pi i df_{ij},
    \end{align*}
    and the gerbe product
    $$g'_{ijk}(y,v)=exp\left(2\pi i q(\hat{n}_{ijk})(v) +h_{ijk}(y)\right)$$
    We can change the representatives of $L_{ij}$ by
    $$\mu_{ij}(y,v)=exp\left(i \pi (G_i-(A^S_{ij})^TG_j)\cdot v  \right),$$
    which only changes the gerbe product by a basic function.
    
    Similarly, we can define $\delta(\hat{L})$ as the flat trivializable gerbe associated to the  $U(1)$-bundles $\{\hat{L}_i\}$ on the cover $q^{-1}(U_i)$. Calculating $\hat{L}_i\otimes \hat{L}_j^{-1}$ on the double overlaps and the gerbe product on the triple overlaps shows again that $\delta(\hat{L})\cong \hat{p}_Z^*\hat{\cG}'$ for a flat gerbe on $\hat{S}$.

    Since locally $\cP_i\cong L_i^{-1}\otimes \hat{L}_i$ Part \ref{higher rank 2} follows. Part \ref{higher rank 3} is an immediate consequence of Theorem \ref{last local} and Part \ref{higher rank 1} and \ref{higher rank 2}.  

    Finally, we have to show that $\cG'$ and $\hat{\cG}'$ are isomorphic to $\cG|_S$ and $\hat{\cG}|_{\hat{S}}$ up to a flat gerbe pulled back from the base. We show this for $\cG'$ and $\cG|_S$ as the T-dual side follows analogously.

    From (\ref{G gerbe product}) the gerbe product for $\cG|_S$ is given by
    \begin{align*} g_{ijk}(y,v)=exp\left(2\pi i q(\hat{n}_{ijk})\cdot v +2\pi i n_{ijk}\cdot b_i -2\pi i f_{ijk}\right)\end{align*}
    and the local line bundles are
    \begin{align*}
    a_{ij}^\cG(y,v;\lambda)&=exp(\pi i q(\hat{c}_{ij})(A^S_{ij}\lambda)),\\
         A_{ij}^\cG(y,v)&=i\pi q(\hat{c}_{ij})(A^S_{ij}dv)+i \pi \hat{c}_{ij}\cdot A_{ij}\cdot db_i+ 2\pi i \epsilon_{ij}.
    \end{align*}
    After changing the representatives of $\cG'$ by $\mu_{ij}$ it becomes clear that $\cG'\otimes \hat{\cG}|_S^{-1}$ is a gerbe pulled back from the base. Moreover, since $p_Z^*\cG'\otimes \hat{p}_Z^*(\hat{cG}')^{-1}\cong p_z^*\cG|_S\otimes \hat{p}_Z^*\hat{\cG}^{-1}$ Part \ref{higher rank 1} follows. 
\end{proof}

\chapter{Conclusions}
\paragraph{Semihomogeneous vector bundles.} As we have mentioned in Remark \ref{semihomog remark} the results of Theorem \ref{last local} are closely tied to the works of Matsushima \cite{matsushima}, Mukai \cite{mukaiSemihomog} and others on semihomogeneous vector bundles. On a complex torus $X=V/\Gamma$ one can define the Heisenberg group $G_H(\Gamma)$ corresponding to any Hermitian pairing $H\in NS(X)\otimes \dQ$, or the Heisenberg group $G_E(\Gamma)$ corresponding to the imaginary part $E$ of $H$. Holomorphic representations of these groups can be reformulated into factors of automorphy, therefore they define holomorphic vector bundles on $X$.

Matsushima showed that there is a bijection between holomorphic representations of $G_H(\Gamma)$ and $G_E(\Gamma)$ and further established that there is a distinguished irreducible representation for any $H$. The vector bundles corresponding to these irreducible representations are precisely the simple semihomogeneous vector bundles on $X$. 

In \cite{mukaiSemihomog} Mukai described the category of semihomogeneous vector bundles on a complex torus. For any element $\delta \in NS(X)\otimes \dQ$ there is a subcategory $S_\delta$ and the category of semihomogeneous vector bundles is given by $\oplus S_\delta$. Within $S_\delta$, up to tensoring with a flat bundle, there is a unique simple vector bundle, which corresponds to the distinguished irreducible representation of the Heisenberg group associated to $\delta$.

In Theorem \ref{last local} the bundles $E$ and $\hat{E}$ should be (at least fiber-wise) analogous to the simple semihomogeneous vector bundles corresponding to irreducible representations of the appropriate Heisenberg groups. Since the curvatures of the connections are not necessarily compatible with a complex structure, we should extend Matsushima's analysis to accommodate this setting.

The work of Matsushima, Mukai, and others on semihomogeneous vector bundles, along with the contributions of Chan, Leung, and Zhang on T-duality and also Theorem \ref{last local}, indicates that the Fourier-Mukai transform of semihomogeneous vector bundles - and similarly, the T-duality of projectively flat bundles with invariant curvatures - can be understood via linear data on the fibers. Such a description could be easily worked out similarly to the Fourier-Mukai transform of line bundles supported on affine subtori (Section \ref{FM transform section}). This might already exist in the literature but the author is currently unaware of any references.

\paragraph{T-duality of generalized branes with non-zero $H$-flux.} 
There are several ways one could extend the results of this thesis concerning generalized branes. The first direction could be to include non-zero $H$-flux and study the locally T-dualizable branes. Suppose that $(M,H)$ and $(\hat{M},\hat{H})$ are T-dual affine torus bundles in the sense of generalized geometry. Then, \cite[Theorem 4.1.]{B2} states that given connections $A$ and $\hat{A}$ on $M$ and $\hat{M}$ respectively, we have
$$\hat{p}^*\hat{H}-p^*H=d\langle \hat{p}^*\hat{A}\wedge p^*A\rangle$$
after possibly changing $H$ and $\hat{H}$ by $B$-field transforms. In this case, however, the connections $A$ and $\hat{A}$ are not flat and we cannot find coordinates for the horizontal distribution.

A generalized brane in $(M,H)$ is a pair $\cL=(S,F)$ such that $S\subset M$ is a submanifold and $F\in \Omega^2(S)$ is a two-form satisfying $dF=H|_S$. Once again we can restrict to branes supported on affine torus subbundles with invariant two-forms. Then, we can still reduce the generalized tangent bundle $\tau_\cL$ and T-dualize it using the T-duality map \ref{tdualitymap}. Then, we can construct an invariant subbundle of $T\hat{M}|_{\pi(S)}\oplus T^*\hat{M}|_{\pi(S)}$ which is still maximal isotropic and Courant integrable. The image of this subbundle under the anchor map is an integrable distribution.  We could also construct this distribution as before, by going through the distribution $\Delta \subset T(S\times_{\pi(S)}\hat{M})$ in the correspondence space. On the other hand, the integrability of $\Delta$ is not clear. Moreover, even if $\Delta$ is integrable, we cannot find coordinates which also span the horizontal distribution. This makes the calculations in coordinates less enlightening. Therefore, to generalize one must find a coordinate-free way to study the distribution $\Delta$ and the T-duals of $(S,F)$. 

\paragraph{T-duality on the twistor space.}
There have been attempts to define the Fourier-Mukai transform on the space of lambda connections \cite{deligne, simpson3}, which serve as an algebraic analogue to the twistor space construction \cite[Section (F)]{hklr} for the Higgs bundle moduli space. 

Given T-dual algebraic integrable systems $(M,0)$ and $(\hat{M},0)$ it is easy to see that their twistor spaces $(Z,0)$ and $(\hat{Z},0)$ are also T-dual in the sense of generalized geometry. Indeed, the twistor space $Z$ of $M$ as a differentiable manifold is just $M\times S^2$, the product of $M$ with the sphere. Therefore, T-duality between $Z$ and $\hat{Z}$ is the pullback of the T-duality relation between $M$ and $\hat{M}$ from the base $B$ to $B\times S^2$.

The semi-flat hyperk\"ahler structures on these spaces induce holomorphic symplectic structures on the twistor spaces. Then, the T-dual of the generalized complex structure (GCS) associated with the holomorphic structure on $Z$ is a mixed-type GCS on $\hat{Z}$. Over the poles of $S^2$, it remains a complex type, while at the other points of the sphere, it becomes a symplectic type GCS.  This is analogous to the  $BBB$-$BAA$ T-duality. 

Via the Atiyah-Ward correspondence, $BBB$-branes already have a well-established description on the twistor space. Then, $BAA$-branes could be studied on the twistor space as special branes of this mixed type GCS. There is potential to develop an analogous geometric framework for the $BAA$ side. 

In \cite{gloverSawon} Glover and Sawon gave a twistor space construction for generalized hyperk\"ahler manifolds. One could extend T-duality to this setting as well.

\paragraph{Applications to the Higgs bundle moduli space.}
One could use the results established in this thesis to study branes on the Higgs bundle moduli spaces $\cM(r,d)$ as well (cf. Example \ref{higgs bundes 1}). The smooth locus of these moduli spaces carries a hyperk\"ahler metric, which we call the Hitchin metric, in which the fibration $\cM(r,d)\ra \cA$ is an algebraic integrable system. The moduli spaces $\cM(r,0)$, of rank $r$ degree zero Higgs bundles, admit a holomorphic Lagrangian section therefore on these components we can construct a semi-flat hyperk\"ahler structure. Then, both the semi-flat and the Hitchin metric induce the same holomorphic symplectic structure on $\cM(r,0)$ and therefore the $BAA$-branes of the semi-flat metric coincide with the $BAA$-branes of the Hitchin metric.

One could then T-dualize the semi-flat $BAA$-branes and study the resulting semi-flat $BBB$-branes. These are going to be a hyperholomorphic bundles with respect to the T-dual semi-flat metric, but not the Hitchin metric. One could study the deformations of the semi-flat hyperholomorphic connections to find a mirror-symmetry relation between $BAA$ and $BBB$-branes of the Hitchin metric. This is also the idea underlying the conjectures of Gaiotto, Moore and Nietzke \cite{GMN1,GMN2}.


\pagestyle{plain}


\begin{thebibliography}{10}

\bibitem{AF}
D.~Arinkin and R.~M. Fedorov.
\newblock Partial {F}ourier--{M}ukai transform for integrable systems with applications to {H}itchin fibration.
\newblock {\em Duke Math. J.}, 165:2991--3042, 2015.

\bibitem{AP}
D.~Arinkin and A.~Polishchuk.
\newblock Fukaya category and {Fourier} transform.
\newblock {\em arXiv: Algebraic Geometry}, 1998.

\bibitem{atiyahWard}
M.~F. Atiyah and R.~S. Ward.
\newblock Instantons and algebraic geometry.
\newblock {\em Communications in Mathematical Physics}, 55:117--124, 1977.

\bibitem{B1}
D.~Baraglia.
\newblock Topological {T-duality} for general circle bundles.
\newblock {\em Pure Appl. Math. Q.}, 10 No. 3:367--438, 2014.

\bibitem{B2}
D.~Baraglia.
\newblock Topological T-duality for torus bundles with monodromy.
\newblock {\em Reviews in Mathematical Physics}, 27(03):1550008, 2015.

\bibitem{BNR}
A.~Beauville, M.~S. Narasimhan, and S.~Ramanan.
\newblock Spectral curves and the generalised theta divisor.
\newblock {\em Journal f{\"u}r die reine und angewandte Mathematik (Crelles Journal)}, 1989:169 -- 179, 1989.

\bibitem{BenBassat}
O.~Ben-Bassat.
\newblock Mirror symmetry and generalized complex manifolds: {Part} {I.} {T}he transform on vector bundles, spinors, and branes.
\newblock {\em Journal of Geometry and Physics}, 56:533--558, 2006.

\bibitem{BLpaper}
C.~Birkenhake and H.~Lange.
\newblock The dual polarization of an abelian variety.
\newblock {\em Archiv der Mathematik}, 73:380--389, 1999.

\bibitem{BL}
C.~Birkenhake and H.~Lange.
\newblock {\em Complex {abelian} varieties}.
\newblock Springer Berlin, Heidelberg, 2004.

\bibitem{BG}
F.~Bischoff and M.~Gualtieri.
\newblock Brane quantization of toric Poisson varieties.
\newblock {\em Communications in Mathematical Physics}, 391:357 -- 400, 2021.

\bibitem{BDV}
F.~A. Bogomolov, R.~N. D{\'e}ev, and M.~Verbitsky.
\newblock Sections of {Lagrangian} fibrations on holomorphically symplectic manifolds and degenerate twistorial deformations.
\newblock {\em Advances in Mathematics}, 2020.

\bibitem{BEM}
P.~Bouwknegt, J.~Evslin, and V.~Mathai.
\newblock T-duality: Topology change from {H}-flux.
\newblock {\em Communications in Mathematical Physics}, 249:383--415, 2003.

\bibitem{BMP1}
U.~Bruzzo, G.~Marelli, and F.~Pioli.
\newblock A {Fourier} transform for sheaves on real tori: {Part} {I}. {The} equivalence {Sky($T$)}{$\cong$}{Loc($\hat{T}$)}.
\newblock {\em Journal of Geometry and Physics}, 39:174--182, 2001.

\bibitem{BMP2}
U.~Bruzzo, G.~Marelli, and F.~Pioli.
\newblock A {Fourier} transform for sheaves on real tori: {Part} {II}. {Relative} theory.
\newblock {\em Journal of Geometry and Physics}, 41:312--329, 2002.

\bibitem{BCG}
H.~Bursztyn, G.~R. Cavalcanti, and M.~Gualtieri.
\newblock Reduction of {Courant} algebroids and generalized complex structures.
\newblock {\em Advances in Mathematics}, 211:726--765, 2005.

\bibitem{cavalcantiTdual}
G.~R. Cavalcanti and M.~Gualtieri.
\newblock Generalized complex geometry and {T}-duality.
\newblock {\em A Celebration of the Mathematical Legacy of Raoul Bott (CRM Proceedings \& Lecture Notes)}, 50:341--366, 2010.

\bibitem{CLZ}
K.~Chan, N.~C. Leung, and Y.~Zhang.
\newblock {SYZ} transformation for coisotropic {A}-branes.
\newblock {\em Adv. Theor. Math. Phys.}, 22, Number 3:509--564, 2018.

\bibitem{CJ}
J.~A. Chen and Z.~J. Jiang.
\newblock Positivity in varieties of maximal albanese dimension.
\newblock {\em Crelle's Journal}, 2018:225--253, 2018.

\bibitem{deligne}
P.~Deligne.
\newblock Letters to {C}arlos {S}impson.
\newblock 1995.

\bibitem{DW}
R.~Y. Donagi and E.~Witten.
\newblock Supersymmetric {Yang}-{Mills} theory and integrable systems.
\newblock {\em Nuclear Physics}, 460:299--334, 1995.

\bibitem{Dui}
J.~J. Duistermaat.
\newblock On global action‐angle coordinates.
\newblock {\em Communications on Pure and Applied Mathematics}, 33:687--706, 1980.

\bibitem{feix}
B.~Feix.
\newblock Hyperk{\"a}hler metrics on cotangent bundles.
\newblock {\em Crelle's Journal}, 2001.

\bibitem{francoHanson}
E.~Franco and R.~Hanson.
\newblock The {Dirac-Higgs} complex and categorification of {(BBB)}-branes.
\newblock {\em International Mathematics Research Notices}, 2024:12919–12953, 2024.

\bibitem{freed}
D.~S. Freed.
\newblock Special k{\"a}hler manifolds.
\newblock {\em Communications in Mathematical Physics}, 203:31--52, 1997.

\bibitem{GMN1}
D.~Gaiotto, G.~W. Moore, and A.~Neitzke.
\newblock Four-dimensional wall-crossing via three-dimensional field theory.
\newblock {\em Comm. Math. Phys.}, 299:163--224, 2010.

\bibitem{GMN2}
D.~Gaiotto, G.~W. Moore, and A.~Neitzke.
\newblock Wall-crossing, hitchin systems, and the WKB approximation.
\newblock {\em Adv. Math.}, 234:239--403, 2013.

\bibitem{GW}
D.~Gaiotto and E.~Witten.
\newblock Probing quantization via branes.
\newblock {\em Surveys in Differential Geometry}, 2021.

\bibitem{twisted}
S.~J. Gates, C.~M. Hull, and M.~Ro{\v{c}}ek.
\newblock Twisted multiplets and new supersymmetric non-linear $\sigma$-models.
\newblock {\em Nuclear Physics}, 248:157--186, 1984.

\bibitem{GJK}
J.~F. Glazebrook, M.~Jardim, and F.~W. Kamber.
\newblock A {Fourier–Mukai} transform for real torus bundles.
\newblock {\em Journal of Geometry and Physics}, 50:360–392, 2004.

\bibitem{gloverSawon}
R.~Glover and J.~Sawon.
\newblock Generalized twistor spaces for hyperk{\"a}hler manifolds.
\newblock {\em Journal of the London Mathematical Society}, 91, 2013.

\bibitem{MGthesis}
M.~Gualtieri.
\newblock Generalized complex geometry.
\newblock {\em DPhil University of Oxford}, 2003.

\bibitem{MGpaper}
M.~Gualtieri.
\newblock Generalized complex geometry.
\newblock {\em Ann. of Math.}, 174:75--123, 2011.

\bibitem{hauselToy}
T.~Hausel.
\newblock Compactification of moduli of Higgs bundles.
\newblock {\em J. Reine Angew. Math.}, 503:169--192, 1998.

\bibitem{hauselUnpub}
T.~Hausel.
\newblock Hitchin map as spectrum of equivariant cohomology and Kirillov algebras.
\newblock {\em In preparation}, 2024.

\bibitem{hauselHitchin}
T.~Hausel and N.~J. Hitchin.
\newblock Very stable Higgs bundles, equivariant multiplicity and mirror symmetry.
\newblock {\em Inventiones mathematicae}, 228:893 -- 989, 2021.

\bibitem{herbst}
M.~Herbst.
\newblock On higher rank coisotropic A-branes.
\newblock {\em Journal of Geometry and Physics}, 62:156--169, 2010.

\bibitem{hhklrlemma}
N.~J. Hitchin.
\newblock Monopoles, minimal surfaces and algebraic curves.
\newblock {\em S{\'e}minaire de math{\'e}matiques supe{\'e}rieures}, 105, 1987.

\bibitem{hitchinHiggs}
N.~J. Hitchin.
\newblock The self-duality equations on a Riemann surface.
\newblock {\em Proceedings of The London Mathematical Society}, 55:59--126, 1987.

\bibitem{hitchinSection}
N.~J. Hitchin.
\newblock Lie-groups and {Teichmueller} space.
\newblock {\em Topology}, 31:449--473, 1992.

\bibitem{hitchinGerbes}
N.~J. Hitchin.
\newblock Lectures on special Lagrangian submanifolds.
\newblock {\em arXiv: Differential Geometry}, 1999.

\bibitem{hitchinCplxLag}
N.~J. Hitchin.
\newblock The moduli space of complex Lagrangian submanifolds.
\newblock {\em Surveys in differential geometry}, 7:327--345, 1999.

\bibitem{hklr}
N.~J. Hitchin, A.~Karlhede, U.~Lindstr{\"o}m, and M.~Ro{\v{c}}ek.
\newblock Hyperk{\"a}hler metrics and supersymmetry.
\newblock {\em Communications in Mathematical Physics}, 108:535--589, 1987.

\bibitem{huaReiner}
L.~Hua and I.~Reiner.
\newblock On the generators of the symplectic modular group.
\newblock {\em Transactions of the American Mathematical Society}, 65:415--426, 1949.

\bibitem{pqGeo}
C.~Hull and U.~Lindstr{\"o}m.
\newblock The generalised complex geometry of (p, q) {Hermitian} geometries.
\newblock {\em Commun. Math. Phys.}, 375:479–494, 2020.

\bibitem{iena}
O.~Iena.
\newblock Vector bundles on elliptic curves and factors of automorphy.
\newblock {\em Rendiconti dell'Istituto di Matematica dell'Universita di Trieste}, 43, 2011.

\bibitem{kamenovaVerbitsky}
L.~Kamenova and M.~Verbitsky.
\newblock Holomorphic {Lagrangian} subvarieties in holomorphic symplectic manifolds with {Lagrangian} fibrations and special k{\"a}hler geometry.
\newblock {\em European Journal of Mathematics}, 8:514 -- 522, 2019.

\bibitem{Kap1}
A.~Kapustin.
\newblock Topological strings on noncommutative manifolds.
\newblock {\em International Journal of Geometric Methods in Modern Physics}, 1:49--81, 2003.

\bibitem{kapustinOrlov}
A.~Kapustin and D.~Orlov.
\newblock Remarks on {A}-branes, mirror symmetry, and the {Fukaya} category.
\newblock {\em Journal of Geometry and Physics}, 48:84--99, 2001.

\bibitem{KW}
A.~Kapustin and E.~Witten.
\newblock Electric-magnetic duality and the geometric {Langlands} program.
\newblock {\em Communications in Number Theory and Physics}, 1:1--236, 2006.

\bibitem{kobayashi}
S.~Kobayashi.
\newblock {\em Differential geometry of complex vector bundles}.
\newblock Iwanami Shoten Publishers and Princeton University Press, 1987.

\bibitem{N=2susy}
U.~Lindstr{\"o}m and M.~Zabzine.
\newblock N=2 boundary conditions for non-linear sigma models and {Landau}-{Ginzburg} models.
\newblock {\em Journal of High Energy Physics}, 2, 2003.

\bibitem{matsushima}
Y.~Matsushima.
\newblock Heisenberg groups and holomorphic vector bundles over a complex torus.
\newblock {\em Nagoya Mathematical Journal}, 61:161 -- 195, 1976.

\bibitem{mukaiSemihomog}
S.~Mukai.
\newblock Semi-homogeneous vector bundles on an abelian variety.
\newblock {\em Journal of Mathematics of Kyoto University}, 18:239--272, 1978.

\bibitem{mukai}
S.~Mukai.
\newblock Duality between {D(X)} and with its application to {P}icard sheaves.
\newblock {\em Nagoya Mathematical Journal}, 81:153 -- 175, 1981.

\bibitem{nitsure}
N.~Nitsure.
\newblock Moduli space of semistable pairs on a curve.
\newblock {\em Proceedings of The London Mathematical Society}, pages 275--300, 1991.

\bibitem{sawon}
J.~Sawon.
\newblock Fourier-{M}ukai transforms, mirror symmetry, and generalized {K3} surfaces.
\newblock {\em arXiv: Differential Geometry}, 2012.

\bibitem{schnell}
C.~Schnell.
\newblock The Fourier–Mukai transform made easy.
\newblock {\em Pure and Applied Mathematics Quarterly}, 2019.

\bibitem{simpson1}
C.~Simpson.
\newblock Moduli of representations of the fundamental group of a smooth projective variety {I}.
\newblock {\em Publications Math{\'e}matiques de l'Institut des Hautes {\'E}tudes Scientifiques}, 79:47--129, 1994.

\bibitem{simpson2}
C.~Simpson.
\newblock Moduli of representations of the fundamental group of a smooth projective variety. {II}.
\newblock {\em Publications Math{\'e}matiques de l'Institut des Hautes {\'E}tudes Scientifiques}, 80:5--79, 1994.

\bibitem{simpson3}
C.~Simpson.
\newblock The Hodge filtration on nonabelian cohomology.
\newblock {\em arXiv:alg-geom/9604005v1}, 1996.

\bibitem{stevenson}
D.~Stevenson.
\newblock The geometry of bundle gerbes.
\newblock {\em arXiv: Differential Geometry}, 2000.

\bibitem{SYZ}
A.~Strominger, S.-T. Yau, and E.~Zaslow.
\newblock Mirror symmetry is {T}-duality.
\newblock {\em Nuclear Physics}, 479:243--259, 1996.

\bibitem{wittenTft}
E.~Witten.
\newblock Mirror manifolds and topological field theory.
\newblock {\em AMS/IP Stud.Adv.Math.}, 9:121--160, 1998.

\bibitem{N=2susygcg}
M.~Zabzine.
\newblock Geometry of D-branes for general {N}=(2,2) sigma models.
\newblock {\em Letters in Mathematical Physics}, 70:211--221, 2004.

\end{thebibliography}



\end{document}